\documentclass[reqno]{amsart}
\usepackage[margin=1in]{geometry}

\usepackage[english]{babel}

\usepackage{amsmath,amssymb,amsthm, comment, mathtools}
\usepackage{mathrsfs}

\usepackage{hyperref}

\numberwithin{equation}{section}

\newtheorem{theorem}{Theorem}[section]
\newtheorem{definition}{Definition}
\newtheorem{prop}[theorem]{Proposition}
\newtheorem{cor}[theorem]{Corollary}
\newtheorem{lemma}[theorem]{Lemma}

\theoremstyle{remark}

\theoremstyle{definition}

\newcommand{\tC}{\widetilde{\chi}}

\newcommand{\tS}{\widetilde{S}}
\newcommand{\Ee}{\mathscr{E}}
\newcommand{\Cc}{\mathscr{C}}
\newcommand{\Eee}{\widetilde{\Ee}}
\newcommand{\Eev}{\E_\ve}
\newcommand{\tV}{\widetilde{V}}
\newcommand{\oE}{\overline{E}}

\newcommand{\w}{\widehat}

\newcommand{\R}{\mathbb{R}}
\newcommand{\Rr}{\mathcal{R}}
\newcommand{\E}{\mathcal{E}}
\newcommand{\W}{\mathcal{W}}
\newcommand{\N}{\mathbb{N}}
\newcommand{\X}{\mathcal{X}}
\renewcommand{\H}{\mathcal{H}}

\newcommand{\T}{{\mathcal{T}^{}}}
\newcommand{\A}{\mathcal{A}}
\newcommand{\C}{\mathcal{C}}
\newcommand{\oC}{\overline{\mathcal{C}}}
\newcommand{\K}{\mathcal{K}}
\newcommand{\coord}{\mathcal{D}} %coordinate vector fields
\DeclareMathOperator{\tr}{tr}

\newcommand{\D}{\mathcal{D}}
\newcommand{\tD}{\widetilde{\D}}
\newcommand{\hrho}{\widehat{\rho}}
\newcommand{\hD}{\widehat{\D}}
\newcommand{\hphi}{\widehat{\phi}}

\newcommand{\U}{\mathcal{V}}
\newcommand{\V}{\mathcal{V}}
\newcommand{\pa}{\partial}
\newcommand{\p}{\partial}
\newcommand{\ve}{\varepsilon}

\newcommand{\fd}{\langle \pa_{\theta{}_{\!}} \rangle}
\newcommand{\fdh}{\fd^{\!1{}_{\!}/2}}
\newcommand{\fdm}{\fd_\mu}
\newcommand{\fdhm}{\fd^{\!1_{\!}/2}_\mu}
\newcommand{\fdhn}{\fd^{1/2}_\nu}

\newcommand{\Dve}{\widetilde{\Delta}}
\newcommand{\Dht}{\widehat{\Delta}}

\newcommand{\Odo}{\Omega^{d_0}}
\newcommand{\I}{{}_{\!}I} %shortened I and II for aesthetic purposes
\newcommand{\II}{{}_{\!}I{\!}I}
\newcommand{\Dveu}{\Dve_{\I}}
\newcommand{\Dvew}{\Dve_{\II}}

\newcommand{\pave}{\widetilde{\pa}}
\newcommand{\paht}{\widehat{\pa}}
\newcommand{\xve}{\widetilde{x}}
\newcommand{\xveu}{\,\widetilde{\!x}_{{}_{\!}I}}
\newcommand{\xvew}{\,\widetilde{\!x}_{{}_{\!}I\!I}}
\newcommand{\Vu}{V_{{}_{\!}\I}}
\newcommand{\Vw}{V_{\II}}
\newcommand{\FD}{\mathfrak{D}}
\newcommand{\Vht}{\widehat{V}}
\newcommand{\Vhtu}{\widehat{V}_{\I}}
\newcommand{\Vhtw}{\widehat{V}_{\II}}
\newcommand{\paveu}{\widetilde{\pa}_{\I}}
\newcommand{\pavew}{\widetilde{\pa}_{\II}}
\newcommand{\wpa}{\widetilde{\pa}}
\newcommand{\pahtu}{\widehat{\pa}_{\I}}
\newcommand{\pahtw}{\widehat{\pa}_{\II}}
\newcommand{\uht}{\widehat{A}}
\newcommand{\q}{\quad}

\newcommand{\wta}{\alpha}

\newcommand{\sm}{J_\ve} %The smoothing operator

\newcommand{\ssm}{S_\ve} %The square of the smoothing operator
\newcommand{\Vve}{\ssm V}

\newcommand{\gve}{\widetilde{g}}
\newcommand{\kve}{\widetilde{\kappa}}
\newcommand{\kveu}{\kve_{I}}
\newcommand{\kvew}{\kve_{\II}}
\newcommand{\khtu}{\widehat{\kappa}_{I}}
\newcommand{\khtw}{\widehat{\kappa}_{\II}}
\newcommand{\gveu}{\gve_{\I}}
\newcommand{\gvew}{\gve_{\II}}
\newcommand{\G}{\mathcal{G}}
\newcommand{\F}{F}

\renewcommand{\P}{\mathcal{P}}

\newcommand{\xht}{\widehat{x}}
\newcommand{\xhtu}{\widehat{x}_{\I}}
\newcommand{\xhtw}{\widehat{x}_{\II}}

\newcommand{\zhtu}{\widehat{z}_{\I}}

\newcommand{\m}{\,\,}

\newcommand{\esssup}{\textrm{ess sup}}
\newcommand{\eprime}{\sigma} %this is what will eventually be e'(h) or e'(h^{\nu-1}).
\newcommand{\eprimee}{\sigma} %this is what will eventually be (e'(h))^{-1} or e'(h^{\nu-1}).

\newcommand{\dm}{d_{m}} %used in construction of strong solutions
\newcommand{\dd}{\dot{d}_{m}} %used in construction of strong solutions
\newcommand{\varphim}{\varphi^{m}} %used in construction of strong solutions
\let\div\relax %The default for ``div'' is the division symbol

\DeclareMathOperator{\curl}{curl}
\DeclareMathOperator{\curlu}{curl_{I}}
\DeclareMathOperator{\curlw}{curl_{II}}
\DeclareMathOperator{\div}{div}
\DeclareMathOperator{\divu}{div_{I}}
\DeclareMathOperator{\divw}{div_{II}}
\DeclareMathOperator{\defo}{def}

\DeclareMathOperator{\tsum}{{\textstyle{\sum}}}
\DeclareMathOperator{\tint}{{\textstyle{\int}}}

\newcommand*{\bigtwo}[1]{\vcenter{\hbox{\scalebox{1.4}{\ensuremath#1}}}}

\title[Compressible, Self-Gravitating Liquid with Free Surface Boundary]{Local Well-Posedness for the Motion of a Compressible, Self-Gravitating Liquid with Free Surface Boundary}

\author{Daniel Ginsberg}\address[D.G.]{Johns Hopkins University, Department of Mathematics, 3400 N.\@ Charles St., Baltimore, MD 21218, USA}
\email{dginsbe5@math.jhu.edu}

\author{Hans Lindblad} \address[H.L.]{Johns Hopkins University, Department of Mathematics, 3400 N.\@ Charles St., Baltimore, MD 21218, USA}
\email{lindblad@math.jhu.edu}

\author{Chenyun Luo}\address[C.L.]{Vanderbilt University, Department of Mathematics, 1326 Stevenson Ctr. Ln., Nashville, TN 37212, USA}
\email{chenyun.luo@vanderbilt.edu}

\thanks{D.G. and H.L. were supported in part by NSF Grant DMS-1500925}

\begin{document}

\mathtoolsset{showonlyrefs=true}
\nocite{*}
\maketitle

\vspace{-0.25in}
\begin{abstract} We establish the local well-posedness for the free boundary
  problem for the compressible
  Euler equations describing the motion of liquid under the influence of
  Newtonian self-gravity.
  We do this by solving a tangentially-smoothed version of Euler's equations in
  Lagrangian coordinates which satisfies uniform
  energy estimates as the smoothing parameter goes to zero.
  The main technical tools are delicate energy estimates and optimal elliptic estimates in terms of boundary regularity, for the Dirichlet
  problem and Green's function.
\end{abstract}

\tableofcontents

\vspace{-0.45in}

\section{Introduction}
\label{intsection}
The motion of a barotropic, self-gravitating fluid occupying a region
$\D = \cup_{0 \leq t \leq T} \{ t \} \times \D_t,
\D_t \subset \R^3$, of space time, is described by the velocity $V = (V^1, V^2, V^3)$,
a non-negative function
$\rho$ known as the density, and an equation of state
$p = p(\rho)$ which specifies the pressure $p$ as a function of $\rho$,
and which is assumed to be non-negative and strictly increasing.
The equations of motion are then given by
Euler's equations:
\begin{equation}
 \rho (\pa_t + V^k\pa_k) v_i + \pa_i p + \rho \pa_i \phi = 0,
 \textrm{ for } i = 1,2,3 \textrm{ in } \D,
 \label{mom}
\end{equation}
and the continuity equation:
\begin{equation}
 (\pa_t + V^k\pa_k) \rho + \rho \div V = 0 \textrm{ in } \D,
 \label{mass}
\end{equation}
where repeated upper and lower indices are summed over, $\pa_i \!=\! {\pa}{}_{\!}/{}_{\!}{\pa x^i}\!$,
$v_i \!=\!\! \delta_{ij\!} V^j\!\!$ and $\div \!V \!\!=\! \pa_{i\!} V^i\!\!$.
Here, with $\chi_{\D_t}$ the characteristic function of $\D_t$,
the Newtonian gravity potential $\phi$ is defined to be the  solution to:
\begin{equation}
 \Delta \phi = -\rho \chi_{\D_t},\qquad  \textrm{ in } \R^3,\quad\text{with}\quad
 {\lim}_{\,|x| \to \infty\,} \phi(x) = 0,
 \quad\text{i.e.}\quad
 \phi(t,x) =
  \frac{1}{4\pi} \int_{\D_t}
  \frac{\rho(t,x')\,dx'}{|x- x'|}\, .
 \label{eq:gravityintro}
\end{equation}

Particles on the boundary $\pa \D_t$ move with the velocity
of the fluid,
and if the body moves
in vacuum then the pressure vanishes outside of $\D$, so
we also require the boundary conditions:
\begin{align}
 (\pa_t + V^k\pa_k)\big|_{\pa \D} &\in T(\pa\D),\label{freebdy}\\
 p = 0,\qquad &\textrm{ on } \pa \D_t,
 \label{vaccuum}
\end{align}
where $\pa \D\!=\!\cup_{0\leq t\leq T}\pa \D_t$ is the space time boundary.
Since the equation of state $p(\rho)$ is strictly increasing, we can alternatively think of the
density as a function of the pressure and then
\eqref{vaccuum} implies that $\rho|_{\pa \D_t}\! =\! \overline{\rho}$
for some constant $\overline{\rho}$, with $p(\overline{\rho})\!=\!0$. We consider the case $\overline{\rho}\! > \!0$,
in which case the fluid is said to be a liquid.

Given an open set $\D_0\subset \R^3$ and a diffeomorphism $x_0:\Omega\to \D_0$ from the unit ball $\Omega$, a function $\rho_0$ which is strictly positive on $\D_0$
so that $p(\rho_0) = 0$ on $\pa \D_0$, and a vector field $V_0$ on
$\D_0$, the \emph{free boundary problem for the compressible
Euler equations in a bounded domain} is to find a domain $\D = \cup_{0\leq t \leq T}
\{t\} \times \D_t$, a vector field $V$ and a function $\rho$ satisfying
\eqref{mom}-\eqref{vaccuum}
 as well as the initial conditions:
\begin{equation}\label{eq:initialconditionseulercoord}
 \{ x : (0, x) \in \D\} = \D_0,\qquad\text{and}\qquad
 V(0, x) = V_0(x),\qquad  \rho(0, x) = \rho_0(x), \quad \textrm{ in } \D_0.
\end{equation}

Since $\rho$ is constant on the boundary it follows that
$(\pa_t \!+\! V^k\pa_k) \rho\!=\!0$  on the boundary so by \eqref{mass}
at $t\! =\! 0$ we must have that $\div\! V_0\! =\! 0$ on the boundary.
Similarly $(\pa_t\! +\! V^k\pa_k)^{2 \!}\rho\!=\!0$ on the boundary, which by \eqref{mass}
implies that $(\pa_t \!+ \!V^k\pa_k)\div \!V\!\!=\!0$ on the boundary, but
taking the divergence of \eqref{mom} gives an expression for
$(\pa_t \!+\! V^k\pa_k)\div \!V\!$ in terms of space derivatives of
$V\!$ and $\rho$, and this expression must vanish on the boundary.
We say that the initial data $V_0, \rho_0$ satisfy
the {\it compatibility condition of order $m$} if there are formal power series
in $t$,  $\hat{\rho}(t,x), \hat{V}\!(t,x),\hat{\phi}(t,x)$ that satisfy \eqref{mom}-\eqref{eq:gravityintro}
 with $\hat{V}\!(0,x) \!=\! V_{{}_{\!}0\,}\!(x)$,  $\hat{\rho}(0,x)\! =\! \rho_0(x)$ and:
\begin{equation}
  (\pa_t + \hat{V}^k\pa_k)^j(\hat{\rho} - \overline{\rho}) \in H^1_0(\D_0),
  \quad\text{for }j = 0,\dots, m.
  \label{int:compat2}
\end{equation}

In addition, this problem is ill-posed (see \cite{E1}) unless the physical (Taylor) sign condition holds:
\begin{equation}
 -\nabla_{\!N} p \geq \delta > 0, \textrm{ on } \pa \D_t,
 \qquad\textrm{ where } \nabla_{\!N} = N^i \nabla_i.
 \label{int:tsc}
\end{equation}
Our apriori bounds hold in general but for the existence of a solution satisfying the compatibility conditions we need to assume that
we are close to the incompressible case, i.e. $\rho(p)$ is close to the constant function:
\begin{equation}\label{eq:closetoincompressible}
0<\rho^\prime(p)\leq \delta_0,\quad \text{and}\quad
|\rho^{(k)}(p)| \leq \delta_0/E_0^{k-1},\quad k=2,\dots ,r,\quad\text{where} \quad
E_0= ||V_{0_{\!}}||_{H^r}\!+\! ||\rho_{0_{\!}}||_{H^r}.
\end{equation}
Our main result is:
\begin{theorem}
  \label{mainthm} There is a constant $\delta_0\!>\! 0$ such that if $\rho(p)$ is a smooth function satisfying $\rho(0)\!>\!0$ and \eqref{eq:closetoincompressible} for $k\!=\!1$ the following hold.
 Suppose that there is a diffeomorphism $x_0\!{}_{\!}:\!\Omega\!\to \!\D_{0\!}$ in $H^{r+1\!}(\Omega)$ and that $V_{{}_{\!}0}, \rho_0\! \in\! H^{r\!}({}_{\!}\D_{0\!}){}_{\!}$
  satisfy the compatibility conditions to order $r -{}_{\!}1\!\geq \!7{}_{\!}$,
  and that \eqref{int:tsc} holds at $t \!=\! 0$ and that  \eqref{eq:closetoincompressible}
  hold for $k\!=\!2,\!...,{}_{\!}r{}_{\!}$.
  Then there is  $T\!{}_{\!}=\! T(||x_{0_{\!}}||_{H^r}\!,\!||V_{0_{\!}}||_{H^r}\!,\! ||\rho_{0_{\!}}||_{H^r}\!)\!
  > \!0$ so that \eqref{mom}-\eqref{vaccuum}
  has a solution $(V\!{}{\!},{}_{\!} \rho,{\!}\D)\!$ with diffeomorphisms  $x(t,\cdot)\!:\!\Omega \!\to\! \D_{t_{\!}}$ in $ H^{r\!}(\Omega)$ and
  $V\!(t,_{\!}\cdot_{{}_{\!}}) \in{}_{\!}
  H^{(r-1,1{}_{\!}/2)}(_{{}_{\!}}\D_{t_{\!}})_{\!}$,
  $\rho(t,_{\!}\cdot_{{}_{\!}})\!\in{}_{\!}
  H^{r\!}(_{{}_{\!}}\D_{t_{\!}})_{\!}$,
  for $0 \!\leq \! t \!\leq\! T\!$.\!
\end{theorem}

The space $H^{(r\!-\!1,1{}_{\!}/2)}$ defined in \eqref{sobspacedef} controls
$r_{\!}-\!1$ full derivatives and half a tangential derivative. While
 $\rho$ is as regular at later times as initially,
%at $t \!= \!0$,
we need to assume  more regularity of $V\!$ and $x$ initially than we get   later. However we prove energy estimates
controlling $x(t,\cdot)\!\in\! H^r(\Omega)$ and $V(t,\cdot)\!\in\! H^{(r-1,1{}_{\!}/2)}(\D_t)$, $\rho(t,\cdot) \!\in\! H^r(\D_t)$ for $t \!\leq \! T$
in terms of these quantities at $t\! =\! 0$.
In the incompressible case, where the compatibility conditions are satisfied, one can regularize the initial data and use this to prove existence in the energy space, see \cite{Nordgren2008}.

Related problems without self gravity have previously been solved using different methods.
In \cite{W99}, Wu proved local well-posedness for
 the incompressible ($\div V \!\!=\! 0$) irrotational ($\curl V\! \!= \!0$)
case,  using complex analysis and spinors.
Lindblad \cite{L05a, Lindblad2005} used a Nash-Moser iteration scheme
to solve the case with $\curl V\!\!\not= \!0$ without self-gravity
in  the incompressible case and the case of a compressible liquid. Later,
Coutand-Shkoller \cite{CS07} and Coutand-Hole-Shkoller \cite{CHS13},
 were able to use tangential smoothing
together with surface tension and
artificial viscosity, and elliptic estimates of the type proven in \cite{Cheng2016}
to avoid the use of a Nash-Moser iteration.

Lindblad-Nordgren \cite{Lindblad2008} proved apriori bounds for an incompressible liquid with self gravity in the two dimensional case.
Nordgren \cite{Nordgren2008} proved local existence for an incompressible liquid with self gravity in three dimensions.
His proof built on the approach of \cite{CS07} but he was able to avoid
the need for artificial viscosity and surface tension using elliptic estimates from \cite{Lindblad2008}. Here we give a complete proof of the local well-posedness for a compressible liquid with self gravity, building on ideas from
\cite{L05a,Lindblad2005,CS07,Lindblad2008,Nordgren2008,CHS13,Lindblad2016}.
In particular we use tangential smoothing but we avoid any extra smoothing by surface tension or artificial viscosity, by using improved elliptic estimates and estimates for a wave equation on a bounded domain. We also include existence proofs for the smoothed Euler and wave equations that are not written in detail elsewhere.

%We include existence proofs for the smoothed and linear problems that are not written %in detail elsewhere.

\subsection{The setup for the proof}
We fix $\Omega$ to be the unit ball in $\R^3$
and a diffeomorphism $x_0 : \Omega \to \D_0$.
We introduce Lagrangian coordinates, see Section \ref{lagsec}, so the boundary is fixed:
\begin{equation}
\frac{d x}{dt}=V(t,x),\qquad x(0,y)=x_0(y),\qquad y\in \Omega.
\end{equation}
We express Euler's equations in these coordinates, using the enthalpy, $h^{\prime\!}(\rho) \!=\! p^{\prime\!}(\rho) /\rho$, $h(\overline{\rho})\!=\!0$,
\begin{equation}\label{eq:eulerlagrangiancoordintro}
D_t V^i=-\delta^{ij}(\partial_j h+\pa_j \phi),\quad\text{in}\quad [0,T]\times\Omega,\quad \text{where}\quad
D_t=\partial_t\big|_{y =const},\qquad \partial_i =\frac{\pa y^a}{\pa x^i} \frac{\pa}{\pa y^a}.
\end{equation}
 If we take the material derivative $D_t$ of the continuity equation
$ D_t\rho=-\rho\div V$ and the divergence of Euler's equations \eqref{eq:eulerlagrangiancoordintro}, using \eqref{eq:gravityintro}, we obtain, with  $e(h) = \log \rho(h)$,
\begin{equation}\label{eq:waveequationintro}
 D_t^2 e(h) - \Delta h = (\pa_i V^j)(\pa_j V^i) -\rho(h),\quad \textrm{ in }
 [0, T]\! \times\! \Omega, \quad\text{with}\quad
 h|_{[0,T]\times\pa \Omega}=0, \quad\text{where}\quad {\Delta}\!=\delta^{ij}\pa_i\pa_j.
\end{equation}
Here $\phi$ is given by
\vspace{-0.1in}
\begin{equation}
  \qquad\quad\phi(t,y) =  \frac{1}{4\pi} \int_{\Omega}
  \frac{ \rho(t, y'){\kappa}(t, y')\,dy'}{|x(t,y)- x(t,y')|}\, ,
 \qquad\text{where}\qquad {\kappa}=|\det{(\pa {x}/\pa y)}|.
\end{equation}

It is possible to obtain  apriori energy bounds for the system
\eqref{eq:eulerlagrangiancoordintro}-\eqref{eq:waveequationintro}
but it is difficult to come up with an iteration
scheme that doesn't lose regularity. We will first smooth out the equations.
Let $S_\varepsilon \!=\!T_\varepsilon^* T_\varepsilon$ be a regularization in directions tangential to the boundary that is self adjoint, see Section
\ref{smoothingsec}.
Given a velocity vector field $V\!$, we define the tangentially regularized velocity
and the regularized coordinates by
\begin{equation}\label{eq:eulerlagrangiancoordxsmoothed}
\widetilde{V}=S_\varepsilon V,\qquad
\frac{d \widetilde{x}}{dt}=\widetilde{V}(t,y),\qquad \xve(0,y)=x_0(y),\qquad y\in \Omega.
\end{equation}
Using these regularized coordinates we defined the smoothed out equations by
\begin{equation}\label{eq:eulerlagrangiancoordsmoothed}
D_t V^i=-\delta^{ij}(\widetilde{\partial}_j h+\pave_j \phi),\quad\text{in}\quad [0,T]\times\Omega,\quad \text{where}\quad
D_t=\partial_t\big|_{y =const},\qquad \widetilde{\partial}_i =\frac{\pa y^a}{\pa \widetilde{x}^i} \frac{\pa}{\pa y^a},
\end{equation}
where $h$ is given by
\begin{equation}\label{eq:waveequationsmoothed}
  D_t^2 e(h) - \widetilde{\Delta} h = (\widetilde{\pa}_i \widetilde{V}^j)(\widetilde{\pa}_j V^i) - \rho(h)\quad \textrm{ in }
 [0, T]\! \times\!\Omega, \quad\text{with}\quad
 h\big|_{ [0, T] \times\pa \Omega}=0,\quad\text{where}\quad \widetilde{\Delta}\!=\delta^{ij}\pave_i\pave_j,\!
\end{equation}
and $\phi$ is given by
\vspace{-0.1in}
\begin{equation}
  \qquad\phi(t,y) =  \frac{1}{4\pi} \int_{\Omega}
  \frac{ \rho(t, y')\widetilde{\kappa}(t, y')\,dy'}{|\xve(t,y)- \xve(t,y')|}\, ,
 \qquad\text{where}\qquad \widetilde{\kappa}=|\det{(\pa \widetilde{x}/\pa y)}|.
\end{equation}
Taking the divergence of \eqref{eq:eulerlagrangiancoordsmoothed} and subtracting
it from \eqref{eq:waveequationsmoothed} shows that $D_t \rho \!= \!-\rho \div \!V\!$ if this holds at $t \!= \!0$.

One can prove uniform apriori energy bounds for the system \eqref{eq:eulerlagrangiancoordxsmoothed}-\eqref{eq:waveequationsmoothed} up to a time $T\!>\!0$, independent of $\varepsilon$.
Moreover, one can prove $\varepsilon$ dependent bounds for the iteration scheme: given
$V$, define $\tilde{V}$ and $\widetilde{x}$ by  \eqref{eq:eulerlagrangiancoordxsmoothed}, and then $h$ and the new $V$ by solving the system  \eqref{eq:eulerlagrangiancoordsmoothed}-\eqref{eq:waveequationsmoothed}.
We will show in Theorem \ref{vlwpthm} that this system has a
unique solution on a time interval of size $\ve$. In Theorem
 \ref{eneestthm}, we show
the solutions  satisfy energy estimates which are uniform in $\ve$.
This allows us to extend the solution to a time
independent of $\ve$, and by taking the limit as $\ve\!\to\! 0$ obtain a solution to the original system \eqref{eq:eulerlagrangiancoordintro}-\eqref{eq:waveequationintro}; see Section \ref{tangsec}.

\subsection{Energy estimates} Let $E$ be the energy for Euler's equations:
\begin{equation}\label{eq:TheEnergy}
E(t)= \int_{\D_t}\!\!\!\big( |V|^2\! +Q(\rho) +\phi \big)\, \rho \, dx
=\int_{\Omega}\!\big( |V|^2\! +Q(\rho) +\phi \big)\,  \rho
\kappa dy,\quad\text{where}\quad Q(\rho)\!=\!2\!\int\! p(\rho) \rho^{-2} d\rho,\quad D_t (\kappa\rho)\!=\!0.
\end{equation}
If we take the time derivative of the integral expressed in the fixed Lagrangian coordinates we get $D_t$ applied to the integrand. We then use Euler's equation $D_t V=-\rho^{-1} \pa p-\pa \phi$ and integrate by parts:
\begin{equation}
\frac{ dE\!}{dt}
=\!\int_{{}_{\!}\D_{t}}\!\!\! 2{V^{{}^{\!}}}^{i{}_{\!}} ({}_{\!}-\pa_i p-\!\rho
\pa_i\phi)\!- Q^{\prime\!}(_{\!}\rho_{\!})\rho D_{t} \rho+{}_{\!}\rho D_{t} \phi \,  dx
=\!\int_{\!\D_{{}_{\!}t}}\!\! 2\div\! V \, p -Q^{\prime\!}(_{\!}\rho_{\!})\rho D_{{}_{\!}t} \rho +({}_{\!} D_{{}_{\!}t} \phi-{}_{\!}2{V^{{}^{\!}}}^{i\!}\pa_i\phi \big) \rho \, dx
+\!\int_{\!\pa\D_{{}_{\!}t}} \!\!\!\! 2  v_i  p {N^{{}^{\!}}}^{i\!} dS.
\end{equation}
Using the continuity equation $D_t \rho\!=\!-\rho\div V\!\!$ and the boundary condition $p\!=\!0$ only terms with $\phi$ remain.
Let us for simplicity assume that $\rho\kappa\!=\!1$, and let $\Phi(z)\!=\!(4\pi |z|)^{-1}\!\!$. With $\Phi_i(z)\!=\!\pa_i \Phi(z)\!=\!-\Phi_i(-z)$ we have
\begin{equation}
\int_{\!\D_t}\!\! \!\! D_t \phi\, \rho \, dx
=\!\! \int\!\!\!\!\int_{\Omega{}_{\!}\times{}_{\!}\Omega}
\!\!\!\!\!\!\!\! D_{t} \Phi\big(x(t,{}_{\!}y)\!-\!x(t,{}_{\!}y')\big) dy dy'\!\!=\!\! \int\!\!\!\!\int_{\Omega{}_{\!}\times{}_{\!}\Omega}
\!\!\!\!\!\!\!\! \big(V(t,{}_{\!}y)\!-\!V(t,{}_{\!}y')\big) \Phi_i\big(x(t,{}_{\!}y)\!-\!x(t,{}_{\!}y')\big) dy dy'\!\!
=\!\!\int_{\!\D_t}\!\! \!\! 2{V^{{}^{\!}}}^{i\!}\pa_i\phi \,  \rho \, dx.
\end{equation}
It follows that $E'(t)\!=\! 0$.
This energy for the smoothed problem with $\!\D_t,dx$ replaced by $\!\tD_t,d\widetilde{x}$ is almost conserved apart from that $D_t(\rho\widetilde{\kappa})\!=\!
\rho\widetilde{\kappa}(\div\!\widetilde{V}\!\!-\div\! V\!)$.
We will obtain energies for derivatives of the smoothed problem which will contain a boundary term where the symmetry of the smoothing matters, see section \ref{energy}.

\section{Lagrangian Coordinates and the wave equation for the enthalphy}
\label{lagsec}
Let $\Omega$ be the unit ball in $\R^3\!$ and $x_0\!_{\!} :\! \Omega \!\to\!
\D_0\!$ be a diffeomorphism.
Suppose that $V(t,_{\!}x_{\!}), p(t,_{\!}x_{\!}), \rho(t,_{\!}x_{\!})$ satisfy \eqref{mom}-\eqref{vaccuum}.
The Lagrangian coordinates $x(t,_{\!}y_{{}_{\!}})$ are given by:
\begin{equation}\label{eq:thelagraniancoord}
  \frac{d}{dt}x(t,y) = V(t, x(t,y)),\qquad
  x(0,y) = x_0(y),\qquad y \in \Omega.
\end{equation}
We define the material derivative:
\begin{equation}
  D_t f(t,y) = \frac{\pa}{\pa t}\Big|_{y = constant\, } f(t,y).
\end{equation}
We will use the letters $i,j,k\dots$ to refer to quantities expressed in
terms of the usual Eulerian coordinates and $a,b,c,\dots$ to refer to
quantities expressed in Lagrangian coordinates, e.g.
\begin{equation}
  \pa_i = \frac{\pa y^a}{\pa x^i} \frac{\pa}{\pa y^a} = \frac{\pa y^a}{\pa x^i}
  \pa_a.
  \label{pai}
\end{equation}
In these coordinates we can now write Euler's equations \eqref{mom} and the continuity equation \eqref{mass} as
\begin{alignat}{2}
  \rho D_t V^i &= -\delta^{ij} (\pa_i p+\rho\pa_i \phi),
  \quad && \text{ on } [0,T] \times \Omega, \label{lagmom}\\
  D_t \rho &= -\rho \div V, \quad && \text{ on } [0,T] \times \Omega,
   \label{lagmass}
\end{alignat}
where $\div V = \pa_i V^i$ and $\pa_i$ acts on functions defined on $\Omega$ by \eqref{pai}, where $x$ is obtained from $V$ by \eqref{eq:thelagraniancoord}. Writing $\kappa \!=\! \det(\pa x/{}_{\!}\pa y)$,
by \eqref{lagmass} and the formula for the derivative of the determinant,
we have $D_t \kappa\! =\! \kappa \div \!V \!\!$.

%We also define $\D_t = x(t,\Omega)$ and write
%$y(t,x)$ for the inverse of the map $y \mapsto x(t,y)$.
The gravitational potential is then given in terms of the fundamental solution
of the Laplacian by:
\begin{align}
 \phi(t,y) =
  \frac{1}{4\pi} \int_{\D_t}
  \frac{\rho(t, y(t,x'))\,dx'}{|x(t,y)- x'|}\,
  =  \frac{1}{4\pi} \int_{\Omega}
  \frac{ \rho(t, y')\kappa(t, y')\,dy'}{|x(t,y)- x(t,y')|}
  =  \frac{1}{4\pi} \int_{\Omega}
  \frac{ \rho_0(y')\kappa_0(y')\,dy'}{|x(t,y)- x(t,y')|}\, .
 \label{phidefsec2}
\end{align}

\subsection{The enthalpy formulation}
\label{enthalpysec}
The pressure is determined from the
mass density, $p = p(\rho)$ for a smooth, increasing function
$p$ and we can alternatively think of $\rho\! =\! \rho(p)$.
With $\bar{\rho}\! = \!\rho^{-1}(0)$, we define the enthalpy by:
\begin{equation}
  h(\rho) = \int_{\bar{\rho}}^\rho \frac{p'(\lambda)}{\lambda} d\lambda.
  \label{hdefi}
\end{equation}
We then have $\pa_i p = \rho \,\pa_i h$, so \eqref{lagmom} becomes:
\begin{equation}
  D_t V^i = -\delta^{ij}(\pa_j h + \pa_j \phi).
  \label{hmom}
\end{equation}
Since we assume that $p'(\lambda)\! >\! 0$,
the function $\rho\! \to\! h(\rho)$ is invertible.
We can then write $\rho \!=\! \rho(h)$ and think of $h$ as the
fundamental thermodynamic quantity. Defining $e(h)\! = \!\log \rho(h)$,
we can re-write \eqref{lagmass} in terms of $h$:
\begin{equation}
  D_t e(h) + \div V = 0.
  \label{hmass}
\end{equation}
Taking the divergence of \eqref{hmom} and the time derivative of \eqref{hmass} using that
$[D_t, \pa_j] = -(\pa_jV^k)\pa_k $,
we get:
\begin{equation}
 D_t^2 e(h) - \Delta h = (\pa_i V^j)(\pa_j V^i)-\rho(h),  \textrm{ in }
 [0, T]\! \times\! \Omega, \quad\text{with}\quad h = 0, \textrm{ on }  [0,T]
 \!\times\! \pa \Omega, \label{wavedef}
\end{equation}
Here, $\pa_i$ is given by \eqref{pai},
$\Delta$ is the Laplacian on $\Omega$ induced by the coordinates
$x$ and the flat metric on $\R^3$:
\begin{equation}
  \Delta h =\delta^{ij}\pa_i\pa_j h= \kappa^{-1}\pa_a\big(\kappa g^{ab}\pa_b h\big),\qquad\text{where}\qquad
  g^{ab} = \delta^{ij}\frac{\pa y^a}{\pa x^i}\frac{\pa y^b}{\pa x^j},\quad\text{and}\quad \kappa = \det(\pa x/\pa y).
\end{equation}
On the other hand, starting with \eqref{wavedef}
and taking the divergence of \eqref{hmom}, \eqref{hmass} is automatically
satisfied if it is satisfied at $t = 0$.
By \eqref{hdefi} respectively \eqref{hmass}, we have that:
\begin{equation}
 h\big|_{t = 0} = e^{-1}(\log \rho_0) \equiv h_0,
\quad\text{and}\quad
 D_t h\big|_{t = 0} = - \div V_0/e'(h_0) \equiv h_1,\qquad \textrm{ in } \Omega.
 \label{wavedefic}
\end{equation}
Assuming that the initial-boundary value problem
\eqref{wavedef}-\eqref{wavedefic}
has a unique solution $h$ for given $V$, the initial-free boundary problem for Euler's equations \eqref{mom}-\eqref{eq:initialconditionseulercoord}
is equivalent to the fixed boundary problem:
\begin{align}
  D_t V^i &= -\delta^{ij}\pa_j h - \delta^{ij}\pa_j\phi, \quad\text{and}\quad D_t x^i =V^i, \qquad\textrm{ in } [0,\! T] \!\times \!\Omega,\\
 D_t^2 e(h) - \Delta h &= (\pa_i V^j)(\pa_j V^i)-\rho(h),\quad\textrm{ in } [0,T]
 \times \Omega,\quad\text{with}\quad
 h = 0, \textrm{ on } [0, T] \!\times\! \pa\Omega,\\
  x(0, y) &= x_0(y),\quad D_t x(0, y) = V_0(y),\quad \text{and}\quad
 h(0, y) =h_0, \quad D_th(0, y) = h_1.
 \end{align}
\subsubsection{Assumptions on the equation of state}
With $\delta_0\!>\!0$ as in Appendix \ref{compat} let $c_1\!>\!0$ be a constant such that
\begin{equation}\label{eq:closetoincompressible2}
0<c_1\leq e^\prime(h)\leq \delta_0,\quad \text{and}\quad
|e^{(k)}(h)| \leq \delta_0/E_0^{k-1},\quad k=2,\dots ,r,\quad\text{where} \quad
E_0= ||V_{0_{\!}}||_{H^r}\!+\! ||\rho_{0_{\!}}||_{H^r}.
\end{equation}

\subsection{Higher order commutators}
\label{hocommutators}
Repeatedly using that $[D_t, \pa_j] = - (\pa_j V^\ell) \pa_\ell$ it follows that
\begin{equation}\label{eq:highercommutatorssec2}
D_t^k \pa_i = {\sum}_{\ell\leq k}S^{jk}_{i\ell} \pa_j D_t^\ell,
\end{equation}
where $S^{jk}_{ik}\!=\delta^{j}_i$,
and for $\ell<k$, we have for some constants $c^{kn}_{\ell \ell_1\dots \ell_n}$
\begin{equation}\label{eq:highercommutatorssec2exp}
S^{jk}_{i\ell}\!= S^{jk}_{i\ell}(\pa V,\dots,\pa D_t^{k-\ell-1} V)=c^{kn}_{\ell \ell_1\dots \ell_n}
(\pa_i D_t^{\ell_1} V^{i_2})\cdots(\pa_{i_n} D_t^{\ell_n} V^j),
\end{equation}
where the sum is over $\ell_1+\cdots+\ell_n=k-\ell-n$ and
$n=1,\dots k$. Here the terms with $n=1$ should be interpreted as
$c^{k1}_{\ell \ell'}(\pa_i D_t^{\ell'}V^j)$ and the terms with $n=2$ should be interpreted as
$c^{k2}_{\ell \ell'\ell''}(\pa_i D_t^{\ell'}V^{i''})(\pa_{i''} D_t^{\ell''}V^j)$.

The potential $\phi=\Phi[\rho \kappa] $ can be expressed in terms an integral operator
\begin{equation}\label{eq:potentialkernel}
\Phi[f](t,y)= \! \!\int_{\Omega}\!\! K(t,y,y') f(t,y')\,dy'\!,\!\!\quad\text{with kernel}\quad
 K(t,y,y')=\frac{1}{4\pi}
  \frac{1 }{|x(t,y)- x(t,y')|} .
\end{equation}
$D_t^k \Phi[f]$ is a sum of integral operators $\Phi_\ell[D_t^{k-\ell\!}\!f]$,
$\ell\!\leq \!k$, with kernels that  are sums over $\ell_1\!+\!\cdots\!+\!\ell_n\!=\!\ell$ of
\begin{equation}\label{eq:Dtkpotentialkernel}
{}_{\!}K_{\!\ell {}_{\!}}( {}_{\!}\delta  x,_{\!}\delta V\!{}_{\!},{\!}\dots{}_{\!},\delta D_t^{\ell\!-\!1\!} V {}_{\!})\!=\!
  d^{\,\ell}_{\ell_{\!1}{\!}\dots{\!} \ell_{\!n}}\!\!\frac{\!\!\!(_{\!}\delta D_t^{\ell_{\!1}} \!x_{\!}\cdot_{\!} \delta D_t^{\ell_2}\!x {}_{\!})\cdots (_{\!}\delta D_t^{{}_{\!}\ell_{\!n\!-\!1}}\!x {}_{\!}\cdot {}_{\!}\delta D_t^{\ell_{{}_{\!}n}}\!x {}_{\!})\!\!\! }{|x(t,y)- x(t,y')|},
  \quad\,\text{where}\quad\,   \delta W{}_{\!}( {}_{\!}t,{}_{\!}y, {}_{\!}y' {}_{\!})\!= \frac{\!\!\! W{}_{\!}( {}_{\!}t, {}_{\!}y {}_{\!})\!-\!W{}_{\!}( {}_{\!}t, {}_{\!}y' {}_{\!} {}_{\!})\!\!\!}{\!|x( {}_{\!}t,{}_{\!}y {}_{\!})\!-\! x( {}_{\!}t,{}_{\!}y' {}_{\!})|\!}.\!
\end{equation}

\subsection{The compatibility conditions}
\label{compcondsec}
The compatibility condition of order $m$ \eqref{int:compat2} can now be expressed in the Lagrangian coordinates as that the formal power series solution in $t$: $\hat{V}(t,y)\!=\!\sum V_k(y) t^k\!/k! $ and $\hat{h}(t,y)\!=\!\sum h_k(y) t^k\!/k!$ and $\hat{\phi}(t,y)\!=\!\sum \phi_k(y) t^k\!/k!$ to the system \eqref{hmom}-\eqref{hmass},\eqref{phidefsec2} satisfy $h_k\big|_{\pa\Omega}\!=0$, for $k=0,\dots,m$. However, since we are looking for solutions in Sobolev spaces this has to be expressed in a weak form:
\begin{equation}
 h_k(y) \in H^1_0(\Omega), \qquad k = 0,..., m.
 \label{maincomp0}
\end{equation}
We would like to think of \eqref{wavedef}-\eqref{wavedefic} as determining
$h$ uniquely as a functional of $V$\!. In order for the initial value problem for the wave equation
\eqref{wavedef}-\eqref{wavedefic} to have a regular enough solution, the
 initial data needs to satisfy the compatibility
conditions \eqref{maincomp0}. These compatibility conditions for $h$ will however depend on the formal power series for $V$\!. We now calculate the formal power series for the coupled system and hence the compatibility conditions. These power series are uniquely determined by
$x_0,V_0,h_0$.  By \eqref{hmom}, using \eqref{eq:highercommutatorssec2}:
\begin{equation}
V_{k+1}= {\sum}_{\ell\leq k} S^{jk}_{i\ell}(\pa V_0,\dots, \pa V_{k-\ell-1}) \pa_j H_\ell,\qquad H_k=h_k+\phi_k.
\end{equation}
Similarly by \eqref{hmass} we have for some function $G_k$:
\begin{equation}
e^\prime(h_0) h_{k+1}= {\sum}_{\ell\leq k} S^{jk}_{i\ell}(\pa V_0,\dots, \pa V_{k-\ell-1}) \pa_j V_\ell+G_k(h_0,\dots,h_k) .
\end{equation}
The relation for $\phi_k$ is not as direct but it is clear from
\eqref{eq:Dtkpotentialkernel} that for some non local functional $\Phi_k$:
\begin{equation}
\phi_k=\Phi_k[ x_0,V_0,\dots,  V_{k-1},h_0,\dots,h_k].
\end{equation}

\section{Tangential smoothing, tangential operators and tangential vector fields }
There is a family of open sets $V_\mu $, $\mu\!=\!1,\dots,N$ that cover $\pa \Omega$ and  onto diffeomorphisms
$\Phi_\mu\!: (-1,1)^2 \!\to V_\mu $. We fix a collection of cutoff functions $\chi_{\mu}\!:\!\pa \Omega
\to \R$ so that $\chi_\mu^2$ form a partition of unity subordinate to the cover $\{V_\mu\}_{\mu = 1}^N$,
 as well as another family of
``fattened'' cutoff functions $\tC_\mu$ so that the support of
$\tC_\mu$ is contained in $V_\mu$ and so that
$\tC_\mu \equiv 1$ on the support of $\chi_\mu$. Recalling that
$\Omega$ is the unit ball, we set $W_{\!\mu} \!=\! \{r\omega, r \in (1/2, 1], \omega
\in V_{\!\mu}\}$
for $\mu = 1,\dots, N$
and let $W_0$ be the ball of radius $3/4$ so that the collection
$\{W_\mu\}_{\mu = 0}^N$ covers $\Omega$.
Writing $\Psi_\mu(z,z_3) = z_3 \Phi_\mu(z)$,
$\Psi_\mu$ is a diffeomorphism from $(-1,1)^2 \times (1/2,1]$ to
$W_\mu$.
Let $\eta:[0,1] \to \R$ be
a bump function so that $\eta(r) = 1$ when $1/2 \leq r \leq 1$ and
$\eta(r) = 0$ when $r < 1/4$.
We define cutoff functions on $\Omega$ by setting
$
\chi_\mu = \chi_\mu \eta .
$

For a linear operator $T'$ defined in coordinate charts we define a global operator $T$ by
\begin{equation}\label{eq:localtoglobal}
Tf=\sum T_\mu f,\quad\text{where}\quad
T_\mu f =
\chi_\mu\big(m_\mu^{-1}T'\big[m_\mu (\chi_\mu f)\circ \Psi_\mu \big]\big)\circ\Psi_\mu^{-1}\!\!,\qquad
m_\mu =|\det{\Phi_\mu^\prime}|^{1/2}r.
\end{equation}
Then $T$ is symmetric with the measure $dy$ if $T'$ is with the measure $dz$ is since $dS(\omega)=m_\mu^2 dz$.

\subsection{Tangential smoothing}
\label{smoothingsec}
 Let
$\varphi \!:\! \R^2 \!\!\to\! \R$ be even, supported in $R = (-1,1)^2$ with $\int_{\R^2}
 \!\varphi =\! 1$ and let
\begin{equation}
 T_\ve f(z) = \int_{\R^2} \varphi_{\ve}(z-z') f(z') dz',\qquad\text{where}\qquad
 \varphi_{\ve}(z)\! =\! {\ve^{-2}} \varphi\big({z}\!/{\ve}\big).
\end{equation}
be a smoothing operator.
Because $\varphi$ is even, $T_\ve$ is symmetric; for any functions
$f, g: \R^2 \to \R$ we have:
\begin{equation}
 \tint T_\ve f(z) g(z)\, dz =
 %\int \int_{\R^2} \varphi_\ve(z-z') f(z') g(z)\, dz'\,  dz =
 \tint f(z) T_\ve g(z)\, dz.
\end{equation}
Furthermore, by Appendix A have:
\begin{equation}
\qquad\qquad \quad |T_\ve(fg)(z) - fT_\ve(g)(z)|
 \leq C \ve ||f||_{C^1(R)} ||g||_{L^2(R)}.
\end{equation}

With notation as in \eqref{eq:localtoglobal}, the smoothing operators we consider on $\Omega$ or $\pa \Omega$ are then defined by:
\begin{align}
 \sm f = {\sum}_{\mu = 0}^N
  T_{\ve,\mu} f,
 \qquad
 \ssm f = \sm \sm f={\sum}_{\mu,\nu = 0}^N
 T_{\ve,\nu} T_{\ve,\mu} f.
 \label{smoothing}
\end{align}
Since $T_\ve $ is symmetric $J_\ve$ is as well.
The following estimates are proved in Section \ref{tangentialsmoothingappendix}:
\begin{lemma}
 With $\sm$ defined by \eqref{smoothing}, if $k \geq m$ then:
 \begin{align}
  ||\sm f||_{H^k(\pa \Omega)} \lesssim \ve^{m-k} ||f||_{H^{m}(\pa \Omega)},
  \qquad
  ||\sm f - f||_{H^k(\pa \Omega)} \lesssim \ve ||f||_{H^{k+1}(\pa \Omega)},
  \label{}
 \end{align}
 and, with $\Sigma = \pa \Omega$ or $\Omega$:
 \begin{equation}
   ||\sm(fg) -f \sm g||_{L^2(\Sigma)} \lesssim \ve ||f||_{C^1(\Sigma)} ||g||_{L^2(\Sigma)}.
 \end{equation}
\end{lemma}

\subsection{The tangential fractional derivatives}
We will need to use fractional tangential derivatives to control our solution
and we will define these operators
in coordinates.
If $F: \R^2 \to \R$, we define:
\begin{equation}
 \fd^s F(z) = \int_{\R^2} e^{iz\cdot \xi} \langle \xi \rangle^s \hat{F}(\xi)\, d\xi,
 \quad \text{where}\quad
 \hat{F}(\xi) = \int_{\R^2} e^{-iz\cdot \xi} F(z)\, dz,
\end{equation}
and we define fractional tangential derivatives on $\Omega$ by:
\begin{equation}
 \fd_\mu^{s} f = \widetilde{\chi}_\mu (\fd^{s} f_\mu)\circ\Psi_{\mu}^{-1},\quad f_\mu=(\chi_\mu f)\circ\Phi, \qquad \mu = 1,..., N.
 \label{fdmudef}
\end{equation}
We also set $\fd_0^s f \!=\!\chi_0( \langle \pa \rangle^{s} f_0)\!\circ\!\Psi_{0}^{-1}\!\!$,
where $\langle \pa \rangle^{s}$ is defined by taking the Fourier transform
in all directions.

For $s \in \R$, $k \in \mathbb{N}$, we define:
\begin{equation}
 || f ||_{H^{(k, s)}(\Omega)}
  = {\sum}_{\mu = 0}^N ||\fdm^s f||_{H^k(\Omega)},
  \quad\text{and}\quad
  || f ||_{H^s(\pa \Omega)} = {\sum}_{\mu = 1}^N ||\fdm^s f ||_{L^2(\pa \Omega)}.
 \label{sobspacedef}
\end{equation}
In Appendix A we prove:
\begin{lemma}
If $T \in \T$, then:
\begin{equation}
 \Big| \int_{\pa\Omega} f T g\, dS(y)\Big|
 \leq C||f||_{H^{1/2}(\pa \Omega)} ||g||_{H^{1/2}(\pa \Omega)},
 \qquad
 \Big| \int_{\Omega} f T g\, dy \Big|
 \leq C||f||_{H^{(0,1/2)}(\Omega)} ||g||_{H^{(0,1/2)}(\Omega)}.
\end{equation}
In addition, with $\Sigma = \pa \Omega$ or $\Omega$,
 \begin{equation}
  ||\fdhm (fg) - f \fdhm g||_{L^2(\Sigma)}
  \leq C ||f||_{H^2(\Sigma)} ||g||_{L^2(\Sigma)}.
 \end{equation}
\end{lemma}

\subsection{The tangential derivatives and tangential norms}
\label{def T and FD}
Since $\Omega$ is the unit ball, the vector fields
\begin{equation}
\Omega_{ab} =  y^a \pa_{y^b} - y^b \pa_{y^a} , \qquad a,b = 1,2,3,
\end{equation}
are
tangent to $\pa\Omega$ and span the tangent space there.
With $\eta$ the cutoff function defined above, we let:
\begin{equation}
 \T = \cup_{a,b = 1,2,3}\{ \eta\, \Omega_{ab}, (1-\eta) \pa_{y^a}\}.
\end{equation}
In analogy with the two dimensional case, when $\T$ is just the derivative with respect to the angle in polar coordinates, we will now introduce some simplified notation for the norms.
Suppose that $\T\!=\!\{T_1,\dots,T_{N'}\!\}$.
If $V\!\!:\!\Omega\!\to\!\bold{R}^3$ is a vector field we will let $\T V$ stand for the map
$\T V\!\!:\!\Omega\to \!\bold{R}^{3N'}\!\!\!$, whose components are $T_jV^i\!\!$, for $i\!={}_{\!}\!1,2,3$, $j\!{}_{\!}={}_{\!}\!1,{}_{\!}...,{}_{\!}N'\!\!$.
Moreover let $\T^r\!\!{}_{\!}={}_{\!}\!\T\!\!\times\!{}_{\!}\cdots\!\times\! \T$($r$ times) and let $T^I\!\!\in\! \T^r\!$ stand
for a product of $r$ vector fields in $\T\!$, where $I\!=\!(i_1,{}_{\!}...,{}_{\!}i_{r\!})\!\in\! [1,N']\!\times\!\cdots\!\times\! [1,N']$ is a multiindex of length $|I|\!=\!r$. Let $\T^r V\!$ stand for the map $\T^r V\!\!:\!\Omega\!\to \!\bold{R}^{3N' \!r}\!\!\!$, whose components are $T^I V^i\!\!$, for $i\!=\!1,2,3$, $1\!\leq\! i_j\!\leq\! N'\!\!$, $j\!=\!1,{}_{\!}...{}_{{}_{\!}},r$. The norm of $\T^r V\!$ is
\begin{equation}\label{eq:simplifiedtangentialnotation}
|\T^r V|^2= \delta_{ij}\T^r V^i \cdot \T^r V^j,\quad\text{where}\quad
\T^r V^i \cdot \T^r V^j= {\sum}_{|I|=r,\,\,T^I\in\T^r} T^I V^i \, T^I V^j.
\end{equation}
We will use similar notation for space time vector fields tangential to the boundary.
Let $\mathfrak{D}\!=\!\T\!\cup\! D_t$, and
$\mathfrak{D}^r\!\! =\!\mathfrak{D}\!{}_{\!}\times \!\!\cdots\!\!\times \!\mathfrak{D}$($r{\!}$ times), $\mathfrak{D}^{r\!,k}\!{}_{\!}=\!\T^r {}_{\!}\!\!\times \!{}_{\!} D_{{}_{\!}t}^k$\!.
For $\!K\!{}_{\!}=\!(_{\!}I{}_{\!},{}_{\!}k{}_{{}_{\!}})$ a multiindex with $|I|\!{}_{\!}=\!r$, we write
$D^K\!\!=\!T^I \!D_{{}_{\!}t}^k\!$, $T^I\!\!\in\!\T^r$\!\!.

\section{The smoothed Euler's equations}
\label{tangsec}
In this section, we introduce the smoothed problem we will use to
construct solutions to \eqref{mom}-\eqref{vaccuum}. This in the incompressible case goes back to Coutand-Shkoller\cite{CS07}, with important improvements due to Nordgren\cite{Nordgren2008}.

\subsection{Tangential smoothing of the coordinates}
With the tangential smoothing operator $\ssm$ defined as in
\eqref{smoothing}, given a vector field
$V$, we define the smoothed coordinate $\xve(t,y)$ by:
\begin{equation}\label{xvedef}
 \xve^i(t,y) = x_0(y) + \int_0^t \ssm V^i(s, y) ds.
\end{equation}
We then define:
\begin{equation}\label{udef}
 A^i_{\m a} =
 \frac{\pa \xve^i}{\pa y^a} , \qquad
 A^a_{\m i} = \big(A^{-1}\big)_{\m i}^a=\frac{\pa y^a}{\pa \xve^i}.
 \end{equation}
 We define $\tD_t = \xve(t, \Omega)$
 and we use the letters $i,j,k,...$ to denote coordinate derivatives
  $\pave_i={\pa}/{\pa \xve^i}$ on $\tD_t$: However all our functions will be functions of $y$ so we will think of $\pave$ as a differential operator on $\Omega$
 \begin{equation}
  \pave_i f(t, y) = A_{\m i}^a(t,y) \pa_a f(t, y),\quad\text{where}\quad
  \pa_a ={\pa}/{\pa y^a}.
  \label{pavedef}
 \end{equation}
 The coordinate $\xve(t,y)$ and Euclidean metric on $\tD_t$ induces a time-dependent metric on $\Omega$:
  $\gve_{ab} =
 \delta_{ij} A^i_{\m a} A^j_{\m b}$.
 We let $\kve dy =\det{(\pa\xve/\pa y)} dy$ be the volume element on $\Omega$
 induced by the volume element $d\xve$ on $\tD_t$.
 Let
\begin{equation}
 \Dve f =\delta^{ij}\wpa_i\wpa_jf = \kve^{-1}\pa_a \big(\kve \gve^{ab} \pa_b f\big),\quad\text{where}\quad \gve^{ab} = \delta^{ij} A_{\m i}^a A_{\m j}^b,\qquad \kve=\det{(\pa\xve/\pa y)},
 \label{laplsm}
\end{equation}
denote the Laplacian.
Given a one-form $\alpha = \alpha_i d\xve^i$ on $\tD_t$, we also write:
\begin{equation}
 \div \alpha =\pave_i( \delta^{ij}\alpha_j),
 \qquad \curl \alpha_{ij} =  \pave_i \alpha_j
 - \pave_j \alpha_i.
 \label{divdef}
\end{equation}
Here,  $a,b,c,...$
correspond to quantities in  the $y$ variables and
$i,j,k...$ to quantities  in  the
$\xve$ variables.

\subsection{The smoothed problem}
\label{approxpbmsec}

Given initial data $(V_0, h_0)$ which are compatible with
\eqref{mom}-\eqref{freebdy} in the sense of \eqref{int:compat2},
we now introduce the smoothed problem we will consider.
Given a vector field $V: [0,T] \times \Omega \to \R^3$, we
define
the tangentially smoothed Lagrangian coordinate
$\xve = \xve[V]$ by \eqref{xvedef} and  $A, \pave$ and $\Dve$ as in
\eqref{udef}-\eqref{laplsm}.

We would like to define
$h = h[V]$ to be the unique solution to:
\begin{align}
 D_t^2e(h) - \Dve h+ \rho(h) &= (\pave_i \ssm V^j)(\pave_j V^i), \qquad
 \textrm{ in } [0, T] \times \Omega,\quad\text{with}\quad    h = 0, \textrm{ on } [0, T] \times \pa \Omega,\label{smwavedef}\\
 h(0,y) &= h^\ve_0(y), \qquad D_t h(0,y) = h^\ve_1(y),
  \textrm{ in } \Omega \label{smwaveic},
\end{align}
for some choice of initial data $h^\ve_0, h^\ve_1$. However,
there are compatibility conditions
that must be satisfied in order for this to have a sufficiently regular
solution. We will define these conditions
momentarily but for now suppose that $V\!$ and $h_0^\ve,h_1^\ve$ are such that
this problem has a unique solution $h$.
We then abuse notation slightly and write
$\rho(t,y)$ instead of $\rho(h(t,y))$.
We also define $\tD_t = \xve(t,\Omega)$ and write
$\widetilde{y}(t,\xve)$ for the inverse of the map $y \mapsto \xve(t,y)$. Next,
we define the gravitational potential $\phi = \phi[V]$ by:
\begin{equation}
 \phi(t, y) =
  \frac{1}{4\pi} \int_{\tD_t}
  \frac{\rho(t, \widetilde{y}(t,x'))\,dx'}{|\xve(t,y)- x'|}\,
  =  \frac{1}{4\pi} \int_{\Omega}
  \frac{ \rho(t, y')\widetilde{\kappa}(t, y')\,dy'}{|\xve(t,y)- \xve(t,y')|}\, ,
  \quad\text{so that}\quad
  \widetilde{\triangle} \phi=-\rho(h).
 \label{phidef}
\end{equation}
Note that $\phi$ depends on $V$ both because $\rho = \rho[V]$
and also through the domain $\tD_t$.

With the above definitions of $\pave[V], h[V], \phi[V]$ in mind:
\begin{definition}
 Given a  vector field $V\!\!$,
  suppose that $h_0^\ve, h_1^\ve$ are given such that the
  the system \eqref{smwavedef}-\eqref{smwaveic}
  has a unique sufficiently regular solution $h = h[V]$.
    We say that $V$ is a \emph{solution to the smoothed problem}
     if:
    \begin{equation}
     D_t V^i = -\delta^{ij}\pave_j h[V] -
     \delta^{ij}\pave_j \phi[V],  \qquad\textrm{ in } [0,T] \!\times\!
     \Omega,\qquad\text{and}\qquad
     V^i(0, y) = V_0^i(y).
     \label{smpbmdef}
    \end{equation}
\end{definition}
Note that subtracting the divergence of
\eqref{smpbmdef} from
\eqref{smwavedef} using \eqref{phidef} gives that $D_t\big(D_t e(h) + \div V\big)\! =\! 0$ so:
\begin{equation}
 D_t e(h) + \div V = 0, \qquad \textrm{ in } [0,T] \times \Omega,
 \label{resultofproj}
\end{equation}
provided that this holds at $t = 0$.

In Section \ref{uniform}, we prove the following a priori estimate
for the problem \eqref{smpbmdef}-\eqref{resultofproj}.
Let $\delta_0$ denote the largest number so that \eqref{int:tsc} holds
with $\delta= \delta_0$ at $t = 0$. Also set:
\begin{equation}
 E_0^r = ||V_0||_{H^{(r-1,1/2)}(\Omega)}^2 +
 ||x_0||_{H^r(\Omega)}^2 + ||\pave h_0||_{H^{r-1}(\Omega)}^2
 + \ve^2 (||V_0||_{H^r(\Omega)}^2
 + ||\pave h_0||_{H^r(\Omega)}^2).
\end{equation}
Writing $H^k = H^k(\Omega), H^{(k,1/2)} = H^{(k,1/2)}(\Omega)$,
in Corollary \ref{useful}, we prove:
\begin{theorem}
  \label{intenergy}
 Let $r \geq 8$ and fix $\ve$ sufficiently small.
 There are strictly positive, continuous functions $\mathscr{T}_r,
 \Cc_r$ with $\mathscr{T}_r$ independent of $\ve$,
 and so that if $V \in C([0,T]; H^r(\Omega))$
 solves the smoothed Euler equations \eqref{smpbmdef}-\eqref{resultofproj}
  for $0 \!\leq\! t \leq T$ with $T \leq \mathscr{T}_r(E_0^r,1/\delta_0)$, then,
  with $||\pave h||_r \!=\! \sum_{k + \ell \leq r} ||D_t^k \pave h||_{H^{\ell}}$:
 \begin{equation}
  ||V(t)||_{H^{(r\!-\!1,1\!/2)}\!}^2
  + ||\xve(t)||_{H^{r\!}}^2 +
  ||\pave h(t)||_{r{}_{\!}-{}_{\!}1\!}^2 +
  ||D_t^r h(t)||_{L^{2\!}}^2
  + \ve^2 \big(||V(t)||_{H^r\!}^2 + ||\pave h(t)||_{r}^2\big)
  \leq \Cc_r(E_0^{{}_{\,}r{}_{\!}-{}_{\!}1}\!\!, \delta_0^{-1\!}) E_0^{{}_{\,}r}.
  \label{intenergyest}
 \end{equation}
\end{theorem}

Before proving existence for the smoothed Euler equations \eqref{smpbmdef}
-\eqref{resultofproj}, we need to ensure that given sufficiently
regular $V$, the wave equation
\eqref{smwavedef}-\eqref{smwaveic} has a unique sufficiently regular solution.

\subsection{Compatibility conditions for the smoothed problem}
\label{wavecompatcondn}
We now define $h_0^\ve, h_1^\ve$ and give a condition that guarantees
that the initial-boundary value problem \eqref{smwavedef}-\eqref{smwaveic} is well-posed.

We say that the initial data $V_0^\ve, h_0^\ve$ satisfy the compatibility
conditions of order $m$ if there is a formal power series solution
$\hat{V}(t,y) = \sum V_k^\ve(y) t^k\!/k!$, along with
$\hat{h}(t,y) = \sum h_k^\ve(y) t^k\!/k!$ and
$\hat{\phi}(t,y) = \sum \phi_k^\ve(y) t^k\!/k!$ which satisfy
\eqref{smpbmdef} and \eqref{resultofproj} at $t = 0$, and moreover
so that:
\begin{equation}
 h_k^\ve \in  H^1_0(\Omega),\quad k = 0,..., m.
 \label{smcompatdef}
\end{equation}

As in Section \ref{hocommutators}, repeatedly using that
$[D_t, \pave_j] = -(\pave_j \ssm V^\ell)\pave_\ell$, we have:
\begin{equation}
 D_t^k \pave_i = \widetilde{S}_{i\ell}^{jk} \pave_j D_t^\ell,
 \label{dtpaSdefsmooth}
\end{equation}
where the sum is over $\ell \leq k$, $\widetilde{S}_{ik}^{jk} = \delta_i^j$
and for $\ell \leq k$,
with the same constants $c^{kn}_{\ell \ell_1 ... \ell_n}$
as in \eqref{eq:highercommutatorssec2exp}, we have:
\begin{equation}
\widetilde{S}_{i\ell}^{jk}  =
\widetilde{S}_{i\ell}^{jk}
(\pave \tV,..., D_t^{k-\ell-1}\tV)  =
c^{kn}_{\ell \ell_1\cdots \ell_n}
(\pave_i D_t^{\ell_1}\tV^{i_1})\cdots
(\pave_{i_n}D_t^{\ell_n}\tV^j),
 \label{sepsdef}
\end{equation}
where the sum is over $\ell_1 + \cdots + \ell_n = k - \ell - n$ and
$n = 1,..., k$, and where we are writing $\tV = \ssm V$.

Using that $\hat{V}$ solves the smoothed-out Euler equations
\eqref{smpbmdef} at $t = 0$, the coefficients $V_\ell^\ve$
must satisfy:
\begin{equation}
 V_{k+1}^\ve = {\sum}_{\ell \leq k} \widetilde{S}_{i\ell}^{jk}(\pave \tV_0^\ve,...,
 \pave \tV_{k-\ell-1}^\ve) \pave_j H_{\ell}^\ve,
 \label{vkepsdef}
\end{equation}
with $H_\ell^\ve = h_\ell^\ve + \phi_\ell^\ve$. Similarly, the condition
that $\hat{h}$ solves the continuity equation \eqref{resultofproj} at $t = 0$ becomes:
\begin{equation}
 e'(h_0^\ve) h_{k+1}^\ve
 = {\sum}_{\ell \leq k}S^{jk}_{i\ell}(\pave \ssm V_0^\ve,...,
 \pave \ssm V_k^\ve)
 + G_k(h_0^\ve,....,h_{k}^\ve),
 \label{hkepsdef}
\end{equation}
for a function $M_k$. We note the explicit formula for $k = 0$:
\begin{equation}
 e'(h_0^\ve) h_1^\ve = -\div V_0^\ve,
 \label{h1epsdef}
\end{equation}
and we take this to be the definition of $h^\ve_1$.
In addition we have that there is a non-local function
$\Phi_k$ so that:
\begin{equation}
 \phi_k^\ve = \Phi_k[x_0, \tV_0^\ve, ..., \tV_{k-1}^\ve, h_0^\ve, ...,
 h_k^\ve].
 \label{phikepsdef}
\end{equation}

To construct a solution to the smoothed Euler's equations
\eqref{smpbmdef}, we will need to consider only vector fields
$V$ whose Taylor expansions in $t$ at $t = 0$ agree with \eqref{vkepsdef}
and we make the following definition:
\begin{definition}
    A vector field $V$ is called \emph{admissible} to order $m$ if, for
    $k = 0,..., m$:
    \begin{equation}
     D_t^k V|_{t = 0} = {\sum}_{\ell \leq k} S^{jk}_{i\ell}(\pave \ssm V_0^\ve,...
     \pave \ssm V_{k-1}^\ve)\pave_j H_\ell^\ve,
     \label{intadmissible}
    \end{equation}
    where the
    $\widetilde{S}^{jk}_{i\ell}$ are defined by \eqref{sepsdef} and
    where $H_\ell^\ve = h_\ell^\ve + \phi_\ell^\ve$,
    with $V^\ve_\ell, h^\ve_\ell,\phi^\ve_\ell$ defined by \eqref{vkepsdef}-
    \eqref{phikepsdef}.
\end{definition}
\noindent In other words, $V\!$ is admissible to order $m$ if it solves the smoothed
Euler equations \eqref{smpbmdef} to order $m$ at $t \!=\! 0$.

In Theorem \ref{nllwpthm} we prove that if $(V_0^\ve\!, h_0^\ve)$ are
compatible to order $m$ in the sense of \eqref{smcompatdef} and $V \!\!=\! V(t,y)$ is
a fixed vector field satisfying \eqref{intadmissible} to order $m$, then
the system \eqref{smwavedef}-\eqref{smwaveic} has a unique solution
$h = h[V]$ on a time interval $[0,T]$ so that $h(t) \!\in\! H^1_0(\Omega)$ for
$t \in [0,T]$ and so that
$D_t^kh\! \in \!C([0,T]; H^{m-k}(\Omega))$ for $k = 0,..., m$.
By Theorem \ref{compatthm}, given $(V_0, h_0)$ which
are compatible
to order $m$, see \eqref{maincomp0}, there is a function $h_0^\ve$ so that if $V$ is
admissible to order $r$, then $(V_0, h_0^\ve)$ are compatible
to order $r$, see \eqref{smcompatdef}, and so that
$h_0^\ve\! \to \!h_0$ as $\ve \!\to \!0$.

\subsection{Solving the smoothed problem}
Suppose that $(V_0, h_0)$ are given and are compatible in the sense of
\eqref{int:compat2} (i.e. for the full nonlinear problem)
to order $r$.
In Appendix \ref{compat}, we construct a sequence
$h_0^\ve$ with $h_0^\ve \to h_0$ as $\ve \to 0$ and so that if $V$ is any vector
field satisfying \eqref{intadmissible} for $k = 1,..., r$,
then $(V_0, h_0^\ve)$ are
compatible to order $r$ in the sense of \eqref{smcompatdef}.
Given initial data $(V_0, h_0^\ve)$ which are compatible
to order $r$ in the sense of \eqref{smcompatdef}
and an admissible vector field $V$, we define
a functional:
\begin{equation}
 \Lambda^i[V](t,y) = V_0^i(y) - \int_0^t \delta^{ij} \pave_j h(s,y) \,ds- \int_0^t \delta^{ij} \pave_j \phi(s,y) \,ds.
 \label{lambdadef}
\end{equation}
with $(\pave,h, \phi) = (\pave[V], h[V],\phi[V])$ as in the previous
section. It is clear that if $V$ is a fixed point of $\Lambda$
then $V$ is a solution of the smoothed problem.
To construct a fixed point of $\Lambda$,
we will use the following norms:
\begin{equation}
 ||V||_{\X^s(T)} =
 {\sup}_{\,0 \leq t \leq T\,} ||V(t)||_{\X^s},
 \quad \text{where}\quad
 ||V(t)||_{\X^s} = {\sum}_{k = 1}^s ||D_t^{k} V(t)||_{H^{s-k}(\Omega)}
 + ||V(t)||_{H^{s-1}(\Omega)}.
\end{equation}

Our first result is then:
\begin{theorem}
  \label{vlwpthmintro}
 \!\!\!Let $r\! \geq\! 7$, $\ve\! >\! 0$ and suppose that
$(V_{\!0}^\ve\!\!, h_0^\ve)$ are compatible to order $r$. Let $\Cc_r$
 be as in Theorem \ref{intenergy} and set
 $\Cc'_r = \Cc_r E_0^{r}$.
Then there is a positive continuous function $T_\ve = T_\ve(E_0^{r+1})$ so that for any
$0 \!\leq\! T \!\leq \! T_\ve$,
the map $\Lambda$ has a unique fixed point in the space:
\begin{equation}
 \mathcal{C}^{r\!}(T) = \big\{ V \!\!: [0,T]\!\times \!\Omega \to \R^3 \big|\,
 V\! \textrm{ satisfies }
 \eqref{intadmissible} \textrm{ and }
 {\sup}_{\,0 \leq t \leq T\,} ||V(t)||^2_{\X^{r+1}} \leq \ve^{-2}
 \Cc_{r}' +\! 1 \big\}.
\end{equation}
\end{theorem}

The rest of the paper is devoted to proving Theorems \ref{intenergy} and \ref{vlwpthmintro}.
We will see in our construction that $T_\ve\! =\! O(\ve)$ and in particular
our proof of existence does not give a uniform time of existence
as $\ve\! \to\! 0$.

\subsection{Existence up to an $\varepsilon$ independent time}

Combining the a priori estimate from Theorem \ref{intenergy} and the
existence result Theorem \ref{vlwpthmintro}, we have:

\begin{proof}[Proof of Theorem \ref{mainthm}]
  Given initial data $(V_0, h_0)$, define $(V_0^\ve, h_0^\ve)$ as in
  Appendix \ref{compat}. For sufficiently small $\ve$,
  let $T_*$ denote the largest time so that the smoothed Euler equations
  \eqref{smpbmdef} have a unique solution $V_\ve(t) \!\in\! H^r(\Omega)$ with
  $V|_{t = 0} \!=\! V^\ve_0$. By Theorem \ref{vlwpthmintro}, $T_*\! \!> \!0$. We claim
  that in fact $T_* \!\geq \!\mathscr{T}_r$ where $\mathscr{T}_r$ is as in Theorem
  \ref{intenergy}. Assuming that this holds for the moment, we now have
  a vector field $V_\ve\! \in\! H^r(\Omega)$ satisfying \eqref{smpbmdef} on a time interval
  $[0,T_{\!E}]$ independent of $\ve$ and moreover by the energy estimate
  \eqref{intenergyest} we have that $||V_\ve||_{H^{r-1,1/2}(\Omega)}$ is uniformly bounded
  in $\ve$. By standard compactness theorems, it follows that there is a vector
  field $V \!\in\! H^{(r-1,1/2)}(\Omega)$ so that $V_\ve \!\to \!V$ strongly in
  $H^{r-1}(\Omega)$. Since $\xve = \ssm x \to x$ as
  $\ve \to 0$ and since $H^{r-1}(\Omega)$ is an algebra,
  it follows that $V$ satisfies Euler's equations
  \eqref{mom}-\eqref{mass}.
  To see that $T_* \!\geq\! \mathscr{T}_r$, we assume that $T_* \!< \! \mathscr{T}_r$.
   By the a priori
  estimate \eqref{intenergyest} and using that $D_t V_\ve = -\pave h - \pave \phi$
  and Theorem \ref{main theorem, ell est D_t phi} to control $\phi$, we have:
  \begin{equation}
    ||V_\ve(t)||_{\X^{r+1}}^2 \leq ||V_\ve(t)||_{H^r(\Omega)}^2 +
    ||\pave h(t)||_{r}^2 + ||\pave \phi(t)||_r^2
    \leq \ve^{-2} \Cc_r', \qquad 0 \leq t \leq T.
  \end{equation}
  Define $(V_{T_*}^\ve\!, h_{T_*}^\ve\!) \!=\! \lim_{t \nearrow T_*} (V(t), h(t))$.
  Since $h$ solves the wave equation \eqref{smwavedef} and $V_\ve$ solves the
  smoothed-out Euler equations \eqref{smpbmdef} it follows that
  the compatibility conditions \eqref{smcompatdef} are satisfied at $t = T_*$ as
  well, so repeating the proof of Theorem \ref{vlwpthmintro} with
  $t$ replaced by $t - T_*$ and $(V_0^\ve\!, h_0^\ve)$ replaced by
  $(V_{T_*}^\ve, h_{T_*}^\ve)$, we see that there is a
  $T_2 > T_{{}_{\!}*}$ so that $V\!\! \in\! \mathcal{C}^{r}(T_2)$ satisfies \eqref{smwavedef}, which
  contradicts the fact that $T_{{}_{\!}*}$ was maximal.
\end{proof}

\section{Elliptic estimates}
\label{ellsec}
In what follows we will need several elliptic estimates, which
are modifications of the estimates from \cite{Lindblad2016} and \cite{Nordgren2008}. We
summarize these here, and their proofs can be found in Appendix \ref{elliptic}.

Let $V\!: \![0,T]\! \times\! \Omega \!\to\!\R^3$ be a vector field on $\Omega$ and let
$\xve$ denote its smoothed flow as in \eqref{xvedef}, and let $A_{\m a}^i $ and $ A_{\m i}^{a}$
be as in \eqref{udef}.
We will assume that we have the following a priori bound:
\begin{equation}
 \tsum_{i,a}|A_{\m i}^a| + |A_{\m a}^i| + {\tsum}_{|I| \leq 3} |\pa_y^I \xve| \leq M_0,
 \label{uwassump}
\end{equation}
and in some of our estimates we will additionally assume the bound:
\begin{equation}
 \tsum_{i,a}|A_{\m i}^a| + |A_{\m a}^i| + {\tsum}_{k + |J| \leq 3}
 |\pa_y^J \xve|
 + |\pa_y^J D_t^k V| \leq M.
 \label{ubd2}
\end{equation}
We write $\pave$ for the
derivative with respect to $\xve$ (as in \eqref{pavedef}) and $\Dve$ for the
Laplacian with respect to $\xve$. For a one-form $\alpha = \alpha_i d\xve^i$
on $\tD_t$, we define $\div \alpha, \curl \alpha$ by \eqref{divdef}.
We will work with the following mixed norms:
\begin{equation}
 || f ||_{k,\ell} = {\sum}_{s \leq k}
 ||D_t^s f||_{H^\ell(\Omega)}, \qquad
 || f ||_r = {\sum}_{k + \ell \leq r} || f ||_{k,\ell}.
 \label{mixednormdef}
\end{equation}
In this section we let $C_0$, $C_s$ for $s\geq 1$ and $C_s'$  denote  continuous functions of arguments indicated below
\begin{equation}\label{eq:constantssection5}
 C_0=C_0(M_0),\qquad C_s=C_s(M_0, ||\xve||_{H^s(\Omega)}),\qquad C_s'=C_s'(M, ||\xve||_{s}), \quad\text{for}\quad s\geq 1.
\end{equation}

As in \cite{CL00} and \cite{Nordgren2008}, we will rely
on the following simple pointwise estimate:
\begin{lemma}
 With the norm of the tangential derivatives $\T\alpha$ as in \eqref{eq:simplifiedtangentialnotation},
 for every (0,1)-tensor $\alpha$ on $\Omega$
 \begin{equation}
  |\pave \wta | \leq C_0\big(|\div \alpha|
  + |\curl \alpha| + |\T \wta|\big).
  \label{ellpw}
 \end{equation}
\end{lemma}
See Lemma \ref{app:pwdiff} for the proof.
Then \eqref{ellpw} can be used to prove (see Proposition \ref{app:sobdiff}):
\begin{lemma}
  Let $s=k+\ell \!\geq\! 1$.
  If $\alpha$ is a (0,1)-tensor on $\Omega$ then, with notation as in
 \eqref{eq:simplifiedtangentialnotation}:
 \begin{equation}
  ||\wta||_{H^{s}(\Omega)}
  \leq C_s \bigtwo(||\div \alpha||_{H^{s-1}(\Omega)} +
  ||\curl \alpha||_{H^{s-1}(\Omega)}
  +  {\sum}_{j \leq s} ||\T^j \wta||_{L^2(\Omega)}\bigtwo),
  \label{sobtensor}
  \end{equation}
  \begin{equation}
   ||\wta||_{k,\ell} \leq C_s'\bigtwo(
   ||\div \alpha||_{k,\ell-1} + ||\curl \alpha||_{k,\ell-1}
   + {\sum}_{k_1 \leq k,\ell_1 \leq \ell}||\FD^{k_1,\ell_1} \wta||_{L^2(\Omega)}\bigtwo).
   \label{sobtensordt}
  \end{equation}
\end{lemma}

If $\wta_i = \pave_i f$
for a function $f$ which vanishes on $\pa \Omega$,
using an
integration-by-parts argument to control the last term on the
right-hand side of \eqref{sobtensor} (resp. \eqref{sobtensordt}) gives
(see Proposition \ref{app:sobfndiff}):
\begin{prop}
 If $f \!\!: \!\Omega \!\to \!\R$ is a function with
 $\!f\!\! =\! 0$ on $\pa \Omega$ then, with $\T \widetilde{\!x\,}\!\!$ defined in
\eqref{eq:simplifiedtangentialnotation},  for $k\!+\!\ell\!=s\! \geq \!1$:
 \begin{equation}
  ||\pave f||_{H^{s}(\Omega)}
  \leq C_s \big( ||\Dve f||_{H^{s-1}(\Omega)} +
  (||\T\xve||_{H^s(\Omega)} +
  ||\xve||_{H^s(\Omega)})||f||_{L^2(\Omega)}\big),\label{sobell}
\end{equation}
\begin{equation}
 ||\pave f||_{k,\ell} \leq C_s' \big(
 ||\Dve f||_{k,\ell-1} + (||D_t \xve||_s + ||\xve||_s)||D_t^k f||_{L^2(\Omega)}\big).
 \label{sobellmix}
\end{equation}
\end{prop}

There are two crucial points in the estimate
\eqref{sobell}. First, we are estimating
$\pave{}_{\!} f$ in $H^{{}_{\!}s{}_{\!}}(\Omega)$ instead of $\!f$ in $H^{s+{}_{\!}1{}_{\!}}(\Omega)$.
For the proof of this we only need to commute the divergence with $s\!-\!1$ instead of $s$ derivatives
with the Laplacian, which would have generate terms with too many derivatives
of $\xve$. Moreover, by first applying \eqref{sobtensor}, we
can replace full $y$-derivatives of $\pave{}_{\!} f$ with tangential derivatives
applied to $\pave{}_{\!} f$. This is why the right-hand side of \eqref{sobell}
involves $||\T \xve||_{H^s(\Omega)}$, which we can control more easily than $||\xve||_{H^{s+1}(\Omega)}$.

We also use \eqref{sobtensor} to
prove the following estimates.
They show that one can control
$\alpha$ in the interior by the divergence and curl of $\alpha$
and either the normal component of $\alpha$ on the boundary or the projection
of $\alpha$ to the tangent space at the boundary.
The first estimate will be used to control $||\sm x||_{H^r(\Omega)}$
in terms of the energies that we define in
Section \ref{energy} and the second will be used to control
$||V||_{H^r(\Omega)}$.
\begin{prop}
  \label{ftprop}
  Fix $r \!\geq\! 5$, $1 \!\leq \!s \!\leq\! r$. If $\alpha$ is a vector field, then
 with notation as in \eqref{eq:simplifiedtangentialnotation} and
  $H^s\! = H^s(\Omega)$:
\begin{equation}
||\wta||_{H^{s{}_{\!}}}^2
\!\leq \! C_{s\!}\Big(\!||\!\div \alpha||_{H^{s\!-\!1\!}}^2\! +\!
 ||\!\curl \alpha||_{H^{s\!-\!1\!}}^2 +\!||\wta||^2_{H^1}
 \!+\!{\sum}_{\mu = 1}^N  \!\int_{\pa \Omega} \!\!({}_{\!}\fdhm \T^{s-{}_{\!}1 \!}\wta^i)\!
 \cdot\!(\fdhm \T^{s-{}_{\!}1 \!}\wta^j)
 N_{{}_{\!}i} N_{\!j} dS
\Big),
\label{ellft1}
\end{equation}
\begin{equation}
||\wta||_{H^{s{}_{\!}}}^2
\leq C_{s\!}\Big( \!||\!\div \alpha||_{H^{s\!-\!1\!}}^2\! +\!
 ||\!\curl \alpha||_{H^{s\!-\!1\!}}^2\! +\!||\wta||^2_{H^1}\!
 + \!{\sum}_{\mu = 1}^N \!\int_{\pa \Omega} \!\!(\fdhm \T^{s -{}_{\!}1\!}\wta^i)\!\cdot\!(\fdhm \T^{s-{}_{\!}1\!} \wta^j)
 \gamma_{ij} dS\!\Big).
\label{ellft2}
\end{equation}
Here $\gamma$ denotes projection to the tangent space at the boundary,
for the definition of $\fdhm\!\!$, see Appendix \ref{tangapp}.
\end{prop}

\subsection{Estimates for differences of solutions}
In Section \ref{vlwpsec}, we will prove that the map $\Lambda$
defined in \eqref{lambdadef} satisfies a Lipschitz estimate. Given two
vector fields $\Vu{}_{\!}, \!\Vw\!:\![0,\!T] {}_{\!}\times{}_{\!} \Omega \to \R$,
define the corresponding smoothed flows $\xveu, \xvew$ as well as
the derivatives
$\paveu, \pavew$ and the Laplacians $\Dveu, \Dvew$. Assume that
$\xveu, \xvew$ both satisfy the estimate \eqref{uwassump} or \eqref{ubd2}
and now let the constants \eqref{eq:constantssection5} depend
on the corresponding norms of both $\xveu, \xvew$.
\begin{prop}
  \label{sobfndiff}
  Fix $r\! \geq\!6$. If $f, g\!:\! \Omega \!\to \!\R$ and $f\! \!=\! g\! =\! 0$ on $\pa \Omega$,
 then
 for $1 \!\leq\!\ell\! \leq \!r\!-\!1$, respectively $k\!+\!\ell\!=\!r$:
 \begin{equation}
  ||\paveu {}_{\!}f\! - \pavew g||_{H^{\ell_{\!}}(\Omega)\!}
  \leq C_r\big({}_{\!} ||\Dveu f\! - \Dvew g||_{H^{\ell\!-\!1\!}(\Omega)}\!
  + ||\xveu||_{H^{r\!}(\Omega)} ||f \!- g||_{L^2(\Omega)}\!
  + ||\xveu\! - \xvew ||_{H^{r\!}(\Omega)} ||\pavew g||_{H^{\ell\!}(\Omega)} {}_{\!}\big),
 \end{equation}
 \begin{equation}
  ||\paveu {}_{\!}f \!- \pavew g||_{H^{r_{\!}}(\Omega)\!}
  \leq C_{r{}_{\!}}\bigtwo({}_{\!} ||\Dveu {}_{\!}f \!- \Dvew g||_{H^{r-_{\!}1\!}(\Omega)\!}
  +{}_{\!} ||\T{}_{\!} \xveu{}_{\!} ||_{H^{r_{\!}}(\Omega)} ||f\! - g||_{L^{{}_{\!}2}(\Omega)\!}
  + ||\T{}_{\!} \xveu\!-\! \T{}_{\!} \xvew{}_{\!} ||_{H^{r_{\!}}(\Omega)}
  ||\pavew g||_{H^{r_{\!}}(\Omega)}{}_{\!}\bigtwo),
 \end{equation}
 \begin{equation}
  ||\paveu {}_{\!}f \!- \pavew g||_{k,\ell}
  \!\leq\! C_{{}_{\!}r}' \big({}_{\!} ||\Dveu{}_{\!}  f {}_{\!} - \Dvew g||_{k,\ell{}{\!}-\!1\!} + \big({}_{\!}||\xveu{}_{\!} ||_{r }+
  ||{}_{\!}D_t \xveu{}_{\!} ||_r{}_{\!}\big)||f\!-g||_{k,0}
  + \big({}_{\!}||\xveu{}_{{}_{\!}}  -\xvew\!||_r + ||{}_{\!}D_{{}_{\!}t}(\xveu {}_{\!}- \xvew{}_{\!})||_{r\!} \big)
  ||\pavew g||_{r}{}_{\!}\big).
 \end{equation}
\end{prop}

\section{Estimates for wave equations}
\label{waveests}
As in the previous section,
we fix a vector field $V = V(t,y)$ on $\Omega$
and let $\xve(t,y)$ denote the tangentially smoothed flow of $V$.
Define $A, \kve, \Dve$ as in \eqref{udef}-\eqref{laplsm}.
We will assume that the a priori assumptions \eqref{ubd2}
hold
Note that \eqref{ubd2} combined with the formula for the derivative
of the inverse \eqref{dinv} implies that $|\pa_y^\ell A_{\m i}^a| \leq C(M)$
for $\ell \leq 2$.
We consider the initial-boundary value problem:
\begin{align}
  \eprime D_t^2 \varphi - \Dve \varphi &= \F, \quad \textrm{ on } [0,\!T] \!\times\!
  \Omega,\quad\text{with}\quad
  \varphi = 0 ,\quad\textrm{ on } [0,\!T] \!\times \!\pa \Omega,\label{phiwave}\\
  \varphi(0,y) &= \varphi_0(y), \qquad D_t \varphi(0, y) = \varphi_1(y), \quad\textrm{ on } \Omega,
  \label{phiwaveic}
\end{align}
where $\F$ is a given function on $[0,T] \!\times \!\Omega$
and
$\eprime\! =\! \eprime(\varphi)$ is a given function satisfying
$0 \!<\! e_1\! \leq \!\eprime\! \leq\! e_2$ for some $e_1, e_2$.
We will suppress the dependence
on $e_1, e_2$ in the following.
In our applications we will have $\eprime = e'(\varphi)$ where $e(\varphi)$
is determined from the equation of state as in Section \ref{enthalpysec}.
We remark that for a linear equation of state $p(\rho) = \rho + c$, we have
$e(\varphi) = \varphi + c$ and so in this case \eqref{phiwave} is a linear
wave equation.

 For the applications we have in mind, we will need to
 allow $\F$ to depend on $\varphi$:
 \begin{equation}
  \F(t,y) =  \F_1(t,y) + \F_2[\varphi,D_t \varphi](t,y),
  \label{Fsplit}
 \end{equation}
 where we assume that $\F_2[\varphi,D_t \varphi](t,y)$ is a nonlocal functional of $(\varphi,D_t\varphi)$ that satisfies the following:
 \begin{equation}
  ||D_t^s\F_2[\varphi,D_t \varphi]||_{L^2(\Omega)}
  \leq P_1(||\varphi||_{s+1,0} + ||\varphi||_s),
  \qquad
  ||\F_2[\varphi,D_t \varphi]||_{s-1}
  \leq P_2 ||\varphi||_{s},
  \label{F2est}
 \end{equation}
 for some polynomials $P_1, P_2$
 depending on $M, L, ||\xve||_{H^s},
 ||V||_{\X^s}, ||\varphi||_{s,0}, ||\varphi||_{s-1}$.
 Recall that the mixed norms $||\cdot||_{k,\ell}$ and $||\cdot ||_s$
 are defined in \eqref{mixednormdef}.
 In Section \ref{enthsec}, we will take $F_2 = e''(\varphi) (D_t\varphi)^2
 + \rho[\varphi]$
 where $\rho$ is determined from $\varphi$ by the equation of state
 as in Section \ref{enthalpysec}, and we will see that this satisfies \eqref{F2est}.

The energy associated to the wave equation \eqref{phiwave} is:
\begin{equation}
 W_{\!s}(t)= \Big(\frac{1}{2}{\sum}_{k \leq s }\int_\Omega\big(
 \eprime(t) \,|D_t^{k+1} \varphi(t)|^2 + \delta^{ij}
 \big( D_t^k \pave_i  \varphi(t)\big)\big(D_t^k\pave_j \varphi(t)\big)
 \big)\,\kve dy\Big)^{1/2}.
\end{equation}
Fixing $T > 0$ and $ s \geq 0$, we will consider solutions $\varphi$ to \eqref{phiwave}-\eqref{phiwaveic}
in the following space:
\begin{equation}
\H_T^{s+1} = \{ \varphi: [0,T] \times \Omega\to \R |
 D_t^k \varphi \in L^\infty\big([0,T],H^{s+1-k}(\Omega)\big), k = 0,..., s+1, 0 \leq t \leq T\}.
 \label{phispace}
\end{equation}

In order for the wave equation \eqref{phiwave}-\eqref{phiwaveic} to have a solution
$\varphi \in \H_T^{s+1}$, the initial data \eqref{phiwaveic} needs to satisfy
compatibility conditions. These are the conditions that there is a formal
power series solution in $t$: $\widehat{\varphi} = \sum \varphi_k t^k$ to \eqref{phiwave}
at $t = 0$ so that $\varphi_k|_{\pa \Omega} = 0$ for $k = 0,..., s$. See \eqref{nlincc'}.

We assume that we have the following a priori estimate for $\varphi$:
\begin{equation}
 {\tsum}_{k + |J| \leq 3} | D_t^{k} \pa_y^J \pave \varphi (t)|
 + |D_t^k \varphi(t)|
 \leq L, \quad \text{ in } [0,T] \times \Omega.
 \label{Lassumpwave}
\end{equation}
This assumption is needed because in general \eqref{phiwave} is a nonlinear
equation. If the equation of state is such that $e'(\varphi)$ is constant,
all of our results hold without \eqref{Lassumpwave}. We also assume the bound for
$M\!=\!M[\widetilde{x}]$ in \eqref{ubd2}.

The first goal of this section is to prove the following theorem:
\begin{theorem}
  \label{mainwavethm}
Fix $s \geq 0$. There are continuous functions
$G_s$, with
\begin{equation}
  G_{\!s} \!=\! G_{\!s}\big(M, L, T, W_{\!s-1}(0),\,
  {\sup}_{\,0 \leq t \leq T\,} (||\xve(t)||_{H^s(\Omega)}
  + ||V(t)||_{\X^s} + ||F_1(t)||_{s-2}) \big),
\end{equation}
so that if $\varphi\in \H^{s+1}_T$ satisfies \eqref{phiwave}, \eqref{Lassumpwave} holds,
then for $0\leq t\leq T$:
 \begin{equation}
  ||\varphi(t)||_{s+1,0}  + ||\pave \varphi(t)||_{s,0}
  \leq G_s \Big( W_s(0) + \int_0^t ||F_1(\tau)||_{s,0}
  + ||F_1(\tau)||_{s-1} + ||V(\tau)||_{\X^{s+1}}\, d\tau
   \Big),
  \label{dtpsibd}
 \end{equation}
 \begin{equation}
  ||\pave \varphi(t)||_{s}  \leq G_s(||\T \xve||_s + 1)\Big(||F_1(t)||_{s-1} +
  W_s(0) + \int_0^t ||F_1(\tau)||_{s,0}
  + ||F_1(\tau)||_{s-1} + ||V(\tau)||_{\X^{s+1}}
  \, d\tau
  \Big).
  \label{allpsibd}
\end{equation}
\end{theorem}

\begin{proof}
  When $s = 0$, by Lemma \ref{dtW} there is a continuous function $G'_0 = G'_0(M)$
  and a polynomial $P_0$ so that:
  \begin{equation}
   \frac{d}{dt} W_0 \leq G'_0
   \big( ||F_1||_{L^2(\Omega)}
   +  P_0(L, ||\varphi||_{L^2(\Omega)} )W_0\big).
   \label{y0}
  \end{equation}
  By Poincar\'{e}'s inequality and H\"{o}lder's inequality,
  we have $||\varphi||_{L^2(\Omega)}
  \leq C(M) ||\pave \varphi||_{L^2(\Omega)} \leq C(M)L$.
  Multiplying both sides of \eqref{y0} by the integrating
  factor $e^{-(G_0' +  P_0(L, C(M)L) )t)}$ and integrating, we get:
  \begin{equation}
   W_0(t) \leq G_0\Big( W_0(0)
   + \int_0^t ||F_1(\tau)||_{L^2(\Omega)} \, d\tau\Big),
  \end{equation}
  for a continuous function $G_0 = G_0(M, L, T)$.
  We now assume that we have the result for $s = 0, ..., m-1$.
  Let $G'_m$ be as in Lemma \ref{dtW}. By the inductive
  assumption, \eqref{dtWest} gives:
  \begin{equation}
   \frac{d}{dt} W_m \leq G'_m
   \Big( \big(1 + P_m(L, W_{m-1}(0), {\sup}_{\,0 \leq t \leq T\,} ||F_1||_{m-2})\big) W_m
   + ||F_1||_{m,0} + ||F_1||_{m-1} + ||V||_{\X^m} \Big),
  \end{equation}
  for a polynomial $P_m$,
  and so multiplying by the integrating factor
  $e^{-G'_m(1 + P_m) t}$ and integrating gives the result
  for $s = m$ as well. The estimate \eqref{allpsibd} then follows from
  \eqref{dtpsibd} and \eqref{alldt1}.
\end{proof}
We have the following energy estimate proven in Section \ref{sec:F3}:
\begin{lemma}
  \label{dtW}
  For each $s \geq 0$,
  there is a continuous function $G'_s(t) \!=\! G'_s( M,
  ||\xve(t)||_s, ||V(t)||_{\X^{s}},W_{s-1}(t))$ and a polynomial $P$ so that
  if \eqref{phiwave}-\eqref{Lassumpwave} hold, then:
  \begin{equation}
   \frac{d}{dt} W_s\leq G'_s \Big(
   W_s + ||F_1||_{s,0} + ||F_1||_{s-1} + ||V||_{\X^{s+1}}  +
    P(L, W_{s-1}, ||F_1||_{s-2})W_{s}\Big).
   \label{dtWest}
  \end{equation}
\end{lemma}
Theorem \ref{mainwavethm} relies on the  following consequence of the elliptic estimate \eqref{sobellmix} and is proven in Section \ref{sec:F3}:
\begin{lemma}
 There is a continuous function
  $G_s'' \!=\! G_s''(M,{}_{\!} ||\xve||_s)$ and $P_{\!s}$
  so that if
  \eqref{phiwave}-\eqref{Lassumpwave} hold, then:
\begin{equation}
 ||\pave \varphi||_s \leq G_s''(||\T \xve||_{H^s} + ||V||_{s})
  \big( ||\varphi||_{s+1,0} +
 ||\pave \varphi||_{s,0}
 + ||F||_{s-1}
 + P_s(L,||\varphi||_{s,0},
 ||\pave \varphi||_{s-1,0}, ||F_1||_{s-2})\big).
 \label{alldt1}
\end{equation}
\end{lemma}

The following result is used in Section \ref{vlwpsec}
to replace the assumption \eqref{Lassumpwave} with an assumption at $t = 0$:
\begin{cor}
  \label{bootstrapped}
  Fix $r \geq 7$. Suppose that for some $T_1, K > 0$, we have the following
  estimate:
  \begin{equation}
   {\sup}_{\,0 \leq t \leq T_1}
   \big(||\xve(t)||_{r} +
   ||V(t)||_{\X^{r+1}} +
   ||F_1(t)||_{r,0} + ||F_1(t)||_{r-1} \big)\leq K,
   {}
  \end{equation}
  and that \eqref{phiwave}-\eqref{phispace} hold  on $[0,T_1]$,
 and
  \begin{equation}
   {\tsum}_{k + |J| \leq 3} |( D_t^k\pa_y^J \pave\varphi)(0)|
   + |D_t^k \varphi(0)| \leq L_0.
   {}
  \end{equation}
There are continuous functions $Q_r$
and $\underline{G}_r = \underline{G}_r(M, L_0, W_5(0),K)$ so that if $T$ satisfies:
\begin{equation}
 T Q_r(M, L_0, W_5(0), K, T_1) \leq 1, \quad \text{ and }
 T \leq T_1,
 \label{Tgr}
\end{equation}
then for $0 \leq t \leq T$ the estimate
\eqref{allpsibd} holds with $G_r$ replaced with $\underline{G}_r$.

\end{cor}
\begin{proof}
  Let $L(t) = \sum_{|J| + k \leq 3} | D_t^k \pa_y^J \pave \varphi(t)|
  + |D_t^k \varphi(t)| $. By Sobolev embedding
  $L(t) \leq C (||\pave \varphi(t)||_{5} + ||\varphi(t)||_5)$. Using
  the product estimate \eqref{product} we have
  $||\varphi||_5 \leq C'(M, ||\xve||_{H^7(\Omega)}, ||V||_{\X^7})
  ( ||\varphi||_{5,0} + ||\pave \varphi||_{4})$. Integrating once in time
  and then using the estimates \eqref{allpsibd} and \eqref{dtpsibd},
  we have:
  \begin{equation}
   L(t) \leq L(0) + C' \int_0^t
    ||\varphi(\tau)||_{6,0} + ||\pave \varphi(\tau)||_6\, d\tau
    \leq L(0) + T P_0,
    \label{bootstrap}
  \end{equation}
  for a polynomial $P_0$ with $P_0 = P_0\big(M, L, W_5(0), \sup_{\,0 \leq t \leq T}
  \big(||\xve(t)||_{H^7(\Omega)} + ||V(t)||_{\X^7} + ||F_1(t)||_7\big)\big)$.

We take $T \leq T_* \equiv \min(T_0, T_1)$ with $T_0$ defined by:
 \begin{equation}
  T_0 P_0 (M, 2L_0, W_5(0), K_1) \leq  L_0/2.
  {}
 \end{equation}
 Let $S \!=\! \{ 0\! \leq \!t \!\leq \!T_*\! : L(t) \leq 2L_0\}$.
 Then $S$ is nonempty, connected and closed. If $t \in S$
 is an interior point then
  \eqref{bootstrap} and the fact that $t \leq T_0$
 shows that $t + \delta \in S$
  for sufficiently small $\delta$,
so
 $L \leq 2 L_0$  for $t \leq T_*$. The result now follows
 from Theorem \ref{mainwavethm} with $\underline{G}_r \!= G_r(M, 2L_0, T_r, W_{r-1}(0), K_1)$
 and $Q_r\! = \max(1,(2L_0)^{-1}) P_0$.
\end{proof}

\subsection{Estimates for differences of solutions}

We will also need to prove a Lipschitz
estimate for $\Lambda$. We fix two vector fields $V_{\!I},\! V_{{}_{\!}\II}$
and, defining $\xve_{{}_{\!}J\!}, \pave_{{}_{\!}J}, \Dve_{{}_{\!}J}$ with $J{}_{\!}\! =\! I{}_{\!}, \II$
as in \eqref{udef}-\eqref{laplsm}, we consider
solutions $\varphi_{{}_{\!}J}$ to
\begin{equation}
 \eprime_J D_t^2 \varphi_J -\Dve_J \varphi_J = \F_J,\quad \textrm{ in } [0,\!T] \!\times\! \Omega, \qquad\text{with}\quad
 \quad \varphi_J = 0, \text{ on } [0,\!T] \!\times\! \pa \Omega,\quad\text{for }
 J\!=\!I,\II,
 \label{phijwave}
\end{equation}
with the same initial data  \eqref{phiwaveic}.
Here $F_{\!J} \!=\! F_{\!J}^1 \!+ \!F_{\!J}^2[\varphi_J,D_t \varphi_J]$ as in \eqref{Fsplit},
where we assume that $F_{\!J}^2[\varphi_J]=F_{\!J}^2[\varphi_J,D_t \varphi_J]$ satisfies \eqref{F2est} and:
\begin{equation}
 ||D_t^{s \!}(F_I^2[\varphi_{{}_{\!}I{}_{\!}}] - F_{\!\II}^2[\varphi_{\II{}_{\!}}])||_{L^2}
 \!\leq \!\overline{P}_{1\!} ||\varphi_{{}_{\!}I{}_{\!}} - \varphi_{\II{}_{\!}}||_{s+\!1\!,0},
 \quad\,\,\,
 ||F_I^2[\varphi_{{}_{\!}I{}_{\!}}] - F_{\!\II}^2[\varphi_{\II{}_{\!}}]||_{s\!-\!1}
 \!\leq \! \overline{P}_{2} ( ||\varphi_{{}_{\!}I{}_{\!}} - \varphi_{\II{}_{\!}}||_{s, 0}
 + ||\paveu \varphi_{{}_{\!}I{}_{\!}} - \pavew \varphi_{\II{}_{\!}}||_{s\!-\!1}),
\end{equation}
where $L^2\!\!=\!L^2(\Omega)$, and $\overline{P}_{\!1}, \overline{P}_{\!2}$ depend on $M$ and $||\xve_J||_{H^{s}(\Omega)},
||V_J||_{\X^s}, ||\pave \varphi_J||_s$, $J = I, \II$, and $L$, where
\begin{equation}
 {\tsum}_{k + |M| \leq 3} |D_t^k \pa_y^{M}\pave \varphi_J|
 + |D_t^k \varphi_J|
 \leq L,
 \qquad \textrm{ in } [0,T] \times \Omega,\quad \text{for}\quad J = I,\II. \label{apriori varphi}
\end{equation}
 Writing $\psi = \varphi_{\I} - \varphi_{\II}$, we have that:
\begin{equation}
  \eprime_I D_t^2 \psi- \Dveu \psi = \F_{\I} - \F_{\II} +
 (\Dveu - \Dvew) \varphi_{\II} + (\eprime_{\I} - \eprime_{\II}) D_t^2 \varphi_{\II},
 \textrm{ in }
  [0,\!T] \!\times\!\Omega,\quad \text{with}\quad
 \psi = 0, \textrm{ on }  [0,\!T] \!\times\!\pa\Omega,
\end{equation}
and that $\psi|_{t = 0} = D_t \psi|_{t = 0} = 0$.
In Lemma \ref{dtphidiffnew} we prove estimates for $\psi$ that are similar
to the estimates in Theorem \ref{mainwavethm} for $\varphi_I, \varphi_{\II}$.
 However because of the terms $(\Dveu - \Dvew)\varphi_{\II}$ and
 $(\eprime_{\I} - \eprime_{\II}) D_t^2 \varphi_{\II}$ we will need to assume an estimate for one more
 derivative of $\varphi_{\II}$ than we get back for $\psi$.

With notation as in the beginning of this section,
we assume that both
$(u, \Vu)$ and $(w, \Vw)$ satisfy the a priori assumption \eqref{ubd2}, and we
also assume that $\varphi_I, \varphi_{II}$ satisfy
\eqref{apriori varphi}.
We define:
\begin{equation}
 ||\alpha||_{C^k_{x,t}(\Omega)} = {\tsum}_{k_1 + k_2 \leq k}
 ||\pa_y^{k_1} D_t^{k_2} \alpha||_{L^\infty(\Omega)}.
\end{equation}

We then have the following energy estimate for differences proven in section \ref{sec:F3}:
\begin{lemma}
 \label{dtphidiffnew}
 Suppose that $(A_I\!, \!\Vu{}_{\!}), (A_{\II}\!,\! \Vw{}_{\!})$  satisfy \eqref{ubd2}
 and that $\varphi_{I\!}, \varphi_{\II}$ satisfy the wave equation
 \eqref{phiwave}-\eqref{phiwaveic}
 with $\Dve$ replaced with $\Dveu, \Dvew$ and  $\F_2$ (defined
 by \eqref{Fsplit}) replaced
 with $\F_2[\varphi_I], \F_2[\varphi_{\II}]$, respectively.
 Define:
 \begin{equation}
  W_s^{I\!,\II}(t) = \Big( \frac{1}{2} {\sum}_{k \leq s}
  \int_\Omega e'(\varphi_I)
  |D_t^{k+1} (\varphi_I - \varphi_{II})|^2
  + |D_t^k \paveu(\varphi_I - \varphi_{II})|^2\, \kve dy\Big)^{1/2}.
 \end{equation}
 For each $s \geq 0$,
 there is a positive, continuous function
 $D_s$ depending on
 \begin{equation}
   M, L, T, W_{s-1}(0), \text{ and }\,\,\,
   {\sup}_{\,0 \leq t \leq T\,}\big(
   ||\xve_J(t)||_{H^{s+1}} +
   ||V_J(t)||_{\X^{s+2}}
   + ||F_J^1(t)||_{s+1}\big), \quad J = I, \II,
 \end{equation}
 so that:
 \begin{equation}
   W_s^{I\!,\II\!}(t) \leq D_{\!s}\!\!
  \int_0^t \!\!
  ||F_{{}_{\!}I}^1(\tau)_{\!} -_{\!} F_{{}_{\!}\II}^1(\tau)||_{s,0}
  + ||F_{{}_{\!}I}^1(\tau)_{\!} -_{\!} F_{{}_{\!}\II}^1(\tau)||_{s-\!1}
  + ||\Vu(\tau)_{\!} -_{\!} \Vw(\tau)||_{s+1} + ||\xveu(\tau)_{\!} -_{\!} \xvew(\tau)||_{C^3_{x,t}}\!
   d\tau,
  \label{energy esti y^12}
 \end{equation}
 \begin{equation}
  ||\paveu \varphi_I - \paveu\varphi_{II}||_{s}
  \leq D_s \big( (||\T(x_I - x_{II})||_{H^s(\Omega)} + ||\xveu - \xvew||_{C^3_{x,t}} +
  1)
  + ||F_I^1 - F_{II}^1||_{s-1} + W_s^{I,II}\big).
  \label{elliptic y^12}
 \end{equation}
\end{lemma}

\section{Estimates for the gravitational potential}
\label{gravitysection}

The estimates for $\phi$ will require integration by parts on
$\tD_t \equiv \xve(t, \Omega)$ and this will yield a boundary term
which is difficult to deal with because there is no boundary
condition for $\phi$ on $\pa \tD_t$. Because of this, following
\cite{Lindblad2008}, our strategy is to extend the domain $\tD_t$ in the radial
direction to a set $\hD_t$ (see \eqref{hD_t def}), and then approximate
$\phi$ with a sequence of functions $\phi_m$ defined $\hD_t$
in such a way that $\T\Delta \phi_m$ and $D_t \Delta \phi_m$ vanish outside
of $\tD_t$. We will also see that $\Phi(x-z)\in L^2_z$ if $x\in \pa\hD_t$
and $z\in \hD_t$ and these facts will allow us to bound
$||\T^s \pa \phi_m||_{L^2(\hD_t)}$ for each $m$ after integrating by
parts.
We will also show that the sequence
$\{\T^s \pa\phi_m\}_{m=1}^\infty$ is a Cauchy sequence in $L^2(\hD_t)$,
which gives an estimate for
 $||\T^s \pa \phi||_{L^2(\tD_t)}$ by letting $m\rightarrow\infty$. Finally, we
  get an estimate for $||\phi||_{H^{s+1}(\tD_t)}$ using \eqref{sobtensor}.

  Similarly to Section \ref{ellsec}, we will let $C_s, C_s^\prime,
  C_s^{\prime\prime}, C_s^{\prime\prime\prime}$ denote
  continuous functions with:
  \begin{equation}
    C_{\!s} \!= \!C_{\!s}(_{\!}M_0,_{\!} ||\xve||_{H^s}\!),
   \quad
   C_{\!s\!}'\!\! =\! C_{\!s}'(_{\!}M_{\!},_{\!} ||\xve||_{s_{\!}}),
   \quad
   C_{\!s}^{\prime\prime} \!\!{}_{\!}=\! C_{s}^{\prime\prime}{}_{\!}(_{\!}M_0,_{\!} ||\xve||_{H^{(_{\!}s-_{\!}1_{\!},1_{\!}/2)}}\!),
   \quad
   C_s^{\prime\prime\prime}\!\!{}_{\!}=\!C_s^{\prime\prime\prime}{}_{\!}(_{\!}M_{\!},_{\!}
   ||\xve||_{H^{(_{\!}s-_{\!}1_{\!},1_{\!}/2)}}{}_{\!},_{\!}
   ||\xve||_{s_{\!}}),
   \label{cprimes}
  \end{equation}
  with $H^k = H^k(\Omega)$ and $H^{(k,1/2)} = H^{(k,1/2)}(\Omega)$
  defined by \eqref{sobspacedef}.
\subsection{Bounds for $\phi$ and the extended domain $\widehat{\Omega}$}
\label{extensionsec}
The following theorem is the main result of this section, and follows
from the elliptic estimate \eqref{sobtensor} and the upcoming
Theorem \ref{usual norms}.
\begin{theorem} \label{main theorem, ell est of phi}
  If $r\geq 5$, then with $\phi$ defined by \eqref{phidef}:
\begin{equation}
||\pave \phi||_{H^{r-1}(\tD_t)} \leq C_r\big(||\rho||_{H^{r-2}(\tD_t)}
+||\rho||_{H^{(r-2, 1/2)}(\tD_t)}\big),
\end{equation}
\begin{equation}
||\pave \phi||_{H^r(\tD_t)} \leq C_r (||\T \xve||_{H^{(r-1, 1/2)}(\Omega)} + 1 )\big(||\rho||_{H^{r-1}(\tD_t)}
+||\rho||_{H^{(r-1,1/2)}(\tD_t)}\big).
\end{equation}
\end{theorem}

We now employ the strategy mentioned above.
Fix $d_0\!>\!0$ and define
 $\Omega^{d_0}\!=\{y_1\!+y_2\in \R^3\!: y_1\!\in \Omega, |y_2|\!<\!d_0\}$.
  Let $\chi_m$ be a smooth radial function whose support is contained in
  $\Omega^{d_0/2}$ with $\chi_m(y)\!=\!1$ whenever $y\!\in\! \Omega$,
   where $m$ is taken large enough that ${1}/{m}\!\leq \!d_0/2$.
  For fixed $r \geq 7$, let $E$ denote an extension operator which is
  bounded from $H^r(\Omega)$ to
  $H^r(\Omega^{d_0})$ (see Appendix \ref{extension section} for the detailed construction of the extension operator $E$). Define
  $\xht(t,y)=\chi_{d_0\!/2}(y)E(\xve(t,\cdot)-x_0(\cdot))(y)+x_0(y)$,
   and define the corresponding velocity by $\widehat{V}\!\!=\!D_t \xht$. With
   these definitions, we have arranged that $\xht(t,y)\!=\!x_0(y)$ for $y\!\in \! \pa\Omega^{d_0}\!\!$,
   and for some $0\!<c\!<\! C\!<\!\infty$:
\begin{equation}
c||\xve||_{H^s(\Omega)}\leq ||\xht||_{H^s(\Omega^{d_0})}\leq C||\xve||_{H^s(\Omega)},\qquad
 ||\T\xht||_{H^s(\Omega^{d_0})}\leq C||\T\xve||_{H^s(\Omega)},
\qquad 0 \leq s \leq r\label{equiv extended norm},
\end{equation}
by Theorem \ref{comparable thm} and similarly for $\widehat{V}$.  Abusing notation, we will also
  let $\chi_m$ denote the analogous function in the Eulerian frame
  and write $\chi_m(\xht)=\chi_m(\xht(t,y))$. We use $\paht$ to denote
   the derivative with respect to $\xht$  and  $D_t=\pa_t+\Vht^k \paht_k$ to denote the material derivative in $\hD_t$.

Assuming that \eqref{uwassump} holds, then taking $d_0$ smaller if necessary,
$\xht(t,\cdot)$ is a homeomorphism from $\Omega^{d_0}$ to $\hD_t$ and the normal
$N$ to $\p\tD_t$ can be extended continuously into the region
between $\pa \tD_t$ and $\p\hD_t$, where:
\begin{equation}
\hD_t=\xht(t, \Omega^{d_0}).
\label{hD_t def}
\end{equation}

 We want to establish an approximation scheme which allows us to control $\phi$.
 Let $\Phi$ be the fundamental solution of the Laplacian and let $\hrho(x) = \rho(x)$
 if $x\in \tD_t$, $\hrho(x) = \bar{\rho}$ if $x\notin \tD_t$. We define
\begin{equation}
\phi_m(t,x) = -\hrho\chi_{m}*\Phi(x),\qquad x\in\widehat{\D}_t.
\label{def wphi}
\end{equation}
 We will show that the sequence $\{\T^{j\!} \paht \phi_m\}_{m = 0}^\infty$ is Cauchy in $L^2(\hD_t)$, which
we will use to  control $||\T^j \pave \phi||_{L^2(\tD_t)}$. The fundamental result we need
is the following inequality, whose proof can be found in Section \ref{gravityproofs1}:
\begin{theorem} \label{ellphi'' 1}
  Fix $r \!\geq \!5$, suppose that \eqref{uwassump} holds and let $\Phi\!$ denote
  the fundamental solution of the Laplacian in $\R^3\!\!$.
  If $g$ is a smooth function supported  in $\xht(t, \Omega^{d_0\!/2})$
  such that $g(x)$ is radial when $x\!\in \!\hD_t\!\!\setminus\!\D_t$,
  then:
\begin{multline}
{\sum}_{j\leq r}||\T^j \paht (g * \Phi)||_{L^2(\hD_t)} \leq C_r^{\prime\prime} (||\T \xve||_{H^{(r-1,1/2)}(\Omega)} + 1)
\\\cdot\bigtwo(
{\sum}_{k\leq r-1}||\T^kg||_{L^2(\hD_t)}+||\T^{r-1} g||_{H^{(0,1/2)}(\hD_t)}+ {\sum}_{k\leq 2}||\T^kg||_{L^6(\hD_t)}
+||g||_{L^\infty(\hD_t)}
\bigtwo).
\label{boundphi T 1}
\end{multline}
\end{theorem}
Applying \eqref{boundphi T 1} to $g = \hrho (\chi_m -\chi_n)$ and using that
by construction $\T (\hrho(\chi_m\! - \chi_n)) \!=\! 0$, we have the following:
\begin{cor} \label{ellphi'''}
  With the same hypotheses and notation as in Theorem \ref{ellphi'' 1} and with
  $\phi_\ell$ defined by \eqref{def wphi}, set
    $\phi_{m,n} = \phi_m-\phi_n$ and $\chi_{m,n}=\chi_m-\chi_n$. Then
\begin{equation}
{\sum}_{j\leq r-1}||\T^j \paht \phi_{m, n}||_{L^2(\hD_t)} \leq
C_r^{\prime\prime}\big(||\hrho \chi_{m,n}||_{L^2(\hD_t)} + ||\hrho \chi_{m,n}||_{L^6(\hD_t)}
+  ||\hrho \chi_{m,n}||_{L^\infty(\hD_t)}\big),
\end{equation}
\begin{equation}
||\T^r\paht \phi_{m,n}||_{L^2(\hD_t)} \leq C_r^{\prime\prime}\big(||\T\xve||_{H^{(r-1, 1/2)}(\Omega)} + 1\big)
\big(||\hrho \chi_{m,n}||_{L^2(\hD_t)} + ||\hrho \chi_{m,n}||_{L^6(\hD_t)}
+  ||\hrho \chi_{m,n}||_{L^\infty(\hD_t)}\big).
\end{equation}
\end{cor}

Corollary \ref{ellphi'''} implies that the sequence  $\{\T^j \paht \phi_m\}_{m=1}^\infty$ is Cauchy in $L^2(\hD_t)$ so
$\T^j \paht \phi_m\rightarrow  \T^j \paht \phi$ in $L^2(\hD_t)$. This allows us to get a bound for $||\T^j \pave \phi||_{L^2(\tD_t)}$ from that of $||\T^j \paht \phi_m||_{L^2(\hD_t)}$.
Although $g\!=\!\hrho \chi_m$ is not smooth, by a regularization procedure in the radial and tangential directions
 \eqref{boundphi T 1} still holds.

 \begin{theorem} \label{usual norms}
 With the same hypotheses and notation as in Theorem \ref{ellphi'' 1},
 writing $\phi=-\rho\chi_{\D_t}*\Phi$, we have:
 \begin{equation}
  {\sum}_{ j\leq r-1}||\T^j \pave \phi||_{L^2(\tD_t)}\leq C_r^{\prime\prime}
  \big(||\rho||_{H^{r-2}(\tD_t)}+||\rho||_{H^{(r-2,1/2)}(\tD_t)}\big),
  \label{usual1}
 \end{equation}
\begin{equation}
  ||\T^r \pave \phi||_{L^2(\tD_t)}\leq C_r^{\prime\prime}\big( ||\T \xve||_{H^{(r-1, 1/2)}(\Omega)} + 1\big)
  \big(||\rho||_{H^{r-1}(\tD_t)}+||\rho||_{H^{(r-1,1/2)}(\tD_t)}\big).
  \label{usual2}
 \end{equation}
\end{theorem}
 \begin{proof}
Substituting $g=-\hrho \chi_m$ into \eqref{boundphi T 1}, we have:
\begin{multline}
{\sum}_{j\leq r}||\T^j \paht \phi_m||_{L^2(\hD_t)} \leq C_r^{\prime\prime} (||\T \xve||_{H^{(r-1,1/2)}(\Omega)} + 1)
\\\cdot\bigtwo(
{\sum}_{k\leq r-1}||\T^k (\hrho\chi_m)||_{L^2(\hD_t)}+||\T^{r-1} (\hrho\chi_m)||_{H^{(0,1/2)}(\hD_t)}+ {\sum}_{k\leq 2}||\T^k (\hrho\chi_m)||_{L^6(\hD_t)}
+||\hrho\chi_m||_{L^\infty(\hD_t)}
\bigtwo).
\end{multline}
Since the right hand side involves only tangential derivatives and because $\hrho\chi_m\rightarrow \rho\chi_{\tD_t}$ as $m\to\infty$, we have that $||\T^{r-1} (\hrho\chi_m)||_{H^{(0,0.5)}(\hD_t)}\to ||\T^{r-1} \rho||_{H^{(0,0.5)}(\tD_t)}$ and
$||\T^k (\hrho\chi_m)||_{L^2(\hD_t)}\to ||\T^k \rho||_{L^2(\tD_t)}$
for $k \leq r-1$, and these are both bounded by the right-hand side of \eqref{usual1}
(resp. \eqref{usual2}).
Similarly, for $k \leq r-1$ we have
 $ ||\T^k (\hrho\chi_m)||_{L^6(\hD_t)}\to  ||\T^k \rho||_{L^6(\tD_t)}$
 and by the Sobolev lemma, this last term is bounded by the right-hand side of
 \eqref{usual1} (resp. \eqref{usual2}). The term involving the $L^\infty$ norm
 can be bounded in the same way.\qedhere
 \end{proof}

\subsection{Bounds for $\phi$ with mixed space and time derivatives }
\label{section 7.2}
The purpose of this section is to estimate $||D_t^{k-1} \pave \phi||_{H^\ell(\hD_t)}$,
extending the result of Theorem \ref{main theorem, ell est of phi}. Recall
the notation from \eqref{cprimes}.

\begin{theorem} \label{main theorem, ell est D_t phi}
Fix $r\!\geq \!7$, $k\! \geq \!1$ and suppose that
\eqref{ubd2} holds. Then with $\phi$ defined by \eqref{phidef}
we have
\begin{align}
|| D_t^{k-1}\pave \phi||_{H^\ell(\tD_t)} \leq C_r^\prime{\sum}_{s\leq k-1}||D_t^s \rho||_{H^{r-s-1}(\tD_t)},\qquad\text{if}\quad k\! + \!\ell \!\leq\! r,
\label{pave D_t mixed1}
\end{align}
\begin{equation}
||D_t^{k-1} \pave \phi||_{H^\ell(\tD_t)} \leq C_r^\prime
(||\T\xve||_{H^{(r-1, 1/2)}(\Omega)} + 1){\sum}_{s\leq k-1}||D_t^s \rho||_{H^{r-s}(\tD_t)},\qquad\text{if}\quad k\! + \!\ell \!=\! r .
\label{pave D_t mixed }
\end{equation}
In addition, $C_r'$ in \eqref{pave D_t mixed1}-\eqref{pave D_t mixed } can be replaced by $C(||\xve||_{H^r(\Omega)}, {\sum}_{s\leq k-1}||D_t^s V||_{H^{r-s}(\Omega)})$.
\end{theorem}
\begin{proof}
  We will just prove \eqref{pave D_t mixed },
  the proof of \eqref{pave D_t mixed1} being similar.
   We proceed by induction: when $k+\ell=1$, this follows from Theorem \ref{main theorem, ell est of phi}. Suppose that we know \eqref{pave D_t mixed }
    for $k+\ell=1,\cdots r$.
   The case $\ell = 0$ follows directly from Theorem \ref{usual norm D_t} so we
   assume that $\ell \geq 1$. By the elliptic estimate \eqref{sobellmix},
   we have:
\begin{equation*}
||D_t^{k-\!1} \pave \phi||_{H^\ell(\tD_t\!)} \leq C(||\xve||_{H^r(\Omega)})\big(||\div\! D_t^{k-\!1} \pave \phi||_{H^{\ell-\!1}\!(\tD_t\!)}
+||\curl D_t^{k-\!1} \pave \phi||_{H^{\ell-\!1}\!(\tD_t\!)}+\!{\sum}_{s\leq \ell}||\T^s D_t^{k-\!1} \pave \phi||_{L^2(\tD_t\!)}\big).
\label{pa D_t phi pre}
\end{equation*}
To control $||\div D_t^{k-1} \pave \phi||_{H^{\ell-1}(\tD_t\!)}$ and $||\curl D_t^{k-1} \pave \phi||_{H^{\ell-1}(\tD_t\!)}$, we use \eqref{ucommEu} and get:
\begin{equation}
||\div D_t^{k-1} \pave \phi||_{H^{\ell-1}(\tD_t)} \leq ||D_t^{k-1} \rho||_{H^{\ell-1}(\tD_t)}+ P({\tsum}_{s\leq k-2}||D_t^s \ssm V||_{H^{r-s}(\Omega)}){\sum}_{s\leq k-2}||D_t^s \pave\phi||_{H^{r-1-s}(\tD_t)},
\end{equation}
\begin{equation}
||\curl D_t^{k-1} \pave \phi||_{H^{\ell-1}(\tD_t)} \leq P({\tsum}_{s\leq k-2}||D_t^s \ssm V||_{H^{r-s}(\Omega)}){\sum}_{s\leq k-2}||D_t^s \pave\phi||_{H^{r-1-s}(\tD_t)}.
\end{equation}
By the inductive assumption, $||D_t^s \pave \phi||_{H^{r-s}(\tD_t)}$ is bounded
by the right-hand side of \eqref{pave D_t mixed1} (resp. \eqref{pave D_t mixed })
when $s \leq k-2$,
and by Theorem \ref{usual norm D_t}, we likewise control $||\T^s D_t^{k-1} \pave \phi||_{L^2(\tD_t)}$
for $s \leq \ell$.
\end{proof}

First, we need a result analogous to Theorem \ref{ellphi'' 1}.
 Let $\FD^r$ be the mixed tangential space and time derivative defined in Section \ref{def T and FD}.
The proof of the following theorem can be found in Appendix \ref{gravityproofs2}:
\begin{theorem} \label{Dtellphi' 1}
  Fix $r \geq 5$, suppose that \eqref{uwassump} holds and let $\Phi$ denote
  the fundamental solution of the Laplacian in $\R^3$. If $g$ is a smooth function whose support is contained in $\xht(t, \Omega^{d_0/2})$
  which additionally satisfies that $g(x)$ is radial whenever $x\!\in \!\hD_t\!\setminus\!\D_t$,
  then with $L^p=L^p(\hD_t)$:
\begin{equation}
{\sum}_{k\leq r}||\mathfrak{D}^k \paht (g * \Phi)||_{L^2} \leq C_r^{\prime\prime\prime}
(||\T\xve||_{H^{(r-1, 1/2)}(\Omega)} + 1)\bigtwo({\sum}_{k\leq r}||\mathfrak{D}^kg||_{L^2}+{\sum}_{k\leq 2}||\mathfrak{D}^kg||_{L^6}+||g||_{L^\infty}\bigtwo).
\end{equation}
\end{theorem}
Similar to the case when $\mathfrak{D}^j = \T^j$, Theorem \ref{Dtellphi' 1} with
$g = \hrho (\chi_m - \chi_n)$ implies that  the sequence
 $(\mathfrak{D}^j \paht \hphi_m)_{m=1}^\infty$ is Cauchy in $L^2(\hD_t)$ for $j\leq r$
 and this gives the following bound:
\begin{cor} \label{usual norm D_t}
With the same hypotheses and notation as in Theorem \ref{Dtellphi' 1}, writing
$\phi = -\rho\chi_{\D_t} * \Phi$, we have:
\begin{equation}
{\sum}_{j\leq r-1}||\mathfrak{D}^j \paht \phi||_{L^2(\tD_t)} \leq C_r^{\prime\prime\prime}
\bigtwo({\sum}_{k\leq r-1}||\mathfrak{D}^k\rho||_{L^2(\tD_t)}+{\sum}_{k\leq 2}||D_t^k\rho||_{H^{3-k}(\tD_t)}\bigtwo),\qquad
\end{equation}
\begin{equation}
{\sum}_{j\leq r}||\mathfrak{D}^j \paht \phi||_{L^2(\tD_t)} \leq C_r^{\prime\prime\prime}
(||\T\xve||_{H^{(r-1, 1/2)}(\Omega)} + 1)\bigtwo({\sum}_{k\leq r}||\mathfrak{D}^k\rho||_{L^2(\tD_t)}+{\sum}_{k\leq 2}||D_t^k\rho||_{H^{3-k}(\tD_t)}\bigtwo).
\end{equation}
\end{cor}

\subsection{Fractional derivative bounds for $\phi$}\label{section 7.3}
We will need an estimate for $||\pave \phi||_{H^{(0,r-1/2)}(\D_t)}$ in Section \ref{uniform}.
The following theorem is an analogue of Theorem \ref{usual norms} and follows from an approximation
argument as above and the estimates in Appendix \ref{tangapp}. See Appendix \ref{gravityproofs}
for the proof.
\begin{theorem} \label{r-0.5 estimate for phi} Fix $r\geq 5$. With $\phi$ defined
  by \eqref{phidef} we have
\begin{equation}
||\T^{r-1}\pave \phi||_{H^{(0,1/2)}(\tD_t)} \leq C_r||\rho||_{H^{r-1}(\tD_t)}.
\end{equation}
\end{theorem}

\subsection{Estimates for differences of solutions}\label{section 7.4}
Let $\Vu,\Vw:\Omega\to\R^3$ be two vector fields and  $\xve_I$, $\xve_{\II}$
their corresponding smoothed flows. Let $\xhtu$, $\xhtw$ be
the corresponding flow maps in the extended domain $\Odo$ and $\Vhtu\!=\!D_t\xhtu$,
$\Vhtw\!=\!D_t\xhtw$ be the associated velocity fields.
For $J \!= \! I, \II$, we define $\tD_{\!Jt}\! = \xve_J(t, \Omega)$ and:
\begin{equation}
  \phi_J(t, x) = \int_{\tD_{J t}}\rho_J(y_J(t, z))\chi_{\tD_{J t}}(z)\Phi(x-
  z)\,dz,
 \label{phijdef}
\end{equation}
where  $y_{\!J}(t,\cdot) \!:\! \tD_{\!Jt} \!\to \!\Omega$ is the inverse of
$\xve_{\!J}(t,\cdot)$. Throughout this section let $D_r$ denote a continuous function
\begin{equation}\label{eq:Dr}
D_r\!=\!D_r(||\xveu||_{H^r(\Omega)},||\xvew||_{H^r(\Omega)}, ||\Vu||_{H^{r-1}(\Omega)}, ||\Vw||_{H^{r-1}(\Omega)}, ||D_t\Vu||_{r-1}, ||D_t\Vw||_{r-1}).
\end{equation}
To prove a Lipschitz estimate for the map $\Lambda$ in Section
\ref{vlwpsec} we have
\begin{theorem}
\label{diffellphi full thm}
For $r\!\geq\! 7$ there is a continuous $D_r$ as in \eqref{eq:Dr}
so that for $k\!+\!\ell\!=r\!$ with $\phi_J$ defined by \eqref{phijdef}:
\begin{equation}
||D_t^{k-1}\paveu \phi_I\! -D_t^{k-1}\pavew \phi_{\II}||_{H^{\ell}(\Omega)} \leq D_r \big(||\rho_I\!-\rho_{\II}||_{r-1}\\
+\big\{||\xveu\!-\xvew||_{H^r(\Omega)}+||\Vu\!-\Vw||_{\X^r}\big\}||\rho_{\II}||_{r-1}\big).
\label{diffellphi full}
\end{equation}
In addition, $D_r$ in \eqref{pave D_t mixed1}-\eqref{pave D_t mixed } can be replaced by $$D_{k,\ell}\big(||\xveu||_{H^r(\Omega)},||\xvew||_{H^r(\Omega)}, {\tsum}_{s\leq k-1}||D_t^s \Vu||_{H^{r-s}(\Omega)},  {\tsum}_{s\leq k-1}||D_t^s \Vw||_{H^{r-s}(\Omega)}\big).$$
\end{theorem}
\begin{proof}
First, if $k=1$, by Lemma \ref{app:sobdiff}, we have:
\begin{equation}
||\paveu \phi_I-\pavew \phi_{\II}||_{H^{r-1}(\Omega)}\leq D_r \big(||\rho_I-\rho_{\II}||_{H^{r-2}(\Omega)}+||\T^{r-1}(\paveu \phi_I-\pavew\phi_{\II})||_{L^2(\Omega)}+ ||\xveu-\xvew||_{H^r(\Omega)}||\pavew \phi_{\II}||_{H^{r-1}(\Omega)}\!\big),
\end{equation}
which is bounded by the right-hand side of \eqref{diffellphi full}.
Second, we consider the case when $k\!\geq \!2$. If $\ell\!=0\!$ then \eqref{diffellphi full} is
Theorem \ref{diffellphi tan thm}. When $\ell\!\geq\! 1$, we set $\alpha\!=\!\paveu \phi_I$ and $\beta\!=\!\pavew \phi_{\II}$
 in Lemma \ref{app:sobdiff} and get with $H^k\!=H^k(\Omega)$:
\begin{multline}
||{}_{\!}D_t^{k-\!1\!}\paveu \phi_{I\!} -{}_{\!}D_t^{k-\!1\!}\pavew \phi_{\II}\!||_{{}_{\!}H^{\ell}}\!\leq\! D_{\!r}\bigtwo(\!||\!\div_{{}_{\!}I} \!D_t^{k-\!1\!}\paveu \phi_{I\!} -\div_{{}_{\!}\II}\! D_t^{k-\!1\!}\pavew \phi_{\II\!}||_{{}_{\!}H^{\ell\!-\!1}}
{}_{\!}+{}_{\!}||\!\curl_{I\!} \!D_t^{k-\!1\!}\paveu \phi_{I\!} -\curl_{\II{}_{\!}} \!D_t^{k-\!1\!}\pavew \phi_{\II}\!||_{{}_{\!}H^{\ell\!-\!1}}\\
+||\T^\ell D_t^{k-1} (\paveu \phi_I\!-\pavew \phi_{\II}\!)||_{L^2}
+||\paveu \phi_I-\pavew\phi_{\II}||_{H^{r-1}}+||\xve_I-\xve_{\II}||_{H^r} ||D_t^{k-1}\pavew  \phi_{\II}||_{{}_{\!}H^{\ell}}\bigtwo),
\end{multline}
which is bounded by the right-hand side of
\eqref{diffellphi full} except for the first two terms. For the first term, we write:
 \begin{multline}
\div_I D_t^{k-1}  \paveu \phi_I-\div_{\II} D_t^{k-1} \pavew \phi_{\II}
= D_t^{k-1} (\rho_I-\rho_{\II})\\
+\sum \big((D_t^{k_1}\pa \xveu)\cdots(D_t^{k_s}\pa \xveu)(\paveu D_t^{\ell'} \paveu \phi_I)\big)-\big((D_t^{k_1}\pa \xvew)\cdots(D_t^{k_s}\pa \xvew)(\pavew D_t^{\ell'} \pavew \phi_{\II})\big),
\end{multline}
with the sum taken over all $k_1+\cdots+k_s+\ell'=k-1$ and $k_1\geq 1$.
The $H^{\ell-1}(\Omega)$ norm is controlled by
\begin{equation}
C_r\bigtwo(||\rho_I-\rho_{\II}||_{k-1,\ell-1}+{\sum}_{s\leq k-2}||D_t^{s}(\paveu \phi_I - \pavew\phi_{\II})||_{H^{r-2-s}}
+\big\{||\xveu-\xvew||_{H^r}+||\Vu-\Vw||_{\X^r}\big\}||D_t^s\pavew \phi_{\II}||_{H^{r-2-s}}\bigtwo),
\end{equation}
by adapting the argument used in the proof of Lemma \ref{diffellphi}. The curl term is controlled similarly.
\end{proof}
With $f_{\!J} \!=\! (g_J * \Phi)\circ \xht_J$, $J\!{}_{\!}=\!I{}_{\!},\II$,
the following theorem allows one to control $||\FD^{r-1\!} \pahtu {}_{\!}f_{\!I}-\FD^{r-1\!} \pahtw{}_{\!} f_{\II}||_{L^2(\Odo\!)}$.
This will be used to get an estimate for $||\FD^{r-1\!} \paveu \phi_I-\FD^{r-1\!} \pavew \phi_{\II}||_{L^2(\Omega)}$
and by Proposition \ref{app:sobdiff} this will allow us to control the full Sobolev norm
of the difference. The proof of this is in Appendix \ref{gravityproofs}.
 \begin{theorem}
 \label{diffellphithm 1}
 For $r\! \geq\! 7$ there is a continuous
 $D_r$ as in \eqref{eq:Dr} so that the following hold.
 For $J\!=\!I,\II$, if $g_J$ are smooth functions supported in $\Omega^{d_0\!/2}\!\!$,
  such that $\FD g_J\!=\!0$ in $ \Odo\!\!\setminus \! \Omega$,
  then $f_{\!J}\!=\!(g_J*\Phi)\!\circ \widehat{x}_J$ satisfy:
 \begin{multline}
 {\sum}_{k\leq r-1}||\FD^k \pahtu f_{I}-\FD^k\pahtw f_{\II}||_{L^2(\Odo)}\\
 \leq D_r\bigtwo({\sum}_{k\leq r-1}||\FD^k (g_I-g_{\II})||_{L^2(\Odo)}
 +{\sum}_{k\leq 2}||\FD^k (g_I-g_{\II})||_{L^6(\Odo\!)}
+||g_I-g_{\II}||_{L^\infty(\Odo\!)}\,\,\\
+\big\{||\xveu-\xvew||_{H^r(\Omega)}+||\Vu-\Vw||_{r}\big\}\big({\sum}_{k\leq r-1}||\mathfrak{D}^kg_{\II}||_{L^2(\Odo)}+{\sum}_{k\leq 2}||\mathfrak{D}^kg_{\II}||_{L^6(\Odo)}+||g_{\II}||_{L^\infty(\Odo)}\big)\bigtwo).
 \label{diff bound phi' 1}
 \end{multline}
 \end{theorem}

Let $\phi_I^m$, $\phi_{\II}^m$ be the extended $\phi_I$ and $\phi_{\II}$, respectively, i.e.,  $\phi_I^m = \int_{\Odo}[\hrho_I\chi_m](\zhtu(t,y'))\Phi(\xhtu(t,y)-\zhtu(t,y'))\,dy'$ and $\phi_{\II}^m$ is defined in a analogous way.
  Then Theorem \ref{diffellphithm 1} with $F=\phi_I^m-\phi_I^n$ and $G=\phi_{\II}^m-\phi_{\II}^n$
  implies that the sequence $(\FD^k \pahtu \phi_I^m-\FD^k \pahtw \phi_I^m)_{m=1}^\infty$
  is Cauchy in $L^2(\Odo)$, and this allows one to get a bound for
  $||\FD^r \paveu \phi_I-\FD^r\pavew \phi_{\II}||_{L^2(\Omega)}$ from that of
  $||\FD^r \pahtu \phi_I^m-\FD^r\pahtw \phi_{\II}^m||_{L^2(\Odo)}$, which gives:
\begin{theorem}
\label{diffellphi tan thm}
If $r\geq 7$, there is a continuous
$D_r$ as in \eqref{eq:Dr}
so that with $\phi_I, \phi_{\II}$ defined by \eqref{phijdef}:
 \begin{multline}
 {\sum}_{k\leq r-1}||\FD^k \paveu \phi_I-\FD^k\pavew \phi_{\II}||_{L^2(\Omega)}
 \leq D_r\bigtwo({\sum}_{k\leq r-1}||\FD^k (\rho_I-\rho_{\II})||_{L^2(\Omega)}+{\sum}_{k\leq 2}||D_t^k(\rho_I-\rho_{\II})||_{H^{3-k}(\Omega)}\,\,\\
+\big\{||\xveu-\xvew||_{H^r(\Omega)}+|\Vu-\Vw||_{r}\big\}\big({\sum}_{k\leq r-1}||\mathfrak{D}^k\rho_{\II}||_{L^2(\Omega)}+{\sum}_{k\leq 2}||D_t^k\!\rho_{\II}||_{H^{3-k}(\Omega)}\big)\bigtwo).
 \end{multline}
\end{theorem}

\section{Estimates for solutions of the enthalpy equation}
\label{enthsec}
With the same notation as in the previous sections,
we now return to the equation:
\begin{align}
\!\!\!\!\! e'(h)D_t^2h - \Dve h &= (\pave_i \ssm V^j)\pave_j V^i\!\! - e''(h) (D_th)^2\!-\rho(h),
 \quad \textrm{in } [0,\!T]\!
 \times\! \Omega,\quad\text{with}\quad
 h= 0,\quad \textrm{on } [0,\!T] \!\times\! \pa\Omega, \quad\label{enth1}\\
 h(0,y) &= h_0^\ve(y), \quad D_t h(0,y)
 = h_1^\ve(y),\qquad \textrm{ on } \Omega.\label{enth3}
\end{align}
We set:
\begin{equation}
 \W_s(t) =  \Big(\frac{1}{2}{\sum}_{k \leq s}
 \int_\Omega e'(h) |D_t^{k+1} h(t)|^2 + \delta^{ij}
 \big( D_t^k \pave_i h(t)\big) \big(
 D_t^k\pave_j  h(t)\big) \kve(t)\, dy\Big)^{1/2}.
 \label{energydef}
\end{equation}

By Lemma \ref{fbdslem},
writing $\F_1 = -(\pave_i \ssm V^j)(\pave_j V^i)$
and $\F_2 = -e'(h) (D_t h)^2-\rho(h)$, we have the estimates:
\begin{align}
 ||\F_1||_{s,0} &= {\sum}_{k \leq s}||D_t^k \F_1||_{L^2}
 \leq C(M)\big(||D_t^s V||_{H^1} + P(||V||_{\X^s})\big),\\
 ||\F_1||_{s-1} &=
 \!\!{\sum}_{k + \ell \leq s-1} ||D_t^k \F_1||_{H^{\ell}}
 \leq C(M)\big(||V||_{s} + ||\xve||_{H^{s}} + P(||V||_{s-1},
  ||\xve||_{H^{s-1}}) \big),
\end{align}
and assuming that $h$ satisfies the a priori assumption \eqref{Lassumpwave} we have
\begin{equation}
 ||\F_2[h]||_{s,0} \leq P_1(L, ||h||_{s,0}, ||h||_{s-1})
 ||h||_{s+1,0} , \qquad
 ||\F_2[h]||_{s-1} \leq P_2(L, ||h||_{s-1}) ||h||_s.
\end{equation}

Combining these estimates with Theorem \ref{mainwavethm}, we have:
\begin{prop}
 \label{enthwavepropen}
 Fix $r \geq 7$, $0 \leq s \leq r$, and $T >0$. Suppose that
  $V \in \X^{r+1}(T)$ and that
 \eqref{ubd2} holds. There is a continuous function $\oC_s$
 depending on
 $M, L, \W_{s-1}(0)$
 and $\sup_{\,0 \leq t \leq T} \big(||V(t)||_{\X^s} + ||\xve(t)||_{H^s}\big)$
 so that if $h$ satisfies the wave equation \eqref{enth1}-\eqref{enth3} and
 the a priori assumption \eqref{Lassumpwave}
 for $0 \leq t \leq T$, then for $s\leq r-1$:
 \begin{align}
  ||D_t h(t)||_{s,0}
  + ||\pave h(t)||_{s,0} &\leq  \oC_s \Big(\W_s(0) + \int_0^t ||V(\tau)||_{\X^{s+1}}
  \, d\tau\Big), \quad 0 \leq t \leq T,\\
  ||\pave h(t)||_{s} &\leq \oC_s \Big( ||V(t)||_{s} + ||\xve(t)||_{H^s}
  + \W_s(0) + \int_0^t ||V(\tau)||_{\X^{s+1}} \, d\tau\Big)
  ,\label{enthellbd}
  \\
  ||\pave h(t)||_{r} &
\leq \oC_r \Big( ||V(t)||_{r}
+ \ve^{-1}||\sm x(t)||_{H^r}
+ \W_r(0) + \int_0^t ||V(\tau)||_{\X^{r+1}} \Big).
 \end{align}

 Moreover, with $h^\ve_k$ defined as in Section \ref{wavecompatcondn}, suppose that:
 \begin{equation}
  {\tsum}_{k + |J| \leq 3} |\pa_y^J \pave h_k^\ve|
  + |h_k^\ve| \leq L_0.
 \end{equation}
 If $T$ satisfies \eqref{Tgr} then
 the
 constants $\C_s\!$ can be taken to depend on $L_0$ instead of $L\!$
 if $T \!\leq\! T_{1}$.
\end{prop}

\subsection{Estimates for differences of solutions}
We now prove the estimates we will need in Corollary
\ref{contcor}. Recall the notation and definitions from
Section \ref{waveests}. Suppose that $h_J$, for $J=I,\II$,
satisfy:
\begin{align}
&e'(h_J)D_t^2h_J - \Dve_J h_J = (\pave_{Ji} \ssm V_J^j)(\pave_{Jj} V_J^i)
 - e''(h_J) (D_th_J)^2\!-\rho(h_J),
 \quad \textrm{in } [0,T]\!
 \times\! \Omega,\\
 &h_J= 0,\quad \textrm{on } [0,T] \!\times\! \pa\Omega,\qquad
  h_J(0,y) = h_0^\ve(y), \quad D_t h_J(0,y)
 = h_1^\ve(y),\quad \textrm{ on } \Omega.
\end{align}

 We  write $\W^I_s,\W^{\II}_s$
for the energy \eqref{energydef} evaluated at $h_I, h_{\II}$,
respectively, and $F_J^1 \!=\! -(\pave_{Ji} \ssm V_J^j)
(\pave_{Jj} V_{J}^i)$ for $J = I, \II$. By the estimate \eqref{fdiff}
we have with $\C_s = \C_s(M, ||\Vu||_s, ||\Vw||_s, ||\xveu||_{H^{\ell+1}},
||\xvew||_{H^{\ell+1}})$:
\begin{align}
 ||F_I^1 - F_{\II}^1||_{s,0} &\leq \C_s \big( ||D_t^s(\Vu - \Vw) ||_{H^{1}(\Omega)}
 + ||\Vu - \Vw||_{C^3_{x,t}}\big),\\
 ||F_I^1 - F^1_{\II}||_{s-1}
 &\leq \C_s \big( ||\Vu - \Vw||_{\X^{s}} + ||\xveu - \xvew||_{H^{s}}
 + ||\xveu - \xvew||_{C^4_{x,t}}\big).
\end{align}

Writing $\F_J^2 \!=\!-e''(h_J) (D_t h_J)^2 - \rho(h_J)$, by
the estimates \eqref{functionalest1}-\eqref{functionalest2}, we also have:
\begin{align}
 ||F_I^2 - F_{\II}^2||_{s,0} &\leq \C_s \big( ||h_I - h_{\II}||_{s+1,0}
 + ||\Vu - \Vw||_{C^3_{x,t}}\big),\\
 ||F_I^2 - F^2_{\II}||_{s-1}
 &\leq \C_s \big( ||\Vu - \Vw||_{\X^{s}} + ||\xveu - \xvew||_{H^{s}}
 + ||\xveu - \xvew||_{C^4_{x,t}}\big).
 \label{fdiffdb2}
\end{align}

Combining these estimates with Lemma \ref{dtphidiffnew}, we have:
\begin{cor}
  \label{mainwavecor2}
Define:
\begin{equation}
 \W_s^{I,II}(t)
 = \Big(\frac{1}{2} {\sum}_{k \leq s}\int_\Omega e'(h_I) |D_t^{k+1} (h_I - h_{\II})|^2
 + |D_t^{k} \paveu (h_I - h_{\II})|^2\, \kveu dy\Big)^{1/2}.
\end{equation}
Fix $r \geq 7$ and suppose that the hypotheses in Proposition \ref{bootstrapped}
hold. Take $T$ small enough that \eqref{Tgr} holds. For
each $s \leq r-1$ there is a positive continuous function $\D_s$ depending
on
   $M, L_0,
   \W_s(0), $ and $
   \sup_{0 \leq t \leq T} \big(||V_J(t)||_{\X^{r+1}}+
   ||\xve_J(t)||_{H^r}\big),$
 for $J = I, \II$, so that:
 \begin{equation}
 {\sup}_{\,0 \leq t \leq T}  W_s^{I\!,\II}(t) \leq
 \D_s \int_0^T\!\!\! ||\Vu(\tau)\! - \! \Vw(\tau)||_{\X^{r+1}}
 + ||\xveu(\tau) - \xvew(\tau)||_{H^{r}(\Omega)}\!
  + ||\xveu\!-\xvew||_{C^4_{x,t}(\Omega)}\, d\tau
 ,
  \label{enthwavediff}
  \end{equation}
  \begin{equation}
  ||\paveu  h_I - \pavew  h_{\II}||_{s}
  \leq \D_s \big(W_s^{I,\II} +
  ||\Vu - \Vw||_{s+1} + ||x_I - x_{\II}||_{H^{s}}\big).
  \label{enthdiff}
\end{equation}
\end{cor}

\section{Existence for the smoothed problem up to a smoothing dependent time}
\label{vlwpsec}
Let $(V_0^\ve, h_0^\ve)$ satisfy the comptibility conditions of order
$r$ (see \eqref{smcompatdef}) for some $r \geq 7$,
and define $h_1^\ve$ by \eqref{h1epsdef}. In this section, we will prove
that there is a unique vector field $V$ solving the smoothed-out
Euler equations \eqref{smpbmdef} with $h$ given by
\eqref{smwavedef}-\eqref{smwaveic}.
We will work with the norms:
 \begin{align}
   ||V||_{\X^s(T)} = {\sup}_{\,0 \leq t \leq T\,} ||V(t)||_{\X^s},
   \quad \text{ where } \quad
  ||V(t)||_{\X^s} = {\sum}_{k = 0}^{s-1} ||D_t^{1+k} V(t)||_{H^{s- k-1}(\Omega)}
  + ||V(t)||_{H^{s-1}(\Omega)}.
  \label{xnorm2}
 \end{align}
 We let $\X^s(T)$ denote the closure of $C^\infty([0,T];
 C^\infty(\overline{\Omega}))$
 with respect to the norm $\X^s(T)$.

For a given vector field $V \in \X^{r+1}(T)$,
we define the tangentially smoothed
flow $\xve: [0, T] \times \Omega \to \R^3$
as in \eqref{xvedef},
$A$ as in \eqref{udef} and
the derivatives $\pave, \Dve$ by \eqref{pavedef} and \eqref{laplsm}.
By Theorem \ref{nllwpthm}, if $h_0^\ve, h_1^\ve$ are compatible to order $r$,
there is a function $h = h[V]$ which solves the
problem:
\begin{align}
 D_t^2 e(h) - \Dve h &= (\pave_i \ssm V^j)(\pave_j V^i)-\rho(h),\quad \textrm{ on }
 [0,T] \times \Omega,\quad\text{and}\quad
 h = 0 \textrm{ on } [0,T] \times \pa \Omega,\label{hdef1}\\
 h(0, y) &= h_0^\ve(y), \quad \,\, D_t h(0,y) = h_1^\ve(y) ,\qquad\textrm{ on } \Omega,
 \label{hdef3}
\end{align}
and the estimates in Proposition \ref{enthwavepropen} hold for $h$.
We also define $\rho = \rho[h] = \rho[V]$ as in Section
\ref{approxpbmsec} and then define $\phi = \phi[V]$ by \eqref{phidef}.
We then define a map $\Lambda$ by:
\begin{equation}
 \Lambda^i(V)(t,y) = V_0^i(y) - \int_0^t \delta^{ij}\pave_j h(s,y)\, ds
 -\int_0^t \delta^{ij}\pave_j \phi(s,y)\,ds.
 {}
\end{equation}
If $V\!$ is a regular fixed point of $\Lambda$ then it satisfies
\eqref{smpbmdef} and the corresponding $h$ satisfies \eqref{hdef1}-\eqref{hdef3}.
Set:
  \begin{align}
   E_0^s = ||V(t)||_{\X^s} \big|_{t = 0},
   \quad\text{and}\quad
   \W_0^s = \Big( \frac{1}{2}{\sum}_{k \leq s}
   \int_\Omega \big(|h_{k+1}^\ve(y)|^2 + |\pave h_k^\ve(y)|^2\big)
    \kve_0(y)\, dy\Big)^{1/2},
  \label{Esdef}
  \end{align}
  with the $h^\ve_s$ defined by \eqref{hkepsdef} and
  $\kve_0 = \det(\pa x_0/\pa y)$. In Lemma
  \ref{initlem}, we show that if $V$ satisfies \eqref{intadmissible}, then
  these quantities are well-defined and bounded by the initial data.
  We also write:
\begin{equation}
   e_0^s = e_1^s + e_2^s + e_2^s,\quad\text{where}\quad
   e_1^s = ||x_0||_{H^s(\Omega)},\qquad e_2^s = ||V^\ve_0||_{H^s(\Omega)},
   \quad \text{and}\quad e_3^s = ||h_0^\ve||_{H^s(\Omega)}.
 \end{equation}
   We remark that
  $x_0$ and $h_0$ are not  independent; we take
  $x_0$ so that $\det(\pa x_0/\pa y) \!=\! 1/\rho(h_0)$, and consequently
  $e_3^s \!\leq \!C(e_3^{s}) e_1^{s+1}\!$
  for a constant $C$ depending on the equation of state. However, it
  is more natural to state our estimates in terms of $x_0^\ve$ rather than
  $h_0^\ve$ in many cases so we will keep the notation separate.

  The main result of this section is then:
\begin{theorem}
  \label{vlwpthm}
 Let $r \geq 7$. If $E_0^{r+1} + \W_r^0 + e_0^r < \infty$
 and $V_0^\ve, h_0^\ve$ satisfy the compatibility conditions \eqref{smcompatdef}
 to order $r$, then for sufficiently small $\ve$, there is a
 positive and continous function $T_\ve =
 T_\ve(E_0^{r+1}, \W_r^0, e_0^r, \ve^{-1})$ so that for any $0 \leq T \leq T_\ve$,
 there is a unique $V \in \X^{r+1}(T)$ which satisfies the smoothed-out
 problem \eqref{smpbmdef}.
 Moreover,
 there is a positive continuous function
 $\F_r$ so that:
 \begin{align}
   {\sup}_{0 \leq t \leq T} \big(||V(t)||_{\X^{r+1}}
   + ||\pave h(t)||_{r} \big)
   &\leq \ve^{-1}\F_r(E_0^{r}, \W^0_{r-1}, e_0^{r-1})
   (E_0^{r+1} + \W^0_r + e_0^{r}) + 1,\\
   {\sup}_{0 \leq t \leq T}
   \W_r(t) &\leq \F_r(E_0^{r}, \W^0_{r-1} e_0^r)\W^0_r + 1.
  \label{normbd}
 \end{align}
\end{theorem}
First, in Section \ref{initialsec}, we show that
$\Lambda(V)$ is admissible (recall the definition in \eqref{intadmissible})
 whenever $V$ is,
and that under the hypotheses of Theorem
\ref{vlwpthm}, the quantities $E_0^{r+1}, \W^0_r$ are bounded.
In Proposition \ref{invariantprop}, we use the estimates from Section
\ref{enthsec} to show that $||\Lambda(V)||_{\X^{r+1}}$
can be bounded in terms of the initial data, $\ve$, and $||V||_{\X^{r+1}}$.
This fact is then used in Corollary \ref{invcor} to show that $\Lambda$
maps a certain Banach space $\C^{r+1} \subset \X^{r+1}$ to
itself
(see \eqref{Crdef}). Finally, in Proposition \ref{contcor}, we prove
that if $\Vu, \Vw \in C^{r+1}$,
$||\Lambda(\Vu) - \Lambda(\Vw)||_{\X^r}$ can be bounded in terms
of $||\Vu||_{\X^{r+1}}, ||\Vw||_{\X^{r+1}}$ and $||\Vu - \Vw||_{\X^{r}}$.

\subsection{The initial data}
\label{initialsec}

Given $(V_0, h_0) \in H^r$ that satisfy the compatibility conditions
\eqref{maincomp0}
to order $r-1$, let $(V_0^\ve, h_0^\ve) \in H^r$ be the data constructed in
Appendix \ref{compat} that satisfy \eqref{smcompatdef}
 to order $r-1$.
Define $h_1^\ve$ by:
\begin{equation}
e'(h_0^\ve) h_1^\ve = -
 \div V_0^\ve,\quad \text{where} \quad\div V_0^\ve = \pave_i V_0^i.
\end{equation}
If $V_0^\ve\!,_{\!}...,\!
V^\ve_r{}_{\!}$ are as in \eqref{vkepsdef}, we will only consider
vector fields $V\!$ which
are admissible to order $r$, meaning:
\begin{equation}
 D_t^k V\big|_{t = 0} = V_k^\ve, \qquad k = 0,..., r.
 \label{admissible}
\end{equation}

 Taking $M_0, L_0$ so that:
 \begin{equation}
  |\pa x_0| + |\pa x^{-1}_0| + {\tsum}_{k + |J| \leq 3}
  |\pa^J V_k|\leq M_0/2,\qquad\text{and}\qquad
  {\tsum}_{ k + |J| \leq 3} |\pa^J \pave h_k|
  + |h_k|
 \leq L_0/2,
  \label{M0L0bd}
 \end{equation}
 we have the following bounds for sufficiently small $\ve$:
  \begin{equation}
   |A(0, y)| + |A^{-1}(0,y)| +
   {\tsum}_{k + |J| \leq 3} |\pa^J V_k^\ve|
  \leq M_0,\qquad\text{and}\qquad
  {\tsum}_{ k + |J| \leq 3} |\pa^J \pave h^\ve_k| + |h_k^\ve|
 \leq L_0.
 \label{M0L0bdreal}
  \end{equation}

  That there are data so that the compatibility conditions \eqref{smcompatdef}
hold follows from Theorem \ref{compatthm}.
We have:
\begin{lemma}
  \label{initlem} Suppose that $(V_0^\ve, h_0^\ve)$ satisfy the compatibility conditions for smoothed Euler \eqref{smcompatdef}
  to order $r-1$. Then if $V\!$ satisfies \eqref{admissible},  $\Lambda(V)$ also satisfies \eqref{admissible}. Moreover,
 $  E_0^{s+1} \!+ \W_0^s \!\leq
  C(M_0)\big(e_1^{r+1}\! \!+
  P(e_0^{r})\big)
  \label{e0bd1}$, $s\!\leq \!r$.
\end{lemma}
\begin{proof}
To see that $\Lambda(_{\!}V)$ satisfies \eqref{admissible}, note that $\Lambda(_{\!}V)\big|_{t = 0}\!
= \!\!V_0^\ve$, and for $k\! \geq \!1$, by the definition of $V_k^\ve$,
\eqref{vkepsdef}:
\begin{equation}
  D_t^k V|_{t = 0} = V_k^\ve
  =
  {\sum}_{\ell \leq k} \widetilde{S}_{i\ell}^{jk}(\pave \tV_0^\ve,...,
  \pave \tV_{k-\ell-1}^\ve) \pave_j H_{\ell}^\ve
 =
 -D_t^{k-1} \pave h\big|_{t = 0}-D_t^{k-1} \pave \phi\big|_{t = 0},
\end{equation}
where the last equality follows from the identity
\eqref{dtpaSdefsmooth}, and the fact that by construction
$D_t^k h|_{t = 0} = h_k^\ve$, $D_t^k \phi|_{t = 0} = \phi_k^\ve$. The right-hand
side here is $D_t^k \Lambda(V)$ by definition and this proves the first point.

  To prove the second point, we start by showing that for $0 \leq s \leq r$,
  \begin{equation}
   E_0^{s+1} + W_0^{s} \leq C(M_0)( e_0^{s+1} + P(L_0, e_0^r)),
   \label{intermediate}
  \end{equation}
  where $L_0$ is as in \eqref{M0L0bd}.
 If $V$ satisfies \eqref{admissible}, then:
  \begin{equation}
   D_tV|_{t = 0} = V^\ve_1 = -\pa h_0^\ve-\pave \phi|_{t=0},
  \end{equation}
  and so:
  \begin{equation}
   E_0^1 = \big(||D_t V(t,\cdot)||_{L^2} + ||V(t,\cdot)||_{L^2}\big)\big|_{t = 0}
   \leq C(M_0)\big( ||h_0^\ve||_{H^1}+||\rho(h_0^\ve )||_{L^2} + ||V_0^\ve||_{L^2}\big),
  \end{equation}
  where we used Theorem \ref{main theorem, ell est of phi} to control $||\phi(0,_{{}_{\!}}\cdot,_{{}_{\!}})||_{\!H^1}\!$.\! That $W_0^0$ is bounded by the right-hand side of \eqref{intermediate}
  is immediate,
  so \eqref{intermediate} hold for $s \!= \!0$. Suppose now that it
  holds for $s \!\leq \! m\!-\!1$. We introduce the notation:
    \begin{equation}
     e^{m^*}_0 = {\sum}_{ k =0}^m ||h^\ve_k||_{H^{m-k}}.
    \end{equation}

  By definition we have $||V_0^\ve||_{H^m} \leq e_0^m$.
  Suppose we know that for some $k \geq 0$:
  \begin{equation}
   ||V_k^\ve||_{H^{m-k}} \leq C(M_0) \big(||H_{k-1}^\ve||_{H^{m-k+1}}
   + P(e^{m^*}_0, e_0^{m-1}, L_0)\big),
  \end{equation}
  where $H_{k-1}^\ve$ is defined in Section \ref{wavecompatcondn},
  then by definition for $F_k$,  and Theorem \ref{main theorem, ell est D_t phi}, we have:
  \begin{equation}
   ||V_{\!k+\!1}^\ve\!||_{H^{m\!-\!k\!-\!1}}\!
   \leq C({}_{\!}M_{0\!}) \big( || H_{k}^\ve||_{H^{m\!-\!k}\!}
   + ||(\pa V_{k}^\ve{}_{\!})\pa H_0^\ve||_{H^{m\!-\!k} \!}+
   ||F_{k+\!1\!}||_{H^{m\!-\!k\!-\!1}}\!\big)\!\leq
   C({}_{\!}M_{0\!}) \big(||h_{k}^\ve||_{H^{m\!-\!k}} \!+ P(e^{m^{\!*}}_0\!\!\!\!, e_0^{m\!-\!1}\!\!\!, L_0)\big).
  \end{equation}
  Therefore \eqref{intermediate} follows from bounding $e_0^{m^*}\!$ by $||V_0^\ve||_{H^m}$ and $||h_0^\ve||_{H^m}$, and hence $e_0^m$.
   To prove this, we use the continuity equation $h_1^\ve \!= \! e'(h_0^\ve)^{-1}\!\div V_0^\ve$ which yields the bound $||h_1^\ve||_{H^{m-1}}\!\leq P(L_0, e_0^m)$. In addition, suppose we know that for some $k\geq 3$ $${\sum}_{\ell\leq k-1}||h_\ell^\ve||_{H^{m-\ell}}\leq P(L_0, e_0^m).$$
We want to show that $||h_k||_{H^{m-k}}$ can be controlled by the same bound. This follows from the wave equation,
\begin{align}
h_{k}^\ve = e'(h_0^\ve)^{-1}\big(\Delta h_{k-2}^\ve +(\pa \ssm V_{k-2})(\pa V_0)+G_{k-2}   \big),
\end{align}
where $G_{k-2}$ is given in Section \ref{wavecompatcondn}. This implies:
\begin{align}
||h_k^\ve||_{H^{m-k}} \leq P(L_0, e_0^m),
\end{align}
where the bound of $||G_{k-2}||_{H^{m-k}}$ follows from Sobolev lemma.
\end{proof}

\subsection{Existence on a time interval of size $O(\ve)$}
\label{vshorttime}

We can now prove Theorem \ref{vlwpthm}.
We start with the following simple lemma, which will be used to
control some low norms of $\xve$ and $V$.
\begin{lemma}
  \label{lownorms}
  Fix $r \!\geq \!5$ and $T_1\! > \!0$ and suppose that $V \!$ satisfies
  \eqref{admissible} and that $||V||_{\X^r(T_{1\!})}\! \leq \!K$.
  If the initial data satisfies \eqref{M0L0bdreal}, then there is a positive,
  continuous
  function $D_0$ so that if $T$ satisfies:
  \begin{equation}
   T D_0(M_0, K, T_1) \leq 1, \text{ and } \quad T \leq T_1,
  \end{equation}
  then:
  \begin{equation}
    {\sup}_{\,0 \leq t \leq T\,}
    \big(|| \pa \xve(t, \cdot)/ \pa y||_{L^{\infty}}
    + || \pa y(t,\cdot)/\pa \xve||_{L^\infty}
    +
    {\tsum}_{ k + |J| \leq 3} ||D_t^k\pa_y^{J} V(t,\cdot)||_{L^\infty}\big)
   \leq 4M_0.
  \end{equation}
\end{lemma}
\begin{proof}
  We integrate in time and use Sobolev embedding to get:
  \begin{equation}
   ||\pa \xve(t,\cdot)/\pa y(t,\cdot)||_{L^\infty}
   + {\tsum}_{k + |J| \leq 3} ||\pa_y^J D_t V(t,\cdot)||_{L^\infty}
   \leq M_0 + \int_0^t ||V(\tau)||_{\X^5}\, d\tau.
  \end{equation}
  The right-hand side is bounded by $2M_0$ if $T \!\leq\! \min(T_{{}_{\!}1},  M_0/\!K)$.
To control $\pa y/_{{}_{\!}}\pa \xve$,
let $N_{{}_{\!}}(t) \!=\! ||\pa y(t,\cdot)/\pa \xve||_{L^\infty(\Omega)}$
and note that by \eqref{dinv} we have ${d} N/{dt} \!\leq C_0 N^2
 ||\pa \ssm V||_{L^\infty(\Omega)}$. Using $N_{{}_{\!}}(0)\! \leq M_0$ and Sobolev embedding,
 this implies that
 $N_{{}_{\!}}(t)\! \leq  M_0 (1\! -\! C_0 M_0 t ||\pa V(t)||_{L^\infty(\Omega)})^{-1}\!$ for a constant
 $C_0$ depending only on $\Omega$.
 Taking $D_0\!=  2 \min( K/\!M_0, 2C_0M_0^{2})$, this implies that
 $N_{{}_{\!}}(t)\! \leq\! 2M_0$ provided $T \leq T_1$ and $TD_0 \leq 1$.
\end{proof}

We can now begin the proof of local well-posedness. We will use the next
estimate to show that $\Lambda$ maps a certain Banach space to itself.
We set $\oE^r_0 = E_0^{r+1} + \W^0_r + e_0^r$.
\begin{prop}
  \label{invariantprop}
  Fix $r \geq 7$ and suppose that $V_0, h_0$ are such that
   $\overline{E}^r_0 < \infty$. There are continuous,
   positive functions $D_r$, $D_r'$  and
   polynomials $\P_1, \P_2$ so that the following statement
   holds: If $T, \ve$ satisfy:
   \begin{equation}
    T D_r(M_0, L_0, \oE_0^{r}) \leq 1,
    \qquad \ve D_r'(M_0, L_0, \oE_0^{r}) \leq 1,
    \label{epsstart}
   \end{equation}
   and $V \in \X^{r+1}(T)$ is any vector field
   satisfying the condition \eqref{admissible} with
   $||V||_{\X^{r+1}(T)} \leq \ve^{-2} E_0^{r+1}$,  then:
 \begin{equation}
  {\sup}_{\,0 \leq t \leq T\,} ||\Lambda(V)(t)||_{\X^{r+1}}
  \leq \ve^{-1} \P_1\big(E_0^{r+1}, e_0^r, \W^0_r
  \big) + \ve^{-1}T\,
  \P_2\big(E_0^{r+1}, e_0^r, \W^0_r,
  ||V||_{\X^{r+1}(T)}\big).
 \end{equation}
\end{prop}
\begin{proof}
To get started, we fix $T_1 \leq 1$ and  $\ve$ small enough
that \eqref{M0L0bdreal} holds, and consider only $V$ so that
$\sup_{\,0 \leq t \leq T_1 } ||V(t)||_{\X^{r+1}}\! \leq \ve^{-2} E_0^{r+1}\!\!$. With
 $\F_1 \!= -(\pave_i \ssm V^j)(\pave_j V^i)$,
 we fix $K_0 \!= K_0(E_0^{r+1}\!\!,e_0^r, \ve^{-1})$ so that:
\begin{equation}
  {\sup}_{\,0 \leq t \leq T_1}\big( ||\xve(t)||_r + ||V(t)||_{\X^{r+1}}
  + ||\F_1(t)||_{r,0} + ||\F_1(t)||_{r-1}\big) \leq K_0,
\end{equation}
whenever $ ||V||_{\X^{r+1}(T_1)} \leq \ve^{-2} E_0^{r+1}$, where
we are bounding $||\xve(t)||_r \leq ||\xve(0)||_r + T_1 ||V||_{\X^{r+1}(T_1)}$.

Let $D_0\!$ be as in Lemma \ref{lownorms}, and take $T\!$ smaller
if needed so that $T D_0(M_0,  T_{\!1}) \!\leq\! 1$.
By Lemma \ref{lownorms}:
  \begin{align}
    {\sup}_{\,0 \leq t \leq T}\big(
    ||\pa \xve(t, \cdot)/\pa y||_{L^\infty}
    + || \pa y(t,\cdot)/\pa \xve||_{L^\infty}
    +
    {\tsum}_{ k + |J| \leq 3} ||D_t^k\pa_y^{J} V(t,\cdot)||_{L^\infty}\big)
    &\leq 4M_0,
   {}
  \end{align}
  for all $V$ satisfying \eqref{admissible} with
  $||V||_{\X^r(T)} \leq \ve^{-2} E_0^{r+1} $. In particular, the assumption
  \eqref{ubd2} holds with $M = 4M_0$.

With $Q_r$ as in Corollary \ref{bootstrapped} and
$G_r'$ as in Theorem \ref{nllwpthm}, we take $T$ smaller again
if necessary so that:
\begin{equation}
 T (Q_r(4M_0, L_0, \W^0_r, K_0, T_1)
 + G_r'(4M_0, L_0, \W^0_r, K_0, T_1)) \leq 1.
 \label{Trestrict}
\end{equation}
By Lemma \ref{initlem} and Theorem \ref{nllwpthm}
the wave equation \eqref{hdef1}-\eqref{hdef3} has
a unique solution $h\! =\! h[V]$ on $[0, T]\!\times\! \Omega $.
By the above calculations and
the first bound in \eqref{Trestrict},
applying Proposition \ref{enthwavepropen}, for $0 \leq t \leq T$ we have:
\begin{equation}
 ||h(t)||_{s+1,0}
 + ||\pave h(t)||_{H^s(\Omega)} \leq
  (1+\ve^{-1}) \D_1 ( ||\sm x(t)||_{H^r(\Omega)}
 + ||V(t)||_{H^r(\Omega)} )
 + T \D_2 ||V||_{\X^{r+1}(T)},
 \qquad 0 \leq s \leq r.
\end{equation}
where $\D_1, \D_2$ depend on $M_0, L_0, \W^0_{r}$ as well as $
\sup_{0 \leq t \leq T }
||\xve(t)||_{H^{r}(\Omega)} + ||V\!(t)||_{\X^{r}}\!$. We now bound
$||V\!(t)||_{H^{r_{\!}}(\Omega)}\! \leq E_0^{r+1} \!\!+ T||V||_{\X^{r+1}(T)}$
and  $||\xve(t)||_{H^{r_{\!}}(\Omega)} \!+ ||V(t)||_{\X^{r}}\!
\leq E_0^{r+1} + e_0^r\! \!+ T ||V||_{\X^{r+1}(T)}$. This gives:
\begin{equation}
 {\sum}_{\,s \leq r-1\,} ||h(t)||_{s+1,0} + ||\pave h(t)||_{H^s(\Omega)}
 \leq (1 + \ve^{-1}) \D'_1 + T \D'_2,
 \label{effectiveest2}
\end{equation}
with $\D_1' = \D_1'(M_0, L_0, E_0^{r+1}, e_0^r, \W^0_r, T ||V||_{\X^{r+1}(T)})$
and $\D_2' = \D_2'(M_0, L_0,
E_0^{r+1}, e_0^r, \W^0_r,  ||V||_{\X^{r+1}(T)})$.

By Theorem \ref{main theorem, ell est D_t phi}, we also have:
\begin{multline}
||\pave\phi||_{r} \leq (1+\ve^{-1})P\big(||\xve||_{H^r(\Omega)},  ||V||_{\X^{r}}, ||V||_{H^r(\Omega)}, ||h||_{r-1}\big)||J_\ve x||_{H^r(\Omega)} ||h||_{r}\\
\leq   (1\!+\!\ve^{-1}) \D_3'(E_0^r, e_0^r)\!
 + \! T \D_4'\big(E_0^r, e_0^r,  ||V||_{\X^{r+1}(T)}\!\big).
 \label{effectiveest3}
\end{multline}

Set $\P_1 = \D_1' + \D_3'$ and $\P_2 = \D_2' + \D_4'$.
Since, for $k \geq 1$
we have $D_t^{k} \Lambda(V)_i = -D_t^{k-1} \pave_i h
- D_t^{k-1} \pave_i \phi ,$
 the estimates for $||D_t^{k}\Lambda(V)||_{H^{\ell}(\Omega)}$ follow from
 \eqref{effectiveest2}-\eqref{effectiveest3}, and the estimate
 \begin{equation}
  ||\Lambda(V)(t)||_{H^r(\Omega)}
  \leq ||V_0||_{H^r(\Omega)} + T\, {\sup}_{\,0 \leq \tau \leq T\,}
  \big(||\pave h(\tau)||_{H^r(\Omega)} + ||\pave \phi(\tau)||_{H^r(\Omega)}\big).
  {}\tag*{\qedhere}
 \end{equation}
\end{proof}

\begin{cor}
  \label{invcor}
  If the hypotheses of the previous theorem hold, there are positive, continuous,
  functions $D_r'', \P_r = \P_r(M_0, \oE_0^{r},\ve)$ so that
  if $T$ satisfies:
  \begin{equation}
   T D_r''(M_0, \oE_0^{r}, \ve) \leq 1,
   \label{lwptassump}
  \end{equation}
  and if $V \in \X^{r+1}(T)$ satisfies \eqref{admissible} as well as the bound:
 \begin{equation}
  {\sup}_{\,0 \leq t \leq T\,} ||V(t)||_{\X^{r+1}} \leq \ve^{-1}\P_r
  + 1,
  \label{invbd}
 \end{equation}
 then $\Lambda(V)$ satisfies:
 \begin{equation}
  {\sup}_{\, 0\leq t \leq  T\,} ||\Lambda(V)||_{\X^{r+1}}
  \leq \ve^{-1} \P_r + 1.
 \end{equation}
\end{cor}
\begin{proof}
  Let $D_r, D_r', \P_1, \P_2$ be as in Proposition \ref{invariantprop}.
  Take $\ve$, $T$ small enough that \eqref{epsstart} holds.
  Let $T^*\!\!, \P_1, \P_2$ be as in Proposition \ref{invariantprop}. By
  Sobolev embedding and the elliptic estimate \eqref{enthellbd}, we have
  that $L_0 \leq  C_0(M_0, \oE_0^r)$, and we take
  $D_r'' = D_r'( M_0, C_0, \oE).$ Now set $\P_r = \P_1$.
  Taking $\ve$ smaller if needed,
  the right-hand side of \eqref{invbd} is smaller than $\ve^{-2}E_0^{r+1}\!\!$,
  and so if $V\!$ satisfies \eqref{invbd} for $T\! \leq \!T^*\!\!$,
  then Proposition \ref{invariantprop} applies and so:
  \begin{equation}
   {\sup}_{\,0 \leq t \leq T\,}
   ||\Lambda(V) (t)||_{\X^{r+1}}
   \leq \ve^{-1} \P_1 + T \big( \ve^{-1}
   \P_2(E_0^{r+1}, e_0^r,\W_r^0,\ve^{-1}\P_1 + 1)\big).
  \end{equation}
  We now take $T$ small enough that this last factor is 1, which
  gives the result.\qedhere
\end{proof}

We now take $T$ small enough that \eqref{lwptassump} holds and define:
\begin{equation}
 \C^{r+1}(T) = \big\{ V\!: [0,T] \!\times \!\Omega \to \R^3 \big|
 V \textrm{ satisfies } \eqref{admissible} \textrm{ and }
 {\sup}_{\, 0\leq t \leq T\,} ||V(t)||_{\X^{r+1}} \!\leq \ve^{-1} \P_r(e_0^{r},
 E_0^{r+1})\! + \!1\big\}.
 \label{Crdef}
\end{equation}
Corollary
\ref{invcor} and Lemma \ref{initlem} imply that $\Lambda :\C^{r+1}(T) \to
\C^{r+1}(T)$.
We now want to show that $\Lambda$ has a fixed point
in $\C^{r+1}(T)$ for $T$ taken small enough.
We start with:
\begin{prop}
  \label{contcor}
  Fix $r \!\geq \! 7$. There
   is a polynomial $\P_3\! =\! \P_3(\oE^r_0, \ve^{-1})$ so that if $T$ satisfies
   \eqref{lwptassump}  then:
 \begin{equation}
  {\sup}_{\,0 \leq t \leq T\,} ||\Lambda(\Vu)(t) - \Lambda(\Vw) (t)||_{\X^r}
  \leq \ve^{-1} T \P_3
  ||\Vu - \Vw||_{\X^r},\quad\text{for}\quad
 \Vu, \Vw \!\in \C^{r+1}(T).
  \label{contbd}
 \end{equation}
\end{prop}
\begin{proof}
 First, note that by Corollary \ref{mainwavecor2}
 and Corollary \ref{invcor}, under our hypotheses we have:
 \begin{equation}
  {\sup}_{\,0 \leq t \leq T}
  (||h_J(t)||_{r+1,0} +
  ||\pave h_J(t)||_{r})
  \leq  \C_1(M_0, \oE_0^{r}, \ve^{-1} ) + 1,
  \label{bounds}
 \end{equation}
 for $J = I, II$ and some positive continuous function $\C_1$.
 For $k \geq 1$, we have:
   \begin{equation}
   D_t^k (\Lambda(\Vu) - \Lambda(\Vw))
   = D_t^{k-1} (\paveu h_I - \pavew h_{II})+D_t^{k-1} (\paveu \phi_I - \pavew \phi_{II}).
 \end{equation}

 By \eqref{enthdiff} combined with \eqref{enthwavediff}, we have:
  \begin{equation}
   ||D_t^{k-1}{}_{\!} (\paveu h_I(t) - \pavew  h_{I\!I}(t))||_{H^{r{}_{\!}-{}_{\!}k}}\!
   \leq\! \C_r' \Big(||\Vu(t)-\Vw(t)||_{\X^{r-1}}+||\xve_I(t)-\xve_{{}_{\!}I\!I}(t)||_{H^{r{}_{\!}-{}_{\!}1}}
    + \!\int_0^t ||\Vu({}_{\!}\tau{}_{\!})-\Vw({}_{\!}\tau{}_{\!})||_{\X^{r}}\!\Big),
  \end{equation}
  and by \eqref{diffellphi full}:
  \begin{equation}
  || D_t^{k-\!1\!} (\paveu \phi_I\! - \pavew  \phi_{I\!I})||_{H^{r\!-\!k}}\leq
  \mathcal{C}_r'P(||h_I||_{r-1}, ||h_{\II}||_{r-1})(||\Vu-\Vw||_{\X^{r}}+||\xve_I-\xve_{{}_{\!}I\!I}||_{H^{r}})||h_I\!-h_{I\!I}||_{r-1}
 ,
  \end{equation}
  where $\C_r' = C_r'\big(M_0, \oE_0^r, \sup_{0 \leq t\leq T}(||\xve_I(t)||_{H^r(\Omega)}
  + ||\xve_{I\!I}(t)||_{H^r(\Omega)}),
  ||V_{\!I}||_{\X^{r+1}(T)},||V_{\!I\!I}||_{\X^{r+1}(T)}
  \big)$ and where we have used that $|\rho_I - \rho_{II}| =
  |\rho(h_I) - \rho(h_{II})| \leq C|\rho'| |h_I - h_{II}|$.
  Combining these with the simple estimate:
  \begin{align}
    {\sup}_{\,0 \leq t \leq T}
    ||\Vu(t) - \Vw(t)||_{\X^{r}}
  + ||\xve_I - \xve_{II}||_{H^{r-1}}
  &\leq 2T \,{\sup}_{\,0 \leq t \leq T}
  ||\Vu(t) - \Vw(t)||_{\X^{r+1}},
  \end{align}
  and using \eqref{bounds}, we have \eqref{contbd}.
\end{proof}
\begin{proof}[Proof of Theorem \ref{vlwpthm}]
 With notation as in Corollary \ref{invcor} and Proposition
  \ref{contcor},  take $\ve$ so
  $\ve/\P_3\!\leq\! 1/\!D_r''$ and set
  \begin{equation}
   T_\ve = 2^{-1}\ve/{\P_3} .
  \end{equation}
  If $T \leq T_\ve$, by Lemma \ref{initlem} and Corollary \ref{invcor},
  for any $V \in \C^{r+1}(T)$, we have that $\Lambda(V) \in \C^{r+1}(T)$.
  Moreover,
  by Proposition \ref{contcor}, for any $\Vu, \Vw \in\C^{r+1}(T)$
  we have that:
  \begin{equation}
    ||\Lambda(\Vu) - \Lambda(\Vw)||_{\X^r(T)}
    \leq  2^{-1}||\Vu - \Vw||_{\X^r(T)}.
   \label{contbd2}
  \end{equation}

With $V_k^\ve$ defined by \eqref{vkepsdef},
define the following sequence:
\begin{align}
 V^{(0)}(t,y) &= {\sum}_{k = 0}^r V_k^\ve(y){t^k}\!/{k!} ,&&
 V^{(N)}(t,y) = \Lambda(V^{(N-1)})(t,y), \quad N \geq 1.
 {}
\end{align}
Noting that $D_t^s V^{(0)}|_{t = 0} = V_k^\ve$,
by Corollary \ref{invcor} and Lemma \ref{initlem}, the
sequence $V^{(N)}$ is well defined and $V^{(N)} \in \C^{r+1}(T)$ for all $N$.
Let
$d_0 = \sup_{\,0 \leq t \leq T} ||\Lambda(V^{(0)})(t) - V^{(0)}(t)||_{\X^r}$.
The estimate \eqref{contbd2} implies:
\begin{equation}
 ||\Lambda(V^{(N)}) - \Lambda(V^{(M)})||_{\X^r(T)} \leq
  2^{1-\min{(M,N)}}d_0,
\end{equation}
for $0 \leq T \leq T_\ve$.
In particular the sequence $V^{(N)}$ is a Cauchy sequence in $\X^r(T)$.
Let $V \in \X^r$ denote the limit. Because the norms $||V^{(N)}||_{\X^{r+1}(T)}$
are uniformly bounded we conclude that $V \in \C^{r+1}(T)$ as well.

The estimate \eqref{normbd} now follows from the definition of
$\C^{r+1}(T)$ and the bounds in Corollary \ref{mainwavecor2}.
\end{proof}

\section{Energy estimates}
\label{energy}
In the previous section, we constructed a solution to the smoothed
problem on a time interval of size $O(\ve)$.  In this section,
we prove the basic energy estimates which control Sobolev norms of the velocity
uniformly in $\ve$. We will not
apply these energy estimates to the solutions $V$ constructed in the previous
section directly but instead to a sequence $V_N$ which converges to $V$ in
an appropriate norm, and so we write our energy estimate in terms of remainders
which we expect to converge to zero. In section \ref{higherordersec} we
prove estimates for the differentiated problem which will be used to show
that these remainders do converge to zero. Finally, in section \ref{uniform}
we implement this strategy and prove the desired estimates for $V, h$ satisfying
Euler's equations.

The below energy estimates are slightly cumbersome, because
we need to control a fractional number of derivatives of the
solution $V$, and since
$\Omega$ does not admit a global
coordinate system, we will need to apply fractional derivatives in each
coordinate patch separately. This unfortunately obscures the
idea behind the estimates so let us explain how the energy estimates work
in a simple case.
The gravitational potential will not enter into the energy estimates
to highest order, so we will ignore it for the moment. Let $T$ be a vector
field which is tangential at the boundary. Using the formula $[T,\pave_i] = -(\pave_i T \xve^k)\pave_k$,
and applying $T$ to the smoothed-out Euler's equations and the
continuity equation \eqref{resultofproj}, we have:
\begin{equation}
 D_t T V^i - \delta^{ij}\pave_j \big( (T \xve^k) \pave_k h - T h\big) =
 - T \xve^k (\delta^{ij}\pave_j\pave_k h),
 \qquad D_t T e(h)  + \div T V = -
 (\pave_i T\xve^\ell) \pave_\ell V^i
 \label{linearized}
\end{equation}
Multiplying the first equation by $T V^i \kve$ and integrating over
$\Omega$ gives:
\begin{equation}
 \frac{1}{2}\frac{d}{dt} ||T V(t)||_{L^2(\Omega)}^2
 - \int_{\Omega} \pave_i \big( (T \xve^k) \pave_k h - Th\big)
 T V^i\, \kve dy = \text{ lower-order terms}.
 \label{}
\end{equation}
Integrating by parts and using that $Th = 0$ on $\pa \Omega$,
we have:
\begin{equation}
 - \int_{\Omega} \pave_i \big( (T \xve^k) \pave_k h - Th\big)T V^i\, \kve dy
 = -\int_{\pa \Omega} (T \xve^k) (T V^i)N_i \pave_k h \kve dy
 +\int_{\Omega} ((T \xve^k) \pave_k h - T h)\pave_i T V^i\, \kve dy.
 \label{heuristicibp}
\end{equation}
We now manipulate the boundary term.
Recall that $\xve = \ssm x = \sm \sm x$ and that $\sm$ is
symmetric with respect to the surface measure $dS$.
Ignoring the commutator $[T, \sm]$
for the moment, the boundary term is:
\begin{multline}
- \int_{\pa \Omega} \sm (T \sm x^k) (T V^i) \pave_k h N_i\,\widetilde{\nu} dS(y)
= -\int_{\pa \Omega} (T\sm x^k)(T \sm V^i) \pave_k h N_i\, \widetilde{\nu} dS(y)\\
 -\int_{\pa \Omega} (T \sm x^k)\big[
\sm \big( T V^i \pave_k h N_i  \widetilde{\nu} \big) - (\sm T V^i) (\pave_k h N_i\widetilde{\nu})\big]
\,dS(y).
 \label{heuristicbdy}
\end{multline}
Because $h = 0$ and $N\cdot \pave h < 0$ on $\pa \Omega$, it follows that
$\pave_k h= -N_k |\pave h|^{-1}$, so the first term here is
the time derivative of a positive term to highest order:
\begin{equation}
- \int_{\pa \Omega}\!\!\!\! (T\!\sm x^k)(T\! \sm V^i) \pave_k h N_i \widetilde{\nu} dS
 = \frac{1}{2}\frac{d}{dt}\!
 \int_{\pa \Omega} \!\!\!(T\!\sm x^k)(T\! \sm x^i) N_k N_i |\pave h| \widetilde{\nu} dS
 + \int_{\pa \Omega}\!\!\!\! (T\!\sm x^k)(T\! \sm x^i)
 D_t (N_k N_i |\pave h|\widetilde{\nu}) dS.
\end{equation}
Since $\sm$ is a convolution with a function supported on a ball of size
$\sim \ve$, one should
expect that the second term in \eqref{heuristicbdy} is bounded
by $C\ve ||T \xve||_{L^2(\pa \Omega)} ||T V||_{L^2(\pa \Omega)}$, with the constant
depending on bounds for $\pave h$.

Using the second equation in \eqref{linearized}, the interior term is:
\begin{equation}
 \int_\Omega ((T \xve^k) \pave_k h - T h) \pave_i T V^i\, \kve dy
 = \int_\Omega ((T \xve^k) \pave_k h - Th) T D_t e(h)\, \kve dy+
 \int_\Omega ((T\xve^k) \pave_k h - Th) (\pave_i T \xve^\ell) (\pave_\ell V^i)\kve dy.
 \label{}
\end{equation}

This leads to an energy identity of the form:
\begin{multline}
 \frac{d}{dt} (||T V(t)||_{L^2(\Omega)}^2 + ||T \sm x(t)\cdot N||_{L^2(\pa \Omega)}^2
 + ||\sqrt{e'(h)}T h||_{L^2(\Omega)}^2)
 \\
 \lesssim
 \bigg|\int_\Omega ((T\xve^k) \pave_k h - Th) (\pave_i T \xve^\ell) \pave_\ell V^i\kve dy
 \bigg| + \ve ||T\xve(t)||_{L^2(\pa \Omega)}||TV(t)||_{L^2(\pa \Omega)} +
 \text{ lower order terms. }
 \label{}
\end{multline}
By the elliptic estimates \eqref{ellft1} the second term on the left controls
$||T\sm x||_{L^2(\pa \Omega)}$ provided we have estimates for
$\div T\sm x, \curl T\sm x$, and so one can think of this as controlling
$||\sm x||_{H^{3/2}(\Omega)}$ and thus $||\xve||_{H^{3/2}(\Omega)}$,
by the trace inequality. The term on the
right-hand side looks problematic because we do not control
$||\xve||_{H^2(\Omega)}$, however we will be able to ``integrate half a derivative
by parts'' in this term using \eqref{inthalf} to control it.

To control the term $\ve ||TV(t)||_{L^2(\pa \Omega)}$, it turns out that
it can be bounded by $||\gamma \cdot TV||_{L^2(\pa \Omega)}$
provided we control the divergence and curl appropriately in the
interior, where $\gamma$ is the projection to the tangent
space at the boundary. We will see that $\ve d/dt ||\gamma \cdot TV||_{L^2(\pa \Omega)}$
is lower-order, uniformly in $\ve$, because to highest order $D_t \gamma \cdot TV \sim
\gamma \cdot T \pave h \sim \gamma \cdot \pave T h + T\pa \xve \cdot \pa h$.
The first term is zero because $h = 0$ on $\pa \Omega$ and using the
smoothing property \eqref{smoothingproperty} the second term is $O(\ve^{-1}$).

\subsection{Higher order energy estimates}
Let $\alpha$ be a vector field on $\Omega$ and $q$ a function with
$q = 0$ on $\pa \Omega$. We suppose that the following hold:
\begin{alignat}{2}
  {|D_t \pave q|}/{|\pave q|} + |q| \leq K, && \quad \text{ on } \pa \Omega,
  \label{qchiassump1}\\
  |D_t^k\pa_y^I q| \leq K, && \quad
  k + |I| \leq 3, \quad \text{ in } \Omega.
 \label{qchiassump2}
\end{alignat}
As in earlier sections, we will also fix a strictly positive
function $\eprime$ so that $|\eprime'| \lesssim \eprime$. We also let $x \in C^1([0,T];
H^r(\Omega))$ for $r \geq 7$ be a given vector field and let $V = D_t x$.

With notation as in \eqref{eq:simplifiedtangentialnotation} and $\fdhm\!$
defined by \eqref{fdmudef},
for each $T^I\! \in\! \T^s$\!\!, we
define $E^I_{\mu} \!= E^I_{\mu,1}\! + E^I_{\mu,2}$, where:
\begin{align}
 E^{I}_{\mu,1}
 &= \frac{1}{2}\int_\Omega \delta^{ij} ( T^I \fdhm \alpha_i)
 ( T^I \fdhm \alpha_j) \kve dy
 + \frac{1}{2} \int_\Omega \eprime| T^I  \fdh_\mu q|^2 \kve dy, \label{e1def}\\
 E^I_{\mu,2} &=
\frac{1}{2} \int_{\pa \Omega} (T^I \fdhm \sm x^i)
(T^I \fdhm \sm x^j)
N_i N_j |\pave q|\, \widetilde{\nu} dS,
 \label{e2def}
\end{align}
as well as:
\begin{equation}
 E^I = {\sum}_{\mu = 0}^N E^I_{\mu,1} + E^I_{\mu,2}, \qquad
 E^s = {\sum}_{|I| \leq s} E^I.
 \label{Edefs}
\end{equation}
Here $\widetilde{\kappa} dy=d\widetilde{x}$ is the volume form on $\widetilde{\mathcal{D}}_t$ and $\widetilde{\nu}$ is such that $\widetilde{\nu}$ times the surface measure on $\pa \Omega$ is equal to the surface measure on $\pa\widetilde{\mathcal{D}}_t$
in the $\widetilde{x}$ coordinates.
We will also need to control the time derivative of
\begin{equation}
 E^I_{\mu, \ve} = \int_{\pa \Omega} \gamma^{ij}
 (\fdhm T^I \alpha_i)(\fdhm T^I \alpha_j) \,\widetilde{\nu}  dS,
 \qquad E^I_\ve = {\sum}_{\mu = K}^N E^I_{\mu, \ve},
 \qquad E^s_\ve = {\sum}_{|I| \leq s} E^I_\mu.
 \label{Edefs2}
\end{equation}
Here, $\gamma^{ij} = A^i_a A^j_b \gamma^{ab}$ where $\gamma^{ab}$
is the projection to $T(\pa \Omega)$.
The fact that the fractional derivative operator $\fdhm$ appears
on the ``outside'' in this definition and the ``inside''
in the definition of the $E^I_{\mu}$ is just to make the computation
simpler and has no special significance; note that the commutator
$[T^I, \fdhm]$ is an operator of lower order by Lemma \ref{fracalg}.
The quantity $E^s_\ve$ will appear in our calculation weighted with
a power of $\ve$ and will be needed in order to show that we have a solution
to the problem \eqref{smpbmdef} on a time interval independent of $\ve$.

We write $T^I = S T^J$ where $S \in \T$ and $|J| = s-1$ and define:
\begin{multline}
 R^I = {\sum}_{\mu = 1}^N
 ||D_t T^I \fdhm \alpha - \pave \big( (T^I \fdhm \xve^j) (\pave_j q)
 +  T^I\fdhm q\big)||_{L^2(\Omega)}\\
 + ||\fdhm\big( \eprime D_t T^J q - (\pave_i T^J \xve^j)\pave_j V^i
 + \div T^J \alpha\big)||_{L^2(\Omega)}
 + ||D_t T^I \fdhm x - T^I \fdhm \alpha||_{L^2(\pa \Omega)},
 \label{Rsdef}
\end{multline}
and $R^s = \sum_{|I| \leq s} R^I$. We will ultimately take
$\alpha = V$ and $q = h$, in which case using \eqref{linearized}, one
expects the first two terms here to be lower order. See Section \ref{higherordersec}.
We also define:
\begin{equation}
 R^I_\ve = {\sum}_{\mu = 1}^N ||\gamma\cdot\fdhm T^I D_t \alpha - (\fdhm T^I \gamma)\cdot
 \pave q||_{L^2(\pa \Omega)},
 \qquad R^s_\ve = {\sum}_{|I| \leq s} R^I_\ve.
 \label{Rsdef2}
\end{equation}

The main result of this section is:
\begin{theorem}
  Suppose that $\xve$ satisfies the assumption \eqref{ubd2},
  that $\alpha$ and $q$ are given as above and that the
  assumptions \eqref{qchiassump1}-\eqref{qchiassump2}
   hold. With $E^I, E^s, E^I_\ve, E^s_\ve$
  defined as in \eqref{Edefs}-\eqref{Edefs2}
   and $R^s, R^s_\ve$ defined by \eqref{Rsdef}-\eqref{Rsdef2},
   there is a continuous function $C = C(M, K)$ so that:
\begin{multline}
 \frac{d}{dt} E^I\!  \leq C
\sqrt{\!E^s} \Big(\!R^s +
||\T^{s-1}\! \alpha||_{H^{(1_{\!},1_{\!}/2)}(\Omega)}
 + ||\T^{s-1}\!\! \sm x||_{H^1(\Omega)} + ||\sm x||_{H^{s+1_{\!}/2}(\pa \Omega)}
 + ||\T^s\! q||_{H^1(\Omega)} + ||D_t \!\T^s\!q||_{L^2(\Omega)}
 \!\Big)\\
 + C E^s + C \ve ||\alpha||_{H^{s+1_{\!}/2}(\pa \Omega)}
 ||\sm x||_{H^{s+1_{\!}/2}(\pa \Omega)},
  \qquad |I| \leq s,
 \label{enident}
\end{multline}
\begin{equation}
 \frac{d}{dt} E^{I}_\ve \leq C\sqrt{E^s_\ve} \big( R_\ve^s + \ve^{-1}
  ||\sm x||_{H^{s+1_{\!}/2}(\pa \Omega)} \big).
 \label{enident2}
\end{equation}

In particular, writing $E = E^s + \ve^2 E^s_\ve$ and $R = R^s + \ve^2 R_\ve^s$,
 we have:
\begin{multline}
 \frac{d}{dt} E
 \leq C
 \sqrt{\!E}\Big(\!
 R + ||\T^{s-1} \alpha||_{H^{(1_{\!},1_{\!}/2)}(\Omega)}
  + ||\T^{s-1}\!\! \sm x||_{H^1(\Omega)} + ||\sm x||_{H^{s+1/2}(\pa \Omega)}
  + ||\T^s q||_{H^1(\Omega)} + ||D_t \T^sq||_{L^2(\Omega)}\!\Big)\\
  + C_s E + \ve ||\alpha||_{H^{s+1_{\!}/2}(\pa \Omega)} ||\sm x||_{H^{s+1_{\!}/2}(\pa \Omega)}.
\end{multline}
\end{theorem}
These estimates appear to lose half a derivative since
$E^s\!$ only controls $||\T^{s\!} q||_{H^{(0_{\!},1_{\!}/2)}(\Omega)}$ and not
$||\T^{s\!} q||_{H^{1\!}(\Omega)}$, but $q$ will satisfy a wave equation which gains enough
regularity to close these estimates, see Lemma \ref{coercivitylem1}.
\begin{proof}
  We will prove that ${d}(E^I_{\mu,1} + E^I_{\mu,2})/{dt} $ is bounded
  by the right-hand side of \eqref{enident}.
  We have:
 \begin{multline}
  \frac{d}{dt} E_{\mu,1}^I
  = \int_{\Omega} \delta^{ij} \Big(  T^I \fdhm D_t \alpha_i
  - \pave_i\big(  (T^I \fdhm  \xve^k)(\pave_k q)
  +  T^I \fdhm  q\big) \Big)
  ( T^I  \fdhm \alpha_j)\, \kve dy\\
  + \int_\Omega
  \eprime (D_t T^I \fdhm q)(T^I\fdhm q)\kve dy
  +
  \int_\Omega \delta^{ij} \pave_i \Big( ( T^I\fdhm  \xve^k)(\pave_k q)
  -T^I \fdhm q\Big) ( T^I\fdhm  \alpha_j)\,\kve dy\\
  +
  \int_\Omega \Big( \delta^{ij} (T^I \fdhm   \alpha_i)( T^I \fdhm \alpha_j)
  + \eprime | T^I\fdhm q|^2\Big)
  (D_t \kve)
  + (D_t \eprime) |T^I \fdhm q|^2 \, \kve dy.
  \label{dte}
 \end{multline}
 The first and the last terms and are bounded by \eqref{enident}.
 After integrating by parts, using Green's formula \eqref{Green's formula}, and
 the facts that
 $ T^I\fdhm q \!= \! 0$  and $\pave_k q\! =\! -N_k |\pave q|$ on $\pa \Omega$,
 the second term on the second line is:
 \begin{equation}
  -\!\int_{\pa \Omega} \!\!\!( T^I \!\fdhm \xve^k)( T^I \!\fdhm \alpha_j) N_{\!j} N_k |\pave q|\, \widetilde{\nu} dS
  - \!\int_\Omega \! ( T^I\!\fdhm  \xve^k(\pave_k q) -
  \fdhm T^I \!q)(\div T^I\!\fdhm  \alpha)
  \, \kve dy.
  \label{afteribp}
 \end{equation}
 In order to handle the interior term,
  we need to perform a few manipulations.
 We start by  writing $\div T^I \!\fdhm \!\alpha = \big(\eprime D_t T^I\!\fdhm \!q
  - (\pave_i T^I\! \fdhm \xve^j)\pave_j V^i + \div T^I\! \fdhm \!\alpha\big)
  - \eprime D_t T^I\! \fdhm \!q + (\pave_i T^I \!\fdhm \xve^j)\pave_j V^i\!\!.$
  Inserting this into \eqref{afteribp}, we get:
  \begin{multline}
    \int_\Omega \big(T^I \fdhm \xve^k (\pave_k q) +
    T^I\fdhm q\big) \big(\eprime  D_t T^I \fdhm q
    - (\pave_i T^I \fdhm \xve^j)\pave_j V^i + \div T^I \fdhm \alpha\big) \, \kve dy\\
    + \int_\Omega \big( T^I \fdhm \xve^k (\pave_k q)\big)(\eprime D_t T^I \fdhm q)
    \,\kve dy
    + \int_\Omega \big( T^I \fdhm q\big) (\pave_i T^I \fdhm \xve^k \pave_kV^i)
    \, \kve dy\\
    + \int_\Omega \big( T^I \fdhm \xve^k \big) \big(\pave_i T^I \fdhm \xve^j)
    \pave_jV^i\, \kve dy.
\end{multline}
We now want to integrate half a derivative by parts so
write $T^I \!= S T^J$, $|J| \!=\! s\!-\!1$. Since $\pave_i\! =\! A_{\m i}^a \pa_a$,
we have:
\begin{equation}
 ||[S, \div] T^J\! \fdhm \alpha||_{L^{\!2}(\Omega)}
 \!\leq C(M) ||\pa T^J\! \fdhm \alpha||_{L^{\!2}(\Omega)},
 \qquad\!\!
 ||[S, \pave] T^J\! \fdhm \xve||_{L^{\!2}(\Omega)}
 \!\leq C(M) ||\pa T^J\! \fdhm \xve||_{L^{\!2}(\Omega)}.
\end{equation}
These terms are bounded by  \eqref{enident}.
Writing $F_{\!1\!} \!= \!\eprime D_t T^J\! \fdhm \!q -(\pave_iT^J\!\fdhm\xve^j)\pave_{j\!} V^i\!
+ \div T^J\! \fdhm \!\alpha$,
 applying \eqref{inthalf}
and the  Leibniz rule \eqref{app:alg2}, we have:
\begin{multline}
 \Big| \int_\Omega (-T^I \fdhm \xve^k (\pave_k q) + \eprime T^I \fdhm q)
 S (F_1) \, \kve dy\Big|
 \\
 \leq CK (|| T^I \fdhm \xve ||_{1/2} +
 ||T^I \fdhm q||_{1/2})
 ||\sqrt{\eprime} D_t T^J \fdhm q-
 (\pave_i T^J \fdhm \xve^j)\pave_jV^i + \div T^J \fdhm \alpha ||_{1/2},
\end{multline}
where $||\cdot||_{1/2} = ||\cdot||_{H^{(0,1/2)}(\Omega)}$,
as well as:
\begin{multline}
 \Big| \int_\Omega (T^I \fdhm\xve^k) (\pave_k q)
 S (\eprime D_t T^J \fdhm q) \kve dy\Big|
 + \Big|\int_\Omega
 (T^I\fdhm q)S ( \pave_i T^J \fdhm \xve^j(\pave_jV^i)\, \kve dy
 \Big|\\
 \leq CK (||\eprime D_t T^J \fdhm q||_{1/2} + ||\eprime T^I \fdhm q||_{1/2})
 (||T^I \fdhm \xve||_{1/2} + ||\pave T^J \fdhm \xve||_{1/2}),
\end{multline}
and finally:
\begin{equation}
 \Big| \int_\Omega (T^I \fdhm\xve^k) (\pave_k q) S (\pave_i T^J \fdhm \xve^j(\pave_jV^i))
 \, \kve dy\Big|
 \leq C ||T^I \fdhm \xve||_{1/2}||\pave T^J \fdhm \xve||_{1/2}.
\end{equation}
Since $||\fdhm\! f||_{1/2} \!\lesssim\!||f||_{H^1(\Omega)}$
 using  Lemma
 \ref{fracalg}, the terms with $\xve$ are bounded by \eqref{enident}.

It remains to control the boundary term in \eqref{afteribp}. Recalling
the definition of $E^{I\!,\mu}_2$ from \eqref{Edefs2},
we have:
\begin{multline}
 \frac{d}{dt}E^{I}_{\mu,2} - \int_{\pa \Omega}
(T^I\fdhm \xve^k)(T\fdhm \alpha^j\big)
N_j N_k |\pave q|\,\widetilde{\nu} dS\\
= \int_{\pa \Omega}
(T^I\fdhm \sm x^k )
(T^I\fdhm \sm D_t x^j)
- (T^I\fdhm \xve^k)(T^I \fdhm \alpha^j) N_j N_k |\pave q|\, \widetilde{\nu} dS\\
+ \frac{1}{2}\int_{\pa \Omega} (T^I \fdhm \sm x^k)(T^I \fdhm \sm x^j)
D_t (N_j N_k |\pave q|\, \widetilde{\nu} dS).
 \label{otherterm}
\end{multline}
The last term is bounded by \eqref{enident}
by the assumption \eqref{qchiassump1}. To handle the second term, we
recall that $\xve = \ssm x$,  $\ssm =\sm^2$ and that $\sm$ is symmetric with respect to
the measure $dS$, so:
\begin{multline}
 \int_{\pa \Omega} (T^I\fdhm \xve^k)(T^I \fdhm \alpha^j) N_j N_k |\pave q|\, \widetilde{\nu}
 dS
 = \int_{\pa \Omega} (T^I \fdhm \sm x^k)(T^I \fdhm \sm \alpha^j)
 N_j N_k |\pave q|\, \widetilde{\nu} dS\\
 + \int_{\pa \Omega} \bigtwo( ([T^I\fdhm, \sm] \sm x^k)(T^I \fdhm \alpha^j) +
 (T^I \fdhm \sm x^k)([\sm, T^I\fdhm] \alpha^j)
 \bigtwo) N_j N_k |\pave q| \, dS\\
 +
 \int_{\pa \Omega} (T^I \fdhm \sm x^k)
 \bigtwo(
 \big(\sm (T^I\fdhm \alpha^j N_j N_k |\pave q|) - (\sm T^I \fdhm \alpha^j)N_j N_k
 |\pave q|\big)\bigtwo)\, dS
 \label{}
\end{multline}
The first term  cancels the first term from \eqref{otherterm}.
Integrating by parts in the first term on the second line and  using
\eqref{jdiffrule} and \eqref{jprodrule}, the terms on the second and third line are bounded
by the right side of \eqref{enident}.

Finally, we control the time derivative of $E^{I}_{\mu,\ve}$. We have:
\begin{equation}
 \frac{d}{dt} E^{I}_{\mu,\ve}
 = \int_{\pa \Omega} (\fdhm T^I \alpha_i)(\fdhm T^I  D_t \alpha_j) \gamma^{ij}\,\widetilde{\nu} dS
 + \frac{1}{2} \int_{\pa \Omega} (\fdhm T^I \alpha_i)(\fdhm T^I  \alpha_j)
 D_t (\gamma^{ij}\widetilde{\nu} )dS.
\end{equation}

The second term is bounded by  \eqref{enident2}.
The idea is that $\gamma^{ij} \fdhm T^I  \pave_i q $ is lower
order because $q \!=\! 0$ on $\pa \Omega$, the operators $T^I\!\!, \fdhm\!$
are tangential, and we are multiplying by the tangential projection
$\gamma$. We have:
\begin{equation}
 ||\fdhm \gamma^{ij}||_{H^k(\pa \Omega)}
 \leq C(M) ||\pa_y \xve||_{H^{k+1/2}(\pa \Omega)},
 \qquad ||\gamma^{ij}||_{H^k(\pa \Omega)} \leq
 C(M) ||\pa_y \xve||_{H^k(\pa \Omega)},
 \label{gammabds}
\end{equation}
which follows from  $\gamma^{ij}\! = \!\gamma^{ab}\! A^i_{\, a} A^j_{\, b}$,
the fractional product rule \eqref{app:alg2}, the formula \eqref{dinv},
and interpolation.

We write $\gamma^{ij\!} \fdhm T^I \!\pave_i q =
\fdhm \!\big( \gamma^{ij} T^I\!\pave_i q) + [\gamma^{ij}, \fdhm] T^I\! \pave_i q
$. The $L^2(\pa \Omega)$ norm  of the
second term here is bounded by $C(\!M{}_{\!}) ||\xve||_{H^{3}(\pa \Omega)}
 ||T^I\!\pave q||_{L^2(\pa \Omega)}$ by \eqref{gammabds} and the fractional
 product rule \eqref{app:alg2}. We then write
 $\gamma^{ij} T^I \pave_i q = T^I(\gamma^{ij}\pave_i q)
 - (T^I \gamma^{ij})\pave_i q + \sum_{J + K = I, |J|, |K| \leq |I|-1}
 (T^J\gamma^{ij})(T^K \pave_i q)$. Applying $\fdhm$ to these terms, we see
 that the first term is zero because $q = 0$ on $\pa \Omega$, while:
 \begin{equation}
  ||\fdhm \big( (T^I \gamma^{ij}) \pave_i q\big)||_{L^2(\pa \Omega)}
  \leq ||\fdhm T^I \gamma^{ij}||_{L^2(\pa \Omega)}
   ||\pave q||_{H^3(\pa \Omega)}
   \leq C(M) ||\fdhm T^I \pa_y \xve||_{L^2(\pa \Omega)}
   ||\pave q||_{H^3(\pa \Omega)}.
 \end{equation}
 Since $|I| \!=\! s$ this last term is higher-order than
 \eqref{enident2}, so we write $T^I\! =\! S T^J$,
 $S \!\in\! \T$ and use the smoothing property
 $||S \sm f||_{L^2(\pa \Omega)} \!\lesssim\! \ve^{-1} ||f||_{L^2(\pa \Omega)}$
 to bound it by $\ve^{-1} ||\fdhm T^J\pa_y \sm x||_{L^2(\pa \Omega)}
 \!\lesssim \!\ve^{-1} ||\pa_y \sm x||_{H^{s-1/2}(\pa \Omega)}$.

 It now remains to control
 $||\fdhm \big((T^J\gamma^{ij})(T^K \pave_i q)\big)||_{L^2(\pa \Omega)}$
 and for this
 we use the Leibniz rule \eqref{app:alg2} in a few different ways. First, when
 $s \leq 5$ we bound the result by $||T^J \gamma||_{H^2(\pa \Omega)}
 ||\fdhm T^K \pave q||_{L^2(\pa \Omega)}$ and this is bounded by
 the right-hand side of \eqref{enident2}.
 If $s \geq 6$ and  $|K| \leq s-3$ we bound this by $||\fdhm T^J \gamma^{ij}||_{L^2(\pa \Omega)}
 ||T^K \pave q||_{H^2(\pa \Omega)}$ and if $|K|\! \geq\! s \!-\!2$ then since
 $s\! \geq \!6$ and $|J|, |K|\!\leq\! s\!-\!1$, we have $|J| \!\leq\! s\!-\!3$ and so we bound
 it by $||T^J \gamma^{ij}||_{H^2(\pa \Omega)} ||\fdhm T^K \pave q||_{L^2(\pa \Omega)}$.
 In each of these cases, applying \eqref{gammabds} we wind up with terms which are
 bounded by the right-hand side of \eqref{enident2}. This completes the
 proof.
\end{proof}

\section{The higher-order equations}
\label{higherordersec}
Fix $r \geq 5$ and let $V \in \X^{r+1}(T)$ be the solution to the smoothed-out
Euler equations \eqref{smpbmdef} constructed in Section \ref{vlwpsec}.
Recall that if $h$ is the corresponding enthalpy, then we have:
\begin{align}
 V &\in L^\infty(0,T; H^r(\Omega)),&&
 D_t^{k} V \in L^\infty(0,T;  H^{r+1-k}(\Omega)), && k = 1,..., r+1,
 \label{regbds1} \\
 D_t^{r+1} h &\in L^\infty(0,T; L^2(\Omega)), &&
 D_t^k \pave h \in L^\infty(0, T; H^{r-k}(\Omega)),
 && k= 0,...,r.
 \label{regbds2}\noeqref{regbds2}
\end{align}

In this section it is convenient to assume that we have a bit more regularity
of $\xve$. We will assume that:
\begin{equation}
 |A^i_{\m a}| + |A^a_{\m i}| + ||\xve||_{H^6(\Omega)} \leq M',
 \quad \text{ on } \Omega.
 \label{strongM}
\end{equation}
The reason we want this assumption is that the fractional product
rule \eqref{app:alg2} involves Sobolev norms. We will use notation similar
but not identical
to that in Sections \ref{ellsec} and \ref{gravitysection} and let $C_0, C_s, s\geq1,
C_s'$ denote
a continuous function of the following arguments:
\begin{equation}
  C_0 = C_0(M'),\qquad
 C_s = C_s(M', ||\xve||_{H^{s}(\Omega)}),\qquad
 C_s' = C_s'(M', ||\xve||_{H^s(\pa \Omega)}).
 {}
\end{equation}

In order to prove that we have uniform energy estimates for $V$,
 we need to show
that we can control $R^s$ and $R^s_\ve$ in terms of the energy. The first
step is the following:
\begin{lemma}
  \label{lowerorder}
  Let $r \geq 5$ and let $V \in\X^{r+1}(T)$ be the solution to
  the smoothed-out Euler equations constructed in the previous section.
  Suppose that \eqref{strongM} holds.
  Fix $1 \leq s \leq r-1$,
  let $T^I \in \T^s$ and write $T^I = S T^J$ for $S \in \T, |J| = s-1$.
  For each
  $\mu,\nu = 0,..., N$, we have:
  \begin{multline}
   ||D_t T^I \fdhm V - \pave\big( (T^I \fdhm \xve^j)(\pave_j h)
   + T^I \fdhm h\big)||_{L^2(\Omega)} \\
   \leq C_s
    ||\xve||_{H^{s}(\Omega)}(||\pave h||_{H^s(\Omega)}
    + ||\pave h||_{H^2(\Omega)})
    + ||\fdhm T^I \pave \phi(h)||_{L^2(\Omega)}, \label{Vest}
 \end{multline}
   \begin{multline}
   ||\fdhn\big( e'(h) D_t T^J\fdhm h - \pave_i \big( (T^J\fdhm  \xve^j)(\pave_j V^i)
   + T^J \fdhm V^i\big)\big)||_{L^2(\Omega)}\\
   \leq C_s
   ||\xve||_{H^{s+1}(\Omega)}
   \big( ||V||_{H^{(s-1,1/2)}(\Omega)}
   + P(||h||_{s-1})||\pave h||_{H^s(\Omega)} \big),\label{divest}
  \end{multline}
   \begin{equation}
   ||\gamma_{ij} D_t \fdhm T^I V^j +
   (\fdhm  T^I \gamma_{ij})\delta^{jk}\pave_k h ||_{L^2(\pa \Omega)}
   \leq C_s' (||\T \xve||_{H^{s+1/2}(\pa\Omega)} +1 )
   ||\pave h||_{H^{s+1/2}(\pa \Omega)}.
   \label{bdyest}
  \end{equation}
\end{lemma}
We recall that the terms on the left-hand sides of \eqref{Vest}-\eqref{divest}
are needed to control $\E^s$. We will eventually show
that $\E^s$ controls $||V||_{H^{(s,1/2)}(\Omega)}, ||\pave h||_{H^s(\Omega)}$
and $||\xve||_{H^{s+1}(\Omega)}$. Similarly we will use the estimate
\eqref{bdyest} to control  $\E^s_\ve$ and we will eventually show
that this controls $\ve^{-1} ||V||_{H^{s+1}(\Omega)}$ and
$\ve^{-1} ||\pave h||_{H^{s+1}(\Omega)}$.

The terms on the right-hand side of \eqref{bdyest} are
higher-order and to deal with them we  need to use tangential smoothing,
which introduces a term behaving like $\ve^{-1}\!\!$. Since we will estimate
$\ve \sqrt{\E^s_\ve}$ this will not cause issues.

The reason we do not commute $\fdhm$ all the way through on the
left-hand side of \eqref{divest} is that it would generate an error term
involving $||V||_{H^{s+1}(\Omega)}$ which can only be controlled
in terms of $\ve^{-1} \sqrt{\E^s_\ve}$, which would not allow us to close
the energy estimates in the next sections on a time interval independent of
$\ve$.

\begin{proof}
  We start by noting that if $V, h$ satisfy \eqref{regbds1}-
  \eqref{regbds2}, then all of the quantities on the right-hand
  sides of \eqref{Vest}-\eqref{bdyest} are finite.
  Therefore, by an approximation argument it suffices to prove
  this result assuming that $V, h$ are smooth.
We first show that
\begin{equation}
  ||\fdhm (D_t T^I V - \pave( (T^I \xve^j)(\pave_j h) )
  + T^I h)||_{L^2}
  \leq C_s ||\xve||_{H^{s+1}}
  (||\pave h||_{H^s} + ||\pave h||_{H^2})
  + ||\fdhm T^I \pave \phi(h)||_{L^2},
 \label{engoal1}
\end{equation}
where $H^s=H^s(\Omega)$ and $L^2=L^2(\Omega)$.
Note that by Lemma \ref{fracalg} we have:
\begin{equation}
 ||[T^I, \fdhm] D_t V||_{L^2(\Omega)} \leq C ||D_t V||_{H^s(\Omega)}
 = C ||\pave h||_{H^s(\Omega)},
 \qquad
 || \pave [T^I, \fdhm] \xve||_{L^2(\Omega)}
 \leq C(M) ||\xve||_{H^{s+1}(\Omega)},
 {}
\end{equation}
for $|I| \leq s$,
and so combining this with \eqref{engoal1} gives \eqref{Vest}.
To prove \eqref{engoal1}, we start by computing $T^I (D_t V + \pave h)$. The vector
fields $T^I$ commute with $D_t$ and so we just need to compute
$T^I \pave h$.
Using \eqref{dinv}, we have:
\begin{equation}
 T^I \pave h -
 \pave \big( T^I h - T^I \xve^j (\pave_j h)\big)
 = (T^I \xve^j)\pave\pave_j h
 + \sum (\pave T^{I_1} \xve)\cdots
 (\pave T^{I_\ell} \xve )(T^{I_{\ell+1}}\pave   h),
 \label{TIexp}
\end{equation}
where the sum is taken over all indices with $|I_1| + ...
+ |I_{\ell+1}| \leq |I|$ so that
$|I_j| \leq s-1$ for  $j\! \leq\! \ell$ and $|I_{\ell+1}| \geq \!1$.

To control the first term on the right-hand side of \eqref{TIexp} we
apply the fractional product rule \eqref{app:alg2}:
\begin{equation}
 ||\fdhm (T^I {}_{\!}\xve \cdot \pave^2{}_{\!} h)||_{L^2(\Omega)}
 \leq C ||T^I{}_{\!} \xve||_{H^2(\Omega)} (||\pave^2{}_{\!} h||_{L^2(\Omega)}
 + ||\fdhm \pave^2 {}_{\!}h||_{L^2(\Omega)})
 \leq C(M) ||\xve||_{H^5(\Omega)} ||\pave h||_{H^{(1{}_{\!},{}_{\!}1/{}_{\!}2)}(\Omega)},
\end{equation}
for $|I| \!= \!s\! \leq \!3$
as required.
If instead $|I| = s \geq 4$,
we use the fractional product rule \eqref{app:alg2} to bound:
\begin{equation}
 ||\fdhm (T^I \xve \cdot \pave^2 h)||_{L^2(\Omega)}
 \leq C||\fdhm T^I \xve||_{L^2(\Omega)}
 ||\pave h||_{H^3(\Omega)},
 {}
\end{equation}
and this is also bounded by the right-hand side of \eqref{Vest}.
It just remains to bound the terms in the sum in \eqref{TIexp}.
Suppose for each $j \leq \ell$, $|I_j| \leq 2$. By the Leibniz
rule \eqref{app:alg2}, we have:
\begin{equation}
 ||\fdhm\big( (\pave T^{I_1}\xve) \cdots (\pave T^{I_{\ell}}\xve)
 (T^{I_{\ell+1}} \pave h)\big)||_{L^2(\Omega)}
 \leq C|| (\pave T^{I_1} \xve)\cdots (\pave T^{I_\ell}\xve)||_{H^2(\Omega)}
 ||\fdhm \pave h||_{L^2(\Omega)}.
 {}
\end{equation}
Since $H^2(\Omega)$ is an algebra, the first factor is bounded
by $C(M) ||\xve||_{H^{5}(\Omega)}^\ell$. This just leaves the case
that there is at least one $j \leq \ell$ with $|I_j| \geq 3$. However note
that in this case we must have that $s \geq 4$ and $|I_{j'}| \leq s-4$ for
$j' \leq \ell +1$, $j' \not= j$, and so using \eqref{app:alg2}
and the algebra property of $H^2(\Omega)$ we have:
\begin{equation}
  ||\fdhm\big( (\pave T^{I_1}\xve) \cdots (\pave T^{I_{\ell}}\xve)
  (T^{I_{\ell+1}} \pave h)\big)||_{L^2(\Omega)}
  \leq
  C(M)||\fdhm \pave T^{I_j}\xve||_{L^2(\Omega)}
  ||\xve||_{H^{s-1}(\Omega)}^{\ell-1}
 ||\pave h||_{H^{s-2}(\Omega)}.
 {}
\end{equation}
The first factor is bounded by $||\xve||_{H^{s+1}(\Omega)}$ since
$I_{j} \leq s-1$, and this completes the proof of \eqref{engoal1}.

The estimate \eqref{divest} is similar. We will actually prove the
slightly stronger estimate:
\begin{equation}
  || \pa_y \big(e'(h) D_t T^J \!h - (\pave_i T^J \xve^j)(\pave_j \ssm V^i)
  + \div T^J V\big)||_{L^2(\Omega)}
 \leq C_s
  \big(||\xve||_{H^s(\Omega)}\! + \!1) (||\pave h||_{H^s(\Omega)}\! +\!
  P(||\pave h||_{H^{s-1}(\Omega)})\big),
\end{equation}
which implies \eqref{divest} since $||\fdhm \!f||_{L^2(\Omega)}
\!\leq\! ||f||_{H^1(\Omega)}$.
We apply $\pa_y T^J\!$ to $e^{\prime\!}(h)D_t h \!+ \!\div\! V\!\! = \!0$.
We start with
\begin{equation}
 ||\pa_y T^J( e'(h) D_t h) - \pa_y( e'(h)D_t T^{J} h)||_{L^2(\Omega)}
 \leq C_s P(||\pave h||_{H^{s-2}(\Omega)})
 (||\pave h||_{H^{s-1}(\Omega)} + ||D_t h||_{H^{s-1}(\Omega)}),
 {}
\end{equation}
which follows from a straightforward modification of the proof
of Lemma \ref{gestlem}.

It just remains to control $[T^J, \div]V$.
We start by writing:
\begin{equation}
 T^J \pave_i V^i = \pave_i T^J V^i
 - (\pave_i T^J \xve^j)(\pave_j V^i)
 + \sum (\pave T^{J_1} \xve)\cdots (\pave T^{J_\ell} \xve)
 ( T^{J_{\ell+1}} \pave V),
\end{equation}
where the sum is over all $J_1 + \cdots J_{\ell+1} = J$
with $|J_{\ell+1}| \geq 1$ and $|J_j| \leq |J| - 1  =s-2$ for
each $j \leq \ell+1$.
Applying $\pa_y$,
it suffices to control the $L^2$ norms of:
\begin{equation}
  (\pave T^{J_1} \xve) (\pave T^{J_2} \xve) \cdots
  (\pa_y T^{J_{\ell+1}}\pave  V),
  \quad
  \text{ and }
  \quad
  (\pa_y \pave T^{J_1} \xve) (\pave T^{J_2} \xve) \cdots
  (T^{J_{\ell+1}} \pave V).
 \label{stuff}
\end{equation}
Estimates for these terms can be obtained in the same way as the
estimates we used above to control the sum in \eqref{TIexp}.
To control the first term in \eqref{stuff},
if $|I_j|  \leq 2$ for each $j \leq \ell$, then we use Sobolev embedding:
\begin{equation}
 ||(\pave T^{J_1} \xve)(\pave T^{J_2} \xve)\cdots
 (T^{J_{\ell+1}} V)||_{L^2(\Omega)}
 \leq C ||(\pave T^{J_1} \xve)(\pave T^{J_2} \xve)
 \cdots(\pave T^{J_\ell} \xve)||_{H^2(\Omega)}
 || \pa_y T^{J_{\ell+1}} \pave V||_{L^2(\Omega)},
 {}
\end{equation}
and the first factor here is bounded by $C(M')$ using the fact that
$H^2(\Omega)$ is an algebra. To control the second factor,
we write it as $ ||\pa_y T^{J_{\ell+1}} (u\cdot \pa_y V)||_{L^2(\Omega)}$
and then note that since $|J_{\ell+1}| \leq s-2$, we can bound
$||\pa_y T^{J_{\ell+1}} (u\cdot \pa_y V)||_{L^2(\Omega)}
\leq C(M) ||\xve||_{H^s(\Omega)} ||V||_{H^s(\Omega)}$ by using
similar arguments to the above. If there is a multi-index $I_j$ with
$|I_j| \geq 3$ then this forces $|I_{j'}| \leq s-4$ for each
$j' \not= j$ and so we put all the factors except $\pave T^{I_{j}} \xve$
into $L^\infty$, apply Sobolev embedding and argue as above.
To control the first type of term from \eqref{stuff} is similar, noting
that in this case there are no more than $2\! + \!(s\!-\!2)$ derivatives falling
on $\xve$ at any point.

We now prove \eqref{bdyest}. We apply $T^I$ to
$\gamma^{ij}D_t v_j|_{\pa \Omega} = \gamma^{ij}\pave_j h|_{\pa \Omega} = 0$ where
$v_i = \delta_{ij}V^j$, and we have:
\begin{equation}
 0 = (T^I \gamma^{ij})\pave_i h
 + \gamma^{ij} T^I \pave_i h + \sum
 (T^{I_1} \gamma^{ij})\cdots (T^{I_\ell} \gamma^{ij})
  (T^{I_{\ell+1}} \pave_i h),
\end{equation}
where $|I_1| + \cdots |I_{\ell+1}| \leq s, |I_{\ell + 1}| \geq 1,
|I_j| \leq s-1, j \leq \ell+1$.

We now recall that $\gamma^{ij} = \gamma^{ab} A^i_{\m a} A^j_{\m b}$ where
$\gamma^{ab} = \delta^{ab}  - N^aN^b$ and in particular $\gamma^{ab}$
is independent of $\xve$. Applying \eqref{dinv} repeatedly, it therefore
suffices to control the $L^2(\pa \Omega)$ norms of:
\begin{equation}
  \fdhm(\gamma^{ij} T^I \pave_i h), \qquad
 \fdhm \big((\pave T^I \xve) (\pave h)\big),
 \qquad
 \fdhm \big((\pave T^{I_1} \xve)\cdots (\pave T^{I_\ell} \xve)
 (T^{I_{\ell+1}} \pave h)
 \big),
 \label{gammaterms}
\end{equation}
with the same conditions on the multi-indices $I_1,..., I_{\ell+1}$
as above.

To deal with the first term in \eqref{gammaterms},
we just use the Leibniz rule \eqref{app:alg2} and control it by:
\begin{equation}
 ||\gamma^{ij}||_{H^2(\pa \Omega)} ||\fdhm T^I \pave h||_{L^2(\pa \Omega)}
 \leq C(M)||\xve||_{H^3(\pa \Omega)}||\fdhm T^I \pave h||_{L^2(\pa \Omega)},
 {}
\end{equation}
since $\gamma$ is quadratic in $A$, where we have use the fact that
$H^2(\pa \Omega)$ is an algebra. By the trace inequality this is
controlled by the right-hand side of \eqref{bdyest}.
To deal with the second term in \eqref{gammaterms}, when
$|I| = s \leq 2$ we use the Leibniz rule \eqref{app:alg2}
and control it by $||\pave T^I \xve||_{H^2(\pa \Omega)}
||\pave h||_{H^{1/2}(\pa \Omega)}$, and this first
factor is controlled by $C(M)||\xve||_{H^5(\pa\Omega)}
\leq C(M)||\xve||_{H^6(\Omega)}$ by the trace inequality.
If $|I| = s \geq 3$ we instead control it by
$||\fdhm \pave T^I \xve||_{L^2(\pa \Omega)} ||\pave h||_{H^2(\pa \Omega)}$,
which is bounded by $C(M)||\T \xve||_{H^{s+1/2}(\pa\Omega)}$
times $||\pave h||_{H^s(\pa \Omega)}$.

Estimates for the third term in \eqref{gammaterms} can be obtained in a
similar fashion. If $|I_j| \!\leq \!2$ for each $j \!\leq \!\ell$, then
\begin{equation}
 ||\fdhm \big( (\pave T^{I_1} \xve)\cdots
 (\pave T^{I_\ell} \xve) (T^{I_{\ell+1}} \pave h)||_{L^2(\pa \Omega)}
 \leq C(M) ||\xve||_{H^5(\pa \Omega)}^\ell ||\fdhm\pave h||_{H^{s-1}(\pa \Omega)},
 {}
\end{equation}
and by the trace inequality this is bounded by the right-hand side of
\eqref{bdyest}. If instead $|I_j| \geq 3$ for some $j \leq \ell$ then this
forces $s \geq 4$ and $|I_{j'}| \leq s-4$ for $j'\not = j$
and so the result is bounded by:
\begin{equation}
 ||\fdhm \pave T^{I_{j}} \xve||_{L^2(\pa\Omega)}
 ||\xve||_{H^{s}(\pa \Omega)}^{\ell-1}
 ||\pave h||_{H^{s-1}(\pa \Omega)}
 \leq C_s||\T \xve||_{H^{s-1/2}(\pa \Omega)} ||\pave h||_{H^{s-1}(\pa\Omega)}.
 \tag*{\qedhere}
\end{equation}
\end{proof}

\section{Uniform Energy estimates for the smoothed problem up to a fixed time}
\label{uniform}
We define:
\begin{equation}
\E^s = \K^s + {\sum}_{\mu = 0}^N {\sum}_{|I| \leq s} \E_\mu^{I}, \qquad
\E^s_\ve = \K_\ve^s + {\sum}_{|I| \leq s} \E^I_{\ve},
\end{equation}
where, with $\fdhm$ defined by \eqref{fdmudef}:
\begin{equation}
\K^s = {\sum}_{\mu = 0}^N ||\curl \fdhm V||_{H^{s-1}(\Omega)}^2,
\qquad
\K_\ve^s = || \curl \pa V||_{H^{s-1}(\Omega)}^2
+ ||\div \pa V||_{H^{s-1}(\Omega)}^2,
\end{equation}
and, with notation as in Sections \ref{smoothingsec} and \ref{tangsec}:
\begin{equation}
\E_{\!\mu}^{{}_{\!}I}\! =\!
\!\int_\Omega\! \delta_{ij} (T^I \!\fdhm V^i)(T^I\!\fdhm V^j)
{}_{\!}+{}_{\!} e^{\prime\!}(h) |T^I \fdhm h|^2\,\, \kve dy
+\! \int_{\pa \Omega}\!\!\!\! N_i N_{\!j} (T^I \! \fdhm \!\sm x^i)
(T^I \!\fdhm \!\sm x^j)\, |\pave h| \widetilde{\nu} dS,
\end{equation}
where $\widetilde{\nu} dS$ is the Eulerian surface measure, and
\begin{equation}
\E^I_{\ve} = \int_{\pa \Omega} \gamma_{ij}
(T^I V^i) (T^I V^j)\, \widetilde{\nu} dS, \qquad T^I \in \T^s.
\end{equation}

To control $h$, we will use:
\begin{equation}
\W^s = {\sum}_{k \leq s} \int_\Omega e'(h)  |D_t^{k+1} h|^2
+ |D_t^k \pave h|^2\, \kve dy,
\end{equation}
and to control $\xve$ we will use:
\begin{equation}
\A^s =
||\div \pa_y \sm  x ||_{H^{s-1}(\Omega)}^2
+ ||\curl\pa_y \sm x ||_{H^{s-1}(\Omega)}^2.
\label{Adef}
\end{equation}
The energy we consider is then:
\begin{equation}
\Ee^s = \A^s + \W^s + \E^s + \ve^2 \Eev^s.
\end{equation}
We will also write $\Ee^s_0$ for the quantity $\Ee^s$ with $V$ replaced
by $V_0^\ve$, $x$ replaced with $x_0$ and $D_t^{k+1} h$ replaced by $h_{k+1}^\ve$, with
$V_0^\ve, h_{k+1}^\ve$ defined by in Section \ref{wavecompatcondn}.
The goal of this section is to prove the following theorem:
\begin{theorem}
 \label{eneestthm}
 Suppose that the initial data
 $(V_0^\ve\!, h_0^\ve)$ are such that $\Ee^s_0 \!< \!\infty$ for some
 $s \!\geq\! 1$.
 There is a positive, continuous function $\mathcal{F}_{\!s}$
 so that the following holds:
 If $V \!\! \in \! \X^{s+1\!}(T)$ is a solution
 to the smoothed Euler's equations \eqref{smwavedef}-\eqref{smpbmdef}
 so that $(V\!{}_{\!}, {}_{\!}h)|_{t = 0}\! =\! (V_0^\ve\!, h_0^\ve)\!$ and
 the a priori assumptions \eqref{qchiassump1}-\eqref{qchiassump2} hold, then:
 \begin{equation}
  \Ee^s(t) \leq \mathcal{F}_s\left(M', L, \delta^{-1},\Ee^{s-1}_0
  \right)
   \Ee^s_0,
  \qquad 0 \leq t \leq T.
  \label{enest1}
 \end{equation}
\end{theorem}
We now take $M_0', L_0, \delta_0 > 0$ so that:
\begin{equation}
|{\pa x_0}/_{\!}{\pa y}|{}_{\!} + {}_{\!}|\pa y_{{}_{\!}}/_{\!} \pa x_{0{}_{\!}}|{}_{\!} +{}_{\!} {\tsum}_{k + |_{\!}I_{{}_{\!}}| \leq 3} |\pa_y^{I} V_k^\ve{}_{\!}| {}_{\!}+{\!}
||\xve_0{}_{\!}||_{H^{6}(\Omega)}\!{}_{\!} \leq \! M_0',
\quad
{\tsum}_{k + |_{\!}I_{{}_{\!}}| \leq 3} |\pa_y^{I} \pave h_{{}_{\!}k}^\ve|{}_{\!}+\! |h_{{}_{\!}k}^\ve|
 \!\leq\! L_0,
 \quad \!-\pave_{N_0} h_0 |_{\pa \Omega}\! \geq\! \delta_0,
\label{aprioriM0L0}
\end{equation}
where we are writing $\pave_{N_0} = N_0^i \pave_i$ with $N_0^i$ the
unit normal to $\pa \Omega$ with respect to the metric $\gve$ at $t = 0$.

We will show that the above energy estimate implies:
\begin{cor}
 \label{useful}
 Let $r \geq 7$. There are continuous, strictly positive
functions $\mathscr{T}_r,\Cc_r,\Cc_r'$
with $\Cc_r, \Cc_r'$
depending on
$M_0', L_0, \delta_0^{-1}, \Ee^{r-1}_0$ so
that if $$T \leq \mathscr{T}_r(M_0', L_0, \Ee^{r-1}_0,\delta^{-1}_0),$$
 and $V \in \X^{r+1}(T)$ satisfies the smoothed-out Euler equations
\eqref{smpbmdef} with initial data satisfying
\eqref{aprioriM0L0}, then:
\begin{equation}
 \Ee^{r-1}(t) \leq \Cc_r \Ee_0^{r-1}, \qquad 0 \leq t \leq T,
 \label{useful1}
\end{equation}
and with $H^s = H^s(\Omega)$
\begin{equation}
 ||V(t)||_{H^{(r-1, 1/2)}}^2 +
 ||\sm x(t)||_{H^r}^2 + ||\pave h(t)||_{r-1}^2
  + \ve^2 (||V(t)||_{H^r}^2
  + ||\pave h(t)||_{H^{r}}^2)
 \leq \Cc_r'\Ee^{r-1}_0, \qquad 0 \leq t \leq T.
 \label{useful2}
\end{equation}
\end{cor}

Before proving Theorem \ref{eneestthm}, we collect a few preliminary
results.
In Lemma \ref{coercivitylem1},
we show that we control $\xve$, $V$ and $h$
provided we control $\A^s$, $\W^s$ and the energies $\E^s$.
In Lemma \ref{coercivitylem2} and Corollary \ref{coercivitylem3},
we show that we control $\A^s$ and $\W^s$ provided that we control $\Ee^s$.

\begin{lemma}
 \label{coercivitylem1}
 Fix $r \!\geq \! 7\!$ and suppose that $V \!\!\in\!\X^{r+1}(T)$ satisfies
 the smoothed-out Euler equations \eqref{smpbmdef} and that
 the apriori assumptions
 \eqref{ubd2},\eqref{Lassumpwave} and the
 Taylor sign condition \eqref{int:tsc} hold.
 For each $0 \!\leq \! s \! \leq \! r\!-\!1$, there is a positive,
 continuous function
 $C_s \!= C_s(M'\!\!, L, \delta^{-1}\!\!, \A^{s-\!1}\!\!, \W^{s-\!1\!}\!, \E^{s-\!1})$
 so that the following estimates hold:
 \begin{equation}
  ||\xve||_{H^{s+1}(\Omega)}^2
  + ||V||_{H^{(s,1/2)}(\Omega)}^2 + ||V||_{\X^{s+1}}^2 +
  ||\pave h||_s^2 + ||D_t h||_s^2
   \leq C_s \big( \A^s + \W^s
  + (1 + \delta^{-1}) \E^s \big),
  \label{easycoer}
 \end{equation}
 and for $\mu = 0,..., N$, we have:
\begin{equation}
 ||\sm x||_{H^{s+1}(\Omega)}^2 +
 ||\fdhm V||_{H^s(\Omega)}^2 + ||\fdhm T^s \pave \phi(h)||_{L^2(\Omega)}
  \leq C_s\big(\A^s + \W^s +  (1 + \delta^{-1})\E^s\big).
  \label{hardcoer}
\end{equation}
Finally:
\begin{equation}
 ||V||_{H^{s+1}(\Omega)}^2 + ||\pave h||_{s+1}^2
 \leq  C_s (\A^s + \W^s + (1 + \delta^{-1})\E^s + \ve^{-2}\E^s_\ve ).
 \label{Vepscoercivity}
\end{equation}
\end{lemma}
\begin{proof}
  The first estimate in \eqref{easycoer} follows from the first
  estimate in \eqref{hardcoer}, since
  $\xve \!=\!\ssm x \!= \!\sm^2 x$ and $\sm$ is bounded
  on Sobolev spaces.
 The second estimate in \eqref{easycoer} follows after summing
 the second estimate in \eqref{hardcoer} over all $\mu\! = \!0,..., N$
 and using Lemma \ref{equivalentnorms}.
 To prove the third estimate
 we note that if $V\!$ solves the smoothed problem \eqref{smpbmdef}
 then $||V||_{\X^{s+1}}\!\! \leq\! ||V||_{H^{s}(\Omega)}\!
 + \!||\pave h||_{s}\! + \!||\pave \phi||_s$. Using Theorem
 \ref{main theorem, ell est of phi} to control $||\pave\phi||_s$,
 this estimate then follows from
 the estimate for $||\pave h||_s$. To prove the estimate
 for $||\pave h||_s$ and $||D_t h||_s$, we
 argue as in the proof of \eqref{alldt1} and suppose that
 \eqref{easycoer} holds
 for $s \!=\! 0,..., m\!-\!1$. By definition $||D_t^m \pave h||_{L^2(\Omega)}^2
 \!+ \!||D_t^{m+1}\! h||_{L^2(\Omega)}^2\!
 \!\leq\! C \W^s$   so we now suppose that
 $||D_t^{k} \pave h||_{H^{\ell}(\Omega)}^2\!
 +\! ||D_t^{k+1\!} h||_{H^{\ell}(\Omega)}^2$ is bounded by the right-hand
 side of \eqref{easycoer} for $k + \ell\! = \!s$ and some $\ell\! \geq\! 0$.
 By induction it suffices to prove that $||D_t^{k-1\!} \pave h||_{H^{\ell+1}(\Omega)}^2
 + ||D_t^{k-1}\! D_t h||_{H^{\ell+1}(\Omega)}^2$ is bounded by
 the right-hand side of \eqref{easycoer}.
 Writing $\pa_a D_t^{k-1}\! D_t h \!=\! A_{\m a}^i \pave_i D_t^{k-1}\! D_t h$
 and then $\pave D_t^{k-1}\! D_t h \!= \!D_t^k \pave h + [\pave, D_t^{k}] h$,
 using \eqref{vectcomm} to handle the commutator and \eqref{product}:
 \begin{equation}
  ||\pa_y D_t^{k-1} D_t h||_{H^{\ell}(\Omega)}
  \leq C(M', ||\xve||_{H^{\ell}(\Omega)}, ||V||_{\X^{m-1}})
   \big( ||D_t^{k} \pave h||_{H^{\ell}(\Omega)} +
  (||V||_{\X^{m}} + ||\xve||_{H^{\ell+1}(\Omega)} )||\pave h||_{m-1}\big).
 \end{equation}
 When $\ell = 0$ then by the inductive assumption and the definition of the energy
 $\W$, all of the terms on the right-hand side are bounded
 by the right-hand side of \eqref{easycoer}.
 The estimate \eqref{easycoer} for $s = m$ now follows from the inductive
 assumption and the following
 estimate, which we claim holds
 whenever $k + \ell =m, \ell \geq 1$:
 \begin{equation}
  ||D_t^k \pave h||_{H^{\ell}(\Omega)}^2
  \leq C
  \big( ||D_t^{k+2} h||_{H^{\ell-1}}
  + (||\xve||_{H^{m+1}} + ||V||_{H^{m}} + ||V||_{\X^{m}})
   (||\pave h||_{m-1} + ||D_t h||_{m-1})\big),
  {}
 \end{equation}
 where $C\! =\! C(M,\! ||\xve||_{H^{m}(\Omega)},\!
 ||V||_{\X^m}\!)$ and $H^s \!\!= \!H^s(\Omega)$. This estimate follows directly
 from the elliptic estimate \eqref{sobellmix}, the fact that
 $D_t^k \Dve h = D_t^{k}\big(e'(h) D_t^2 h - (\pave_i \ssm V^j)(\pave_j V^i)
 \big)$
 and Lemmas \ref{fbdslem} and \ref{gestlem} to control these.

 We now prove the first estimate in \eqref{hardcoer}.
When $s \leq 6$ there is nothing to prove
since $|| \sm x||_{H^6(\Omega)} \leq M'$, so we assume $s \geq 6$. In
fact the below argument works provided $s \geq 2$ and this assumption
is only needed to ensure that the trace map is continuous. The point
of the below manipulations is to replace the derivative $\pa_y$ with
tangential vector fields $T$. Using
\eqref{ftang1}, we have that:
\begin{equation}
||\pa_y \sm x||_{H^s(\Omega)}^2
\leq C_s \big( \A_s + ||\pa_y \sm x||_{H^{s-1/2}(\pa \Omega)}^2 +
||\sm x||_{H^1(\Omega)}^2\big),
\end{equation}
with $C_s = C_s(M', ||\xve||_{H^s(\Omega)})$. To control the boundary
term here it suffices to control
$
||T \pa_y \sm x||_{H^{s-3/2}(\pa \Omega)}$ for $s \geq 2$ and any
$T \in \T$, and by
the trace inequality \eqref{trace}, this is under control if we control
$||T \sm x||_{H^{s}(\Omega)}$. Finally, we note that because of the boundary
term in the energy, for each $T \in \T$ we have:
\begin{equation}
||T \sm x||_{H^s(\Omega)}^2
\leq C_s \big( ||\div T \sm x||_{H^{s-1}(\Omega)}^2
+ ||\curl T \sm x||_{H^{s-1}(\Omega)}^2 +
\delta^{-1} \E_s\big),
{}
\end{equation}
again with $C_s = C_s(M', ||\xve||_{H^{s}(\Omega)})$.
The first and second terms here are bounded by $\A^s$ and using
induction and the first estimate in \eqref{easycoer},
this implies the first estimate in \eqref{hardcoer}.

To prove the second estimate in \eqref{hardcoer},
we note that by the elliptic estimate \eqref{app:hotpw}, we have:
\begin{equation}
||\fdhm V||_{H^s(\Omega)}^2
\leq C(M', ||\xve||_{H^s(\Omega)} )\big(||\div \fdhm V||_{H^{s-1}(\Omega)}^2
+ ||\curl \fdhm V||_{H^{s-1}(\Omega)}^2 +
\E^s\big).
\end{equation}
The last two terms are controlled by the right-hand side of
\eqref{hardcoer}. For the first term, we use
Lemma \ref{fracalg}:
\begin{equation}
||[\div, \fdhm ]V||_{H^{s-1}(\Omega)}
\leq C(M', ||\xve||_{H^s(\Omega)})
\big(||\xve||_{H^{s+1}(\Omega)} + ||V||_{H^s(\Omega)}\big),
\end{equation}
and so using $\div V = -e'(h) D_t h$, it just remains
to bound $||\fdhm (e'(h) D_t h)||_{H^{s-1}(\Omega)}$. We first bound this
by $||e'(h) D_t h||_{H^s(\Omega)}$ and then using induction and Lemma
\ref{gestlem}, this is controlled by
$C(M'\!\!, L, \W^{s-1}) ||D_t h||_{H^s(\Omega)}$. We write
$\pa_y^I \!=\! \pa_y^J (u \cdot \pave )$, where $|I| \!=\!s, |J| \!=\! s\!-\!1$, then apply
\eqref{product}
and the commutator estimate \eqref{leibpave}
to control this by $C_s ||\pave h||_{s}$.
Since $\rho\!=\!\rho(h)$,  the third estimate in \eqref{hardcoer} is a consequence of Theorem \ref{r-0.5 estimate for phi}, \eqref{easycoer}.

The estimate \eqref{Vepscoercivity} follows from the definition of
$\E^s_\ve$, the elliptic estimate \eqref{ellft2} and \eqref{easycoer}.
\end{proof}

We now control the energy for the wave equation $\W^s$ in terms of
$\E^s, \A^s$:
\begin{lemma}
 \label{coercivitylem2}
 With the same hypotheses as Lemma \ref{coercivitylem1},
 there is a constant $C'_s$ depending on $M',$ $L,$ $\delta^{-1}$, $T$,
  $\sup_{\,0 \leq t \leq T}\A^{s-1}(t) +  \E^{s-1}(t)$ and
   $\W^{s-1}(0)$ so that:
 \begin{equation}
  \W^s(t) \leq C'_s \Big( \W^s(0) + \int_0^t
   \A^s(\tau) + \E^s(\tau) \, d\tau\Big), \quad 0 \leq t \leq T.
  \label{hcoercive}
 \end{equation}
\end{lemma}

\begin{proof}

By Lemma \ref{dtW}, writing $\F = -(\pave_i \ssm V^j)(\pave_i V^i)$,
we have:
\begin{equation}
\frac{d}{dt} \sqrt{\W^s} \leq C^1_s \big( ||\F||_{s,0} +
||\F||_{s-1} + ||V||_{\X^{s+1}} +
P(L, ||\pave h||_{s-1}, \sqrt{\W^{s-1}})\big),
{}
\end{equation}
with $C^1_s = C^1_s(M', L, T, ||\xve||_{H^s(\Omega)},
||V||_{\X^{s}})$.
Using Lemma \ref{fbdslem}
to control $\F$ and Lemma \ref{coercivitylem1} to control $V$, $\xve$
and $\pave h$ in terms of $\A, \W$ and $\E$, this implies:
\begin{equation}
\frac{d}{dt} \sqrt{\W^s}
\leq C_s^2\big(M', L, \delta^{-1}, T, \sqrt{\W^{s-1}}, \sqrt{\A^{s-1}},
\sqrt{\E^{s-1}}\big) \big( \sqrt{\W^s} + \sqrt{\A^s} +  \sqrt{\E^s} \big).
{}
\end{equation}
Multiplying by the integrating factor
$e^{-tC_s^2}$, integrating from 0 to $T$ and using induction
gives
\eqref{hcoercive}.
\end{proof}

We will need the following estimate to control $\xve$:
\begin{lemma}
 \label{divcurlbounds}
 For $s \geq 0$,
 there are constants $C'_{\!s\!}$ depending on  $M'\!\!,L,
 \delta^{-1}\!\!,T$ and  $\sup_{0 \leq t \leq T}\big(
  \A^{s-1}(t)+ \E^{s-1}(t)\big)$
 so that if \eqref{strongM} is satisfied, then for $0 \leq t \leq T$:
 \begin{equation}
||D_t \div  \sm \pa_y x||_{H^{s-1}(\Omega)}^2
\leq
   C_{\!s}'\big(
   ||\div V||_{H^{s}(\Omega)}^2\!
   + || \big(\sm \div \pa_y V\!\! - \div \sm \pa_y V\big)||_{H^{s-1_{\!}}(\Omega)}^2\!
   + \A^s\!{}_{\!}  + \E^s  \big).
  \label{arbitrarydiv}
\end{equation}
In addition, for any multi-index $I$ with $|I| = s-1$
and $\mu = 0,...., N$ there
is a two-form $R = R^{I}_{ij}$ with
$||R||_{L^2(\Omega)} \leq C_s'(\A^s + \E^s)$
so that for $0 \leq t \leq T$:
\begin{equation}
  ||D_t^2 \pa_y^I ( \curl \!\pa_y \sm x)
  - D_t R^{I\!}||_{L^2(\Omega)}^2 \! \leq
    C_{\!s\!}'\big( ||\! \curl {}_{\!}D_t V||_{H^{s}(\Omega)}^2\!
    + ||\sm{}_{\!} \curl \pa_y V\!\! - \curl \sm\pa_y V||_{H^{s\!-\!1_{\!}}(\Omega)}^2\!
    + \A^s\! + \E^s  \big).
   \label{arbitrarycurl}
\end{equation}
\end{lemma}
\begin{proof}
 We start by writing
 $
 D_t \div \pa_y \sm x =
 -(\pave_i \ssm V^j )\pave_j \pa_y \sm x^i
 + \div \pa_y \sm V$. Applying
 $s-1$ derivatives to this expression, we first prove:
 \begin{equation}
  ||\pa_y^J (\pave \ssm V) \pa_y^K (\pave \pa_y \sm x)||_{L^2(\Omega)}
  \leq C_s'(\A^s + \E^s), \quad |J| + |K| = s-1.
  \label{prodest}
 \end{equation}
  When $|J| \leq 2$ we bound the first factor
 in $L^\infty$ by $M'$ and the second factor by
 $||\pave \pa_y \sm x||_{H^{s-1}(\Omega)} \leq C(M)||\sm x||_{H^{s+1}(\Omega)}$.
 If  $|J| \!\geq\! 3$ then $|K| \!\leq \!s\!-\!4$ and so we bound the
 first factor in $L^2(\Omega)$ by $||V||_{H^{s}(\Omega)}$ and the
 second factor by $||\pa_y^K \pave \pa_y \sm x||_{L^\infty(\Omega)}
 \leq C||\pave \pa_y \sm x||_{H^{s-2}(\Omega)} \leq C(M)||\xve||_{H^s(\Omega)}
 ||\sm x||_{H^{s}(\Omega)}$.
 By Lemma \ref{coercivitylem1}, we control all of these
 terms by the right-hand side of \eqref{arbitrarydiv}.

 We now control $||\div \pa_y \sm V||_{H^{s-1}(\Omega)}$.
 Noting that $||(\sm \div - \div \sm)V||_{H^{s-1}(\Omega)}$ appears on
 the right-hand side of \eqref{arbitrarydiv}, and that
 $[\pa_y, \!\sm]$  and  $\sm$ are bounded
 operators on $H^{s\!-\!1\!}(\Omega)$, it remains to control
 $||[\div\!, \pa_y] V||_{H^{s\!-\!1\!}(\Omega)}$. Writing
 $[\div\!, \pa_y]V\!\!=\! -(\pa_y A_{\m i}^a)\pa_a \!V^i\!$ and arguing as above,
 we have $||[\div \!, \pa_y] V\!||_{H^{s\!-\!1\!}(\Omega)}\!\!
 \leq\! C(\!M\!) ||\xve||_{\!H^{s\!+\!1\!}(\Omega)} ||V\!||_{\!H^s(\Omega)}$, and
 again using Lemma \ref{coercivitylem1} this is bounded by the
 right-hand side of \eqref{arbitrarydiv}.

 To prove \eqref{arbitrarycurl}, we start by writing:
 \begin{multline}
  D_t^2 (\curl \pa_y \sm x)_{ij} = D_t \big(
   \big(\!\!-(\pave_i \ssm{}_{\!} V^\ell) \pave_\ell \pa_y\sm x_j +
   (\pave_j \ssm{}_{\!} V^\ell)\pave_\ell
   \pa_y \sm x_i \big)
   + (\curl \pa_y\sm D_t x)_{ij}\big)\\
   \!=\!D_t
   \big(\!
   (\pave_j \ssm {}_{\!} V^\ell)\pave_\ell
   \pa_y \sm x_i-(\pave_i \ssm {}_{\!}V^\ell) \pave_\ell \pa_y \sm x_{{}_{\!}j{}_{\!}}  \big)
   - (\pave_i \ssm{}_{\!} V^\ell)\pave_\ell \pa_y \sm D_t x_{{}_{\!}j{}_{\!}}
   + (\pave_j \ssm {}_{\!} V^\ell)\pave_\ell
   \pa_y \sm D_t x_{{}_{\!}i{}_{\!}}
   + (\curl \pa_y \sm D_t^2 x)_{ij}.
 \end{multline}
 The last two terms will be too high-order after
 we apply $s-1$ derivatives since we do not want an estimate that
 involves
 $||V||_{H^{s+1}(\Omega)}$. To handle this, for each of
 these terms,
 we write
  $(\pave \ssm V)\cdot \pave \pa_y D_t \sm x
 = D_t (\pave \ssm V \cdot \pave \pa_y \sm x) - (\pave \ssm D_t V)\cdot \pave \pa_y \sm x
 + (\pave \ssm V)\cdot (\pave \ssm V)\cdot \pave \pa_y \sm x$. Writing:
 \begin{equation}
  R_{ij} = 2\big( -(\pave_i \ssm V^\ell) \pave_\ell \pa_y \sm x_j +
  (\pave_j \ssm V^\ell)\pave_\ell
  \pa_y\sm x_i\big),
 \end{equation}
 we have shown that for some constants $\alpha^{ijk\ell mn}, q^{ijk\ell}$:
 \begin{equation}
  D_t^{2\!} \curl \pa_y \sm x_{ij} \!-\! D_t R_{ij}\!
  =\!\curl \pa_y \sm D_t^2 x_{ij}
  + \!\!\sum\!\! \alpha_{j\ell}^{ik mn} (\pave_i\ssm V^j)(\pave_k \ssm V^\ell)
  \pave_m \pa_y \sm x_n
  \!+\! q_{jk}^{i \ell}(\pave_i \ssm D_t V^j)\pave_k \pa_y \sm x_\ell,
 \end{equation}
 For multi-index $I$ with $|I| \!=\! s-1$, we
 define $R^{I} \!\!= \! \pa_y^I R$. The estimate for $R^{I}$
 follows exactly as above estimates.

 Using \eqref{prodest}, we have that $R^{I}$
 satisfies the stated estimate so
 it just remains to control the terms in the sum after applying
 $\pa_y^I$. To control the second term in the sum, we note that we also have
 \eqref{prodest} with $V$ replaced by $D_tV = -\pave h - \pave \phi$,
 using the estimates in Lemma \ref{coercivitylem1} for $\pave h$.

To control the first term in the sum, we argue as in the proof of
\eqref{prodest}.
 If $|J| + |K| + |L| = s-1$ and
 either $|J|, |K| \leq 2$ then $||\pa_y^J \pave \ssm V||_{L^\infty(\Omega)}
 ||\pa_y^K \pave \ssm V||_{L^\infty(\Omega)} ||\pa_y^L \pave \pa_y \sm x||_{L^2(\Omega)}
 \leq C(M) ||\sm x||_{H^{s+1}(\Omega)}$ and if instead one of $|J|, |K| \geq
 3$ then without loss of generality it is $|J|$ and then $|K|, |L| \leq s-4$,
 so by Sobolev embedding, $||\pa_y^J \pave \ssm V||_{L^2(\Omega)}
 ||\pa_y^K \pave \ssm V||_{L^\infty(\Omega)}
 ||\pa_y^L \pave \pa_y \sm x||_{L^\infty(\Omega)}
 \leq C(M) ||V||_{H^s(\Omega)}^2 ||\sm x||_{H^{s}(\Omega)}$,
 as required. Finally, using the same arguments as above
 we can re-write $\pa_y^I \curl \pa_y \sm D_tV$ in terms of
 $\pa_y^I \curl \pa_y D_tV$ and terms with $L^2$ norms
 bounded by the right-hand side of \eqref{arbitrarycurl}. Finally, we
 note that:
 \begin{equation}
  ||\pa_y^I ([\curl, \pa_y] D_tv)||_{L^2(\Omega)}
  \leq C(M, ||\xve||_{H^s(\Omega)})
  ||\xve||_{H^{s+1}(\Omega)} ||D_tV||_{H^s(\Omega)},
 \end{equation}
 which follows from the fact that $([\curl,\pa_y]D_tV)_{ij} = (\pa_y A_{\m j}^a)\pa_a
 D_t v_i - (\pa_y A_{\m j}^a)\pa_a D_tv_j$ with $v_i = \delta_{ij} V^j$,
 and using the above arguments. Using the smoothed-out Euler's equations
 $D_t V \!= -\pave h - \pave \phi$ and Lemma \ref{coercivitylem1}, we have
 \eqref{arbitrarycurl}.
\end{proof}

For the next estimate, we write $\Eee^s\!
= \W^s\! + \E^s\! + \ve^2 \E^s_\ve$ for the part of $\Ee$ that does not
involve $\A$.
We  have:
\begin{cor}
 \label{coercivitylem3}
For each $s \geq 0$, there is a continuous function
$\C_s$ depending on $M'\!$, $L$, $\delta^{-1}\!$, $T$,
$\A^{s-1}(0)$, $\W^{s-1\!}(0)$, $\sup_{\,0 \leq t \leq T\,} \Eee^{s-1}(t)$
so that if $V\!\! \in\! \X^{s+1}(T)$ satisfies \eqref{smpbmdef}, then with $\A_s\!$
as in \eqref{Adef}:
\begin{equation}
 \A_s(t)
 \leq \C_s \big(\A_s(0) +  \int_0^t (1+\tau) \Eee^s(\tau)\,  d\tau\big).
 \label{Aestimate}
\end{equation}
\end{cor}
\begin{proof}
 Using Lemmas \ref{smproperties} and \ref{coercivitylem1}, we have:
 \begin{equation}
  ||[\sm, \div] V||_{H^{s-1}}^2 +
  ||[\sm, \curl] V||_{H^{s-1}}^2
  \leq \ve^2 C' (||\xve||_{H^{s+1}}^2
  + ||V||_{H^{s+1}}^2)
  \leq C'
  (\A^s\! + (1+\delta^{-1}) \Eee^s),
  \label{weighted1}
\end{equation}
 with $C' =  C'(M', L, \delta^{-1}, \A^{s-1}, \Eee^{s-1})$, and with
 $H^{k} = H^k(\Omega)$,
 noting that the highest-order term in the second inequality is multiplied by
 $\ve^2$. Integrating \eqref{arbitrarydiv} once in time, we have:
 \begin{equation}
  ||\div \sm \pa_y x(t)||_{H^{s-1}(\Omega)}^2
  \leq  ||\div \sm\pa_y x(0)||_{H^{s-1}(\Omega)}^2+
  \int_0^t ||D_t \div \sm \pa_y x(\tau)||_{H^{s-1}(\Omega)}^2\, d\tau.
  {}
 \end{equation}
 If $|I| = s-1$ then with $R^{I}$ as defined in
 Lemma \ref{divcurlbounds}, then integrating \eqref{arbitrarycurl} twice in
 time, we also have:
 \begin{multline}
  ||\pa_y^I \curl \sm \pa_y x(t)||_{L^2}^2
   \leq ||\pa_y^I \curl \sm \pa_y x_0||_{L^2}^2
    + \int_0^t ||D_t \pa_y^I \curl \pa_y \sm x_0 - R^{I}_0||_{L^2}^2
   + ||R^{I}(\tau)||_{L^2}^2\, d\tau\\
   + \int_0^t \int_0^{\tau} ||D_t^2 \pa_y^I \curl \pa_y \sm x(\tau')
   - D_t R^{I}(\tau')||_{L^2}^2\, d\tau' d\tau,
 \end{multline}
 with $L^2 = L^2(\Omega)$ and $R^{I}_0 = R^{I}|_{t = 0}$.
We have $$|| D_t \pa_y^I
 \curl \pa_y \sm x(0) - R^{I}(0)||_{L^2(\Omega)} \leq
 \C_s''(M, L, \delta^{-1}\!\!, \A^{s-1}(0),  \Eee^{s-1}_0)(\A^{s}(0) \!+ \Eee^s_0).$$
 We now use the facts that $\div\! V \!\!= -e'(h) D_t h, \curl D_t V \!\!= 0$,
 the estimates \eqref{arbitrarydiv}-
 \eqref{arbitrarycurl}.
  Using \eqref{weighted1} and  Lemma \ref{divcurlbounds}
  for $R$, we get:
 \begin{align}
  \A^s(t) &\leq \A^s(0)
  + \C_s'\Big(\int_0^t \A^s(\tau) + \Eee^s(\tau)\, d\tau
  + \int_0^t \int_0^\tau \A^s(\tau') + \Eee^s(\tau')\, d\tau' d\tau\Big)\\
  &\leq
   \A^s(0) + \C_s' \Big(\int_0^t (1 + \tau) \A^s(\tau)\, d\tau
  + \int_0^t (1 + \tau) \Eee^s(\tau)\, d\tau\Big).
  \label{thisone}
 \end{align}
 with $\C_s' = \C_s'(M', L, \delta^{-1}, \A^{s-1}(t), \Eee^{s-1}(t), \A^{s-1}(0))$.
 We now assume that we have the estimate \eqref{Aestimate} for $s = 0,..., m-1$.
 By the inductive assumption, \eqref{thisone} holds with $s = m$
 and with $\C_m'$ replaced with
 $\C_m''$ depending on $M'\!\!, L, \delta^{-1}\!\!, T, \A^{m-1}(0), \W^{m-1}(0), \sup_{0 \leq t \leq T} \Eee^{m-1}(t)$.
 Making this substitution into \eqref{thisone} with $s = m$ and letting
 $H(t)$ denote the right-hand side, we have that
 $H'(t) \leq \C_m'' ( (1 + t) \Eee^m(t) + (1+t) H(t))$. Multiplying
 both sides by the integrating factor
 $e^{-(t + t^2\!/2)\C_m'' }$ and integrating gives the result.
\end{proof}

Combining Lemmas \ref{coercivitylem1}, \ref{coercivitylem2} and Corollary
\ref{coercivitylem3}, we have:
\begin{cor}
 \label{usefulcoercivity}
With the same hypotheses as Lemmas \ref{coercivitylem1}-\ref{coercivitylem2},
there are continuous functions $\C_s$ with
$\C_s = \C_s\big(M, L, \delta^{-1},T,  \A^{s-1}(0),\W^{s-1}(0),
\sup_{ \,0 \leq t \leq T\,} \Ee^{s-1}(t)\big)$ so that for $1 \leq s \leq r-1$:
\begin{equation}
 ||\xve(t)||_{H^{s+1}(\Omega)}^2 +
 ||\pave h(t)||_{s}^2
 + ||V(t)||_{H^{(s,1/2)}(\Omega)}^2
 + \ve^2 ||V(t)||_{H^{s+1}(\Omega)}^2
 \leq \C_s
{\, \sup}_{\,0 \leq t\leq T\,} \Ee^s(t), \qquad 0 \leq t \leq T.
 {}
\end{equation}
\end{cor}

\begin{proof}[Proof of Theorem \ref{eneestthm}]
We will prove that:
\begin{equation}
\Ee^s(t) \leq \F'_s\big(M', L, \delta^{-1}, T, {\sup}_{\,0 \leq t \leq T\,}
\Ee^{s-1}(t)\big)  \Big( \Ee^s_0 +
\int_0^t (1 + \tau)^2 \Ee^s(\tau)\, d\tau\Big), \qquad 0 \leq t \leq T,
\label{energygoal}
\end{equation}
for a continuous function $\F'_s$ .
If the estimate
\eqref{enest1} holds for $s = 0,..., m-1$ then \eqref{energygoal} implies:
\begin{equation}
\Ee^m(t) \leq \F''_m(M', L, \delta^{-1}, T, \Ee^{m-1}_0)
\Big( \Ee^m_0 + \int_0^t (1 + \tau)^2 \Ee^s(\tau)\, d\tau \Big),
\qquad 0 \leq t \leq T.
\end{equation}
 Letting $H_m(t)$ denote the right-hand side of
this expression then ${d} H_m/{dt} \leq (1 + t)^2 \F''_m H_m$ and so multiplying
by $e^{-((1 + t)^3\!/3 - 1)\F''_m }$ and integrating shows that \eqref{enest1}
holds for $s = m$ as well.

By Lemma \ref{coercivitylem2} and Corollary \ref{divcurlbounds}, we have shown
that $\A^s, \W^s$ are bounded by the right-hand side of \eqref{energygoal}
and so it just remains to prove that $\E^s + \ve^2 \E^s_\ve$ is bounded
by the right-hand side of \eqref{energygoal}.
 We will prove that, with $\C_s' = \C_s'\big(M', L, \delta^{-1}, T, \sup_{\,0 \leq t \leq T\,}
 \Ee^{s-1}(t)\big)$:
 \begin{equation}
  \frac{d}{dt} (\E^{s_{\!}}(t) + \ve^2 \E^{s\!}_\ve(t)) \leq \C_{\!s}' \Big( \! \E^{s\!}(t) +
  \ve^2 \E^{s\!}_\ve(t) + \A^{s\!}(t) + \W^{s\!}(t) +
  \!\int_0^t \!\!\!(1_{\!} +\tau)\big(\E^{s\!}(\tau_{{}_{\!}}) + \ve^2 \E^{s\!}_\ve(\tau_{{}_{\!}})
  + \A^{s\!}(\tau_{{}_{\!}}) + \W^{s\!}(\tau_{{}_{\!}})\big) d\tau\! \Big).
  \label{venergygoal}
 \end{equation}
Multiplying both sides of \eqref{venergygoal} by
the integrating factor $e^{-t \C_s'}$ gives:
 \begin{equation}
  \E^s_\ve(t) + \ve^2 \E^s_\ve(t) \leq \F_s'
  \Big(\E^s(0) + \ve^2 \E^s_\ve(0) + \int_0^t (1 + \tau)^2 \Ee^s(\tau)
  \, d\tau
  \Big),
 \end{equation}
 with $\F_s' \!\!=\! \F_s'\big(M'\!\!, L, \delta^{-1}\!\!, T\!, \sup_{\,0 \leq t \leq T} \Ee^{s-1}(t)\big)$.
  Together with the estimates for
  $\A, \W$, this proves \eqref{energygoal}.

  We start by controlling the time derivative of $\K^s$.
  By \eqref{smpbmdef}, $\curl D_t V\!\! = 0$ and so
  $D_t \curl V \!\!= -(\pave \ssm V)(\pave V)$ by \eqref{dinv}. Using also
 \eqref{leibpave} to control $[\fdhm, \curl] V$ and the product estimate
 \eqref{product}, it follows that:
 \begin{equation}
  ||D_t \curl \fdhm V||_{H^{s-1}(\Omega)}
  \leq C(M, ||\xve||_{H^s(\Omega)}, ||V||_{H^s(\Omega)})
  ||\fdhm V||_{H^s(\Omega)},
 \end{equation}
 and so using Corollary \ref{usefulcoercivity} and induction, this implies that
  ${d}\K^s /{dt}  \leq C(M) \Ee^s$. To control $\K_\ve^s$, the same argument
  allows us to control the curl term and to control the divergence term, we
  write:
  \begin{equation}
   D_t \div \pa V = -\pa (\Dve h + \Dve \phi) - [D_t, \div] \pa V
   + [\div, \pa] (\pave h + \pave \phi).
   {}
  \end{equation}
  Since $[D_t, \div] \pa V = -(\pave_i \ssm V^k)\pave_k \pa V^i$ and
  $[\div, \pa](\pave h + \pave \phi)
  = -(\pave_i \pa \xve^k)\pave_k (\pave_i h + \pave_i \phi)$, after using
  the product rule \eqref{product}, the wave equation \eqref{smwavedef}
  along with Lemmas \ref{fbdslem},\ref{gestlem},
  the definition $\Dve \phi = 4\pi \rho$, and Corollary \ref{usefulcoercivity},
  we can bound $\ve^2 {d}_{\,} ||\div \pa V||_{H^{s-1}(\Omega)}^2/{dt}$
  by the right-hand side of \eqref{venergygoal}.

  It remains to prove that for $\mu \!=\! 0,..., N$ and
  $|I| \!=\! s-\!1$, ${d}(\E^{I\!,\mu}\! + \ve^2 \E^I_\ve)/{dt}$ is
  bounded by the right-hand side of
  \eqref{venergygoal},
  and for this we use the energy identity \eqref{enident}
and an approximation argument. We could approximate $V, h, \xve$ by smooth
functions but since $\xve$ is smooth in tangential directions and we only apply
tangential derivatives, it will suffice to just approximate $V, h$.
We start by noting that under
our hypotheses, $V, D_t V, D_t h, \pave h \!\in\!
H^{r\!}(\Omega)$.
Indeed, $D_t V \!= -\pave h - \pave \phi$ and so by
Corollary \ref{usefulcoercivity} and
Theorem \ref{main theorem, ell est of phi}, we have $D_t V, \pave h \!\in\! H^r(\Omega)$.
To see that $D_t h \!\in\! H^r(\Omega)$, we write
$\pa_a D_t h = A_{\m a}^i \pave_i D_t h = A_{\m a}^i D_t \pave_i h
+ A_{\m a}^i \pave_i \ssm V^k \pave_k h$. The $H^{r-1}(\Omega)$ norm of the first term here
is bounded by $C(M')||\xve||_{H^{r+1}(\Omega)}||\pave h||_{r}$ using
\eqref{dinv} and the fact that $H^{r+1}(\Omega)$ is an algebra. The second
term here is bounded by $C(M) ||V||_{H^r(\Omega)} ||\pave h||_r$ for the
same reason. By \eqref{coercivitylem1} we control $||\pave h||_r$ and thus
$||D_t h||_{H^r(\Omega)}$.

Therefore, there is a sequence of smooth vector fields $V_{(n)}$
and a sequence of smooth functions $h_{(n)}$ with $h_{(n)}|_{\pa \Omega} = 0$
 so that $(V_{(n)}(t,\cdot), D_t V_{(n)}(t,\cdot), D_t h_{(n)}(t,\cdot),
 \pave h_{(n)}(t,\cdot))$
 $\to  (V(t,\cdot), D_t V(t,\cdot), D_t h(t,\cdot), \pave h(t,\cdot) )$ in $H^r(\Omega)$.
 We claim that for all $I, J$ with $|I| =s \leq r-1$, $|J| = s-1$
 and all $\mu = 0,..., N$, we have:
 \begin{equation}
  \pave \fdhm T^I h_{(n)} \to \pave \fdhm T^I h,
  \qquad \fdhm \div T^J \fdhm V_{(n)} \to \fdhm \div T^J \fdhm V,
  \qquad \text{ in } L^2(\Omega).
 \end{equation}
 These claims follow after writing $\fdhm \!\div T^J \fdhm V_{(n)} \!= \fdhm\fdhm T^J\! \div V_{(n)} + \fdhm [T^J\fdhm\!\!, \div] V_{(n)}$ and
 $\pave \fdhm T^I h_{(n)} = \fdhm T^I \pave h_{(n)} + [\fdhm T^i, \pave] h_{(n)}$.
 In each of these expressions, the first term converges in $L^2(\Omega)$.
 The commutator terms involve tangential derivatives of $\xve$ to highest order
 and lower-order norms of $V_{(n)}$ and so these converge as well.
 By the continuity of the trace map:
 \begin{equation}
  T^I \fdhm V_{(n)} \to T^I \fdhm V,
  \qquad
  D_t T^I \fdhm V_{(n)} \to D_t T^I \fdhm V,
  \qquad \text{ in } L^2(\pa \Omega).
 \end{equation}

 We now apply Proposition \ref{enident} with
 $
  \alpha\! =\!  V_{(n)}, q \! =\!  h_{(n)},  \chi\!  = \! \pave \ssm\!  V.
 $
 For sufficiently large $n$, the assumptions \eqref{qchiassump1}-\eqref{qchiassump2} hold with $K \! \! =\!  2\delta \! + \! 2M' \! \! + \! 2L$. With $\E^I_n,
 \E^I_{n, \ve}$
 defined by \eqref{e1def}-\eqref{e2def} with $\alpha \! =\!  V_{(n)}$, $q \! =\!  h_{(n)}$ let
 $\Rr_n \!+\! \ve \Rr_{n,\ve} = \!\sum_{\mu = 0}^N \sum_{|I| \leq s} \Rr^{I\!,\mu}_n\!
 + \ve \Rr^{I\!,\mu}_{n,\ve}$, where
 $\Rr^{I\!,\mu}_{n,\ve}\! =\! ||\gamma\!\cdot \!\fdhm D_tT^I V_{(n)} \!-\!
(\fdhm T^I \gamma)\!\cdot\! \pave h_{(n)}||_{L^2(\pa \Omega)}$ and
\begin{multline}
\Rr_n^{I,\mu} = ||D_t \fdhm T^I V_{(n)} - \pave \big( (T^I \fdhm \xve^j )
(\pave_j h_{(n)}) + T^I \fdhm h_{(n)}\big)||_{L^2(\Omega)}\\
+ ||\fdhm  \big(e'(h) D_t T^J h_{(n)} -
(\pave_i T^I\fdhm \xve^j)(\pave_j \ssm V^i) +
 \div T^J \fdhm V_{(n)}\big)||_{L^2(\Omega)}\\
+ ||D_t T^I \fdhm x - T^I \fdhm V_{(n)}||_{L^2(\pa \Omega)},
\end{multline}
 and $T^I\!\! = \!ST^J$ for $S \!\in \!\T\!, |J|\! =\! s\!-\!1$.
 The energy inequality \eqref{enident} then gives, with $C_0=C_0(M'\!\!, \delta^{-1})$,
 \begin{multline}
  \frac{d}{dt} \E^I_n
  \leq \sqrt{\E^s_n}
  C_0\Big( \Rr_n \!+ ||V_{(n)}||_{H^{(s,1_{\!}/2)}(\Omega)} +
  ||\sm x||_{H^{s}(\Omega)}
  + ||\sm x||_{H^{s+1_{\!}/2}(\pa \Omega)}
  + ||h_{(n)}||_{H^{s+1}(\Omega)}+
  || D_{{}_{\!}t} h_{(n)}||_{H^s(\Omega)}\Big)\\
  + C_s \E^s_n + C_s \ve||V_{(n)}||_{H^{s+1/2}(\pa \Omega)}
  ||\sm x||_{H^{s+1/2}(\pa \Omega)},
 \end{multline}
 \begin{equation}
  \frac{d}{dt} \E^I_{n,\ve}
  \leq \sqrt{\E^s_{n,\ve}} (\Rr_{n,\ve} +
  \ve^{-1} ||\sm x||_{H^{s+1/2}(\pa \Omega)}).
 \end{equation}
 By the above, using Lemma \ref{lowerorder}
 to control $\lim_{n \to \infty} \Rr_n$ and Corollary \ref{usefulcoercivity}
 again, this implies \eqref{venergygoal}.
\end{proof}

Before proving Corollary \ref{useful}, we prove the following
simple lemma:
\begin{lemma}
 \label{bootlem1}
 Fix $r \geq 6$, and write $V_k^\ve = D_t^k V|_{t = 0},
 h_k^\ve = D_t^k h|_{t = 0}$. Suppose that
 the bound \eqref{aprioriM0L0} holds for $x_0, V_k^\ve$ and $h_k^\ve$.
 There is a continuous function $\T_r = \T_r(M_0', L_0, \delta_0^{-1},
 \Ee^r_0)$ so that if $T \leq \T_r, $ and $V \in \X^{r+1}(T)$ satisfies
 the smoothed-out Euler equations \eqref{smpbmdef}, then
 for $ 0\leq t \leq T$:
 \begin{align}
   ||\pa \xve(t,\cdot)/\pa y||_{L^\infty} +
   ||\pa y(t,\cdot)/\pa \xve||_{L^\infty} +
  {\tsum}_{|I| + k \leq 3}||\pa_y^{I} D_t^k V(t,\cdot)||_{L^\infty} +
  ||\xve(t,\cdot)||_{H^6(\Omega)} &\leq 4M_0',
  \\
  {\tsum}_{|I| + k \leq 3} ||\pa_y^{I} \pave h(t,\cdot)||_{L^\infty} +
  ||D_t^k h(t,\cdot)||_{L^\infty} &\leq 2L_0,\\
  -\pave_N h(t,\cdot)|_{\pa \Omega} &\geq  \delta_0/2.
 \end{align}
\end{lemma}
\begin{proof}
Let $M_1(t)\! =\! ||\pa_y \xve(t)||_{L^\infty(\Omega)}\!+\!\sum_{|I| + k \leq 3}
||\pa_y^I D_t^k V(t)||_{L^\infty(\Omega)}$ and  $M_2(t)\! = \! ||\pa y(t)/\pa \xve||_{L^\infty(\Omega)}$. Further, let
$L(t)\! =\! \sum_{|I| + k \leq 3} ||\pa_y^I D_t^k \pave h(t)||_{L^\infty(\Omega)}
\!+\! ||D_t^k h(t)||_{L^{\!\infty}(\Omega)}$
and $\nu(t) \!= \! || (-\pave_N h(t))^{-_{\!}1}||_{L^{\!\infty}(\pa\Omega)}$.
Note that by the definition of $\xve$ in \eqref{xvedef} and the definition
of $V_k^\ve$,  we have $M_1(0) \!+\! M_2(0)\! \leq \! M_0'$
and $L(0)\! =\! L_0$ and $\nu(0)\! \leq\! \delta_0^{-1}\!\!$.
By Sobolev embedding, the fundamental theorem of calculus, and the fact that
the operator $\sm$ is bounded:
\begin{equation}
M_1(t)
\leq  M_0'
+ C_1 \Big(\int_0^t ||V(\tau)||_{6}
\, d\tau\Big),
\qquad
L(t) \leq  L_2 + C_2 \Big(\int_0^t || D_t h(\tau)||_{6} + ||\pave h(\tau)||_6
 \, d\tau\Big).
\label{MLbds}
\end{equation}
Using the trace inequality \eqref{trace}, we also have:
\begin{equation}
\nu(t) \leq \delta_0^{-1} + \int_0^t \nu(\tau)^2
||D_t \pave h(\tau,\cdot)||_{L^\infty(\pa\Omega)}\, d\tau
\leq \delta_0^{-1} + C_3 \int_0^t \nu(\tau)^2
||\pave h(\tau)||_{4}\, d\tau.
\end{equation}
Finally, integrating in time, using \eqref{dinv} and Sobolev embedding, we have:
\begin{equation}
 M_2(t) \leq M_0' + C_3 \int_0^t M_2(\tau)^2 ||V(\tau)||_4\, d\tau.
\end{equation}

If $V \in \X^{r+1}(T_1)$ solves the smoothed Euler equations
\eqref{smpbmdef} for some $T_1 > 0$ then
Corollary \ref{usefulcoercivity} combined with Theorem
\ref{eneestthm} gives a continuous function $\F_r'$  so that:
\begin{equation}
||V(\tau)||_{6}^2 + ||D_t h(\tau)||_6^2
+ ||\pave h(\tau)||_4^2 \leq \F_r'(M', L, \Ee^{r-1}_0) \Ee^r_0,
\qquad 0 \leq \tau \leq T_1.
\end{equation}
Here, the constants $C_1, C_2, C_3$ depend only on $\Omega$.
Set
$\F_r''\!\! =\! C_0\F_r'(4M_0', 2L_0, (2\delta_0)^{-1}\!\!,
\Ee^{r-1}_0)(\Ee^{r}_0 \!+ (2\delta_0)^{-2})$ with $C_0 \!=\! C_1 \!+ \!C_2 \!+\! C_3$, and define
$ \T_{\!\!\!r}\!= \!\min(M_0', 1/M_0', L_0, \delta_0) (16\F_r'')^{-1}\!\!$.
Take $T \!\leq\! \min(\T_{\!\!\!r},T_{\!1})$ and consider the set:
\begin{equation}
S = \{0 \leq t \leq T : M_1(t) + M_2(t) \leq 4M_0', L(t) \leq 2L_0,
\nu(t) \leq 2\delta_0^{-1}\}.
\end{equation}
Then $S$ is nonempty, since it contains $t\! = \!0$, and it is connected and closed by
continuity of the functions
$M_{{}_{\!}1\!}(_{{}_{\!}} t_{{}_{\!}} ), M_2(_{{}_{\!}}t_{{}_{\!}}), L(_{{}_{\!}}t_{{}_{\!}}), \nu(_{{}_{\!}}t_{{}_{\!}})$. If $t\!{}_{\!} \in \!S$ then
the assumption $T\!{}_{\!} \leq \!\T_{\!\!\!r}$ and \eqref{MLbds} imply:
\begin{equation}
M_1(t) \leq  M_0' + T \F_r''
\leq M_0' + M_0' (16\F_r'')^{-1} \F_r'',
\qquad
L(t) \leq L_0 + T \F_r''
\leq L_0 + L_0 (16 \F_r'')^{-1}\F_r'',
\end{equation}
and similarly
\begin{equation}
  M_2(t) \leq M_0' + (4M_0')^2 (16 M_0' \F_r'')^{-1} \F_r'',
\qquad
\nu(t) \leq  \delta_0^{-1}
+ 2\delta_0^{-2} \delta_0(16 \F_r'')^{-1} \F_r''.
{}
\end{equation}
In particular  $M_{{}_{\!}1}(t) \!+\! M_2(t)\! \leq\! 3 M_0',\, L(t) \!\leq\! 3L_0/2$
and $\nu(t)\! \leq\! 3\delta_0^{-1\!}\!/2$. Hence  $S$ is also open so $S \!=\! \{0 \!\leq\! t\! \leq \!T\}$.
\end{proof}

\begin{proof}[Proof of Corollary \ref{useful}]
 Let $\T_{\!\!r}$ be as in Lemma \ref{bootlem1} and with
 $\mathcal{F}_r$ as in Theorem \ref{eneestthm}, define:
 \begin{equation}
   \mathscr{T}_{\!r} = \T_{\!r},\qquad
  \Cc_r = \mathcal{F}_r(4M_0', 2L_0, (2\delta_0)^{-1}, \Ee_0^{r-1}).
  {}
 \end{equation}
 By \eqref{enest1} and  Lemma \ref{bootlem1}, this proves
 \eqref{useful1}. The estimate \eqref{useful2} follows from
 \eqref{useful1} and Corollary \ref{usefulcoercivity}.
\end{proof}

\appendix

\section{Fractional tangential derivatives and tangental smoothing}
\label{tangapp}

There is a family of open sets $U_\mu $, $\mu\!=\!1,\dots,N$ that cover $\pa \Omega$ and  onto diffeomorphisms
$\Phi_\mu\!:\! (-1,\!1)^2 \!\to \!U_\mu $. We fix a collection of cutoff functions $\chi_{\mu}\!{}_{\!}:\!\pa \Omega\!
\to \!\R$ so that $\chi_\mu^2$ form a partition of unity and $\operatorname{supp} \chi_\mu\!{}_{\!}\subset\! U_\mu$,
 as well as  two other families of
 cutoff functions such that $\tC_\mu\! \equiv\! 1$ on $\operatorname{supp} \chi_\mu$, $\overline{\chi}_\mu\!\equiv\! 1$ on
$\operatorname{supp} \tC_\mu$ and $\operatorname{supp}\overline{\chi}_\mu \!\!\subset \!U_\mu$. Recalling that
$\Omega$ is the unit ball, we set $W_{\!\mu} \!=\! \{r\omega, r \!\in\! (1/2, 1], \omega
\!\in\! U_{\!\mu}\}$
for $\mu \!= \!1,{}_{\!}...,{}_{\!} N$
and let $W_0$ be the ball of radius $3/{}_{\!}4$ so that
$\{W_{\!\mu}\}_{\mu = 0}^N$ covers $\Omega$.
Writing $\Psi_\mu(z,z_3) = z_3 \Phi_\mu(z)$,
$\Psi_\mu$ is a diffeomorphism from $(-1,{}_{\!}1)^2\! \times (1{}_{\!}/2,{}_{\!}1]$ to
$W_{\!\mu}$.
Let $\zeta\!:\![0,\!1] \to \R$ be
a bump function so that $\zeta(r) \!= \!1$ when $1/2\! \leq\! r\! \leq \!1$ and
$\zeta(r)\! = \!0$ when $r \!<\! 1/4$.
We extend the above cutoffs to $\Omega$ by setting $\chi_\mu(y) = \chi_\mu(y/|y|) \zeta(|y|)$  for $\mu = 1,..., N$ and $\chi_0
=\! 1\!-\!\zeta$, and we similarly extend $\tC_\mu$ and $\overline{\chi}_\mu$. We abuse notation by writing $\chi_\mu$ also for the function $\chi_\mu\circ \Psi_\mu$

\subsection{Fractional derivatives}
For a function $F: \R^2 \to \R$, we set:
\begin{equation}
 \fdh F(z) = \int_{\R^2} \langle \xi \rangle^{1/2} \hat{F}(\xi) e^{iz\cdot \xi}\,d\xi,
 \quad \text{where } \hat{F}(\xi) = \int_{\R^2} e^{-iz\cdot \xi} F(z)\, dz.
\end{equation}
Given a function $f: \Omega \to \R$,  we define $\fdhm f: \Omega \to \R$ for $\mu = 1,..., N$ by:
\begin{equation}\label{eq:localcoordinatefmu}
 \fdhm f = \widetilde{\chi}_\mu (\fdh f_\mu)\circ \Psi_\mu^{-1},\quad\text{where}\quad f_\mu = (\chi_\mu f)\circ \Psi_\mu:\R^2 \to \R.
\end{equation}

With the cutoff function $\zeta$ defined above, we let $\T$ denote
the following family of vector fields, which span the tangent space to the
boundary and in the interior span the full tangent space:
\begin{equation}
  \zeta(y) (y^a\pa_{y^b} - y^b \pa_{y^a}), \quad
  (1 - \zeta(y)) \pa_{y^a}, \quad a,b =1,2,3.
 \label{}
\end{equation}

We work in terms of the following Sobolev norms, for $s \in \R$:
\begin{equation}
 ||f||_{H^s(\pa \Omega)}^2
 = {\sum}_{\mu = 1}^N ||\fd^{s} f_\mu||_{L^2(\R^2)}^2
 = {\sum}_{\mu = 1}^N \int_{\R^2} |\langle \xi\rangle^{s} \hat{f}_\mu(\xi)|^2\, d\xi,
 \label{sobspacedef0}
\end{equation}
and if $s \in \R, k \in \N$ we set:
\begin{equation}
 ||f||_{H^{(k,s)}(\Omega)}^2
 = {\sum}_{|I| \leq k} \int_0^1 ||\pa_y^I (\zeta f)(r,\cdot)||_{H^s(\pa \Omega)}^2
 \, r^2 dr
 + ||(1-\zeta) f||_{H^{k+s}(\Omega)}^2,
\end{equation}
where for non-integer $s$, $H^{k+s}(\Omega)$ is defined in the usual
way by taking the
Fourier transform in all variables.
We collect here the basic properties of the operators
$\fdhm$ and the norms $H^s(\pa \Omega), H^{(k,s)}(\Omega)$:
\begin{lemma}
  \label{fracalg}
If $T \in \T$, then:
\begin{equation}
 \Big| \int_{\pa\Omega} f T g\, dS(y)\Big|
 \leq C||f||_{H^{1/2}(\pa \Omega)} ||g||_{H^{1/2}(\pa \Omega)},
 \qquad
 \Big| \int_{\Omega} f T g\, dy \Big|
 \leq C||f||_{H^{(0,1/2)}(\Omega)} ||g||_{H^{(0,1/2)}(\Omega)}.
 \label{inthalf}
\end{equation}
In addition, with $\Sigma = \pa \Omega$ or $\Omega$,
 \begin{equation}
  ||\fdhm (fg) - f \fdhm g||_{L^2(\Sigma)}
  \leq C ||f||_{H^2(\Sigma)} ||g||_{L^2(\Sigma)},
  \label{app:alg2}
 \end{equation}
 and, with notation as in \eqref{eq:simplifiedtangentialnotation} and $T^I \in
 \coord^k$ or $\T^k$:
 \begin{equation}
  ||\fdhm(T^I f) - T^I \fdhm f||_{L^2(\Sigma)}
  \leq C||f||_{H^{k}(\Sigma)}.
  \label{commutewithpa}
 \end{equation}
 In particular, if $||\xve||_{H^3(\Omega)} \leq M$ then:
 \begin{equation}
  ||\fdhm \pave f - \pave \fdhm f||_{L^2(\Sigma)}
  \leq C (M) ||f||_{H^1(\Sigma)}.
  \label{leibpave}
 \end{equation}
\end{lemma}

These estimates all rely on the following ``Leibniz
rule''. This lemma and its proof can be found in \cite{Nordgren2008}.
\begin{lemma}
  If $F, G: \R^2 \to \R$ have compact support, then:
  \begin{equation}
   ||\fdh(FG) - F \fdh G||_{L^2(\R^2)}
   \leq C ||F||_{H^2(\R^2)} ||G||_{L^2(\R^2)}.
   \label{leibbasic}
  \end{equation}
\end{lemma}
\begin{proof}
 By the elementary estimate
 $|\langle\xi\rangle^{1/2} - \langle \xi-\eta\rangle^{1/2}| \leq C
 \langle \eta\rangle^{1/2}$, we have:
 \begin{equation}
  |\langle \xi\rangle^{1/2} \widehat{ FG} (\xi)
  - \widehat{ (F \fdh G)}(\xi)|^2
  \lesssim \Big(\int_{\R^2} \!\!\!\langle \eta\rangle^{1/2}
  |\hat{F}(\eta)|  |\hat{G}(\xi\!-\!\eta)|\, d\eta\Big)^2\!\!
  \lesssim  \int_{\R^2}  \!\!\!\langle \eta\rangle^{4}
  |\hat{F}(\eta)|^2\, d\eta
  \int_{\R^2} \!\!\! \langle\eta\rangle^{-3}  |\hat{G}(\xi\!-\!\eta)|^2\, d\eta.
  \label{identity}
 \end{equation}
 Integrating in $\xi$, changing variables, and using the fact that
 $\int_{\R^2} \langle\xi-\eta\rangle^{-3} \, d\xi\! \leq \!C$,
 we have:
 \begin{equation}
 ||\langle\xi\rangle \widehat{FG} - \widehat{(F \fdh G)}||_{L^2(\R^2)}^2 \leq
 C ||F||_{H^2(R)}^2
   \int_{\R^2}
  \int_{\R^2} \langle\xi-\eta\rangle^{-3}  |\hat{G}(\eta)|^2\, d\eta
  \, d\xi
  \leq C||F||_{H^2(R)} ||G||_{L^2(R)}.
  \end{equation}
  The result now follows from Plancherel's theorem.
\end{proof}

\begin{proof}[Proof of Lemma \ref{fracalg}]

Since $\sum \chi_\mu^2 = 1$, we have:
\begin{equation}
 \int_{\pa \Omega}
 fT g \, dS(y) =
 {\sum}_{\mu = 1}^N
 \int_{\pa \Omega} (\chi_\mu f) (\chi_\mu T g)\, dS(y)=
 {\sum}_{\mu = 1}^N
 \int_{\pa \Omega} \chi_\mu f T(\chi_\mu g) \, dS(y)
 - \int_{\pa \Omega} \chi_\mu f g T\chi_\mu\, dS(y).
\end{equation}
The second term is bounded by $C||f||_{L^2(\pa \Omega)} ||g||_{L^2(\pa \Omega)}$.
To deal with the first term, we use
\eqref{eq:localcoordinatefmu} and write:
\begin{equation}
 \int_{\pa \Omega} \chi_\mu f T(\chi_\mu g) \, dS(y)
 = \int_{\R^2} f_\mu T^\alpha \pa_{z^\alpha} g_\mu |\det{\Phi'_\mu}| \,dz,
 \quad\text{where}\quad
T= T^\alpha \pa_{z^\alpha}.
\end{equation}
 With $F^\alpha = f_\mu T^\alpha |\det{\Phi'_\mu}|$
and $G = g_\mu$, by Plancherel's theorem we have:
\begin{equation}
 \int_{\R^2} F^\alpha(z) \pa_{z^\alpha} G(z) dz
 = \int_{\R^2} \hat{F}^\alpha(\xi) i\xi_\alpha \hat{G}(\xi)\, d\xi
 \leq || \langle \xi \rangle^{1/2} \hat{F}||_{L^2(\R^2)}
 || \langle \xi \rangle^{1/2} \hat{G}||_{L^2(\R^2)}.
\end{equation}
By \eqref{leibbasic} and  \eqref{sobspacedef0}, this is bounded by
$(||\fdh \!f_{\mu\!}||_{L^2(R)}\! +\! ||f||_{L^2(\pa \Omega)})||g||_{{}_{\!}H^{1\!/2}(\pa \Omega)} $. The case $\Sigma \!= \! \Omega$
is similar.

We now prove \eqref{app:alg2}.
  Writing $\overline{f}_\mu =\overline{\chi}_\mu  f\circ\Psi_\mu$,
  where $\overline{\chi}_\mu \equiv 1$ in the support of $\tC_\mu$ in \eqref{eq:localcoordinatefmu},
  we have:
  \begin{equation}
   ||{}_{\!}\fdhm (f{}_{\!}g_{{}_{\!}})\! -\! f \fdhm {}_{\!}g||_{L^2{}_{\!}(\pa \Omega)}
  \! \lesssim\! ||{}_{\!}\fdh ({}_{\!}\overline{f}_{\!\!\mu} g_{\mu\!})\!
   - \!\overline{f}_{\!\!\mu\!} \fdh {}_{\!}g_{\mu\!}||_{L^2{}_{\!}(\R^{2\!})}
   \!\lesssim\! \|{}_{\!}\overline{f}_{\!\!\mu\!}\|_{H^2(\R^{2\!})} \|g_{\mu\!}\|_{L^2(\R^{2\!})}
   \!\lesssim\! ||f||_{H^2(\!\pa \Omega)}||g_{{}_{\!}}||_{L^2(\!\pa \Omega)},
  \end{equation}
 by  \eqref{leibbasic}, which gives
  \eqref{app:alg2} for $\Sigma \!= \!\pa \Omega$. The case $\Sigma\! =\!
  \Omega$ follows from the case $\Sigma  \!=\! \pa \Omega$ by the definition
   \eqref{sobspacedef0}.

We now prove \eqref{commutewithpa}. We first prove the case $k \!=\! 1$
with $T\in \T$ and $\Sigma = \pa \Omega$.
Since $\pa_{z^\alpha} \fdh = \fdh \pa_{z^\alpha}$:
  \begin{equation}\label{eq:commutatordifference}
    T^\alpha\pa_{z^\alpha} \fdh f_\mu - \fdh(T f)_\mu
    = T^\alpha \fdh(\pa_\alpha f_\mu)
    - \fdh(T^\alpha  \pa_\alpha f_\mu)+\fdh((T^\alpha  \pa_\alpha\chi_\mu)f),
  \end{equation}
  Applying \eqref{leibbasic}, the $L^2$ norm of the right-hand side is bounded
  by
  $C ||f||_{H^1(\pa \Omega)}$.
  The commutator of $T \fdhm\! f\! -\! \fdhm T\! f$ just contribute an additional term
  $(T\tC_\mu)\fdh \!f_\mu $ compared to \eqref{eq:commutatordifference} and \eqref{commutewithpa} follows.

  To prove \eqref{commutewithpa} when $T = \pa_{y^a}$ for some $a = 1,2,3$
  and $\Sigma = \pa \Omega$, close to the boundary we write $\pa_{y^a} =
   \sum_{T \in \T} c_a^T(y) T\!
  + \!c(y) \pa_r$ for some smooth functions $c_a^T$ and $c$. By what we have
  just proven and \eqref{app:alg2}, it is enough to prove  the estimate
  with $T$ replaced by $\pa_r$. This follows  from the definition after
  noting that close to the boundary, the cutoff functions $\tC_\mu,
  \chi_\mu$ are independent of $r$.
  The case $|I| \geq 2$ follows similarly.
\end{proof}

  The operators $\fdhm$ can be used to control
  fractional Sobolev norms:
  \begin{lemma} We have
  \begin{equation}
   ||(1-\tC_\mu) \fdh f_\mu||_{L^2(\R^2)}\lesssim ||f_\mu||_{L^2(\R^2)}.
  \end{equation}
   Moreover, there are constants $0<C_1< C_2<\infty$ so that:
   \begin{equation}
    C_1 \bigtwo({\sum}_{\mu = 1}^N ||\fdhm f||_{L^2(\pa \Omega)}+||f||_{L^2(\pa \Omega)}\bigtwo)
    \leq ||f||_{H^{1/2}(\pa \Omega)} \leq C_2
   \bigtwo( {\sum}_{\mu = 1}^N ||\fdhm f||_{L^2(\pa \Omega)}
   +||f||_{L^2(\pa \Omega)}\bigtwo).
   \label{equivalentnorms}
   \end{equation}
   The same estimate holds with $\pa \Omega$ replaced by
   $\Omega$ and $H^{1/2}(\pa \Omega)$ replaced with
   $H^{(0,1/2)}(\Omega)$.
  \end{lemma}
  \begin{proof}
     Since $\tC_\mu =1$ on the support of $\chi_\mu$ and hence on the support of
    $f_\mu$ it follows from \eqref{leibbasic} that
    \begin{equation}
     ||(1-\tC_\mu) \fdh f_\mu||_{L^2(R)}
     =|| \fdh (\tC_\mu f_\mu)-\tC_\mu \fdh f_\mu||_{L^2(R)}
     \leq C ||f_\mu||_{L^2(R)}\leq C ||f||_{L^2(\pa \Omega)}.
      \tag*{\qedhere}
    \end{equation}
  \end{proof}

\subsection{Tangential smoothing}
\label{tangentialsmoothingappendix}

Let $\varphi \!:\! \R^2 \!\!\to\! \R$ be an even smooth
function, supported in $R = (-1,1)^2$,
with $\int_{\R^2}
 \!\varphi =\! 1$ and define the smoothing operator
\begin{equation}
 T_\ve f(z) = \int_{\R^2} \varphi_{\ve}(z-z') f(z') dz',\qquad\text{where}\qquad
 \varphi_{\ve}(z)\! =\! {\ve^{-2}} \varphi\big({z}\!/{\ve}\big).
\end{equation}
Because $\varphi$ is even, $T_\ve$ is symmetric; for any functions
$f, g: \R^2 \to \R$ we have:
\begin{equation}
 \int_{\R^2} T_\ve f(z) g(z) dz =
 \int_{\R^2} \int_{R} \varphi_\ve(z-z') f(z') g(z)
 dz'\, dz =
 \int_{\R^2} f(z) T_\ve g(z) dz.
\end{equation}
From the fact that $\|\pa^k \varphi_\ve\|_{L^1}\lesssim \ve^{-k}$ it follows that
for $k\geq m$
\begin{equation}
\| T_\ve f\|_{H^k}\lesssim \ve^{m-k} \|f\|_{H^m}.
\label{smoothingproperty}
\end{equation}
Furthermore, we have:
\begin{equation}
 |T_\ve(fg)(z) - fT_\ve(g)(z)|
 \leq C \ve ||f||_{C^1(R)} ||g||_{L^2(R)},
 \label{diffprop}
\end{equation}
which follows from the fact that $|z'| \leq \ve$ in the support of $\varphi_\ve$,
after writing:
\begin{equation}\label{eq:theaboveintegral}
 T_\ve(fg)(z) - f(z)T_\ve(g)(z)
 = \int_{\R^2} \varphi_\ve(z')g(z-z')\big( f(z-z') - f(z)\big)
 dz'.
\end{equation}
Moreover from using \eqref{eq:differencequotionet} and Minkowski's integral inequality in \eqref{eq:theaboveintegral} with $g\!=\!1$ it  follows that
\begin{equation}
\|T_\ve f-f\|_{H^{k}}\lesssim \ve \|f\|_{H^{k+1}}.
\label{nofreelunch}
\end{equation}

For a linear operator $T'$ defined in coordinate charts we define a global operator $T$ by
\begin{equation}\label{eq:localtoglobalappendixA}
T\!f\!=\!\sum T_{\!\mu} f,\quad\text{where}\quad
T_{\!\mu} f \!=
\chi_\mu\big(m_\mu^{-1}T'\big[m_\mu f_{\!\mu} \big]\big)\circ\Psi_{\!\mu}^{-1}\!\!,\qquad
f_{\!\mu} \!=\! (\chi_\mu f)\circ \Psi_{\!\mu},\quad
m_\mu \!=\!r|\!\det{\Phi_{\!\mu}^\prime}|^{1/2}\!\!.
\end{equation}
Then $T$ is symmetric with the measure $dy$ if $T'$ is with the measure $dz$ is since $dS(\omega)=m_\mu^2 dz$.
With notation as in \eqref{eq:localtoglobalappendixA}, the smoothing operators we consider on $\Omega$ or $\pa \Omega$ are then defined by:
\begin{align}
 \sm f = {\sum}_{\mu = 0}^N
  T_{\ve,\mu} f,
 \qquad
 \ssm f = \sm \sm f={\sum}_{\mu,\nu = 0}^N
 T_{\ve,\nu} T_{\ve,\mu} f.
 \label{smoothingappendixA}
\end{align}
Since $T_\ve $ is symmetric $J_\ve $ is as well, w.r.t. $dy$.

The smoothing operator has the following important properties:
\begin{lemma}
  \label{smproperties}
   If $f, g : \Omega \to \R$, then with $\Sigma = \pa \Omega$ or $\Omega$
   \begin{equation}
  ||J_\ve(fg) - f J_\ve(g)||_{L^2(\Sigma)}
  \leq C \ve ||f||_{C^1(\Sigma)} ||g||_{L^2(\Sigma)}.
  \label{jprodrule}
 \end{equation}
 \begin{equation}
  \Big| \int_{\pa\Omega}\big( f (\ssm g)
  - (\sm f)(\sm g) \big)\,\widetilde{\nu} dS(y)\Big|
  \leq C \ve||\widetilde{\nu}||_{C^1(\pa\Omega)} ||f||_{L^2(\Omega)} ||g||_{L^2(\pa\Omega)},
  \label{almostsaapp}
 \end{equation}
 \begin{equation}
  \Big| \int_{\Omega}\big( f (\ssm g)
  - (\sm f)(\sm g) \big)\,\kve dy\Big|
  \leq C \ve||\kve||_{C^1(\Omega)} ||f||_{L^2(\Omega)} ||g||_{L^2(\Omega)},
  \label{almostsaapp2}
 \end{equation}
 Further, if $T^I \in \T^k$ for $k \geq 0$:
   \begin{equation}
  ||T^I J_\ve f - J_\ve \big(T^I f\big)||_{L^2(\Sigma)}
  \leq C  ||f||_{H^{k-1}(\Sigma)},
  \label{jdiffrule}
 \end{equation}
  \begin{equation}
  ||\fd^{1/2}_\mu J_\ve f - J_\ve \big(\fd^{1/2}_\mu f\big)||_{L^2(\Sigma)}
  \leq C  ||f||_{L^2(\Sigma)}.
  \label{jdiffrulehalf}
 \end{equation}
\end{lemma}
\begin{proof} The estimate \eqref{jprodrule} is a straightforward consequence of \eqref{diffprop}.

To prove \eqref{almostsaapp} and \eqref{almostsaapp2} note first that $J_\ve$ is symmetric w.r.t. to the measure $dS(y)$ since $m_\mu^2 dS(y)=dz$:
\begin{equation}
\int_{\pa\Omega} \!\!\!\!   f {}_{\!} J_{\ve} g\, dS
\!= \!\!{\sum}_{\mu\!}\!\int_{\pa\Omega}  \!\!\!\! f{}_{\!}\chi_{{}_{\!}\mu{}_{\!}}  \big( {}_{\!} m_\mu^{-\!1} T_{\!\ve} [m_\mu g_{\mu{}_{\!}}]\big){}_{\!}\circ{}_{\!} \Psi_\mu^{-\!1\!} dS
 \!=\!\! {\sum}_{\mu\!} \!\int_{\!R} \!\! m_\mu f_{\!\mu}   T_{\!\ve}[ m_\mu g_{\mu{}_{\!}}] dz
 \!=\!\!{\sum}_{\mu\!}\!\int_{\!R} \! T_{\!\ve}[ m_\mu f_{\!\mu{}_{\!}}]   m_\mu g_\mu dz
  \!=\!\!\int_{\pa\Omega}  \!\!\!{}_{\!}J_{\ve{}_{\!}}f\, g\, dS.
\end{equation}
\eqref{almostsaapp} follows from this applied to $\widetilde{\nu} f\! $ in place of $f\!$ and then \eqref{jprodrule} with $\widetilde{\nu}$ in place of $f\!$ and $f\!$ in place of $g$.

Changing coordinates, using
that $\pa_z (\varphi_\ve \!* \!F)\! =\! \varphi_\ve\! *\! (\pa_z F)$ for any function
$F\!\!:\! (-1,1)^2 \!\to \R$ and using \eqref{jprodrule}, a straightforward calculation as in the proof of Lemma \ref{fracalg} shows that
$||T^I \! J_\ve f - J_\ve T^I \!f||_{L^2(\Sigma)} \lesssim ||f||_{H^{k-1}(\Sigma)}$.

\eqref{jdiffrulehalf} follows from that
$[\fd^{{}_{\!}1_{\!}/2}\!\!, T_{\!\ve}]\!=\!0$ and $\sum_\nu \chi_\nu^2\!=\!1$, after repeatedly using \eqref{app:alg2} and
\eqref{jprodrule} in
\begin{equation}
\fd^{1/2}_\mu J_\ve f ={\sum}_\nu \tC_\mu\bigtwo(
\fd^{1/2}\bigtwo[ \chi_\mu  \chi_\nu m_\nu^{-1} T_{\!\ve}\big[m_\nu f_{\!\nu} \big]\bigtwo]\bigtwo)\circ\Psi_{\!\nu}^{-1},
\end{equation}
\begin{equation}
J_\ve \fd^{1/2}_\mu f ={\sum}_\nu \chi_\nu
\bigtwo(m_\nu^{-1}
 T_{\!\ve}\bigtwo[ \chi_\nu  m_\nu\tC_\mu\fd^{1/2}[ f_{\!\mu}] \big]\bigtwo]\bigtwo)\circ\Psi_{\!\nu}^{-1}.  \tag*{\qedhere}
\end{equation}
\end{proof}

  \subsection{Interpolation and Sobolev Inequalities}
Here we collect some standard inequalities we will use.

We will use the Sobolev inequalities on both $\Omega$ and
$\pa \Omega$. For any tensor field $\alpha$ on either $\Omega\cup\pa \Omega$
or $\pa \Omega$:
\begin{align}
  ||\alpha||_{L^{3p/(3-kp)}(\Omega)}
  &\leq C {\sum}_{|I| \leq k}||\pa_y^I \alpha||_{L^p(\Omega)},
  &&1 \leq p < {3}/{k},\label{sob2}\\
  ||\alpha||_{L^\infty(\Omega)}
  &\leq C {\sum}_{|I| \leq k} ||\pa_y^I \alpha||_{L^p(\Omega)},
  && k > {3}/{p},\\
  ||\alpha||_{L^{2p/(2-kp)}(\pa\Omega)}
  &\leq C {\sum}_{|I| \leq k} ||\pa_y^I \alpha||_{L^p(\pa \Omega)},
  && 1 \leq p < {2}/{k},\\
  ||\alpha||_{L^\infty(\pa \Omega)}
  &\leq {\sum}_{|I| \leq k} ||\pa_y^I \alpha||_{L^p(\pa \Omega)},
  && k> {2}/{p}.
 {}
\end{align}
By, e.g. the results in the appendix of \cite{CL00}, the constants above depend only
on the injectivity radius of $\Omega$.

We also have the following alternative characterization of the Sobolev spaces
\begin{equation}
\| D_c^h F\|_{L^2}\lesssim \|\pa_c F\|_{L^2}\lesssim {\sup}_h \| D_c^h F\|_{L^2}
, \quad \text{where}\quad
  D_c^h F(z) = \big(F(z + h e_c) - F(z)\big)/{h}, \label{eq:differencequotionet}
\end{equation}
denotes  the difference quotient in the direction of a unit
  vector $e_c$, see \cite{Evans2010}.

We will also need the trace inequality (see, e.g. \cite{Taylor2011}):
\begin{equation}
 ||f||_{H^{s-1/2}(\pa \Omega)} \leq C ||f||_{H^s(\Omega)},\qquad s>1/2.
 \label{trace}
\end{equation}
  We will only
apply this when $s$ is a positive integer and in that case the
right-hand side is defined in the usual way and the left-hand side is
defined by \eqref{sobspacedef}.
We will use the following Sobolev inequalities.
\begin{lemma}
 If $s \geq 2$, then:
 \begin{equation}
  ||f||_{L^\infty(\Omega)} \leq C ||f||_{H^s(\Omega)}.
  \label{regsob}
 \end{equation}
 Further, with notation as in Section \ref{def T and FD},
 if $s \geq 2$ then:
 \begin{equation}
  ||f||_{L^\infty(\Omega)}
  \leq C ||\T^{{}_{\,}s} f||_{H^1(\Omega)}.
  \label{vectsob}
 \end{equation}
 If $ k < {3}/{p}$ and ${1}/{q} = {1}/{p} - {k}/{3}$, then:
 \begin{equation}
  ||f||_{L^p(\Omega)} \leq C {\sum}_{|I| \leq k}
  || \pa_y^I f||_{L^q(\Omega)}.
  \label{sobstar}
 \end{equation}
\end{lemma}

\begin{proof}
 The estimates \eqref{regsob} and \eqref{sobstar} are the usual Sobolev
 inequalities.
 The estimate \eqref{vectsob} follows after applying
 the one-dimensional Sobolev inequality in the
radial direction and the two-dimensional Sobolev inequality in the
tangential directions.
\end{proof}

We also have the following product rule:
\begin{lemma}
  Suppose that $|\pa_y^I D_t^k f| \leq K$ in $\Omega$ for all
  $|I| + k \leq 3$. Then, if $k + \ell =s$, we have:
 \begin{equation}
  ||fg||_{k,\ell} \leq (||f||_{k,\ell} + K) (||g||_{k,\ell} +
  ||g||_{s-1}).
  \label{product}
 \end{equation}
\end{lemma}
The right-hand side can also be bounded by $(||f||_s + L) ||g||_s$, but
for some our applications it is more useful to keep track of which types
of derivatives land on $f$.
\begin{proof}
 We need to bound $|| (D_t^{k_1\!} \pa_y^{J_1} \!f)
 ( D_t^{k_2} \pa_y^{J_2} g)||_{L^2(\Omega)}$ where
 $k_1\! + \!k_2 \!+\! |J_1| \!+\! |J_2| \!= \! s$. If $k_1\! +\! |J_1|
 \!\leq 3$, we bound this by
 $||D_t^{k_1\!} \pa_y^{J_1}\! f||_{L^\infty(\Omega)}
 ||D_t^{k_2} \pa_y^{J_2} g||_{L^2(\Omega)}$ which is bounded
 by the right-hand side of \eqref{product}. If instead
 $k_1\! + \! |J_1| \geq 4$, we bound it by
 $|| D_t^{k_1\!} \pa_y^{J_1}\! f||_{L^2(\Omega)}
 ||D_t^{k_2\!} \pa_y^{J_2}\! g||_{L^\infty(\Omega)}
 \leq ||f||_{k,\ell} ||g||_{2 + k_2 + |J_2|}$. Since
 $k_1 + |J_1| \geq 4$ and $k_1 + k_2 + |J_1| + |J_2| = s$,
 it follows that $2+ k_2 + |J_2| \leq s$, as required.
\end{proof}

\subsection{The extension operator}
\label{extension section}
Fix an integer $s \geq 0$. Let $\eta = \eta(r)$ be a smooth cutoff function which
is one when $r \leq 1 + 1/(4+4s) $ and zero when $ r \geq 1+ 1/(2+2s)$.
Let $\lambda_0,..., \lambda_s$ be the solution to the system
$\tsum_{j = 0}^s \lambda_j (-(j+1))^\ell = 1$ for $\ell = 0,..., s$.
If $f: \Omega \to \R$, we extend $f$ to a function $Ef = E_sf$ on
$\R^3$ by setting $Ef(y) = f(y)$ when
$|y| \leq 1$ and when $|y| \geq 1$, write $f(y) = f(r, \omega)$
where $r = |y|, \omega = y/|y|\in \mathbb{S}^2$ and define:
\begin{equation}
 Ef(r , \omega) =
 {\sum}_{j = 0}^s \lambda_j f(r - (j+1)(r-1), \omega) \eta(r), \quad r \geq 1.
 \label{extension operator}
\end{equation}

Let $\zeta = \zeta(r)$ be a smooth function with $\zeta(r) = 0, r \leq 1/4$
and $\zeta = 1$ for $r \geq 1/2$.
For $f:\R^3\to \R$, we define
$
 ||f||_{H^{(k,s)}(\R^3)}^2
 = {\sum}_{|I| \leq k} \int_0^\infty ||\pa_y^I (\zeta f)(r,\cdot)||_{H^s(\pa \Omega)}^2
 \, r^2 dr
 + ||(1-\zeta)f||_{H^{k+s}(\Omega)}^2,
 $
and we have:
\begin{theorem}
\label{comparable thm}
Fix $s \geq 2$ and define $E = E_s$ by \eqref{extension operator}.
Then $E$ is continuous as a map $H^s(\Omega) \to H^s(\R^3)$ and
$H^{(s,1/2)}(\Omega) \to H^{(s,1/2)}(\R^3)$
and there are constants $0 < C_1 < C_2 < \infty$ depending only on $s$ so that
\begin{equation}
 C_1 ||Ef||_{H^{(s,a)} (\R^3)} \leq  ||f||_{H^{(s,a)}(\Omega)}
 \leq C_2||Ef||_{H^{(s,a)}(\R^3)}
 \qquad
 \text{ where } a = 0,1/2,
 \label{extnorm}
\end{equation}
there is a constant $C$ depending only on $s$ so that if $T$ is any vector field on $\R^3$ with
$T|_{\Omega} \in \T$, then:
\begin{equation}
 ||T Ef||_{H^{(s, a)}(\R^3)} \leq C(|| ET f||_{H^{(s,a)}(\R^3)}+|| E f||_{H^{(s, a)}(\R^3)}),
 \quad \text{ where } a=0, 1/2.
 \label{comparable}
\end{equation}
 \end{theorem}

\begin{proof}
  We have:
  \begin{equation}
   \pa_r^\ell (Ef)(r,\omega) = {\sum}_{j = 0}^s \lambda_j \pa_r^\ell f(r - (j+1)(r-1),\omega)
   \big(-(j+1)\eta(r)\big)^\ell +
   g_\ell(r,\omega),\quad r \geq 1,
   \label{}
  \end{equation}
  where $g_\ell(1,\omega) \!= \!0$, so by the definition of the
  $\lambda_j$ and the fact that $\eta(1)\! =\! 1$, it follows that
  $\pa_r^k(Ef)(1,\omega) \!= \! \pa_r^k f(1,\omega)$ for $1\! \leq \!k \!\leq \!s$ and
  $\omega\! \in\! \mathbb{S}^2$. This implies the estimate \eqref{extnorm}.
  The estimate \eqref{comparable} follows from the fact that near the boundary,
  $T \in \T$ commutes with $E$ since $(y^a\pa_b - \pa_b y^a) |y|^2  = 0$.
\end{proof}

\subsection{The Green's formula}
We conclude this section by recording the following Green's formula which will be frequently used throughout this manuscript. Let $f, g: \D\rightarrow \R$ be $C^1$ functions, then:
\begin{multline}
\int_{\Omega} \pave_i f(\xve(t,y)) g(\xve(t, y)) \kve dy  = \int_{\tD_t} \pave_i f(\xve)g(\xve)d\xve = -\int_{\tD_t} f(\xve)\pave_i g(\xve)d\xve +\int_{\pa\tD_t} N_i f(\xve)g(\xve)dS(\xve)\\
=-\int_{\Omega} f(y)\pave_i g(y)\kve dy+\int_{\pa\Omega} N_i f(y)g(y)\widetilde{\nu} dS(y).
\label{Green's formula}
\end{multline}

\section{Proofs of Elliptic estimates for the Dirichlet Problem}
\label{elliptic}
Here we prove the elliptic estimates we need. We will
use these to prove that $\Lambda$ is a continuous map on a certain
Banach space and to prove that
 $\Lambda$ is a contraction, in Section \ref{vlwpsec}.
The basic estimates we need for the contraction estimates  imply
the estimates for the operator norm  so we start with the contraction estimates.

Let
$\Vu{}_{\!},{\!} \Vw{\!}:{}_{\!} [0,T] \!\times \!\Omega \!\to \R^3$ be two vector fields on
$\Omega$ and let
$\xveu, \xvew$ denote their smoothed flows \eqref{xvedef}. Set
\begin{equation}
 A_{{}_{\!}I\, a}^{\,\,i} = \frac{\pa \xveu^i}{\pa y^a},\qquad
  A_{{}_{\!}I\, i}^{\,\,a}= \frac{\pa y^a}{\pa \xveu^i}\qquad\text{and}\qquad
 A_{\II\, a}^{\,\,\,i}= \frac{\pa \xvew^i}{\pa y^a},\qquad
  A_{\II\, i}^{\,\,\,a}= \frac{\pa y^a}{\pa \xvew^i}.
\end{equation}
We will assume that:
\begin{equation}\label{uwbd}
 {\tsum}_{ k + |J| \leq 3} |\pa^J_y \xve_I| + |\pa^J_y \xve_{II}| \leq M_0.
\end{equation}
By the formula for the derivative of the inverse \eqref{dinv}
this implies that $|A^{\m\, a}_{\I\, i}| + |A_{\II\, i}^{\,\,\,a}| \leq C(M_0)$.
We define
\begin{equation}
  \wpa_{I i}= A_{{}_{\!}I\, i}^{\,\,a} \frac{\pa}{\pa y^a},\qquad\text{and}\qquad \wpa_{\II i}=A_{\II\, i}^{\,\,a}\frac{\pa}{\pa y^a},
\end{equation}
\begin{equation}
 \gveu^{ab} = \delta^{ij}A_{{}_{\!}I\, i}^{\,\,a}A_{{}_{\!}I\, j}^{\,\,b},\qquad\text{and}\qquad\gvew^{ab} = \delta^{ij}
 A_{\II\, i}^{\,\,\,a} A_{\II\, j}^{\,\,\,b},
\end{equation}
as well as:
\begin{equation}
 \Dveu f = \delta^{ij}\wpa_{I i}\wpa_{I j} f = \pa_a \big(\gve_I^{ab}\pa_b f\big),\qquad
 \Dvew f = \delta^{ij}\wpa_{I i}\wpa_{\II j} f = \pa_a \big(\gve_{\II}^{ab}\pa_b f\big).
\end{equation}

We define:
\begin{align}
 \divu \alpha &= \delta^{ij}\wpa_{I i} \alpha_j,
 &&
 \divw \alpha = \delta^{ij} \wpa_{\II i}  \alpha_j, \label{div12}\\
 (\curlu \alpha)_{ij} &= \wpa_{I i}  \alpha_j - \wpa_{I j}  \alpha_i,
 &&
 (\curlw \alpha)_{ij} = \wpa_{\II i} \alpha_j - \wpa_{\II j} \alpha_i,
 \end{align}
 and
 \begin{equation}
 \gamma_I^{ij} = A_{{}_{\!}I\, a}^{\,\,i} A_{{}_{\!}I\, b}^{\,\,j} \gamma^{ab}, \qquad\text{and}\qquad \gamma_{\II}^{ij} = A_{\II\, a}^{\,\,\,i} A_{\II\, b}^{\,\,\,j}
 \gamma^{ab}.
\end{equation}
Here, we are writing $\gamma^{ab}$ for the cometric on $\pa \Omega$ extended
to the interior of $\Omega$. Fixing a smooth radial function $\chi$ with
$\chi(r) = 0$ for $r \leq \frac{1}{2}$ and $\chi(r) = 1$ for $r \geq \frac{3}{4}$,
then:
\begin{equation}
 \gamma^{ab} = \delta^{ab} - \chi(r) N^a N^a,
 \label{}
\end{equation}
with $N$ the unit normal to $\pa \Omega$.

Recalling the notation $\T^r$ from Section \ref{def T and FD}, we will use the following norms:
\begin{equation}
 ||\alpha||_{H^k(\Omega)}^2
 =\! {\sum}_{|J| \leq k} \int_\Omega\!
 \delta^{ij}\pa_y^J \alpha_i
 \pa_y^J
 \alpha_j\, dy,\qquad
 ||\alpha||_{C^k(\Omega)} = \!{\sum}_{\ell = 0}^k ||\pa_y^\ell \alpha||_{L^\infty(\Omega)},\qquad
  ||\T^r \!\alpha||_{L^2(\Omega)}^2
 = \!\int_{\Omega} \!|\T^r\! \alpha|^2 dy.
\end{equation}

In what follows we will use the convention that the components of $\alpha$
will be expressed in terms of the $\xveu$ frame and $\beta$ will be expressed
in terms of the $\xvew$ frame and we will just write $\alpha, \beta$
instead of $\alpha_I, \beta_{II}$.
We now list the elliptic estimates we use. Proofs can be found in the following
sections.
\begin{lemma}
  \label{app:pwdiff}
  With the above definitions, if $\alpha, \beta$ are (0,1)-tensors
  on $\Omega$ then on $[0 ,T] \times \Omega$:
  \begin{equation}
   |\paveu \alpha - \pavew \beta|
   \leq C(M')\big(
   |\!\divu \alpha - \divw \beta|+
   |\!\curlu \alpha - \curlw \beta|
   +  |\T \!\alpha - \T \beta|
   + ||\xveu - \xvew||_{C^{1}(\Omega)} |\pavew \beta|\big).
   \label{pwdiffbd}
  \end{equation}
\end{lemma}

There is a higher-order version of Lemma \ref{app:pwdiff} in Sobolev spaces
and with mixed space and time derivatives:
\begin{lemma}
  \label{app:sobdiff}
  Fix $r \geq 7$ and let $1 \leq \ell \leq r$.
  Suppose $\xveu, \xvew \in H^r(\Omega)$ satisfy
  \eqref{uwbd}. If $\alpha - \beta \in H^\ell_{loc}(\Omega)$
  and:
  \begin{align}
  \divu \alpha - \divw \beta,\,\, \curlu \alpha - \curlw \beta\,\,
  \in H^{\ell-1}(\Omega),
  &&
  T (\alpha \!-\! \beta) \in H^{\ell-1}(\Omega), \textrm{ for all }
  T \!\in \T,
  && \pavew \beta \in H^{\ell-1}(\Omega),
  \label{hyp}
 \end{align}
  then $\alpha - \beta \in H^{\ell}(\Omega)$ and there is a constant
  $C_r = C_r(M_0, ||\xveu||_{H^{r}(\Omega)}, ||\xvew||_{H^{r}(\Omega)})$,
  so that
 \begin{multline}
  ||\alpha - \beta||_{H^{\ell}(\Omega)} \leq C_r
  \big( ||\divu \alpha - \divw \beta||_{H^{\ell-1}(\Omega)}
  + ||\curlu \alpha - \curlw \beta||_{H^{\ell-1}(\Omega)}\\
  + ||\T^{\ell-1} (\paveu \alpha - \pavew\beta)||_{L^2(\Omega)}
  + (||\xveu - \xvew||_{C^2(\Omega)} + ||\xveu - \xvew||_{H^{\ell}(\Omega)})
  ||\pavew \beta||_{H^{\ell}(\Omega)}\big).
  \label{app:hotpw}
 \end{multline}

 Similarly, if $k +\ell = s \leq r$, $D_t^{k'}\pave \beta \in H^{\ell'}(\Omega)$ for any $k'+\ell'\!\leq \! s$ and:
 \begin{equation*}
 D_t^k (\divu \!\alpha - \divw \beta)\!\in\! H^{\ell\!-\!1\!}(\Omega),\qquad D_t^k(\curlu\! \alpha - \curlw \beta)
\! \in\! H^{\ell\!-\!1\!}(\Omega),\qquad
 D_t^k T (\alpha - \beta) \!\in\! H^{\ell\!-\!1\!}(\Omega), \!\!\!\!\quad\text{for all }
 T \!\in\! \T\!\!,
 \label{hyp2}
\end{equation*}
then $D_t^k(\alpha - \beta) \in H^\ell(\Omega)$ and there is a constant
$C_r' = C_r'(M_0, ||\xveu||_r, ||\xvew||_r)$,
 so that:
 \begin{multline}
  ||\alpha - \beta||_{k,\ell} \leq C_s'
  \big( || (\divu \alpha - \divw \beta)||_{k,\ell-1} +
  ||\curlu \alpha - \curlw \beta||_{k,\ell-1}\\
  + ||\FD^{k,\ell-1} (\paveu \alpha - \pavew \beta)||_{L^2}+
  ||\alpha - \beta||_{s,0}
  +(||\xveu - \xvew||_{C^2(\Omega)} + ||\xveu - \xvew||_{s})
  (||\beta||_{s,0} + ||\pave \beta||_{s-1}) \big).
  \label{mixedellipticest}
 \end{multline}
\end{lemma}

In the special case that $\alpha = \pa f, \beta = \pa g$ for functions
$f, g \in H^1_0(\Omega)$, $\paveu\! f - \pavew g \in H^\ell_{loc}(\Omega)$,
 we have:
\begin{prop}
  \label{app:sobfndiff}
 Suppose $\xve_I, \xve_{\II}\! \in\! H^s(\Omega)$, $s\! \geq\! 1$, satisfy \eqref{uwbd},
 $f\!-_{\!}g \!\in \!H^1_0(\Omega)$,
 $\paveu\! f \!-_{\!} \pavew g \!\in\! H^s_{loc}(\Omega)$ and that:
 \begin{align}
  \Dveu f - \Dvew g \in H^{s-1}(\Omega),&&
  \pavew g \in H^{s}(\Omega),&&
  T^J (\paveu  f - \pavew g) \in L^2(\Omega), \textrm{ for all } |J| \leq  s.
 \end{align}
 Then $\paveu f\!  - \pavew g \!\in \! H^{s}(\Omega)$
 and there is a constant $C_s\! = \!C_s(M_0, ||\xveu||_{H^s(\Omega)}, ||\xvew||_s)$
 so that
\begin{multline}
 ||\paveu f - \pavew g||_{H^{s}(\Omega)}
 \leq C_s ||\Dveu f - \Dvew g||_{H^{s-1}(\Omega)}
 +C_s ||\T(\xveu - \xvew)||_{H^s(\Omega)}
 ||\pavew g||_{H^s(\Omega)}\\
 + C_s ||\T \xveu||_{H^s(\Omega)}
 \big(||\xveu - \xvew||_{H^s(\Omega)}
  ||\pavew g||_{H^{s-1}(\Omega)}
  + ||f - g||_{L^2}\big).
\end{multline}

Similarly, if $k + \ell = s$, the assumption
\eqref{Massumpapp} holds,
$D_t^k(\paveu f - \pavew g) \in H^{\ell}_{loc}(\Omega)$ and:
\begin{equation}
  D_t^k(\Dveu f- \Dvew g) \in H^{\ell-1}(\Omega),
  \qquad
  D_t^k \pavew g \in H^{\ell}(\Omega),
  \qquad
  T^J(\paveu f - \pavew g) \in L^2(\Omega),
  \text{ for all } T^J \in \FD^s,
\end{equation}
then $D_t^k(\paveu f - \pavew g ) \in H^\ell(\Omega)$ and
there are constants $C_s' =
C_s'(M, ||\xveu||_{s},
||\xvew||_{s})$ so that if
$k + \ell = s$:
\begin{multline}
 ||\paveu f - \pavew g||_{k,\ell}
 \leq C_s' \big( ||\Dveu f - \Dvew g||_{k-1, \ell}
 + ||f - g||_{s+1,0} + ||\paveu f - \pavew g||_{s-1,1}
 + ||\T \xvew||_s ||f - g||_{s}\big)\\
 + C_r'||\T(\xveu - \xvew)||_{s}\big( ||\pavew g||_{s} +
 ||g||_{s+1,0}).
 \label{ibpmixedelliptic}
\end{multline}
\end{prop}

We also need a result to build regularity for a function $f$ with
$\Dve f \in H^{\ell-1}(\Omega)$ but with a priori only $f \in H^1_0(\Omega)$. Note that we are \emph{not} assuming that
$f \in H^\ell_{loc}(\Omega)$. This result is needed to prove a local-wellposedness
result
for the wave equation \eqref{smwavedef}-\eqref{smwaveic} (see Appendix \ref{linwaveexist}).
Writing $\xve = \xve_I$, we have:

\begin{prop}
  \label{app:sobdirich}
  Suppose  $\xve \in H^r(\Omega)$, $r\! \geq \! 5$, satisfies \eqref{uwbd}.
  If $f \in H^1_0(\Omega)$ and $\Dve f \in H^{\ell-1}(\Omega)$ for some $0 \leq \ell \leq r$,
  then $\pave f \in H^\ell(\Omega)$ and
  \begin{equation}
   ||\pave f||_{H^\ell(\Omega)} \leq C(M_0, ||\xve||_{H^r(\Omega)})
   \big( ||\Dve f||_{H^{\ell-1}(\Omega)} +
   ||\T \xve||_{H^r(\Omega)}
   ||f||_{L^2(\Omega)}\big).
   {}
  \end{equation}

  Similarly, if $f \in H^1_0(\Omega)$,
  $D_t^k f \in L^2(\Omega)$ and
$D_t^k \Dve f \in H^{\ell-1}(\Omega)$, then $D_t^k \pave f \in H^\ell(\Omega)$
and
\begin{equation}
 ||D_t^k \pave f||_{H^\ell(\Omega)} \leq C(M_0, ||\xve||_r)
 \big( ||D_t^k \Dve f||_{H^{\ell-1}(\Omega)}
 + ||\T \xve||_{r} ||D_t^k f||_{L^2(\Omega)}\big).
 {}
\end{equation}
\end{prop}

We also need estimates which involve fractional derivatives on $\pa \Omega$.
\begin{prop} \label{div-curl}
 Let $\alpha$ be a vector field on $\Omega$. Fix $r \geq 5$.
 Then, for $1 \leq \ell \leq r$, there are continuous functions $C_\ell = C_\ell\big(M_0,
 ||\xve||_{H^r(\Omega)}\big)$ so that
 \begin{multline}
   ||\wta||_{H^{\ell{}_{\!}}}^2
   \!\leq \! C_{\ell}\Big(||\!\div \alpha||_{H^{\ell\!-\!1\!}}^2\! +
    ||\!\curl \alpha||_{H^{\ell\!-\!1\!}}^2 +\!||\wta||^2_{H^1}
    +\!{\sum}_{\mu = 1}^N  \!\int_{\pa \Omega} \!\!({}_{\!}\fdhm \T^{\ell -\!1 \!}\wta^i)\!
    \cdot\!(\fdhm \T^{\ell-\!1 \!}\wta^j)
    N_{{}_{\!}i} N_{\!j} dS
   \Big),
 \label{fractang1}
\end{multline}
 \begin{multline}
    ||\wta||_{H^{\ell{}_{\!}}}^2
    \leq C_{\ell}\Big( ||\!\div \alpha||_{H^{\ell\!-\!1\!}}^2\! +
     ||\!\curl \alpha||_{H^{\ell\!-\!1\!}}^2\! +\!||\wta||^2_{H^1}
     + {\sum}_{\mu = 1}^N \!\int_{\pa \Omega} \!\!(\fdhm \T^{\ell -1 \!}\wta^i)\!\cdot\!(\fdhm \T^{\ell-1\!} \wta^j)
     \gamma_{ij} dS\Big).
 \label{fractang2}
\end{multline}
\end{prop}

We will need the following lemma to exchange normal and tangential components
of vector fields on $\pa \Omega$. This estimate appears
in Lemma 5.6 of \cite{CL00}.
\begin{lemma}
  If $\alpha$ is a (0,1)-tensor
  on $\Omega$ and $\gamma$ denotes the metric on $\pa \Omega_t$, then:
 \begin{equation}
  \Big|\int_{\pa \Omega} \big(\gamma^{ij} - N^iN^j\big)\alpha_i
  \alpha_j d\mu_\gamma\Big|
  \leq
  \big( ||\div \alpha||_{L^2(\Omega)} + ||\curl \alpha||_{L^2(\Omega)}
  + K ||\alpha||_{L^2(\Omega)}\big) ||\alpha||_{L^2(\Omega)}.
 \end{equation}
\end{lemma}

Finally, in Section \ref{linwaveexist}, we will need the following elliptic estimate
in $H^2(\Omega)$:
\begin{lemma}
  \label{crosstermlem}
 Let $\Delta_y\! = \pa_{{}_{\!}y^1}^2\! + \pa_{{}_{\!}y^2}^2\! + \pa_{{}_{\!}y^3}^2$ be
 the flat Laplacian in the $y$ coordinates. If $f \!\in\! H^{{}_{\!}1}_0(\Omega) \!\cap\!
 H^2(\Omega)$, then:
 \begin{equation}
  ||\pave f||_{H^1(\Omega)} \leq C(M)
  \big( (\Delta_y f, \Dve f)_{L^2(\Omega)} + ||\pave f||_{L^2(\Omega)}
  + ||f||_{L^2(\Omega)}\big).
  \label{crossterm}
 \end{equation}
\end{lemma}

\subsection*{Proof of Lemma \ref{app:pwdiff} }
  The case with $\beta = 0$ is Lemma 5.5 in \cite{CL00},
  and this version is Lemma B.4.1
  of \cite{Nordgren2008}. For the reader's convenience, we include the proof here.
We start by setting:
   \begin{equation}
    \big(\defo_I \alpha\big)_{ij} = \wpa_{I{}_{\!}i} \alpha_{I{}_{\!}j}
    +
     \wpa_{I{}_{\!}j} \alpha_{I{}_{\!}i},\qquad
    (D_I \alpha)_{ij} = \div_I \alpha \delta_{ij},\qquad
    (\widehat{D_I} \alpha)_{ij} = \big(\defo_I \alpha - \frac{2}{3}D_I\alpha
    \big)_{ij},
   \end{equation}
   with a similar definition for $\defo_{\II}, D_{\II},$ and $\widehat{D_{\II}}$.
   We write:
   \begin{align}
    \paveu \alpha - \pavew \beta
    &=
     \frac{1}{3}\big(D_I \alpha - D_{\II} \beta\big)
    + \frac{1}{2}\big(\curl_I \alpha - \curl_{\II} \beta\big)
    + \frac{1}{2}\big(\widehat{D_{I}} \alpha - \widehat{D_{\II}} \beta\big).
   \end{align}
   The first and second terms are bounded by the right-hand side of
   \eqref{pwdiffbd}, and we now show how to control the last term.
   Let $S_{ij} = (\widehat{D_I}\alpha - \widehat{D_{II}} \beta)_{ij}$. Writing
   $\delta^{ij} = \gamma_I^{ij} - N^i_I N^j_I$ and using that
   $S$ is symmetric, we have:
   \begin{equation}
    \delta^{ij}\delta^{k\ell} S_{ik}S_{j\ell} =
    \big(\gamma_I^{ij}\gamma_{I}^{k\ell} + 2 \gamma_I^{ij} N_I^{k\ell}
    + N_I^iN_I^jN_I^kN^\ell_I\big) S_{ik} S_{j\ell}.
    \label{mainexpdiff}
   \end{equation}
   Now, because $\delta^{ij} S_{ij} = 0$, the last term is:
   \begin{equation}
    \big(N_I^iN_I^j S_{ij}\big)^2 = \big(\delta^{ij}S_{ij} - \gamma_I^{ij}S_{ij}
    \big)^2
    = \big(\gamma_I^{ij} S_{ij}\big)^2 \leq 2 \gamma_I^{ij}
    \gamma_I^{k\ell} S_{ik}S_{j\ell},
   \end{equation}
   where we have used that if $T\!$ is a symmetric
   matrix then
   $
    (\tr T)^2\! {}_{\!}\leq\! \textrm{rank} T {}_{\!}\tr (T^2)
    $.
   Returning to \eqref{mainexpdiff}, we have:
   \begin{equation}
    |S|^2 \leq 2 \gamma_I^{ij}\big(\gamma_I^{k\ell} + N_I^k N_I^\ell\big)
    S_{ik}S_{j\ell}
    = 2\gamma_I^{ij} \delta^{k\ell}S_{ik}S_{j\ell}.
   \end{equation}
   We now write:
   \begin{equation}
    S_{ij} = \big( \defo_I \alpha - \defo_{II}\beta\big)_{ij}
    - \frac{2}{3} \big( D_I\alpha - D_{II}\beta\big)_{ij}
    \equiv S^1_{ij} + S^2_{ij}.
   \end{equation}
   Since $|S^2| \leq C |\div_I \alpha - \div_{II}\beta|$, it suffices
   to control $S^1$. We have:
   \begin{equation}
    \gamma_I^{ij}\delta^{k\ell}S^1_{ik}S^1_{j\ell}   = \gamma_I^{ij}\delta^{k\ell}
    \big( \wpa_{I i} \alpha_{I k} - \wpa_{\II i} \beta_{I\!I k }
    + \wpa_{I k} \alpha_{I i} - \wpa_{\II k} \beta_{I\!I  i}
    \big)
    \big( \wpa_{I j}\alpha_{I\ell} - \wpa_{\II j}\beta_{II\ell}
    +  \wpa_{I \ell}\alpha_{Ij} -  \wpa_{\II \ell}\beta_{IIj}\big)
    \label{diffbig}.
    \end{equation}
    To bound the product of the first term in the first
    factor with the first term in the second factor, we
    replace $\pave_{II}\beta_{II}$ with $\pave_{I} \beta_{II}$, which
    generates terms that are bounded by the last term on the right-hand side of
    \eqref{pwdiffbd}. The resulting term only involves tangential derivatives
    of $\alpha, \beta$ but these are with respect to
    $\xveu$. However we can replace these with tangential
    derivatives with respect to $y$ up to terms that are bounded by the
    last term on the right-hand side of
    \eqref{pwdiffbd}.
    For the product of the second term in the first factor and the
    second term in the second factor we instead note that it can be
    controlled in terms of $|\!\curl_I\!\alpha \!- \curl_{II}\! \beta|^2$
    along with
    the third and fourth terms on the right-hand side of \eqref{pwdiffbd}
    The other terms in \eqref{diffbig} can be handled similarly.

\subsection*{Proof of Lemma \ref{app:sobdiff}}
\label{ftangpf}
Both estimates have essentially the same proof, so we will
just prove the second. The first one follows from the same argument,
but one uses the commutator estimate \ref{vectcommlemma} with
$\U \!=\! \{\pa_{y^1\!}, \pa_{y^2\!}, \pa_{y^3}\!\}$ instead of $\U \!=\! \coord$. The only
difference is that
in the proof of \eqref{app:hotpw} no time derivatives enter.

  We argue by induction. When $s = 1$, the result follows
  from the pointwise estimate after writing:
  \begin{equation}
   \pa_{a} (\alpha - \beta)=
   A_{{}_{\!}I\, a}^{\,\,i} (  \wpa_{I{}_{\!}i}\alpha - \wpa_{\II{}_{\!}i}\beta)
   + (A_{{}_{\!}I\, a}^{\,\,i} - A_{\II\, a}^{\,\,\,i})\wpa_{I{}_{\!}i}\beta.
  \end{equation}
  We now assume that we have the
  result for $s \!\leq\! m\!-\!1$.
  We write $T^I\!\! = D_t^k \pa_y^J \!\!\in \D^{k,\ell}$ where
  $k \!+ \! |J|\! = m$. If $|J|\!=\! 0$ there is nothing to prove, so we consider
  $|J| \!\geq\! 1$. We then write $D_t^k \pa_y^J \!\!= \pa_{a} D_t^k \pa_y^{J'}\!$ where
  $J \!= (a, J')$ and $\pa_{a} \!= A_{{}_{\!}I\, a}^{\,\,i} \pave_{i}$. Applying
  the pointwise estimate \eqref{pwdiffbd} and integrating over an arbitrary
  $U \!\subset\subset \Omega$, we have:
  \begin{multline}
   ||T^I (\alpha - \beta)||_{L^2(U)}
   \leq C(M_0) \big( ||\divu D_t^k \pa_y^{J'} \alpha -
   \divw D_t^k \pa_y^{J'} \beta||_{L^2(\Omega)}
  + ||\curlu D_t^k \pa_y^{J'} \alpha -
  \curlw D_t^k \pa_y^{J'} \beta||_{L^2(\Omega)} \\
  +
  ||\T D_t^k \pa_y^{J'} (\alpha - \beta)||_{L^2(\Omega)}
  + ||\xveu - \xvew||_{C^2(\Omega)} ||\pavew D_t^k \pa_y^{J'}\beta||_{L^2(\Omega)}\big).
   \label{afterpw}
  \end{multline}
  Using the commutator estimate from Lemma \ref{vectcommlemma} with
  $\U = \coord$, the last term is bounded by the right-hand side of
  \eqref{mixedellipticest}.
  To deal with the first two terms, we apply the commutator estimate
  \eqref{vectcomm}
  with $\U = \coord$:
  \begin{equation}
   ||\divu\! D_t^k \pa_y^{J'}\!\alpha -
   \divw\! D_t^k\pa_y^{J'} \!\beta||_{L^2}
   \leq ||D_t^k \pa_y^{J'}\! (\divu\! \alpha - \divw \beta)||_{L^2}
   +  C_s \big( ||\paveu \alpha - \pavew \beta||_{m-2}
   + ||\xveu \!- \xvew||_{s} ||\pavew \beta||_{m-2}\big),
  \end{equation}
  where $L^2=L^2(\Omega)$ and
   $C_s = C_s(M, ||\xveu||_s, ||\xvew||_s)$,
  along with a similar estimate for the curl. All of these terms
  are bounded by the right-hand side of \eqref{mixedellipticest}.
  To deal with the last term on the right-hand side of \eqref{afterpw},
  we commute the tangential derivative with $D_t^k\pa_y^{J'}$:
  \begin{equation}
   |\T D_t^k \pa_y^{J'} (\alpha - \beta)|
   \leq {\sum}_{T \in \T} |D_t^k \pa_y^{J'} T(\alpha - \beta)|
   + C |D_t^k\pa_y^{J'} (\alpha - \beta)|.
  \end{equation}
  The second term here is bounded by the right-hand side of \eqref{mixedellipticest}
  by the
  inductive assumption. To control the first term in $L^2$,
  we apply the inductive assumption with $\alpha, \beta$ replaced by
  $T \alpha, T\beta$, and this gives:
  \begin{multline}
   ||D_t^k\pa_y^{J'} T(\alpha - \beta)||_{L^2(\Omega)}
   \leq C_s ||\divu T\alpha - \divw T \beta||_{k,\ell-2}
   + ||\curlu T\alpha - \curlw T \beta||_{k,\ell-2}
   + ||\FD^{k,\ell}(\alpha - \beta)||_{L^2(\Omega)}\\
   + C_s(||\xveu - \xvew||_{C^2(\Omega)} +
   ||\xveu -\xvew||_{r} )||\pavew T\beta||_{m-1}.
   \label{last}
  \end{multline}
  We now write $\div T (\alpha \!-\! \beta) = T \div (\alpha\! -\! \beta)
  - TA_{{}_{\!}I\, i}^{\,\,a} \pa_a (\alpha^i\! -\! \beta^i)$, and use the product rule
  \eqref{product} and \eqref{dinv}:
  \begin{equation}
   ||(TA_{{}_{\!}I\, i}^{\,\,a})\pa_a(\alpha^i - \beta^i)||_{k,\ell-2}
   \leq C(M, ||\xveu||_s)||\alpha - \beta||_{m-1}.
  \end{equation}
  Arguing as with the other terms in \eqref{last},
  recalling that we are integrating over any $U \!\subset\subset
  \!\Omega$ gives the result.

\subsection*{Proof of Proposition \ref{app:sobfndiff}}
\label{sobdiffpf}

To motivate the proof, first consider the case that $\xve_{\II} \!= \!\xve_{\I}$
and $g \!=\! 0$. If $\xve_{\I}$ was smooth, one could get a version of this
estimate without tangential derivatives by straightening the boundary
and using a standard integration by parts argument.
Because the coordinate $\xve_{\I}$ is only smooth in tangential directions, the
idea is instead to first use the estimate \eqref{app:hotpw} to
replace the derivatives of $\pave f$ with derivatives of $\Delta f$ and tangential
derivatives of $\pave f$, and then apply the integration by parts argument
to this. One then has to deal with commutators $[\T^r\!\!, \pave ]f$.
To highest order, this behaves like $(\T^r \pa_y \xve_{\I})\pa_y f$, and because
the derivatives $\T$ are tangential this term can be handled. Also
note that since $\T^r\! f\!=\! 0$ on $\pa \Omega$, the boundary terms that arise
when integrating by parts vanish so we avoid the need to straighten the
boundary.

We start with the following estimate:
\begin{lemma}
 Under the hypotheses of Proposition \ref{app:sobfndiff}, we have:
 \begin{equation}
  ||\paveu f - \pavew g||_{L^2}^2
  \leq C(M_0)\big( ||\Dveu f - \Dvew g||_{L^2}^2
  + ||\xveu - \xvew||_{C^2(\Omega)}^2 ||\pavew g||_{L^2}^2\big).
  \label{zeroth}
 \end{equation}
\end{lemma}
\begin{proof}

We write $\pavew g = \paveu g + (A_{\II}-A_I)\cdot \pa_y g$ and since $||\alpha||_{L^2(\Omega)}^2$ is comparable to $\int_\Omega |\alpha|^2 \kve\, dy$,  so:
\begin{multline}
 ||\paveu f - \pavew g||_{L^2(\Omega)}^2
 \lesssim  \int_{\Omega} \delta^{ij}
 \big(\wpa_{I i} f - \wpa_{\II i} g\big)
 \big( \wpa_{I j} f - \wpa_{\II j}g\big)\kveu dy = \int_{\Omega} \delta^{ij}
 \big( \wpa_{I i}  f - \wpa_{\II i} g\big)
 \big( \wpa_{I j}  f - \wpa_{\II j}g\big)\kveu dy\\
 + 2\int_{\Omega} \delta^{ij} (A_{{}_{\!}I\, i}^{\,\,a} - A_{\II\, i}^{\,\,\,a})
 A_{{}_{\!}I\, j}^{\,\,b} (\pa_a g)\pa_b (f-g)\kveu\, dy
 + \int_{\Omega} \delta^{ij} (A_{{}_{\!}I\, i}^{\,\,a} - A_{\II\, i}^{\,\,\,a})
 (A_{{}_{\!}I\, j}^{\,\,b} - A_{\II\, j}^{\,\,\,b}) (\pa_a g)(\pa_b g) \kveu dy.
\end{multline}

The terms on the last line are bounded by
the second term on the right-hand side of \eqref{zeroth}, using Lemma
\ref{inversedifflemma} and Sobolev embedding.
To control the terms on the first line, we integrate by parts:
\begin{align}
 \int_\Omega \delta^{ij} A_{{}_{\!}I\, i}^{\,\,a} A_{\m i}^a \pa_a (f-g) A_{{}_{\!}I\, j}^{\,\,b}\pa_b(f-g) \kveu dy
 = -\int_\Omega (f-g) \frac{1}{\kveu} \pa_a \big( \kveu \delta^{ij} A_{{}_{\!}I\, i}^{\,\,a}A_{{}_{\!}I\, j}^{\,\,b} \pa_b (f-g)\big)
 \kveu dy.
\end{align}
The second factor here is
$\Dveu(f-g) = \big(\Dveu f - \Dvew g \big)+ (\Dveu - \Dvew)g$.  Since
we
want a bound that only involves one derivative of $g$, we further write:
\begin{equation}
 (\Dveu - \Dvew )g =
 \frac{1}{\kveu} \pa_a \Big(  \kveu (\gveu^{ab} \pa_b g)
 - \kvew (\gvew^{ab}\pa_b g) \Big)
 + \Big(\frac{1}{\kvew} - \frac{1}{\kveu}\Big)
 \pa_a \big( \kvew \gvew^{ab}\pa_b g\big),
\end{equation}
and then integrate by parts and use Poincar\`{e}'s inequality again, which shows that:
\begin{equation}
 \Big| \int_\Omega (f-g) (\Dveu  - \Dvew) g\kve dy\Big|
\leq C(M_0)||\xveu -\xvew||_{C^2(\Omega)} ||\paveu f - \pavew g||_{L^2}
 ||\pavew g||_{L^2}.\tag*{\qedhere}
\end{equation}
\end{proof}

We now consider the case $\alpha = \paveu f, \beta = \pavew g$ for
functions $f,g \in H^1_0(\Omega)$. We then have:
\begin{prop}
  \label{mixedestimates}
 With the hypotheses of Proposition \ref{app:sobfndiff}, for each $s$
 there are constants\linebreak $C_s =
 C_s(M, ||\xveu||_{H^s(\Omega)}, ||D_t \xveu||_{s},
 ||\xvew||_{H^s(\Omega)}, ||D_t \xvew||_s)$ so that if
 $k + \ell = s$:
 \begin{multline}
  ||\paveu f - \pavew g||_{k,\ell}
  \leq C_s\big( ||\Dveu f - \Dvew g||_{k-1, \ell}
  + ||f - g||_{s,0} + ||\paveu f - \pavew g||_{s-1,1}
  + ||\T \xvew||_s ||f - g||_{s}\\
  + ||\T(\xveu - \xvew)||_{s}\big( ||\pavew g||_{s} +
  ||g||_{s+1,0}\big)\big).
 \end{multline}
\end{prop}

This proposition follows from \eqref{mixedellipticest} and the
following lemma:
\begin{lemma}\label{justtangential}
 With the hypotheses as above,
  there is a constant $C_{\!s}(M,_{\!}
 ||\xveu||_{s},_{\!} ||\xvew||_{s})$ so that for any $\delta\! > \!0$:
 \begin{multline}
  ||\FD^{k,\ell} (\paveu f - \pavew g)||_{L^2(\Omega)}
  \leq C_s \big(||\Dveu f- \Dvew g||_{k-1,\ell} + \delta
  ||\paveu f - \pavew g||_{k,\ell}+ \delta^{-1}||\T(\xveu -\xvew)||_s ||\pavew g||_{s}\\
  +  \delta^{-1}||\T \xveu||_s (||\paveu f - \pavew g||_{s,0}
  + ||f -  g||_{s,0})
  \big),\quad s\!=\!k\!+\!\ell.
  \label{justtangentialest}
 \end{multline}
\end{lemma}
\begin{proof}[Proof of Lemma \ref{justtangential}]
  For the purposes of the below proof, the commutator $[T, \pa_a]$
  for $T \in \T$ will be ignored for notational convenience.
  We argue by induction.
When $s = 1$, we fix a multi-index $I$ with
$|I| = 1$. If $T^I = D_t$ there is nothing to prove so we assume
that $T^I = S \in \T$. We start by writing:
\begin{equation*}
  ||S \paveu f\! - S \pavew g||_{L^2(\Omega)}^2
 = \int_\Omega \!\!(\paveu Sf\! - \pavew Sg)\cdot
 S( \paveu f \!- \pavew g)\,dy
 + \!\int_\Omega \!\! \big([\paveu, S] f \!- [\pavew, S] g\big)
 \cdot S (\paveu f \!- S \pavew g)\,
 dy.
\end{equation*}
To deal with the first term, we integrate by parts and use
that $Sf \!=\! Sg\!=\! 0$ on $\pa \Omega$, which gives:
\begin{equation}
 \int_\Omega \!\!(\paveu S f - \pavew Sg)\!\cdot\! S
 (\paveu{}_{\!} f \!- \pavew g)  dy\\
 = \!\int_\Omega\!\! Sf  \pa_a\big(
 \delta^{ij}\! A_{{}_{\!}I\, i}^{\,\,a} \{ S \paveu{}_{\!j} f \!- S\pavew{}_{\!j} g\}\big) dy
 - \!\int_\Omega \!\! Sg\,
 \pa_a\big(
 \delta^{ij}\! A_{\II\, i}^{\,\,\,a} \{S \paveu{}_{\!j} f \!- S\pavew{}_{\!j} g\}\big) dy.
\end{equation}
We write the first term on the right-hand side as:
\begin{equation}
  \int_\Omega
 \delta^{ij} \big( (Sf)A_{{}_{\!}I\, i}^{\,\,a} - (Sg) A_{\II\, i}^{\,\,\,a}) \pa_a \big( S \paveu{}_{\!j}f
 - S\pavew{}_{\!j} g\big)\,  dy
 +\int_\Omega \delta^{ij} \big(Sf \pa_a( A_{{}_{\!}I\, i}^{\,\,a}) - S g\pa_a( A_{\II\, i}^{\,\,\,a})\big)
 \big(S\paveu f - S\pavew g\big)\, dy.
\end{equation}
The second term here is bounded by the right-hand side of
\eqref{justtangentialest}. We now re-write the first term as:
\begin{multline}
 \int_\Omega\!\! (Sf) \paveu {}_{\!}\cdot{}_{\!} (S \paveu f\! - S\pavew g)
 - (Sg) \pavew{}_{\!} \cdot{}_{\!} (S \paveu f\! - S \pavew g)
=\!\! \int_\Omega  \!(Sf) S (\Dveu f\! - \paveu\!\cdot{}_{\!} \pavew g)
- (Sg) S(\pavew\!\cdot \paveu {}_{\!}f \!- \Dvew g)\\
+\!\! \int_\Omega  \!(S{}_{\!}f) [\paveu, S] {}_{\!}\cdot{}_{\!} ( \paveu {}_{\!}f \!- \pavew g)
-(Sg) [\pavew, S] {}_{\!}\cdot{}_{\!} (\paveu {}_{\!}f \! - \pavew g).
\end{multline}
Finally, we re-write the first
term on the right-hand side as:
\begin{equation}
 \int_\Omega (Sf - Sg) S (\Dveu f - \Dvew g)
 + \int_\Omega (Sf - Sg) S (\paveu - \pavew)\cdot \pavew g,
\end{equation}
and integrate $S$ by parts in each of these terms. Applying Cauchy's
inequality, the result of the above is:
\begin{multline}
 ||S(\paveu f - \pavew g)||_{L^2(\Omega)}^2
 \leq C_1 \big(||\Dveu f - \Dvew g||_{L^2(\Omega)}^2
 + \delta^{-1}||\paveu f - \pavew g||_{L^2(\Omega)}^2\big)+ C_1
 \delta^{-1}||\xveu - \xvew||_{C^2(\Omega)}
 ||\pavew g||_{H^1(\Omega)}\\
 + C_1\delta \big(||\paveu f - \pavew g||_{H^1(\Omega)}^2 +
 ||[\paveu, S] f - [\pavew, S] g||_{L^2(\Omega)}^2 +
 ||[\paveu, S]\cdot \paveu f - [\pavew, S]\cdot \pavew g||_{L^2(\Omega)}^2\big).
\end{multline}
Here, and in what follows, we will use $C_k$ to denote a constant
which depends on $M, ||\xveu||_k, ||\xvew||_k$.
Applying the commutator estimate \eqref{vectcomm}, every term here is bounded by the
right-hand side of \eqref{justtangentialest}.

We now suppose we have the result for $s = 1,..., m-1$, and fix
$T^I = D_t^k T^J$ where $T^K \in \T^\ell$ with $k + \ell = m$.
If $|K| = 0$ there is nothing to prove
so we assume that $T^I = ST^J$ for some $S \in \T$ and $T^J \in \FD^{k,\ell-1}$.
The proof now follows in nearly the same way as above, so we just indicate the
main points. First, we write:
\begin{equation}
 \int_\Omega \!\! T^I \!(\paveu {}_{\!}f \!- \pavew g)
T^I\! (\paveu  {}_{\!} f\! - \pavew g)\, dy
= \int_\Omega \!\! (\paveu T^I\! {}_{\!}f\! - \pavew T^I \! g)
T^I\!(\paveu {}_{\!}f\! - \pavew g) dy +\!
\int_\Omega\!\!  ([\paveu,\! T^I]f\! - [\pavew,\! T^I] g)
T^I (\paveu {}_{\!}f\!- \pavew g)  dy.
\end{equation}
Integrating by parts in the first term yields, in addition to lower-order terms:
\begin{equation}
 \int_\Omega (T^I f)\paveu \cdot (T^I \paveu f - T^I \pavew g)
 - (T^I g) \pavew \cdot (T^I \paveu f - T^I \pavew g)\, dy.
\end{equation}
We now write $T^I = S T^J$ in the second factor in each
term and then commute $S$ with $\paveu, \pavew$,
and obtain:
\begin{equation}
 \int_\Omega (T^J f - T^J g) S(T^J\Dveu f - T^J \Dvew g)
 + \int_\Omega (T^J f - T^J g) S T^J\big( (\paveu - \pavew) \cdot \pavew g\big).
\end{equation}
Integrating $S$ by parts and bounding:
\begin{equation}
 ||T^J\big((\paveu - \pavew) \cdot \pavew g\big)||_{L^2(\Omega)}
 \leq ||\xveu - \xvew||_{r} ||\pavew g||_{m},
\end{equation}
shows that $||T^I (\paveu f - \pavew g)||_{L^2(\Omega)}$ is bounded by:
\begin{multline}
 C_m \big(
 ||T^J (\Dveu f - \Dvew g)||_{L^2(\Omega)} +
 (1 + \delta^{-1})||T^J(\paveu f - \pavew g)||_{L^2(\Omega)}^2+
 (1 + \delta^{-1}) ||\T(\xveu - \xvew)||_{m}^2 ||\pave g||_{k,\ell}^2\\
 + C_m \delta \big( ||\paveu f - \pavew g||_{k,\ell}^2 +
 || [\paveu, S] T^J f - [\pavew, S] T^J g||_{L^2(\Omega)}
 + ||[\paveu, S] T^J \paveu f -
 [\paveu, S] T^J \pavew g||_{L^2(\Omega)} \big).
\end{multline}
The result now follows after using the commutator estimate \eqref{vectcomm}
and induction.
\end{proof}

\subsection*{Proof of Proposition \ref{app:sobdirich}}
We just prove the $k = 0$ case, as the $k \geq 1$ case follows
using similar arguments.
This would be a consequence of the Proposition \ref{app:sobfndiff}
with $g = 0$ if we knew that
$\pave f \in H^m_{loc}(\Omega)$ and $\T^I \pave f\in L^2(\Omega)$
for all $|I| \leq m$. In the following lemma we prove that this is
the case. See Section \ref{tangapp} for the definitions of the sets
 $U_\alpha$ and the vector fields $T \in \T$.

\begin{lemma}
Fix $s \!\geq\! 0$ and suppose that $\xve \!\in \! H^s(\Omega),
T \!\xve\! \in\! H^s(\Omega)$
for all $T \!\in \!\T\!$ and that \eqref{uwbd} holds.
Suppose also that $f\! {}_{\!}\in \! H^{s{}_{\!}}(\Omega)$,
$\Dve f\!{}_{\!} \in\! H^{s-\!1\!}(\Omega)$.
Then $\pave {}_{\!}f\! {}_{\!}\in \! H^s_{loc}(\Omega)$
and $T^{I{}_{\!}} \pave {}_{\!}f {}_{\!}\!\in \! L^2(\Omega)$ for all $|I|\! \leq \!s$
and there is a constant
$C_s \!=\! C_s(M_0, ||\xve||_{H^s(\Omega)})$ so that
with notation as in \eqref{eq:simplifiedtangentialnotation},
the following inequalities hold for any $V\! \subset\subset \!\Omega$:
\begin{align}
 ||\pave f||_{H^s(V)} + ||\T^s\pave f||_{L^2(\Omega)} &\leq
 C_s\big( ||\Dve f||_{H^{s-1}(\Omega)} +
 (||\T \xve||_{H^s(\Omega)} + ||\xve||_{H^s(\Omega)})
 ||f||_{H^s(\Omega)}\big).\label{wkellloc}
\end{align}
  \end{lemma}
\begin{proof}
  We will follow the proof in \cite{Evans2010}. Both of the above statements have
  essentially the same proof and so we will just prove the second one.
  For the case $s = 1$, we want to show:
  \begin{equation}
   {\sum}_{T \in \T} ||T \pave f||_{L^2(\Omega)}
   \leq C(M')\big( ||\Dve f||_{L^2(\Omega)} + ||\pave f||_{L^2(\Omega)}\big).
   \label{goal}
  \end{equation}

  We fix one of the open sets $U = U_\mu$ with $\mu \geq 1$ and write
  $F = f\circ \psi_\mu$. Then, arguing
  as in \cite{Evans2010}, to  prove \eqref{goal} it suffices to prove
  that for every $V \subset U$,  with
  a constant independent of $h$,
  \begin{equation}
   ||D_c^h \pave F||_{L^2(V)} \leq C
   \big( ||\Dve f||_{L^2(\Omega)} + ||\pave f||_{L^2(\Omega)}\big),
   \quad\text{for}\quad c = 1,2,
   \label{diffquot}
  \end{equation}
 for $\!D_{\!c}^{h}\!$ denoting the difference quotient in the direction of a unit
  vector $e_c$
  \begin{equation}
   D_c^h F(z) = \big(F(z + h e_c) - F(z)\big)/{h}.
  \end{equation}

  Let $\rho$ denote
  a cutoff function which is 1 on $V$ and zero outside of $U$, and
  set $v = -D_c^{-h}(\rho^2 D_c^h F)$. Note that $v \in H^1_0(U)$.
  Now we have:
\begin{equation}
 \int_U \!\!\Dve F v = -\!\int_U\!\!\delta^{ij} (\pave_i F)
 \pave_j \big\{ D_c^{-h} (\rho^2 D_c^h F)\big\}
 =\! \int_U\!\!  \delta^{ij}\!\underbrace{ (\pave_i F) D_c^{-h} \pave_j\{\rho^2 D_c^h F\}}_{I}-
 \int_U \!\!\delta^{ij}\! \underbrace{(\pave_i F) (D_c^{-h}\! A_{\m j}^a) \pa_a \{\rho^2 D_c^h F\}}_{II}.
\end{equation}
Now,
\begin{multline}
 I = \int_U (D_c^h \pave_iF) \pave_j (\rho^2 D_c^h F)
 = \int_U \rho^2 \delta^{ij} (D_c^h \pave_i F)(\pave_j D_c^h F)
 + \int_U 2 \delta^{ij} (D_c^h \pave_i f)
 (D_c^h f)(\rho \pave_j \rho)\\
 = \int_U \rho^2 \delta^{ij} (D_c^h \pave_i F)(D_c^h\pave_j F)
 + \int_U  \delta^{ij} (D_c^h \pave_i F) \big\{ 2\rho(\pave_j \rho) (D_c^h F)
 - \rho^2 (D_c^h A_{\m j}^a)\pa_a F)\big\}.
\end{multline}
The first term is:
\begin{equation}
 \int_U \rho^2 |D_c^h \pave F|^2.
\end{equation}
The second term is bounded by:
\begin{equation}
  C(_{\!}M_{0\!})||\rho D_{\!c}^h {}_{\!}\pave F||_{L^{\!2}(U{}_{\!})}
  \big(||D_{\!c}^h\! F||_{L^{\!2}(U{}_{\!})} {}_{\!}+{}_{\!} ||D_{\!c}^h \! A||_{L^{\!\infty\!}(U{}_{\!})}
  ||f||_{H^{1\!}}{}_{\!}\big)\!
  \leq
\!  C(_{\!}M_{0\!}) ||\rho D_{\!c}^h {}_{\!}\pave F||_{L^{\!2}(U{}_{\!})}||f||_{H^1\!(\Omega)}(1 {}_{\!}+{}_{\!}
   ||\pa^2 \xve||_{L^{\!\infty\!}}).
\end{equation}
Next, writing $\pa_a D_c^h F = D_c^h \pa_a F = D_c^h(A_{\m a}^\ell \pave_\ell F)$:
\begin{multline}
  II = \int_U \delta^{ij}(\pave_j F)(D_c^{-h} \! A_{\m j}^a)\big\{
  2\rho \pa_a \rho D_c^h F
   + \rho^2 \pa_a D_c^h F\big\}\\
   =
   \int_U \delta^{ij} (\pave_j F)(D_c^{-h}\! A_{\m j}^a)
   \big\{ 2\rho \pa_a \rho D_c^h F +
   \rho^2 A_{\m a}^\ell D_c^h \pave_\ell F -
    \rho^2 (D_c^h A_{\m a}^\ell)\pave_\ell F\big\}
\end{multline}
so we have:
\begin{equation}
 |II| \leq C(M_0)||\pa^2 \xve||_{L^\infty(U)}||\pave F||_{L^2(U)}
 \big(||F||_{H^1(U)}
 +  ||\rho D_c^h \pave F||_{L^2(U)} + ||\pa^2 x||_{L^\infty(U)}
 ||\pave F||_{L^2(U)}\big).
\end{equation}

Finally, we have:
\begin{equation}
  \int_U |\Dve F| |v|\, dy
 \leq ||\Dve F||_{L^2(U)} ||D_c^{-h} (\rho^2 D_c^h F)||_{L^2(U)}.
\end{equation}
Using similar arguments to the above, we can show:
\begin{equation}
 ||D_c^{-h} (\rho^2 D_c^h F)||_{L^2(U)}
 \leq C(M')\big( ||\rho D_c^h \pave F||_{L^2(U)} +
 ||\pa^2 \xve||_{L^\infty(U)}||F||_{H^1}\big),
\end{equation}
so that:
\begin{equation}
 \int_U \rho^2 |D_c^h \pave F|^2
 \leq C(M_0) \big( ||\rho D_c^h \pave F||_{L^2(U)}
 \big\{(1 + ||\pa^2 x||_{L^\infty(U)})||F||_{H^1(U)} +
 ||g||_{L^2(U)}\big\}
 + ||f||_{H^1(U)}^2\big).
\end{equation}
Absorbing this first factor into the left-hand side we have, for any $h$
small enough:
\begin{equation}
 \int_V |D_c^h \pave F|^2 \leq
 \int_U \rho^2 |D_c^h \pave F|^2
 \leq C(M_0)\big( (1 + ||\pa^2 \xve||_{L^\infty(U)})^2 ||F||_{H^1(U)}^2
 + ||g||_{L^2(U)}^2\big),
\end{equation}
which implies the $s = 1$ case of the theorem.

Now suppose that $T^J \pave F \in L^2(\Omega)$ for all $|J| \leq s-1$.
Fix a multi-index $I$ with $|I| = s-1$ and write $F' = T^I F$.
Note that $F' = 0 $ on $\pa \Omega$ in the trace sense and also that:
\begin{equation}
 ||\pa_y F'||_{L^2(\Omega)}
 \leq C(M_0)\pave F'||_{L^2(\Omega)}
 \leq C(M_0)\big( ||T^I\pave F||_{L^2(\Omega)} + ||[\pave, \T^I]F||_{L^2(\Omega)}
 \big).
 {}
\end{equation}
The commutator can be bounded using Lemma \ref{vectcommlemma}:
\begin{equation}
 ||[\pave, T^I] F||_{L^2(\Omega)}
 \leq C(M_0, ||\xve||_{H^s(\Omega)}) \big(  ||\T \xve||_{H^s(\Omega)}  +
 ||\xve||_{H^s(\Omega)} \big)||F||_{H^{s-1}(\Omega)}.
 {}
\end{equation}
In particular this implies that $F' \!\!\in\! H^1_0(\Omega)$. We also have:
\begin{equation}
 \Dve F' = T^I \Dve F + [T^I,\Dve] F,
 \label{fprime}
\end{equation}
and
\begin{equation}
 ||[T^I, \Dve] F||_{L^2(\Omega)}
 \leq C(M_0, ||\xve||_{H^s(\Omega)}) ( ||\T \xve||_{H^s(\Omega)} + ||\xve||_{H^s(\Omega)}) ||F||_{H^s(\Omega)},
 {}
\end{equation}
Therefore we have that $F'\! \in\! H^1_0$ is the weak solution
to the problem \eqref{fprime} and $\Dve F'\! \in\! L^2(\Omega)$, so by the
$|I| \!=\! 1$ case we have $\T\! \pave F' \!\in\! L^2(\Omega)$ and:
\begin{equation}\label{eq:theaboveinequality}
 ||T \pave F'||_{L^2(\Omega)}
 \leq C(M_0)\big( ||\Dve F'||_{L^2(\Omega)} + ||\pave F'||_{L^2(\Omega)}\big).
 {}
\end{equation}
We write:
\begin{equation}
 T \pave_i F' = T (\pave_i \T^I F)
 = T T^I \pave_i F + (T T^I A_{\m i}^a)\pa_a F + R,
 {}
\end{equation}
where the $L^2$ norm of $R$ is bounded by the right side of \eqref{wkellloc}.
Combining this with \eqref{eq:theaboveinequality} gives \eqref{wkellloc}.
To prove the first estimate in \eqref{wkellloc} we argue in the same way,
but we also prove \eqref{diffquot} also for $c \!=\! 3$.
\end{proof}

\subsection*{Proof of Proposition \ref{div-curl}}

We will need a few preliminary results. First, we fix a function
$d$ with $d = 0$ on $\pa \Omega, d < 0 $ in $\Omega$ and
$|\nabla d| > 0$ everywhere, so that the normal can be written as:
\begin{equation}
 N_i = {\pave_i d}\, /|{\pave d}|
= {A_{\m i}^a \pa_ad}\,/|\wpa  d|,\qquad \text{where}\quad
|\wpa d|^2=\delta^{ij}\widetilde{\pa}_i d\, \widetilde{\pa}_j d\, =\gve^{ab} \pa_a d\,  \pa_b d.
 {}
\end{equation}
By \eqref{dinv}
and Lemma \ref{inverselemma}, this implies the estimates:
\begin{equation}
 ||N||_{C^\ell(\pa\Omega)} \leq C(M') ||\xve||_{C^{\ell+1}(\pa \Omega)}
 \leq C(M')||\xve||_{H^{\ell+4}(\Omega)},
 \quad \textrm{ and } \quad ||N||_{H^{\ell}(\pa \Omega)}
 \leq C(M') ||\xve||_{H^{\ell+1}(\pa \Omega)}.
 \label{Nbds}
\end{equation}
where in the first inequality we used Sobolev embedding on
$\pa \Omega$ and the
trace inequality \eqref{trace}.
Recalling the definition $\gamma_{ij} = \delta_{ij} - N_i N_j$, there
are similar estimates for derivatives of $\gamma$.

The basic result we need is the following consequence of Green's formula:
\begin{lemma}
 If $\alpha$ is a vector field then:
 \begin{equation}
   ||\pave \alpha||_{L^2(\tD_t)}^2   =
   ||\div \alpha||_{L^2(\tD_t)}^2 + \frac{1}{2} ||\curl \alpha||_{L^2(\tD_t)}^2
  +\int_{\pa \tD_t} \Big(\alpha^j (\gamma_j^k \pave_k \alpha_i) N^i
  - \alpha_i (\gamma_j^k \pave_k \alpha^j )N^i\Big).
 \end{equation}
\end{lemma}
\begin{proof}
  Integrating by parts:
  \begin{equation}
   ||\pave \alpha||_{L^2(\tD_t)}^2
   = -\int_{\tD_t} \delta^{ij}
   \alpha_i \Dve \alpha_j + \int_{\pa \tD_t}
   \delta^{ij} \alpha_i N^k \pave_k \alpha_j.
   \label{ibpapp1}
  \end{equation}
  We insert the identity:
  \begin{equation}
   \Delta \alpha_j = \delta^{k\ell}\pave_k(\pave_\ell \alpha_j)
   = \delta^{k\ell} \pave_k \big( \pave_j \alpha_\ell + \curl \alpha_{\ell j}\big)
   = \pave_j \div \alpha + \delta^{k\ell}\pave_k \curl \alpha_{\ell j},
  \end{equation}
into the first term in \eqref{ibpapp1}
  and integrate by parts again:
  \begin{equation}
   \int_{\tD_t}\delta^{ij} \alpha_i \Dve \alpha_j
   = \int_{\pa \tD_t} N^i \alpha_i \div \alpha
   + \delta^{ij} N^\ell \alpha_i \curl \alpha_{\ell j} dS
   - \int_{\tD_t} (\div \alpha)^2
   + \delta^{k\ell}\delta^{ij} \pave_k \alpha_i
   \curl \alpha_{\ell j}.
   \label{green1}
  \end{equation}
  Note that by the antisymmetry of curl:
  \begin{equation}
   \delta^{k\ell}\delta^{ij} \pave_k \alpha_i
   \curl \alpha_{\ell j}\!=\frac{1}{2} \delta^{k\ell}\delta^{ij} (\pave_k \alpha_i+\pave_i \alpha_k)
   \curl \alpha_{\ell j}+\frac{1}{2} \delta^{k\ell}\delta^{ij} (\pave_k \alpha_i-\pave_i \alpha_k)
   \curl \alpha_{\ell j}\!=\frac{1}{2} \delta^{k\ell}\delta^{ij}  \curl \alpha_{k i}
   \curl \alpha_{\ell j},
  \end{equation}
  so \eqref{ibpapp1} becomes:
  \begin{equation}
   ||\pave \alpha||_{L^2(\tD_t)}^2
   = ||\div \alpha||_{L^2(\tD_t)}^2 + \frac{1}{2} ||\curl \alpha||_{L^2(\tD_t)}^2
   +\int_{\pa \tD_t}
   N^k \alpha^j \pave_k \alpha_j -
   N^i \alpha_i \div \alpha
    - N^\ell \alpha^j
   \curl \alpha_{\ell j} .
   \end{equation}
   Here:
   \begin{multline}
    N^k \alpha^j \pave_k \alpha_j -
   N^i \alpha_i \div \alpha
    - N^\ell \alpha^j
   \curl \alpha_{\ell j}= N^k \alpha^j \pave_j \alpha_k -
   N^i \alpha_i \div \alpha\\
   =N^k \alpha_\ell N^\ell  N^j\pave_j \alpha_k +N^k \alpha_\ell \gamma^{\ell j}\pave_j \alpha_k -
   N^i \alpha_i( N^k N^\ell + \gamma^{\ell k})\pave_k \alpha_\ell
   =N^k \alpha_\ell \gamma^{\ell j}\pave_j \alpha_k -
   N^i \alpha_i \gamma^{\ell k}\pave_k \alpha_\ell.\tag*{\qedhere}
   \end{multline}
\end{proof}

\begin{lemma} There is a constant
  $C_{\!1}$ depending on $M'\!\!$ and $||\xve||_{H^5(\Omega)}$ so that
  if $\alpha$ is a vector field on $\Omega$ then
\begin{equation}
 \! ||\alpha||_{H^{1{}_{\!}}(\Omega)}^2
 {}_{\!}\leq\! C_{{}_{\!}1\!}\Big(\! ||\!\div \alpha||_{L^2(\Omega)}^2\!
 +{}_{\!} ||\!\curl \alpha||_{L^2(\Omega)}^2\!
 + {\sum}_{\mu=1}^N\!\int_{\pa \Omega}\!\!\! \big(\fdhm \alpha^i\big)
 \big(\fdhm \alpha^j\big) N_i N_j dS
 + ||\alpha||_{L^2(\pa \Omega)}^2\!
 + ||\alpha||_{L^2(\Omega)}^2\!\Big)
 \label{ftang1},
\end{equation}
\begin{equation}
 ||\alpha||_{H^{1{}_{\!}}(\Omega)}^2
 \!\leq\! C_1 \Big( ||\!\div \alpha||_{L^2(\Omega)}^2\!
 + ||\!\curl \alpha||_{L^2(\Omega)}^2
 +{\sum}_{\mu=1}^N\!\int_{\pa \Omega} \!\!\!(\fdhm \alpha^i)(\fdhm \alpha^j) \gamma_{ij} dS
 + ||\alpha||_{L^2(\pa \Omega)}^2\!
 + ||\alpha||_{L^2(\Omega)}^2\!\Big).
 \label{ftang2}
\end{equation}
\end{lemma}

\begin{proof}
  These estimates follow from the fact that for any $\epsilon > 0$:
 \begin{multline}
   \big|\int_{\pa \Omega}\!\!\!\!\! \alpha^j (\gamma_j^k \pave_k \alpha_i)
   N^i \!\!- \alpha_i(\gamma_j^k \pave_k \alpha^j)N^i\, dS \big|
   \\
   \leq C(M', ||\xve||_{H^5(\Omega)})\Big(
  \frac{1}{\epsilon} {\sum}_{\mu=1}^N||(\fd^{1/2}\alpha) \cdot N||_{L^2(\pa\Omega)}^2
  +  \epsilon {\sum}_{\mu=1}^N||(\fd^{1/2}\alpha)\cdot \gamma||_{L^2(\pa\Omega)}^2
  + ||\alpha||_{L^2(\Omega)}^2\Big).
  \label{goalft}
\end{multline}
To see that this estimate implies
\eqref{ftang1}, we use \eqref{Nbds} and the trace inequality \eqref{trace}
to control the second term
by $C(M') ||\alpha||_{H^{1/2}(\pa \Omega)} \leq C(M')||\alpha||_{H^1(\Omega)}$,
 and then take $\epsilon$ sufficiently small.
 The estimate \eqref{ftang2} follows
 by instead using the trace estimate on the first term
and taking $\epsilon$ sufficiently large.

  To prove \eqref{goalft}, we write $\gamma_j^k \pave_k \alpha^j\!\! =
  \gamma_j^k \pave_k(\gamma^j_\ell \alpha^\ell) -
  \gamma^k_j (\pave_k\gamma^j_\ell)\alpha^\ell
  -\gamma_j^k (\pave_k N^j) N_{\!\ell}\alpha^\ell$
  and  the left-hand side as:
  \begin{align}
   \int_{\pa \Omega}
   \alpha^j \gamma_j^k\pave_k(\alpha_i N^i)
   -
   \gamma_j^k (\pave_k(\gamma^j_\ell \alpha^\ell) ) \alpha_iN^i
   + \int_{\pa \Omega}
   \gamma^k_j (\pave_k\gamma^j_\ell)\alpha^\ell \alpha_i N^i
   +\gamma_j^k (\pave_k N^j) N_\ell\alpha^\ell \alpha_i N^i
   -\alpha^j \alpha_i \gamma_j^k \pave_k N^i.
   {}
  \end{align}
The  second integral is bounded by the right-hand
side of \eqref{goalft}, by \eqref{Nbds}. The first integral
is bounded by the right-hand side of \eqref{goalft} using the fractional
product rules \eqref{inthalf} and \eqref{app:alg2} - \eqref{leibpave}.
\end{proof}
\begin{proof}[Proof of Proposition \ref{div-curl}]
  By the previous lemma we have the result for $\ell = 1$. Assume that we
  have the result for $\ell = 1, ..., m-1$. To prove it for $\ell = m$,
  we write $\pa_y^m \alpha = \pa_y^{m-1} \pave \alpha + [\pa^{m-1},\pave]\alpha$.
  This second term can be bounded by the third term on the right-hand side of
  \eqref{fractang1} (resp. \eqref{fractang2}) by using Lemma \ref{inverselemma} and arguing
  as in the proof of Proposition \ref{app:sobdiff}.
  To control the first term, we apply \eqref{app:hotpw} and we need to control
  $||\pave \div \alpha||_{H^{m-2}(\Omega)},||\pave \curl \alpha||_{H^{m-2}(\Omega)}$
  and $||\T^J \pave \alpha||_{L^2(\Omega)}$ for all multi-indices with
  $|J| = m-1$. Writing $\pave = A\cdot \pa_y$ and arguing as above, the
  first two terms are bounded by the right-hand side of \eqref{fractang1}
  (resp. \eqref{fractang2}). It therefore just remains to control
  the third term. We commute $T^J$ with $\pave$, apply \eqref{app:dxtu}
  and again argue as in the proof of Proposition \ref{app:sobdiff}. Applying
  \eqref{ftang1} (resp. \eqref{ftang2}) and repeating the same
  argument as above completes the proof of Proposition
  \ref{div-curl}.
\end{proof}

\begin{proof}[Proof of Lemma \ref{crosstermlem}]
  It suffices to prove the claim for $f \in C_c^\infty(\Omega)$ by
  an approximation argument. Integrating by parts twice and
  using that $\pa_a \pave_i = \pave_i \pa_a -(\pa_a A_{\m i}^c)\pa_c$,
  we have:
  \begin{multline}
   (\Dve f, \Delta f)_{L^2(\Omega)}
   = \int_\Omega \delta^{ij} \delta^{ab}
   (\pave_i \pave_j f)(\pa_a \pa_b f)\\
   =
   \int_\Omega \delta^{ij} \delta^{ab}
 (\pa_a \pave_j f)(\pa_b  \pave_i f)
 + \int_\Omega \delta^{ij} \delta^{ab}
 (\pave_j f)(\pa_a A_{\m j}^c)(\pa_c \pa_b f)
 -\int_\Omega \delta^{ij} \delta^{ab}
 (\pa_a \pave_j f)(\pa_b A_{\m i}^d)(\pa_d f).
   {}
  \end{multline}
  This implies that:
  \begin{equation}
   (\Dve f, \Delta f)_{L^2(\Omega)}
   \geq C(M) \big( ||\pave f||_{H^1(\Omega)}^2 -
   ||\pave f||_{H^1(\Omega)} ||f||_{H^1(\Omega)}\big),
   {}
  \end{equation}
  and the result follows.
\end{proof}

\section{Proofs of Elliptic estimates for the Newton potential} \label{gravityproofs}
In this section we record the elliptic estimates that are needed to control $\phi$ in Section \ref{gravitysection}.
We will use the convention in \eqref{cprimes} for functions $C_s, C_s^\prime,
C_s^{\prime \prime}, C_s^{\prime\prime\prime}$ throughout this section.

\subsection{Estimates for Section \ref{extensionsec}}
\label{gravityproofs1}
Let $\hD_t$ be the extended fluid domain (see Section \ref{gravitysection}) and $\paht$ be the associated spatial derivative.
\begin{lemma} \label{ellphi} Suppose $r\geq 5$. If $\Dht f=g$ in $\hD_t$, then for $j\leq r-1$:
\begin{equation}
 ||\paht \T^j \paht f||_{L^2(\hD_t)}\leq C_r \bigtwo({\sum}_{k\leq j+1}||\T^{k}\paht f||_{L^2(\hD_t)}\!+{\sum}_{k\leq 2}||\T^k g||_{L^6(\hD_t)}
 +||\T^j\! g||_{L^2(\hD_t)}+||g||_{L^\infty(\hD_t)}\bigtwo).
  \label{ellphi 1}
 \end{equation}
 If $j\! \leq\! r\!-\!2$ then \eqref{ellphi 1} holds without the $L^6\!$ norms,
 and if $j \! \leq r\!-1$ it holds without the $L^\infty$ norm. In addition,
\begin{multline}
 \!\!\! ||\paht \T^r \paht f||_{L^2(\hD_t)}\leq C_r
 \bigtwo({\sum}_{k\leq r+1}||\T^{k}\paht f||_{L^2(\hD_t)}\!+||\T^r\! g||_{L^2(\hD_t)}\\
 +||\T\xve||_{H^r(\Omega)}\big[{\sum}_{k\leq 4}||\T^k \paht f||_{L^2(\hD_t)}+{\sum}_{k\leq 2}||\T^k g||_{L^6(\hD_t)}+||g||_{L^\infty(\hD_t)}\big]\bigtwo).
  \label{ellphi 2}
 \end{multline}
  Moreover, for $0\leq \ell \leq 2$,
 \begin{equation}
||\paht \T^{\ell} \paht f||_{L^6(\hD_t)} \leq C_r\bigtwo({\sum}_{k\leq \ell}||\T^k g||_{L^6(\hD_t)}+{\sum}_{k\leq \ell+2}||\T^k \paht f||_{L^2(\hD_t)}\bigtwo),
\label{L^6 phi est}
\end{equation}
as well as
\begin{equation}
||\paht^2 f||_{L^\infty(\hD_t)}\leq C_0\big(||g||_{L^\infty(\hD_t)}+{\sum}_{|J|\leq 1}||\paht^J\T \paht f||_{L^6(\hD_t)}\big).
\label{Linfty f}
\end{equation}
 The above estimates also hold in the domain $\tD_t$ with $\pave$ instead of $\paht$.
 \end{lemma}
 \begin{proof}
   The estimate \eqref{Linfty f} follows from
   the pointwise estimate \eqref{ellpw} and Sobolev embedding:
  \begin{align}
  ||\paht^2 f||_{L^\infty(\hD_t)} \leq C_0\big( ||g||_{L^\infty(\hD_t)}+||\T \paht f||_{L^\infty(\hD_t)}
  \big) \leq
  C_0\big( ||g||_{L^\infty(\hD_t)}+{\sum}_{|J|\leq 1}||\paht^J \T \paht f||_{L^6(\hD_t)}\big).
  \end{align}
By \eqref{ellpw} we also have:
\begin{multline}
||\paht \T^{\ell} \paht f||_{L^6(\hD_t)} \leq
C_0\big(||\div \T^{\ell} \paht f||_{L^6(\hD_t)}+||\curl \T^{\ell} \paht f||_{L^6(\hD_t)}+||\T^{1+\ell}\paht f||_{L^6(\hD_t)}
\big)\\
\leq C_0\big(||\T^{\ell} g||_{L^6(\hD_t)}
+{\sum}_{k_1+k_2=\ell-1}||(\T^{1+k_1} \uht)(\paht \T^{k_2}\paht f)||_{L^6(\hD_t)}+|| \T^{\ell+1}\paht f||_{H^1(\hD_t)}
\big),
\end{multline}
where the sum is not there if $\ell\! =\! 0$. Putting $\T^{1+k_1}\!\! \uht$
into ${L}^{\!\infty}\!$ and using induction, this implies that for $\ell \!\leq\! 2$:
\begin{equation}
  ||\paht \T^{\ell} \paht f||_{L^6(\hD_t)} \leq
  C_0\big(||\T^{\ell} g||_{L^6(\hD_t)}+{\sum}_{\ell' \leq \ell} ||\T^{\ell'+1}\paht f||_{H^1(\hD_t)}
  \big).
 \label{L^6 step 1}
\end{equation}

We now prove \eqref{ellphi 1}, which, combined with \eqref{L^6 step 1} will
also prove \eqref{L^6 phi est}.
 We proceed by induction: for $j\!=\!0$, \eqref{ellphi 2}
 without the $L^6$ and $L^\infty$ norms is a direct consequence of \eqref{ellpw}.
  Now suppose that \eqref{ellphi 2} is known for $j=0,1,\cdots, m-1 \leq r-1$.
  Using the pointwise estimate \eqref{ellpw} we have:
 \begin{align}
 ||\paht \T^m \paht f||_{L^2(\hD_t)}\leq
 C_0\big(||\div \T^m \paht f||_{L^2(\hD_t)}+||\curl \T^m \paht f||_{L^2(\hD_t)}
 +||\T^{m+1}\paht f||_{L^2(\hD_t)}\big).
\end{align}
Here $\div$ and $\curl$ stand for the divergence and curl with respect to $\paht$.
Since $\div \T^m \paht_{\!} f\!=\!\T^m \!g+\sum (\T^k \!\!\uht)\paht \T^\ell \paht_{\!} f$,
where $\uht=(\uht_{\m i}^a)$ and the sum is over $k+\ell=m$ with $k\geq 1$, we have
\begin{equation}
\div \T^m \paht f = \T^m g+\sum (\T^{k_1}\pa \xht)\cdots(\T^{k_s}\pa \xht)(\paht \T^\ell \paht f).
\label{expanddiv}
\end{equation}
The above sum is over $k_1+\cdots+k_s+\ell=k+\ell=m, k \geq 1$ , and $\pa$ denotes the
Lagrangian spatial derivative $\pa_y$. This is because $\T^k\uht$ is a sum of
terms of the form $(\T^{k_1}\pa \xht)\cdots(\T^{k_s}\pa \xht)$. Now, we need to
control  $\sum (\T^{k_1}\pa \xht)\cdots(\T^{k_s}\pa \xht)(\paht \T^\ell \paht f)$
in $L^{\!2}(\hD_{t\!})$. When $\ell\!\geq\! 3$, then $k_1,\cdots\!, k_s\!\leq\! r\!-\!3$, so all
terms involving $\xht$ can be controlled in
$L^{\!\infty}$ by $||\xht||_{H^r(\Odo)}$ and we control
$||\paht \T^\ell \paht \!f||_{L^{\!2}(\hD_{\!t\!})}\!$ by the inductive assumption
since $\ell \!\leq \!m\!-\!1$.

We now consider the case that at least one of
$k_1,\cdots, k_s\geq  r-2$ so that $\ell \leq 2$.
Since $r\geq 5$, at most one of the $k_j$, say
$k_1$, can be greater than or equal to $r-2$. If $k_1=r-2$ or $k_1=r-1$,
then by Sobolev embedding we
control $||\T^{k_1}\pa \xht||_{L^{3}(\Omega^{d_0})} \leq C
 ||\T\xht||_{H^{(r-1, 1/2)}(\Omega^{d_0})}$, and the other terms involving
 $\xht$ can be controlled in $L^{\infty}$ and hence by $||\xht||_{H^{r-1}
 (\Odo)}$. Using the estimate \eqref{L^6 step 1}, the inductive assumption
 and H\"{o}lder's inequality $||f_1f_2||_{L^2}\! \leq\! ||f_1||_{L^6}||f_2||_{L^3}$,
 we control the $L^2$ norm of right-hand side of \eqref{expanddiv}
 by the right-hand side of \eqref{ellphi 1}.

The only remaining case is when $k_1=r$, and to deal with this
we bound $\T^{r}\pa \xht$ in $L^2$ and use \eqref{Linfty f} to bound
the $L^\infty$ norm of the term involving $f$, which gives:
 \begin{multline}
 ||\!\div \!\T^{m\!} \paht f||_{L^2(\hD_t\!)}\\
\!\! \leq\! C_r\bigtwo(\! ||\T^{m\!} g||_{L^2(\hD_t\!)}\!+\!\!{\sum}_{\ell=0,1,2}||\T^\ell \! g||_{L^6(\hD_t\!)}\!+\!||g||_{{L}^{\!\infty}(\hD_t\!)}
 \!+\!(||\T\xht||_{H^{r\!}(\Omega^{d_0}\!)} \!+\! 1){\sum}_{k\leq m-\!1}||\paht \T^k \paht f||_{L^2(\hD_t\!)}\!\bigtwo).
 \end{multline}
 By the inductive assumption,
   $||\!\div \!\T^m \paht f||_{L^2(\hD_t)}$ is  controlled by the right-hand
  side of \eqref{ellphi 2}. A similar argument shows that
  $||\!\curl\! \T^m \paht f||_{L^2(\hD_t)}$ is bounded
  by the right-hand side of \eqref{ellphi 1} (resp. \eqref{ellphi 2})
  with $s\! =\! m$ but with $||\xve||_{H^m(\Omega)},
  ||\T\xve||_{H^m(\Omega)}$ replaced by $||\xht||_{H^m(\Odo)},
  ||\T\xht||_{H^m(\Odo)}$. Using \eqref{equiv extended norm} completes the proof.
 \end{proof}
 We also need the following estimate for the Newton potential:
  \begin{lemma}
 \label{bdy estimate for D_tphi_m} If $g$ is a smooth function supported  in $\xht(t, \Omega^{d_0\!/2})$, then there is a constant $C$ with:
 \begin{equation}
 |\paht^s (g * \Phi)(x)| \leq  C||g||_{L^2(\hD_t)},\qquad x\in \pa\hD_t,\qquad  s\geq 0.
 \end{equation}
 \end{lemma}
 \begin{proof}
 Since there exists $c_0>0$ such that $d(\xht(t, \Omega^{d_0/2}), \pa \hD_t)\geq c_0$, we have that $d(x, z)\geq c_0$ for each $z\in \text{supp}(g)\subset\xht(t, \Omega^{d_0/2}) $,  and so $\paht^s \Phi(x-\cdot)\in L^2(\xht(t, \Omega^{d_0/2}))$. Therefore,
 \begin{equation}
 |\paht^s (g*\Phi)| \leq ||g||_{L^2(\hD_t)}||\paht^s\Phi (x-z)||_{L^2_z(\xht(t, \Omega^{d_0/2})} \leq C||g||_{L^2(\hD_t)}.\tag*{\qedhere}
 \end{equation}
 \end{proof}

\begin{proof}[Proof of Theorem \ref{ellphi'' 1}]
We proceed by induction. Write $f = g * \Phi$.
 When $j=0$ we have
 \begin{equation}
 ||\paht f||_{L^2(\hD_t)}^2=\int_{\hD_t}\delta^{ij}(\paht_i f)\cdot(\paht_j f)\,dx
 =\int_{\pa \hD_t} N^i (\paht_i f)f\,dS(x)-\int_{\hD_t}gf\,dx.
 \label{j=0}
 \end{equation}
 By Lemma \ref{bdy estimate for D_tphi_m},
the boundary integral in \eqref{j=0} is bounded by $C||g||_{L^2(\hD_t)}^2$.
The second term in \eqref{j=0} is bounded by $||g||_{L^2(\hD_t)}||f||_{L^2(\hD_t)}$,
and by Young's inequality:
 \begin{align*}
 ||f||_{L^2(\hD_t)}=||g*\Phi||_{L^2(\hD_t)}\leq C||g||_{L^{2}(\hD_t)}
 ||\Phi||_{L^{1}(\hD_t)}\leq C||g||_{L^2(\hD_t)},
 \end{align*}
 By \eqref{j=0}, this implies:
 \begin{align}
 ||\paht f||_{L^2(\hD_t)} \leq C||g||_{L^2(\hD_t)}.
 \end{align}

Suppose that we now know that
$||\T^j \paht g||_{L^2(\tD_t)}$ is bounded
by the right-hand side of \eqref{ellphi'' 1} for $j= 0,\cdots, m-1 \leq r-1$.
To prove that it holds for $j = m$ as well, we integrate by parts:
\begin{multline}
||\T^m\paht f||_{L^2(\hD_t)}^2 = \int_{\hD_t}(\T^m\paht_i f)(\T^m\paht^i f)\,dx\\
=\!\int_{\hD_t}\underbrace{\!\!\!\delta^{ij} (\T^m\paht_i f)\paht_j\T^m \!f dx}_{I}
-\!\int_{\hD_t}\underbrace{ \!\!\!\delta^{ij}(\T^m\paht_i f)
(\paht_j\T^m \xht)\paht f dx}_{II}+\!\sum\! \int_{\hD_t}\underbrace{\!\!\!(\T^m \paht f)(\pa \T^{\ell_1}\w x)\cdots(\pa \T^{\ell_{s-1}}\w x)
\T^{\ell_s}\paht f dx}_{III},
\label{errorphi 1}
\end{multline}
where the sum is over
$\ell_1+\cdots+\ell_s=m$ and $\ell_1,\cdots, \ell_s\leq m-1$, $\ell_1\geq 1$.
To control $III$, we note that if $\ell_1,\cdots, \ell_{s-1}\leq r-3$, then
 $III\leq C(||\xve||_{H^{r-1}(\Omega)})||\T^m \paht f||_{L^2(\hD_t)}
 ||\T^{\ell_s}\paht f||_{L^2(\hD_t)}$, and we control $||\T^{\ell_s}\paht f||_{L^2(\D_t)}$
  by the inductive assumption. On the other hand, since $r \geq 5$, there can be at most
  one $j$ with $\ell_j \geq r-2$ and without loss of generality it is
  $\ell_1$ in which case $\ell_j\leq 2$
  for $j=2,3,\cdots, s$. We then bound $III\leq C(||\xve||_{H^{r-1}(\Omega)})
  ||\T^m \paht f||_{L^2(\hD_t)}||\pa \T^{\ell_1}\xht||_{L^3(\hD_t)}
  ||\T^{\ell_s}\paht f||_{L^6(\hD_t)}$. By Sobolev embedding,
   $||\pa \T^{\ell_1}\xht||_{L^3(\hD_t)}
  \leq C||\T \xve||_{H^{(r-1, 1/2)}(\Omega)} $, and $||\T^{\ell_s}\paht f
  ||_{L^6(\hD_t)}$ can be controlled using Lemma \ref{ellphi}.

To control $I+II$, we integrate by parts and get
\begin{equation}
I+II\!=-\!\int_{\hD_t}\!\!\!{}_{\!}\delta^{ij\!}\!\underbrace{(\paht_i \T^{m\!}\paht_j f)\T^m\! f dx}_{I_1}
+\!\int_{\hD_t}\!\!\!{}_{\!}\delta^{ij\!}\!\underbrace{(\paht_i\T^{m\!}\paht_j f)(\T^m \xht^k)\paht_k f dx}_{II_1}
+\!\int_{\hD_t}\!\!\!{}_{\!}\delta^{ij\!}\!\underbrace{(\T^m\paht_i f)(\T^m \xht^k)\paht_j\paht_k f dx}_{II_2}
+\mathcal{B},\label{errorphi 3}
\end{equation}
where
\begin{equation}
\mathcal{B}=\int_{\pa\hD_t}(N^i \T^m\paht_if)\big(\T^m f-(\T^m\xht^k)(\paht_k f)\big).
\label{errorphibound}
\end{equation}
To control $II_1$, we have:
\begin{equation}
\delta^{ij}\paht_i\T^m\paht_j f=
\T^m \Delta f + (\pa \T^m \xht)(\paht^2 f)+\sum (\pa \T^{\ell_1} \xht)\cdots(\pa
 \T^{\ell_{s-1}} \xht)(\paht \T^{\ell_s}\paht f), \label{pa T^r pa f}
\end{equation}
where the sum is over $\ell_1+\cdots+\ell_s=m$ and $\ell_1,\cdots, \ell_s\leq m-1$.
The terms in the sum can be controlled similarly to how we controlled the sum
 in \eqref{expanddiv}. The two main terms that are left in $II_1$ are
\begin{equation}
\int_{\hD_t}(\T^m g)(\T^{m} \xht)(\paht f)\,dx+\int_{\hD_t}(\pa \T^m \xht)
(\paht^2f)(\T^m \xht)(\paht f)\,dx. \label{THE ONLY}
\end{equation}
To control the second term in \eqref{THE ONLY}, we commute one $\T$ to the outside
which gives:
\begin{equation}
\!\int_{\hD_t}\underbrace{ (\T\pa \T^{m-1} \xht)(\paht^2 f)(\T^{m} \xht)(\paht f)\,dx}_{II_{11}}
+\!\int_{\hD_t}\underbrace{ (\pa^2 \xht)(\pa \T^{m-1} \xht)(\paht^2f)(\T^{m} \xht)
(\paht f)\,dx}_{II_{12}}.
\label{II11 II12}
\end{equation}
To control $II_{11}$, we integrate half a tangential derivative by parts using
\eqref{inthalf} and get:
\begin{align}
II_{11} \leq C||\pa \T^{m-1} \xht||_{H^{(0,1/2)}(\Omega)}
||(\paht^2 f)(\T^m \xht)(\paht f)||_{H^{(0,1/2)}(\hD_t)}.
\end{align}
Using the fractional product rule \eqref{app:alg2}, for each $\mu$ we have
with $L^2=L^2(\hD_t)$
\begin{equation}
||\fdhm \big((\paht^2\! f)(\T^m \xht)\paht  f\big)||_{L^2} \!\leq C
||(\paht^2 \! f)(\fdhm\T^m \xht)\paht f||_{L^2}\!
+||\T^m \xht||_{L^2(\Omega)} {\sum}_{\ell\leq 2}||\T^{\ell\!}\big((\paht^2\! f)
\paht f\big)||_{L^2}.
\label{j1=r}
\end{equation}
The first term on the right hand side can be controlled by
 $||\T \xve||_{H^{(m-1, 1/2)}(\Omega)}||\paht^2 f||_{L^\infty(\hD_t)}
 ||\paht f||_{L^\infty(\hD_t)}$. Using  \eqref{Linfty f},
 the Sobolev inequality $||\paht f||_{L^\infty(\hD_t)}\leq
  C{\sum}_{|I|\leq 1}||\paht^{I+1} f||_{L^6(\hD_t)}$ and
  \eqref{L^6 phi est}, we control this term.
  To control the second term in \eqref{j1=r}, we just show how to control
   $||(\T^\ell \paht^2 f)(\paht f)||_{L^2(\hD_t)}$ for $\ell \leq 2$ since
   the remaining terms are similar. For $\ell \leq 2$ we have:
\begin{equation}
||\T^\ell \paht^2 f||_{L^6(\hD_t)} \leq
||\paht \T^\ell \paht f||_{L^6(\hD_t)}+{\sum}_{\ell\leq 2}
{\sum}_{j_1+j_2=\ell, j_1\geq 1}||(\T^{j_1} \uht)(\paht\T^{j_2}\paht f)||_{L^6(\hD_t)}.
\end{equation}
By \eqref{L^6 phi est} we control the first term here, and after bounding the
term involving $\uht$ in $L^\infty$ and using \eqref{L^6 phi est} again we also control
the second term.
To control the term $II_{12}$ from \eqref{II11 II12}, we have:
\begin{equation}
II_{12}\leq P(||\T\xve||_{H^{r-1}(\Omega)})||\paht^2f||_{L^\infty(\hD_t)}||\paht f||_{L^\infty(\hD_t)},
\end{equation}
 and then use \eqref{Linfty f}.
 To control the first term in \eqref{THE ONLY} we use \eqref{inthalf}
 and then bound:
\begin{equation}
||\fdhm \big((\T^m\xht)(\paht f)\big)||_{L^2}
\leq C
\big(||(\fdhm\T^m\xht)(\paht f)||_{L^2}\!+\!
||\fdhm\big((\T^m\xht)(\paht f)\big)\!-\!(\fdhm\T^m\xht)(\paht f)||_{L^2}\big),
\end{equation}
where $L^2=L^2(\hD_t)$,
and then:
\begin{equation}
||(\fdhm \T^m\xht)(\paht f)||_{L^2(\hD_t)}\leq ||\T \xve||_{H^{(r-1, 1/2)}(\Omega)}
||\paht f||_{L^\infty(\hD_t)}\leq ||\T \xve||_{H^{(r-1, 1/2)}(\Omega)}
{\sum}_{\ell\leq 1}||\paht^{\ell+1}f||_{L^6(\hD_t)},
\end{equation}
and by \eqref{app:alg2},
\begin{equation}
||\fdhm \big((\T^m\xht)(\paht f)\big)-(\fdhm \T^mx)(\paht f)||_{L^2(\hD_t)} \leq
C||\T^m \xve||_{L^2(\Omega)}{\sum}_{\ell\leq 2}||\T^{\ell}\paht f||_{L^2(\hD_t)}.
\end{equation}

To control $II_2$ in \eqref{errorphi 3}, we have
\begin{equation}
II_2 \leq ||\T^m \paht f||_{L^2(\hD_t)}||\T^m \xht||_{L^3(\D_t)}||\paht^2 f||_{L^6(\D_t)},
\end{equation}
and $||\paht^2 f||_{L^6(\hD_t)}$ under control, using \eqref{L^6 phi est}.

To control $I_1$, we substitute \eqref{pa T^r pa f} into $I_1$ and get,
to highest order:
\begin{equation}
\int_{\hD_t} (\T^{m} g)(\T^{m} f)+\int_{\hD_t}(\T\pa \T^{m-1} \xht)(\paht^2 f)(\T^{m} f).
\label{I_1}
\end{equation}
We write $\T = \T^a \pa_{y^a}=\T^a \uht_a^i \paht_i$, so that:
\begin{equation}
\T^m \!f\!= (\T^a \uht_a^i \paht_i)\T^{m\!-\!1}\!f\!=\T^a \! \uht_a^i\Big(\T^{m\!-\!1}\paht_i f+
(\paht_i\T^{m\!-\!1}\xht)(\paht f)+{\sum}_{\ell_1\!+\cdots+\ell_s\leq m\!-\!2}
(\pa \T^{\ell_1}\xht)\cdots(\pa\T^{\ell_{s-1}}\xht)\T^{\ell_s}\paht f\Big).
\end{equation}
Substituting this into \eqref{I_1}, to highest order the result is:
\begin{multline}
\int_{\hD_t}(\T^m g)(\T^a\uht_a^i)(\T^{m-1}\paht_i f)+
\int_{\hD_t}(\T^m g)(\T^a\uht_a^i)(\paht_i \T^{m-1} \xht)(\paht f)\\
+\int_{\hD_t}(\T \pa \T^{m-1} \xht)(\paht^2 f)(\T^a \uht^i_a)(\T^{m-1}\paht_i f)
+\int_{\hD_t}(\T \pa \T^{m-1} \xht)(\paht^2 f)(\T^a \uht^i_a)(\paht_i \T^{m-1}
\xht)(\paht f).
\end{multline}
 The first and third terms can be controlled after integrating $\T$ by parts
 and using H\"{o}lder's inequality. The other terms can be controlled after
 integrating half a tangential derivative by parts using \eqref{inthalf} and
  \eqref{app:alg2}.

Finally, to control the boundary term $\mathcal{B}$ in \eqref{errorphibound},
we use Lemma \ref{bdy estimate for D_tphi_m}:
\begin{equation}
\mathcal{B}  = \int_{\pa\hD_t}\!\!\!(N^i \T^r\paht_if)\Big(\T^r f-(\T^r\xht)(\paht f)\Big)
\leq C\Big(||g||_{L^2(\hD_t)}^2+ ||g||_{L^2(\hD_t)}^2\int_{\pa\hD_t}\!\!\!|\T^{r}\xht|\Big),
\end{equation}
which is controlled by
$
C(||\T\xve||_{H^{(r-1, 0.5)}(\Omega)} + 1)||g||_{L^2(\hD_t)}
$
by the trace lemma \eqref{trace} and Theorem \ref{comparable thm}.
\end{proof}

\subsection{Estimates for Section \ref{section 7.2}}
\label{gravityproofs2}
 Let $\FD^r$ be the mixed tangential space and time derivative defined in Section \ref{def T and FD}. We have:
 \begin{lemma} \label{Dtellphi} Suppose that $r\geq 5$. If $\Dht f=g$ in $\hD_t$, then for $j\leq r-1$:
\begin{equation}
||\paht \mathfrak{D}^j \paht f||_{L^2(\hD_t)}\leq C^\prime_r
 \bigtwo({\sum}_{k=0}^{j+1}||\mathfrak{D}^k\paht f||_{L^2(\hD_t)}\!+{\sum}_{k\leq 2}||\mathfrak{D}^k g||_{L^6(\hD_t)}+||\mathfrak{D}^j\! g||_{L^2(\hD_t)}+||g||_{L^\infty(\hD_t)}\bigtwo).
 \end{equation}
  In addition, we have:
\begin{multline}
 ||\paht \mathfrak{D}^r \paht f||_{L^2(\hD_t)}\leq P\big(||\xve||_{H^{r}(\Omega)}, {\tsum}_{k\leq r-1}||D_t^k\ssm V||_{H^{r-k}(\Omega)}\big)\\
 \cdot\bigtwo({\sum}_{k=0}^{r+1}||\mathfrak{D}^k\paht f||_{L^2(\hD_t\!)}\!+||\mathfrak{D}^r\! g||_{L^2(\hD_t\!)}+||\T \xve||_{H^{r\!}(\Omega)}\big[{\sum}_{k\leq 4}||\FD^k \paht f||_{L^2(\hD_t\!)}+\!{\sum}_{k\leq 2}||\mathfrak{D}^k \! g||_{L^6(\hD_t\!)}+||g||_{L^\infty(\hD_t\!)}\big]\bigtwo).
 \end{multline}
 Moreover, for $0\leq \ell \leq 2$,
 \begin{equation}
||\paht \mathfrak{D}^\ell \paht f||_{L^6(\hD_t)} \leq C_r^\prime
\bigtwo({\sum}_{k\leq \ell}||\mathfrak{D}^k g||_{L^6(\hD_t)}+{\sum}_{k\leq \ell+2}||\mathfrak{D}^k \paht f||_{L^2(\hD_t)}\bigtwo).
 \label{Dt ellphi 2}
\end{equation}
\end{lemma}
\begin{proof}
 It suffices to prove
 \begin{multline}
||\paht \mathfrak{D}^{r-1}D_t \paht f||_{L^2(\hD_t)}\leq C\big(||\xht||_{H^{r}(\Omega^{d_0})}, {\tsum}_{k\leq r-1}||D_t^k\widehat{V}||_{H^{r-k}(\Omega^{d_0})}\big)
 \bigtwo({\sum}_{k=0}^{r+1}||\mathfrak{D}^k\paht f||_{L^2(\hD_t)}\!+||\mathfrak{D}^r\! g||_{L^2(\hD_t)}\\
 +||\T \xve||_{H^{(r-1, 0.5)}(\Odo)}\big[{\sum}_{k\leq 4}||\FD^k \paht f||_{L^2(\hD_t)}+{\sum}_{k\leq 2}||\mathfrak{D}^k g||_{L^6(\hD_t)}+||g||_{L^\infty(\hD_t)}\big]\bigtwo).
 \label{Dt ellphi core}
\end{multline}
because \eqref{Dt ellphi 2} will then follow from this estimate and Lemma \ref{ellphi}.
Suppose that \eqref{Dt ellphi core} is known for $||\paht \mathfrak{D}^{r-1}D_t \paht f||_{L^2(\hD_t)}$ with $j=1,\cdots, r-2$, then for $j=r-1$, we have
 \begin{equation}
 ||\paht \FD^{r-1}D_t \paht f||_{L^2(\hD_t)}\leq  ||\div \FD^{r-1} D_t \paht f||_{L^2(\hD_t)}\\
 + ||\curl \FD^{r-1}D_t \paht f||_{L^2(\hD_t)}+ ||\T \FD^{r-1}D_t\paht f||_{L^2(\hD_t)}.
\end{equation}
Here $\div$ and $\curl$ stand for the divergence and curl with respect to $\paht$. We only need to control the div term, because the curl term can be treated similarly. Since $\div \FD^{r-1} D_t \paht_{\!} f\!=\!\FD^{r-1}D_t g+\sum (\FD^k \uht)(\paht \FD^\ell \paht_{\!} f)$, where $\uht=(\uht_{\m i}^a)$ and the sum is over $k+\ell=r$ such that $k\geq 1$, we have
\begin{equation}
\div \FD^{r-1} D_t \paht f = \FD^{r-1} D_t g+\sum (\FD^{k_1}\pa \xht)\cdots(\FD^{k_s}\pa \xht)(\paht \FD^\ell \paht f).
\label{controllable error}
\end{equation}
The above sum is over $k_1+\cdots+k_s+\ell=k+\ell=r$, which needs to be controlled
in $L^2(\hD_t)$.   If $\ell\geq 3$, then $k_1,\cdots, k_s\leq r-3$, and so all
terms involving $\xht$ can then be controlled in $L^\infty$ by either
$||\xht||_{H^r(\Omega^{d_0})}$ or ${\sum}_{k\leq r-3}||D_t^k \Vht||_{H^{r-k}(\Omega^{d_0})}$.
Furthermore, when at least one of  $k_1,\cdots, k_s\geq  r-2$, since $r\geq 5$, there is at most
one term, say $k_1$, can be greater than or equal to $r-2$. If $k_1=r-2$ or
$k_1=r-1$,  we control $||\FD^{k_1}\pa \xht||_{L^{3}(\Omega^{d_0})}$ by either
$||\T\xht||_{H^{(r-1, 0.5)}(\Odo)}$ or ${\sum}_{k\leq r-2}||D_t^k \Vht||_{H^{r-k}(\Odo)}$,
and other terms involving $\xht$ are of lower order. In addition to this, we control
$\paht \FD^\ell \paht f$ for $\ell\leq 2$ in $L^6$ because by the pointwise inequality
\eqref{ellpw} we have:
\begin{multline}
||\paht \FD^\ell \paht f||_{L^6(\hD_t)} \leq C(M)\big(
||\div \FD^\ell \paht f||_{L^6(\hD_t)}+
||\curl \FD^\ell \paht f||_{L^6(\hD_t)}+|
|\T\FD^\ell\paht f||_{L^6(\hD_t)}\big)\\
\leq C(M)\bigtwo(||\FD^\ell g||_{L^6(\hD_t)}
+{\sum}_{\ell_1+\ell_2=\ell, \ell_1\geq 1}||(\FD^{\ell_1} \uht)(\paht
\FD^{\ell_2}\paht f)||_{L^6(\hD_t)}+||\T\FD^\ell\paht f||_{H^1(\hD_t)}\bigtwo),
\end{multline}
where the second term is not present if $\ell = 0$. The second and third terms
can be bounded by the right-hand side of \eqref{Dt ellphi 2} by the
inductive assumption.
On the other hand, when $k_1\!=\!r$, $\FD^{k_1}\!$ involves at least one $D_t$, and so
we control $\FD^{k_1}\!\pa \xht$ in $L^2$ by ${\sum}_{k\leq r-\!1}
||D_t^k \Vht||_{H^{r-k}(\Odo)}$. We also control $\paht^2 \!f$ in
$L^{\!\infty}\!\!$, as in Lemma \ref{ellphi}.
\end{proof}

\begin{lemma} \label{D_t^r f lem}
Fix $r \geq 7$. If $g$ is a smooth function such that $\text{supp}(g)\subset \xht(t, \Omega^{d_0/2})$, then:
\begin{equation}
||D_{\!t}^r \!(g * \Phi)||_{L^{\!2}(\hD_t)} \leq C_r' \bigtwo({\sum}_{k\leq r\!-1}||\FD^{k} \paht f||_{L^{\!2}(\hD_t)}+{\sum}_{k\leq r}||\FD^k \! g||_{L^{\!2}(\hD_t)}+{\sum}_{k\leq 2}||\FD^k \! g||_{L^6(\hD_t)}+||g||_{L^\infty(\hD_t)}\!\bigtwo).
\end{equation}
\end{lemma}
\begin{proof}
 Since $\widehat{\Delta} f = g$ in $\hD_t$, commuting $D_t^r$ through this and get
\begin{equation}
 \widehat{\Delta}D_t^r f = (D_t^rg)+ [\widehat{\Delta}, D_t^r] f.
 \label{lap D_t^s f}
\end{equation}
In addition, since $D_t=\pa_t+\Vht^k\paht_k$ in $\hD_t$, we have $[\paht, D_t]=\paht \Vht\cdot \paht$, which can then be used to compute
\begin{multline}
[\Dht, D_t^r] = {\sum}_{\ell_1+\ell_2=r-1}c_{\ell_1,\ell_2} (\Dht D_t^{\ell_1} \Vht)\cdot \paht D_t^{\ell_2}+ {\sum}_{\ell_1+\ell_2=r-1}c_{\ell_1,\ell_2} (\paht D_t^{\ell_1} \Vht)\cdot \paht D_t^{\ell_2}\paht\\
+ {\sum}_{\ell_1+\cdots+\ell_n= r-n+1, \, n\geq 3} d_{\ell_1,\cdots,\ell_n}(\paht D_t^{\ell_3} \Vht)\cdots (\paht D_t^{\ell_{n}} \Vht)\cdot (\paht^2 D_t^{\ell_1} \Vht)\cdot D_t^{\ell_2}\paht\\
+{\sum}_{\ell_1+\cdots+ \ell_n= r-n+1, \, n\geq 3} e_{\ell_1,\cdots,\ell_n}(\paht D_t^{\ell_3} \Vht)\cdots (\paht D_t^{\ell_n} \Vht)\cdot (\paht D_t^{\ell_1} \Vht)\cdot \paht D_t^{\ell_2}\paht .
\label{com Dht D_t^k}
\end{multline}
Since $\xht(t,y)\!=\!x_0(y)$ in $\Odo\!\setminus\!\Omega^{d_0/2}$,
$[\Dht, D_t^r]f$ is compactly supported in $\xht(t,\Omega^{d_0/2})$. Therefore, \eqref{lap D_t^s f} yields:
\begin{equation}
D_t^r f = (D_t^rg)*\Phi+ ([\widehat{\Delta}, D_t^r] f)*\Phi.
\label{D_t^s f}
\end{equation}
The first term on the right can be controlled by $ C(\text{Vol}(\hD_t))||D_t^r g||_{L^2(\hD_t)}$ using Young's inequality. In addition,
by \eqref{com Dht D_t^k}, to control the $L^2(\hD_t)$ norm of the second term it suffices to consider
\begin{equation}
||[(\paht^2 \!D_t^{\ell_{\!1\!}} \Vht)_{\!}\cdots_{\!} (\paht D_t^{\ell_{\!n\!-\!1\!}} \Vht)\!\cdot\!D_t^{\ell_{\!n\!}}\paht f]\!*\!\Phi||_{L^{\!2}(\hD_t\!)}\qquad \text{and}\qquad ||[(\paht D_t^{\ell_{\!1\!}} \Vht)_{\!}\cdots_{\!} (\paht D_t^{\ell_{\!n\!-\!1\!}} \Vht)\!\cdot\!\paht D_t^{\ell_{\!n\!}}\paht f]\!*\!\Phi||_{L^{\!2}(\hD_t\!)},
\label{products}
\end{equation}
where $\ell_1+\cdots+\ell_n=r+1-n$ and $n\geq 2$. For the first term in
 \eqref{products}, when $\ell_n\geq 3$, we must have $\ell_j\leq r-4$ for $j\leq n-1$.
  In this case, we bound the $\Vht$ terms in $L^\infty(\hD_t)$  and then use the
   Sobolev lemma:
\begin{equation}
||[(\paht^2 \!D_t^{\ell_{\!1\!}} \Vht)_{\!}\cdots_{\!} (\paht D_t^{\ell_{\!n\!-\!1\!}} \Vht)\!\cdot\!D_t^{\ell_{\!n\!}}\paht f]\!*\!\Phi||_{L^{\!2}(\hD_t\!)}
\leq C ||(\paht^2 \!D_t^{\ell_{\!1\!}} \Vht)_{\!}\cdots_{\!} (\paht D_t^{\ell_{\!n\!-\!1\!}} \Vht)\!\cdot\!D_t^{\ell_{\!n\!}}\paht f||_{L^{\!2}(\hD_t\!)} \leq C'_r|| D_t^{\ell_{\!n\!}}\paht f||_{L^{\!2}(\hD_t\!)}.
\end{equation}
When $\ell_n=1,2$, the worst case scenario is when $n=2$ and $D_t^{r-1-\ell_n}$
falls on $\paht^2 \Vht$. In other words, we only need to control
$||[(\paht^2 D_t^{r-1-\ell_n} \Vht)(D_t^{\ell_n} \paht f)]*\Phi||_{L^2(\hD_t)}$.
Writing
\begin{multline}
[(\paht^2 D_t^{r-1-\ell_n} \Vht)(D_t^{\ell_n} \paht f)]*\Phi = \paht[(\paht D_t^{r-1-\ell_n} \Vht)(D_t^{\ell_n} \paht f)]*\Phi-(\paht D_t^{r-1-\ell_n} \Vht)(\paht D_t^{\ell_n} \paht f)*\Phi\\
=[(\paht D_t^{r-1-\ell_n} \Vht)(D_t^{\ell_n} \paht f)]*(\paht \Phi)-(\paht D_t^{r-1-\ell_n} \Vht)(\paht D_t^{\ell_n} \paht f)*\Phi,
\end{multline}
and using that $\paht \Phi$ and $\Phi$ belong to $L^1(\hD_t)$, Young's inequality implies that
\begin{equation}
||[(\paht D_t^{r\!-\!1\!-\!\ell_{\!n\!}} \Vht)D_t^{\ell_{\!n\!}} \paht{}_{\!} f]*\paht \Phi||_{L^{\!2}(\hD_t\!)}\!+||[(\paht D_t^{r\!-\!1\!-\!\ell_{\!n\!}} \Vht)\paht D_t^{\ell_{\!n\!}} \paht {}_{\!}f]*\Phi||_{L^{\!2}(\hD_t\!)}
\!\lesssim\! {\sum}_{k\leq 1}||(\paht D_t^{r\!-\!1\!-\!\ell_{\!n\!}} \Vht)\paht^k \! D_t^{\ell_{\!n\!}}\paht{}_{\!} f||_{L^{\!2}(\hD_t\!)}.
\end{equation}
Next, to control the term on the right hand side, we have:
\begin{equation}
{\sum}_{k\leq 1}\!||(\paht D_t^{r\!-\!1\!-\!\ell_{\!n\!}} \Vht)\paht^k \! D_t^{\ell_{\!n\!}}\paht{}_{\!} f||_{L^{\!2}(_{\!}\hD_t\!)}\!
\leq \!
C_r'{\sum}_{k\leq 1}\!||\paht D_t^{\ell_{\!n\!}}\paht{}_{\!} f||_{L^{\!6}(_{\!}\hD_t\!)}\!
\leq \! C_{\!r\!}'\big(||\paht D_t^{\ell_{\!n\!}}\paht{}_{\!} f||_{L^{\!6}(_{\!}\hD_t\!)}
 \!+\!{\sum}_{k\leq 1}\!||\paht^k \! D_t^{\ell_{\!n\!}}\paht{}_{\!} f||_{L^{\!6}(_{\!}\hD_t\!)}\big),
\end{equation}
which can be controlled using Lemma \ref{Dtellphi}.
When $\ell_n=0$, the worst-case scenario is when $n=2$ and $D_t^{r-1}$ falls on $\paht^2 \Vht$. In other words, we only need to control $||(\paht^2 D_t^{r-1} \Vht)(\paht f)||_{L^2(\hD_t)}$. By a similar argument as above, we need to control ${\sum}_{k=1,2}||(\paht D_t^{r-1} \Vht)(\paht^k f)||_{L^2(\hD_t)}$, and this requires the control of $||\paht^k f||_{L^\infty(\hD_t)}$ for $k=1,2$. The case when $k=2$ is treated in Lemma \ref{ellphi}, and when $k=1$, we have by Young's inequality:
\begin{equation}
||\paht f||_{L^\infty(\hD_t)} \leq ||g*(\paht \Phi)||_{L^\infty(\hD_t)} \leq C||g||_{L^\infty(\hD_t)}.
\end{equation}
To control the $L^2(\hD_t)$ norm for the second product in \eqref{products}, when $\ell_n=r-1$ and $n=2$, we write:
\begin{equation}
[(\paht \Vht)(\paht D_t^{r-1}\paht f)]*\Phi=[(\paht \Vht)(D_t^{r-1}\paht f)]*(\paht \Phi)-[(\paht^2 \Vht)(D_t^{r-1}\paht f)]*\Phi,
\end{equation}
whose $L^2(\hD_t)$ norm can then be controlled by $C_r'{\sum}_{k\leq r-1}||D_t^k \paht f||_{L^2(\hD_t)}$ using Young's inequality and Sobolev's lemma.
When $r-2\geq \ell_n\geq 2$ (and so $\ell_j\leq r-3$ for $j=1,\cdots, n-1)$,  we have:
\begin{equation}
||[(\paht D_t^{\ell_1} \Vht)\cdots (\paht D_t^{\ell_{n-1}} \Vht)\cdot(\paht D_t^{\ell_n} \paht f)]*\Phi||_{L^2(\hD_t)} \leq  C'_r||\paht D_t^{\ell_n} \paht f||_{L^2(\hD_t)},
\end{equation}
using Young's inequality and Sobolev's lemma. The right hand side is controlled by Lemma \ref{Dtellphi}. If $\ell_n\!=\!1$, it suffices to consider $||[(\paht D_t^{r-2} \Vht)\paht D_t \paht f]\!*\!\Phi||_{L^2(\hD_{t\!})}$, which is bounded by $C'_r||\paht D_t\paht f||_{L^6(\hD_{t\!})}$.  When $\ell_n\!=\!0$, we need to control $||(\paht D_t^{r-1}\Vht)\paht^2 \!f]\!*\!\Phi||_{L^{\!2}(_{\!}\hD_{t\!})}$,
which requires control of $||\paht^2\! f||_{L^{\!\infty}(\hD_{t\!})}$ as in in Lemma \ref{ellphi}.
\end{proof}

\begin{lemma}
 \label{bdy estimate for D_tphi_m'} There is a constant $C\!$ so that if $g$ is smooth and supported  in $\xht({}_{\!}t,{}_{\!} \Omega^{d_0\!/2})$ and
 $f\!{}_{\!}=\!g{}_{\!}*{}_{\!}\Phi$ then
 \begin{equation}
 |\paht^s D_t^k (g*\Phi)(x)| \leq C||D_t^k g||_{L^2(\hD_t)},\qquad x\in \pa\hD_t\qquad  k, s\geq 0.
 \label{bdy est'}
 \end{equation}
 \end{lemma}
 \begin{proof}  We have $[\Dht, D_t^k] f(x)\!=\!0$ when $x\!\in\! \pa\hD_t$
 since $\Vht\!\!=\!0$ near $\pa \hD_t$,.  Therefore, \eqref{D_t^s f} yields $\paht^s \! D_t^k\! f(x)\!=\!(D_t^k g)\!*\!(\paht^s\Phi)(x)$ and so \eqref{bdy est'} follows from a similar argument as in the proof of Lemma \ref{bdy estimate for D_tphi_m}.\!\!\!
 \end{proof}

\begin{proof}[Proof of Theorem \ref{Dtellphi' 1}]
It suffices to prove that for $j \leq r-1$:
\begin{multline}
||\mathfrak{D}^jD_t \paht f||_{L^2(\hD_t)} \leq C\big(
||\xve||_{H^{r}(\Omega)},  {\tsum}_{k\leq r-1}
||D_t^k V||_{H^{r-k}(\Omega)}\big)\\\cdot||\T\xve||_{H^{(r-1, 1/2)}(\Omega)}
\bigtwo({\sum}_{k\leq r}||\mathfrak{D}^kg||_{L^2(\hD_t)}+{\sum}_{k\leq 2}
||\mathfrak{D}^kg||_{L^6(\hD_t)}+||g||_{L^\infty(\hD_t)}\bigtwo).
\label{boundphi T D_t}
\end{multline}
When $j=r-1$, we have:
\begin{equation}
||\FD^{r-1}D_t\paht f||_{L^2(\hD_t)}^2
=\int_{\hD_t}\underbrace{\delta^{ij}(\FD^{r-1} D_t\paht_i f)(\paht_j\FD^{r-1} D_tf)\,dx}_{I}
+\int_{\hD_t}\underbrace{\delta^{ij}(\FD^{r-1}D_t\paht_i f)([\FD^{r-1}D_t, \paht_j]f)}_{II}.
\end{equation}
We then control $II$ by applying Corollary \ref{commutator FD pave lemma}.
To control $I$, we integrate by parts and get:
\begin{align}
-\int_{\hD_t}\underbrace{\delta^{ij}(\paht_i\FD^{r-1} D_t\paht_j f)(\FD^{r-1} D_tf)
\,dx}_{I_1}+\int_{\pa\hD_t}\underbrace{(N^i \mathfrak{D}^{r-1}D_t\paht_if)
(\mathfrak{D}^{r-1}D_t f)}_{\mathcal{B}}.
\end{align}
The interior term $I_{\!1}$ is equal to
$
\int_{\hD_t}(\mathfrak{D}^{r\!-\!1}\!D_t g )(\mathfrak{D}^{r\!-\!1}\!D_t f)
$
to highest order. The error terms here are as in
\eqref{controllable error}, and the $L^2$ norm of these terms contribute
$||\T\xve||_{H^{(r-1, 1/2)}(\Omega)}$ in \eqref{boundphi T D_t} using \eqref{inthalf}.  When $\FD^{r-1}=D_t^r$, this term can be controlled by $||D_t^r g||_{L^2(\hD_t)}||D_t^r f||_{L^2(\hD_t)}$, and then we may bound $||D_t^r f||_{L^2(\hD_t)}$ using Lemma \ref{D_t^r f lem}. In addition, when $\FD^{r-1}=\T\FD^{r-2}$, we control $I_1$ by integrating $\T$ by parts, similar to the control of \eqref{I_1} in the proof of Theorem \ref{ellphi'' 1}.
 Finally, we use Lemma \ref{bdy estimate for D_tphi_m'} to control $\mathcal{B}$.\!\!\!\!
\end{proof}

\subsection{Estimates for Section \ref{section 7.3}}
\begin{theorem}  If $r\geq 5$, then for each $\mu = 0,..., N$ and $j\leq r-1$:
\begin{equation}
||\fdhm\T^{j\!} \paht (g*\Phi)||_{L^2(\hD_t)} \!\leq \! P(||\xve||_{H^{r_{\!}}(\Omega)})
\bigtwo(\!||g||_{L^\infty(\hD_t)}\!+\!\!{\sum}_{k\leq 2}||\T^k \!g||_{L^6(\hD_t)}\!+\!\!{\sum}_{k\leq r_{\!}-_{\!}1}||\T^k \! g||_{L^2(\hD_t)}\!\bigtwo).
\label{boundphi 1/2}
\end{equation}
\end{theorem}
\begin{proof}
Suppose that we know \eqref{boundphi 1/2} holds for $0\leq j\leq r-2$, when $j=r-1$, we have
\begin{multline}
||\fdhm \T^{r-1}\paht f||_{L^2(\hD_t)}^2 = \int_{\hD_t}(\fdhm \T^{r-1}\paht_i f)(\fdhm \T^{r-1}\paht^i f)\,dx\\
=\int_{\hD_t}\underbrace{(\fdhm \T^{r-1}\paht_i f)\fdhm (\paht^i\T^{r-1}f)\,dx}_{I}
-\int_{\hD_t}\underbrace{(\fdhm \T^{r-1}\paht_i f)
\fdhm \big((\paht^i\T^{r-1} \xht)(\paht f)\big)\,dx}_{II}\\
+\sum \int_{\hD_t}\underbrace{(\fdhm \T^{r-1} \paht f)\fdhm \big((\pa \T^{\ell_1}\w x)\cdots(\pa \T^{\ell_{s-1}}\w x)
(\T^{\ell_s}\paht f)\big)\,dx}_{III},
\end{multline}
where the sum of over $\ell_1+\cdots+\ell_s =r-1$, $\ell_1,\cdots,\ell_s\leq r-2$.
Invoking \eqref{app:alg2}, one has with $L^2=L^2(\hD_t)$:
\begin{equation}
II\!\leq\!\! \int_{\hD_t}\!\!\!(\fdhm \T^{r_{\!}-_{\!}1}\paht f)(\fdhm \paht \T^{r_{\!}-_{\!}1}\! \xht)(\paht_{\!} f)dx
+C||\fdhm \T^{r_{\!}-_{\!}1}\paht_{\!} f||_{L^2}||\paht \T^{r_{\!}-_{\!}1}\! f||_{L^2}{\sum}_{k\leq 2}||\T^k \paht_{\!} f||_{L^2}.
\label{II0.5}
\end{equation}
The last term on the right hand side is of the correct form that we control,
while the main term is controlled as the corresponding term (i.e., $II$) in the
 proof of Theorem \ref{ellphi'' 1} and a repeated use of \eqref{app:alg2}. In addition,
\begin{equation}
I \!\leq \!\!\int_{\hD_t}\!\!\!(\fdhm \T^{r_{\!}-_{\!}1}\paht_i f)(\paht^i\fdhm \T^{r-1}\!f)dx
+C||\fdhm \T^{r_{\!}-_{\!}1}\paht_{\!}  f||_{L^2(\hD_t)}||\paht \T^{r_{\!}-_{\!}1}\! f||_{L^2(\hD_t)}{\sum}_{k\leq 2}||\T^k \uht||_{L^2(\hD_t)}.
\label{I0.5}
\end{equation}
The last term on the right hand side is of the form that we control,
while the main term can be controlled similarly to how we controlled the corresponding term (i.e., $I$) in
the proof of Theorem \ref{ellphi'' 1} after a repeated use of \eqref{app:alg2}.
Finally, we need to control the $L^2$ norm of
$\sum \fdhm ((\pa \T^{\ell_1}\w x)\cdots(\pa \T^{\ell_{s-1}}\w x)
(\T^{\ell_s}\paht f))$ in $III$. When $\ell_s\geq 3$, then $\ell_1,\cdots,
\ell_{s-1}\leq r-4$, and so we let $\fdhm$ fall on $\p \T^{\ell_1} \xht$ by
applying \eqref{app:alg2} and then control the terms involving $\xht$ in $L^{\infty}$.
 Moreover, if at least one of $\ell_1,\cdots, \ell_{s-1}$, say $\ell_1$, is
 greater than or equal to $r-3$, we let $\fdhm$ falls on $\p \T^{\ell_1} \xht$
 by applying \eqref{app:alg2} and control this term in $L^3$, and so $\T^s \paht
 f$ is controlled in $L^6$. But this can then be treated using Sobolev embedding
 and then Lemma \ref{ellphi}.
\end{proof}
\subsection{Estimates for Section \ref{section 7.4}}
\begin{lemma} \label{diffellphi} Suppose that $r\!\geq\! 7$ and  $f_J$ satisfy $\Dht_J f_J\!=\!g_J$ for $J=I,\II$. Then for $j\!\leq \! r-1$, we have:
\begin{multline}
||\pahtu \mathfrak{D}^j \pahtu f_I-\pahtw \mathfrak{D}^j \pahtw f_{\II}||_{L^2(\Odo)}
\leq D_r\bigtwo({\sum}_{k\leq r}||\FD^k \pahtu f_I-\FD^k\pahtw f_{\II}||_{L^2(\Odo)}+||\FD^{r-1} (g_I-g_{\II})||_{L^2(\Odo)}\\
+{\sum}_{k\leq 2}||\FD^k (g_I-g_{\II})||_{L^6(\Odo)}
+||g_I-g_{\II}||_{L^\infty(\Odo)}
+\bigtwo\{|| \xveu-\xvew||_{H^r(\Omega)}+{\sum}_{k\leq r-2}||D_t^k (\Vu-\Vw)||_{H^{r-k}(\Omega)}\bigtwo\}\\
\cdot \big({\sum}_{k\leq r}||\FD^k \pahtw f_{\II}||_{L^2(\Odo)}+||\FD^{r-1}g_{\II}||_{L^2(\Odo)}+
{\sum}_{k\leq 2}||\FD^k g_{\II}||_{L^6(\Odo\!)}+||g_{\II}||_{L^\infty(\Odo\!)}\big)\bigtwo).
\label{diffbound phi}
 \end{multline}
 where $D_r\!=\!D_r\big(||\xveu\!||_{H^r(\Omega)}, \!||\xvew\!||_{H^r(\Omega)},\! {\sum}_{k\leq r-2}||D_t^k\Vu\!||_{H^{r\!-\!k}(\Omega)},\! {\sum}_{k\leq r-2}||D_t^k\Vw\!||_{H^{r\!-\!k}(\Omega)}\big)$. For  $0\!\leq\!\ell\!\leq \!2$, we have:
 \begin{multline}
 ||\pahtu \mathfrak{D}^\ell \pahtu f_I-\pahtw \mathfrak{D}^\ell \pahtw f_{\II}||_{L^6(\Odo)}\leq D_r\bigtwo({\sum}_{k\leq \ell}||\FD^k (g_I-g_{\II})||_{L^6(\Odo)}+{\sum}_{k\leq \ell+2}||\FD^k \pahtu f_I-\FD^k \pahtw f_{\II}||_{L^2(\Odo)}\\
 +||\xveu-\xvew||_{H^r(\Omega)}{\sum}_{k\leq \ell}||\pahtw \FD^k \pahtw f_{\II}||_{L^6(\Odo)}\bigtwo),
 \end{multline}
 as well as
 \begin{multline}
 ||\pahtu^2 f_I-\pahtw^2 f_{\II}||_{L^\infty(\Odo)}\lesssim ||g_I-g_{\II}||_{L^\infty(\Odo)}\\
 +{\sum}_{\ell\leq 1}||\pahtu^\ell \T\pahtu f_I-\pahtw^\ell\T\pahtw f_{\II}||_{L^6(\Odo)}
 +||\xveu-\xvew||_{H^r(\Omega)}{\sum}_{\ell\leq 1}||\pahtw \FD^\ell \pahtw f_{\II}||_{L^6(\Odo)}.
 \end{multline}
 \end{lemma}
 \begin{proof}
 For $j\!=\!0$ \eqref{diffbound phi} follows from \eqref{pwdiffbd}.
 Suppose that \eqref{diffbound phi} hold for $j\!\leq \!r\!-\!2$. When $j\!=\!r\!-\!1$ we have:
 \begin{multline}
 ||\pahtu \mathfrak{D}^{r-1\!} \pahtu f_I-\pahtw \mathfrak{D}^{r-1\!} \pahtw f_{\II}||_{L^2(\Odo)}\\
 \lesssim ||\div_I \!\FD^{r-1\!} \pahtu f_I-\div_{\II}\!\FD^{r-1\!}\pahtw f_{\II}||_{L^2(\Odo)}
 +||\curl_I\! \FD^{r-1\!} \pahtu f_I-\curl_{\II}\!\FD^{r-1\!}\pahtw f_{\II}||_{L^2(\Odo)}\,\,\,\,\,\\
 +||\T\FD^{r-1\!} \pahtu f_I-\T\FD^{r-1\!} \pahtw f_{\II}||_{L^2(\Odo)}+||\xhtu-\xhtw||_{H^r(\Odo)}||\pahtw \FD^{r-1} \pahtw f_{\II}||_{L^2(\Odo)}.
 \end{multline}
 It suffices to control the $\div$ term, since the $\curl$ term can be controlled similarly. We have:
 \begin{multline}
\div_{\!I} \!\FD^{r-1\!}  \pahtu{}_{\!} f_I\!-\div_{\II}\! \FD^{r-1\!} \pahtw f_{\II}\\
= \FD^{r-1}\! (g_I\!-g_{\II})
+\!\sum \Big((\FD^{k_1\!}\pa \xhtu)\cdots(\FD^{k_s\!}\pa \xhtu{}_{\!})\pahtu \FD^\ell \pahtu{}_{\!} f_I\!-(\FD^{k_1\!}\pa \xhtw{}_{\!})\cdots(\FD^{k_s\!}\pa \xhtw)\pahtw \FD^\ell \pahtw f_{\II}\Big).
\end{multline}
where the sum is over $k_1\!+\!\cdots\!+\!k_s\!+\!\ell\!=\!r\!-\!1$, $k_1\!\geq \! 1$. To control the sum in $L^2(\Odo\!)$ we only need to consider
\begin{align}
\mathcal{A}=(\FD^{k_1}\pa \xhtu)\cdots(\FD^{k_s}\pa \xhtu)(\pahtu \FD^\ell \pahtu f_I-\pahtw \FD^\ell \pahtw f_{\II}),\\
\mathcal{B}=(\FD^{k_1}\pa \xhtu-\FD^{k_1}\pa \xhtw)(\FD^{k_2}\pa \xhtw)\cdots(\FD^{k_s}\pa \xhtw)(\pahtw \FD^\ell \pahtw f_{\II}).
\end{align}
Now, if $\ell\geq 2$, then $k_1,\cdots, k_s\leq r-3$, and so all terms involving $\xht$ can then be controlled in $L^\infty$, i.e.,
\begin{align}
||\mathcal{A}||_{L^2(\Odo)} \leq D_r||\pahtu \FD^{\ell} \pahtu f_I-\pahtw \FD^\ell \pahtw f_{\II}||_{L^2(\Odo)},\\
||\mathcal{B}||_{L^2(\Odo)} \leq D_r||\FD^{k_1}\pa \xhtu-\FD^{k_1}\pa \xhtw||_{L^\infty(\Odo)}||\pahtw\FD^\ell\pahtw f_{\II}||_{L^2(\Odo)}.
\end{align}
 Moreover, since $r\geq 7$, there is at most one of $k_1,\cdots, k_s$, say $k_1$, that can be $\geq r-2$. If $k_1=r-2$, then
\begin{align}
||\mathcal{A}||_{L^2(\Odo)}\leq D_r ||\FD^{k_1}\pa \xhtu||_{L^3(\Odo)}||\pahtu \FD^{\ell} \pahtu f_I-\pahtw \FD^\ell \pahtw f_{\II}||_{L^6(\Odo)},\\
||\mathcal{B}||_{L^2(\Odo)} \leq D_r ||\FD^{k_1}\pa \xhtu-\FD^{k_1}\pa \xhtw||_{L^3(\Odo)}||\pahtw\FD^\ell\pahtw f_{\II}||_{L^6(\Odo)},
\end{align}
and since $\ell\leq 2$, we have:
\begin{multline}
||\pahtu \FD^{\ell} \pahtu \! f_I-\pahtw \FD^\ell \pahtw f_{\II}{}_{\!}||_{L^6(\Odo\!)}\!\lesssim\!
 ||\div_{\!I}\! \FD^\ell \pahtu\! f_{\!I}-\div_{\II}\! \FD^{\ell} \pahtw{}_{\!} f_{\II}{}_{\!}||_{L^6(\Odo\!)}+||\curl_I \!\FD^\ell \pahtu \!f_{\!I}-\curl_{\II}\! \FD^{\ell} \pahtw {}_{\!} f_{\II}{}_{\!}||_{L^6(\Odo\!)}\,\,\,\\
+||\T\FD^\ell \pahtu f_I-\T\FD^\ell \pahtw f_{\II}||_{L^6(\Odo)}+||\xhtu-\xhtw||_{H^r(\Odo)}||\pahtw \FD^\ell \pahtw f_{\II}||_{L^6(\Odo)}
\lesssim ||\FD^\ell (g_I-g_{\II})||_{L^6(\Odo)}\\
+{\sum}_{\ell_1+\ell_2=\ell, \ell_1\geq 1}\Big(||(\FD^{\ell_1}\uht_I)(\pahtu\FD^{\ell_2}\pahtu f_I-\pahtw\FD^{\ell_2}\pahtw f_{\II})||_{L^6(\Odo)}+||(\FD^{\ell_1}[\widehat{A}_{\I}-\widehat{A}_{\II}])(\pahtw\FD^{\ell_2}\pahtw f_{\II})||_{L^6(\Odo)}\Big)\\
+||\T\FD^\ell \pahtu f_I-\T\FD^\ell \pahtw f_{\II}||_{H^1(\Odo)}+||\xhtu-\xhtw||_{H^r(\Odo)}||\pahtw \FD^\ell \pahtw f_{\II}||_{L^6(\Odo)},
\end{multline}
where the sum is of lower order and
\begin{equation}
||\pa_y(\T\FD^\ell \pahtu f_I\!-\T\FD^\ell \pahtw f_{\II})||_{L^2(\Odo)}\lesssim ||\pahtu\T\FD^\ell \pahtu {}_{\!}f_I\!-\pahtw\T\FD^\ell \pahtw f_{\II}||_{L^2(\Odo)}+||(\pahtw \!-\!\pahtu)\T\FD^\ell\pahtw f_{\II}||_{L^2(\Odo)},
\label{new H1}
\end{equation}
which is of the form we control.
Finally, if $k_1=r-1$, we need to control $\pahtu^2 F-\pahtw^2 G$ in $L^\infty$, i.e.,
\begin{equation}
||\pahtu^2 f_I-\pahtw^2 f_{\II}||_{L^\infty(\Odo)}\lesssim ||g_I-g_{\II}||_{L^\infty(\Odo)}+||\T(\pahtu f_I-\pahtw f_{\II})||_{L^\infty(\Odo)}+||\xhtu-\xhtw||_{H^r(\Odo)}||\pahtw^2 f_{\II}||_{L^\infty(\Odo)},
\end{equation}
where $||\T(\pahtu{}_{\!} f_I\!\!-\!\pahtw f_{\II})||_{L^{{}_{\!}\infty}(\Odo\!)}\lesssim {\sum}_{\ell\leq 1}||\pa_y^
\ell\T(\pahtu {}_{\!}f_I\!\!-\!\pahtw f_{\II})||_{L^{6_{\!}}(\Odo\!)}$, and this can be controlled as above.
 \end{proof}
 \begin{lemma}
 \label{D_t^r-1 F-G lem}
 Let $f_J=(g_J*\Phi)\circ \xht_J$ for $J=I,\II$, where $g_J$ are smooth functions supported in $\Omega^{d_0\!/2}$ satisfying $\FD g_J=0$ in $ \Odo\!\!\setminus \! \Omega$. Then:
 \begin{equation}
 ||f_I-f_{\II}||_{L^2(\Odo)}\leq D_r(||g_I-g_{\II}||_{L^2(\Odo)}+||\xveu-\xvew||_{H^r(\Omega)}||g_{\II}||_{L^2(\Odo)}),
 \label{F-G}
 \end{equation}
 and for $r\geq 7$, we have:
 \begin{multline}
 ||D_t^{r-1} f_I-D_t^{r-1} f_{\II}||_{L^2(\Odo)}\leq  D_r\bigtwo({\sum}_{k\leq r-1}||\FD^k \pahtu f_I-\FD^k\pahtw f_{\II}||_{L^2(\Odo)}+{\sum}_{k\leq r-1}||\FD^k (g_I-g_{\II})||_{L^2(\Odo)}\\
+{\sum}_{k\leq 2}||\FD^k (g_I-g_{\II})||_{L^6(\Odo)}
+||g_I-g_{\II}||_{L^\infty(\Odo)}
+\big\{|| \xveu-\xvew||_{H^r(\Omega)}+{\sum}_{k\leq r-2}||D_t^k (\Vu-\Vw)||_{H^{r-k}(\Omega)}\big\}\\
\!\cdot \!\big({\sum}_{k\leq r\!-\!1}||\FD^k\pahtw \! f_{\II\!}||_{L^2(\Odo\!)}\!+\!\!{\sum}_{k\leq r\!-\!1}||\FD^k \! g_{\II\!}||_{L^2(\Odo\!)}\!
+\!\!{\sum}_{k\leq 2}||\FD^k\! g_{\II\!}||_{L^6(\Odo\!)}\!
+||g_{\II\!}||_{L^\infty(\Odo\!)}\big)\!\bigtwo).
 \label{D_t r-1 F-G}
 \end{multline}
\end{lemma}
\begin{proof}
We prove \eqref{F-G} first. Writing $f_J=\int_{\Odo}g_J(t,y')\Phi(\xht_J(t,y)-\xht_J(t,y'))\widehat{\kappa}_J\,dy$, we have
\begin{multline}
f_I-f_{\II}=\int_{\Odo}\underbrace{g_{\II}(t,y')\big(\Phi(\xhtu(t,y)-\xhtu(t,y'))
-\Phi\big(\xhtw(t,y)-\xhtw(t,y')\big)\big)\khtu\,dy'}_{I_1}\\
+\!\!\int_{\Odo}\underbrace{\!\!\!\!\!\!(g_I(t,y')\!-g_{\II}(t,y'))
\Phi\big(\xhtu(t,y)\!-\xhtu(t,y')\big)\khtu dy'\!\!}_{I_2}\,
+\!\!\int_{\Odo}
\underbrace{\!\!\!\!\!g_{\II}(t,y')\Phi\big(\xhtw(t,y)\!-\xhtw(t,y')\big)
(\khtu\!-\khtw) dy'\!\!}_{I_3}.
\label{F-G expression}
\end{multline}
By Young's inequality, we have:
\begin{align}
||I_2||_{L^2(\Odo)}\leq D_r||g_I-g_{\II}||_{L^2(\Odo)},\qquad  ||I_3||_{L^2(\Odo)}\leq D_r ||\xveu-\xvew||_{H^r(\Omega)}||g_{\II}||_{L^2(\Odo)}.
\end{align}
 To control $I_1$, we write
\begin{multline}
\big|\Phi(\xhtu(t,y)-\xhtu(t,y'))-\Phi(\xhtw(t,y)-\xhtw(t,y'))\big|\\
=\frac{1}{4\pi}
\Big|\frac{1}{|\xht_I(t,y)-\xht_I(t,y')|}-\frac{1}{|\xht_{\II}(t,y)-\xht_{\II}(t,y')|}
\Big|
\leq \frac{1}{4\pi}\frac{|\xhtw(t,y)-\xhtu(t,y)|+|\xhtw(t,y')-\xhtu(t,y')|}
{|\xht_I(t,y)-\xht_I(t,y')||\xht_{\II}(t,y)-\xht_{\II}(t,y')|}.
\end{multline}
Since this is in $L^{\!1\!}(\Odo\!)$, we have $||I_{\!1\!}||_{L^2(\Odo\!)}\!\leq\! D_r(||\xveu\!||_{H^r(\Omega)})||\xveu-\xvew||_{H^r(\Omega)}||g_{\II}||_{L^2(\Odo\!)}$ using Young's inequality. Now, for \eqref{D_t r-1 F-G}, we write:
\begin{align}
D_t^{r-1}f_I= (D_t^{r-1}g_I)*\Phi\circ \xhtu +([\widehat{\Delta}_I, D_t^{r-1}]f_I)*\Phi\circ \xhtu,\\
D_t^{r-1}f_{\II}= (D_t^{r-1}g_{\II})*\Phi\circ \xhtw +([\widehat{\Delta}_{\II}, D_t^{r-1}]f_{\II})*\Phi\circ \xhtw.
\end{align}
To control $||D_t^{r-1} f_I\!-D_t^{r-1} f_{\II}||_{L^2(\Odo)}$, we need the bounds for $||(D_t^{r-1}\!g_I)*\Phi\circ \xhtu-(D_t^{r-1}\!g_{\II})*\Phi\circ \xhtw||_{L^2(\Odo)}$ and $||([\widehat{\Delta}_I, D_t^{r-1}]f_I)*\Phi\circ \xhtu-([\widehat{\Delta}_{\II}, D_t^{r-1}]f_{\II})*\Phi\circ \xhtw||_{L^2(\Odo\!)}$. As above, they are controlled by
\begin{align}
D_r \big\{||D_t^{r-1}g_I-D_t^{r-1}g_{\II}||_{L^2(\Odo)}
+||\xveu-\xvew||_{H^r(\Omega)}||D_t^{r-1} g_{\II}||_{L^2(\Odo)}\big\},\\
D_r \big\{||[\widehat{\Delta}_I, D_t^{r-1}]f_I-[\widehat{\Delta}_{\II}, D_t^{r-1}]f_{\II}||_{L^2(\Odo)}+||\xveu-\xvew||_{H^r(\Omega)}||[\widehat{\Delta}_{\II}, D_t^{r-1}]f_{\II}||_{L^2(\Odo)}\big\},
\end{align}
respectively, where $[\widehat{\Delta}_{\II}, D_t^{r-1}]f_{\II}$ can be treated by adapting the proof for Lemma \ref{D_t^r f lem}.
Moreover, since
for each $J=I, II$, $[\Dht_J, D_t^{r-1}]$ consists
\begin{equation}
(\paht_J^2 D_t^{\ell_1} \Vht_J)\cdots (\paht_J D_t^{\ell_{n-1}} \Vht_J)\cdot( D_t^{\ell_n} \paht_J )\q \text{and}\q (\paht_J D_t^{\ell_1} \Vht_J)\cdots (\paht_J D_t^{\ell_{n-1}} \Vht_J)\cdot(\paht_J D_t^{\ell_n} \paht_J),
\end{equation}
where $\ell_1+\cdots+\ell_n=r-n$, the control of $||[\widehat{\Delta}_I, D_t^{r-1}]f_I-[\widehat{\Delta}_{\II}, D_t^{r-1}]f_{\II}||_{L^2(\Odo)}$ requires that of
\begin{align}
K_1=||(\pahtu^2 D_t^{\ell_1} \Vhtu)\cdots (\pahtu D_t^{\ell_{n-1}} \Vhtu)\cdot( D_t^{\ell_n} \pahtu f_I )-(\pahtw^2 D_t^{\ell_1} \Vhtw)\cdots (\pahtw D_t^{\ell_{n-1}} \Vhtw)\cdot( D_t^{\ell_n} \pahtw f_{\II} )||_{L^2(\Odo)},\\
K_2=||(\pahtu D_t^{\ell_1} \Vhtu)\cdots (\pahtu D_t^{\ell_{n-1}} \Vhtu)\cdot( \pahtu D_t^{\ell_n} \pahtu f_I )-(\pahtw D_t^{\ell_1} \Vhtw)\cdots (\pahtw D_t^{\ell_{n-1}} \Vhtw)\cdot( \pahtw D_t^{\ell_n} \pahtw f_{\II})||_{L^2(\Odo)}.
\end{align}
It suffices to consider the case when $n=2$ only. To control $K_1$, we have
with $L^2=L^2(\Odo)$
\begin{multline}
||(\pahtu^2 D_t^{\ell_1} \!\Vhtu)D_t^{\ell_2} \pahtu{}_{\!}
{}_{\!} f_I\! -(\pahtw^2 D_t^{\ell_1}\! \Vhtw) D_t^{\ell_2} \pahtw {}_{\!} f_{\II}||_{L^2}\\
\leq ||(\pahtu^2 D_t^{\ell_1}\! \Vhtu\!-\pahtw^2 D_t^{\ell_1}\! \Vhtw) D_t^{\ell_2} \pahtw{}_{\!} f_{\II} ||_{L^2}+||(\pahtu^2 \! D_t^{\ell_1} \! \Vhtu)(D_t^{\ell_2\!} \pahtu {}_{\!}f_{\!I} \!-D_t^{\ell_2\!} \pahtw {}_{\!} f_{\II})||_{L^2}\\
\leq \underbrace{||\big((\pahtu^2\!-\pahtw^2) D_t^{\ell_1} \!\Vhtu \!\big)D_t^{\ell_2} \pahtw {}_{\!}f_{\II}||_{L^2}}_{K_{12}}+\underbrace{||\big(\pahtw^2 D_t^{\ell_1} \! (\Vhtu\!-\!\Vhtw)\big) D_t^{\ell_2} \pahtw{}_{\!} f_{\II} ||_{L^2}}_{K_{13}}+
\underbrace{||(\pahtu^2\! D_t^{\ell_1} \! \Vhtu)(D_t^{\ell_2} \pahtu {}_{\!}f_{\!I} \!-\!D_t^{\ell_2\!} \pahtw {}_{\!} f_{\II})||_{L^2}}_{K_{11}}\!.
\end{multline}
When $\ell_1\!\leq \!r\!-\!4$, we bound $\Vht\!$ factors in $L^{\!\infty\!}$ and use Sobolev's lemma. Then
$K_{11}\!\leq \! D_r||D_t^{\ell_2} (\pahtu {}_{\!}f_{\!I} \!- \pahtw {}_{\!} f_{\II\!})||_{L^2}$, and $K_{12}\!\leq \! D_r||\xveu\!-\!\xvew||_{H^r(\Omega)}|| D_t^{\ell_2} \pahtw{}_{\!} f_{\II}||_{L^2}$,
 and $K_{13}\!\leq\! D_r ||D_t^{\ell_1}\!(V_{\I}\!-\!V_{\II})||_{H^{r\!-\!\ell_1}} || D_t^{\ell_2} {}_{\!} \pahtw f_{\II}||_{L^{\!2}}$.
 When $\ell_1\!=\!r\!-\!3$ (and $\ell_2\!=\!1$), we bound $\Vht$ terms in $L^3(\Odo\!)$ and use Sobolev's lemma.  In this case,
$K_{11}\!\leq\!
D_r||D_t (\pahtu{}_{\!} f_{\!I}\! - \pahtw{}_{\!} f_{\II\!})||_{L^6(\Odo\!)}$,
and  $K_{12}\!\leq\! D_r||\xveu\!-\!\xvew||_{H^r(\Omega)}|| D_t\pahtw{}_{\!}  f_{\II\!}||_{L^6(\Odo\!)}$ and $K_{13}\leq D_r ||D_t^{\ell_1}\!(V_{\I}\!-\!V_{\II})||_{H^{r-3}(\Omega)} || D_t\pahtw f_{\II}||_{L^6(\Odo\!)}$.
 By Sobolev's lemma
$
||D_t (\pahtu\! f_{\!I} -{}_{\!}\pahtw{}_{\!} f_{\II})||_{L^6(\Odo\!)}\!\lesssim\! ||D_t(\pahtu f_{\!I \!}-\pahtw f_{\II})||_{H^{1\!}(\Odo\!)},
$
 where we have
$
||\pa_y D_t (\pahtu f_{\!I} - \pahtw f_{\II})||_{L^2(\Odo)}\leq  D_r\big(||\pahtu D_t \pahtu {}_{\!}f_I\!-\pahtw D_t \pahtw f_{\II}||_{L^2(\Odo)}+||(\pahtw -\pahtu)D_t\pahtw f_{\II}||_{L^2(\Odo)}\big),
\label{}
$
which can be controlled by the right hand side of \eqref{D_t r-1 F-G} using Lemma \ref{Dtellphi}, and $||D_t\pahtw f_{\II}||_{L^6(\Odo)}$ can be treated in a similar way. When $\ell_1\!=\!r\!-\!2$ (and $\ell_2\!=\!0$), we bound $\Vht\!$ terms in $L^2(\Odo)$, so we need to control $||\pahtu \! f_I\!-\pahtw f_{\II}||_{L^\infty(\Odo)}$ and $||\pahtw f_{\II}||_{L^\infty(\Odo)}$.  By Sobolev's lemma,
$
||\pahtu f_I \!\!-\pahtw f_{\II}||_{L^\infty(\Odo)} \lesssim {\sum}_{\ell\leq 1}||\pa_y^\ell(\pahtu f_I\! -\pahtw f_{\II})||_{L^6(\Odo)},
$
where
$
||\pa_y(\pahtu f_I\! -\pahtw f_{\II})||_{L^6(\Odo)}\leq ||\pahtu^2 f_I\!-\pahtw^2 f_{\II}||_{L^6(\Odo)}+||(\widehat{A}_{\II}\!-\widehat{A}_{I})\pa_y \pahtw f_{\II}||_{L^6(\Odo)},
$
which is of the form that we control thanks to Lemma \ref{Dtellphi}, and $||\pahtw f_{\II}||_{L^\infty(\Odo)}$ can be treated in a similar fashion.  Finally, we control $K_2$ by adapting a similar argument as above.
\end{proof}
\begin{lemma}
\label{bdy estimate diff} Let $g_J$, $J=I,\II$ be smooth functions supported in $\Omega^{d_0\!/2}$ and $y\in \pa\Odo$. Then:
\begin{equation}
|\pa_y \FD^k(f_I(t,y)-f_{\II}(t,y)) | \lesssim ||\FD^k(g_I-g_{\II})||_{L^2(\Odo)}+||\xveu-\xvew||_{H^r(\Omega)}||\FD^k g_{\II}||_{L^2(\Odo)} ,\q k\geq 0.
\end{equation}
\end{lemma}
\begin{proof}
Since $\xht(t,y)=x_0(y)$ and $\Vht(t,y)=0$ when  $y\in \pa\Odo$, we have that $[\Dveu, \FD^k]f_I(y)=[\Dvew, \FD^k]f_{\II}(y)=0$. Therefore, $\pa_y\FD^k f_I(t,y)=(\FD^k g_I)*(\pa_y \Phi)(t,y)$ and $\pa_y\FD^k f_{\II}(t,y)=(\FD^k g_{\II})*(\pa_y \Phi)(t,y)$, and the control of $|\pa_y \FD^k(f_I(t,y)-f_{\II}(t,y))|$ follows from a similar argument that is used to control \eqref{F-G expression} since $\pa_y\Phi(\xht(t,y)-\xht(t,y'))$ is away from its singularity when $y'\in \Omega^{d_0/2}$ and $y\in \pa\Odo$.
\end{proof}
 \begin{theorem}
 \label{diffellphithm}
  With the same assumptions as in Lemma \ref{D_t^r-1 F-G lem},  if $r\geq 7$, we have with $L^p=L^p(\Odo)$:
 \begin{multline}
 {\sum}_{k\leq r-1}||\FD^k \pahtu f_I-\FD^k\pahtw f_{\II}||_{L^2}
 \leq D_r\bigtwo({\sum}_{k\leq r-1}||\FD^k (g_I-g_{\II})||_{L^2}
 +{\sum}_{k\leq 2}||\FD^k (g_I-g_{\II})||_{L^6}
+||g_I-g_{\II}||_{L^\infty}\,\,\\
+\!\big(||\xveu\!{}_{\!}-\!\xvew\!||_{H^{r\!}(\Omega)}\!+\!\!{\sum}_{k\leq r-2}||D_t^k (\Vu\!-\!\Vw\!)||_{H^{r\!-\!k}(\Omega)}\big)
\big({\sum}_{k\leq r\!-\!1}||\mathfrak{D}^k\! g_{\II}\!||_{L^2}\!+\!\!
{\sum}_{k\leq 2}||\mathfrak{D}^k\! g_{\II}\!||_{L^6}\!+\!||g_{\II}\!||_{L^\infty}\!\big)\!\bigtwo).
 \label{diff bound phi'}
 \end{multline}
 \end{theorem}
 \begin{proof}
 When $k=0$, this is done as in the proof of Lemma \ref{zeroth}. However, one needs to estimate $||f_I-f_{\II}||_{L^2}$ directly without using Poincar\'e's inequality, which has been done in Lemma \ref{D_t^r-1 F-G lem}. Next, suppose that \eqref{diff bound phi'} is known for $k=0,\cdots,r-2$. When $k=r-1$, we have:
\begin{multline}
||\FD^{r-1}\pahtu f_I-\FD^{r-1}\pahtw f_{\II}||_{L^2(\Odo)}^2 =\int_{\Odo} \!\! \delta^{ij}\!\underbrace{\big(\FD^{r-1} \pahtu{}_{\!i} f_I-\FD^{r-1}\pahtw{}_{\!i} f_{\II}\big)\big(\pahtu{}_{\!j} \FD^{r-1} f_I-\pahtw{}_{\!j} \FD^{r-1} f_{\II}\big)\,dy}_{I}\\
+\int_{\Odo} \!\!\delta^{ij}\big(\FD^{r-1} \pahtu{}_{\!i} f_I-\FD^{r-1}\pahtw{}_{\!i} f_{\II}\big)\big([\FD^{r-1}, \pahtu{}_{\!j}] f_I-[\FD^{r-1},\pahtw{}_{\!j}] f_{\II}]\big)\,dy.
\label{diffI}
\end{multline}
The second term can be bounded using Lemma \ref{commutator FD pave lemma diff} together with the bounds for $||\pahtu f_{\!I}\! -\!\pahtw f_{\II}||_{L^\infty(\Odo)}$ and ${\sum}_{\ell\leq 2}||\FD^\ell \pahtu {}_{\!} f_{\!I}\!-\!\FD^\ell {}_{\!}\pahtw {}_{\!}f_{\II}||_{L^6(\Odo)}$.
Here,
$
||\pahtu {}_{\!}f_I \!-\pahtw {}_{\!}f_{\II}||_{L^\infty(\Odo)} \lesssim {\sum}_{\ell\leq 1}||\pa_y^\ell(\pahtu {}_{\!}f_{\!I}\! -{}_{\!}\pahtw {}_{\!} f_{\II})||_{L^6(\Odo)},
$
where
\begin{equation}
||\pa_y(\pahtu f_I -\pahtw f_{\II})||_{L^6(\Odo)}\leq ||\pahtu^2 f_I-\pahtw^2 f_{\II}||_{L^6(\Odo)}+||(\widehat{A}_{\II\, i}^{\,\,\,a}-\widehat{A}_{{}_{\!}I\, i}^{\,\,a})\pa_{y^a} \pahtw f_{\II}||_{L^6(\Odo)},
\end{equation}
which is of the form that we control by Lemma \ref{diffellphi}.
In addition, for each $\ell\leq 2$, we have with $L^p=L^p(\Odo)$:
\begin{equation}
||\FD^\ell \pahtu{}_{\!} f_{\!I\!}-\FD^\ell{}_{\!} \pahtw {}_{\!} f_{\II}\!||_{L^6}\lesssim ||\pa_y(\FD^\ell \pahtu \! f_{\!I}\!-\FD^\ell \pahtw {}_{\!} f_{\II})||_{L^2}
\leq ||\pahtu \FD^\ell \pahtu f_I\!-\pahtw \FD^\ell \pahtw f_{\II}\!||_{L^2}+||(\widehat{A}_{\II\, i}^{\,\,\,a}\!-\widehat{A}_{{}_{\!}I\, i}^{\,\,a})\pa_{a} \FD^\ell \pahtw f_{\II}||_{L^2},
\end{equation}
which is again of the form that we control by Lemma \ref{diffellphi}.
To deal with the first term in \eqref{diffI}, one writes $\FD^{r-1}\pahtw{}_{\!i} f_{\II}\!=\!\FD^{r-1}\pahtu{}_{\!i} f_{\II}+\FD^{r-1}[(\widehat{A}_{\II\, i}^{\,\,\,a}\!-\widehat{A}_{{}_{\!}I\, i}^{\,\,a})\pa_{a}f_{\II}]$ and $\pahtw{}_{\!i}\FD^{r-1}  f_{\II} \!=\!\pahtu{}_{\!i}\FD^{r-1} f_{\II}+(\widehat{A}_{\II\, i}^{\,\,\,a}\!-\widehat{A}_{{}_{\!}I\, i}^{\,\,a})\pa_{a}\FD^{r-1} f_{\II}$, and
\begin{multline}
I\!=\!\int_{\Odo}\!\!\!\! \delta^{ij}\big(\pahtu{}_{\!i} \FD^{r-1} (f_I\!-f_{\II})\big)\big(\FD^{r-1} \pahtu{}_{\!j} (f_I\!- f_{\II})\big)dy
+ \int_{\Odo}\!\!\!\! \delta^{ij}\big(\pahtu{}_{\!i} \FD^{r-1} (f_I\!-f_{\II})\big)\big(\FD^{r-1} \big[(\widehat{A}_{\II\, j}^{\,\,\,a}-\widehat{A}_{{}_{\!}I\, j}^{\,\,a})\pa_{a} f_{\II}\big]\big)dy\\
+\!\int_{\Odo}\!\!\!\!\!\!\delta^{ij}\big((\widehat{A}_{\II\, i}^{\,\,\,a}\!-\widehat{A}_{{}_{\!}I\, i}^{\,\,a})\pa_{a} \FD^{r-1\!} f_{\II}\!\big)\big(\FD^{r\!-\!1\!} \pahtu{}_{\!j} (f_{\!I\!}- f_{\II}\!)\big)dy
+\!\int_{\Odo}\!\!\!\!\!\!\delta^{ij}\big((\widehat{A}_{\II\, i}^{\,\,\,a}\!-\widehat{A}_{{}_{\!}I\, i}^{\,\,a})\pa_{a} \FD^{r\!-\!1\!} f_{\II}\!\big)\big(\FD^{r\!-\!1\!} \big[(\widehat{A}_{\II\, j}^{\,\,\,a}-\widehat{A}_{{}_{\!}I\, j}^{\,\,a})\pa_{a} f_{\II}\!\big]\big)dy.
\end{multline}
It is straightforward to control the last three terms, and the first term is equal to:
\begin{equation}
\int_{\Odo}\!\!\!\! \delta^{ij}\pa_{a}\big((\widehat{A}_{{}_{\!}I\, i}^{\,\,a} \FD^{r\!-\!1} \! (f_{\!I}\!- \! f_{\II})\big)\FD^{r\!-\!1} \pahtu{}_{\!j} (f_{\!I}\!-\!f_{\II})\big) dy-\!\int_{\Odo}\!\!\!\!\delta^{ij} (\pa_{a}\widehat{A}_{{}_{\!}I\, i}^{\,\,a})\big( \FD^{r\!-\!1 \!} (f_{\!I}\!- \! f_{\II})\big)\big(\FD^{r-1\!} \pahtu{}_{\!j} (f_{\!I}\!-\! f_{\II})\big)\,dy.\,\,\,
\label{diffII}
\end{equation}
Integrating the first term by parts gives:
\begin{equation*}
\int_{\Odo} \delta^{ij}\underbrace{\big( \FD^{r-1} \!( f_I\!- f_{\II})\big)\pahtu{}_{\!i}\big(\FD^{r-1\!} \pahtu{}_{\!j} (f_I\!- f_{\II})\big)\, dy}_{II}+\!\!\int_{\p\Odo}\delta^{ij}\underbrace{N_a \widehat{A}_{{}_{\!}I{}_{\,} i}^{\,\,a} \big( \FD^{r-1\!}  (f_I\!-  f_{\II})\big)\big(\FD^{r-1} \pahtu{}_{\!j} (f_I\!- f_{\II})\big)\, dy}_{\mathcal{B}}.
\end{equation*}
 Here,  modulo controllable error terms, $II$ is equal to
$
\int_{\Odo} (\FD^r f_I-\FD^r f_{\II})(\FD^r g_I-\FD^r g_{\II})\,dy.
$
When $\FD^{r-1}$ contains at least one $\T$, one can integrate this $\T$ by parts and control the resulting integral as what is done to the control of \eqref{I_1} in the proof of Theorem \ref{ellphi'' 1}.
When $\FD^{r-1}=D_t^{r-1}$, this is bounded by $||D_t^{r-1} f_I-D_t^{r-1} f_{\II}||_{L^2(\Odo)}||D_t^{r-1} g_I-D_t^{r-1} g_{\II}||_{L^2(\Odo)}$, where $||D_t^{r-1} f_I-D_t^{r-1} f_{\II}||_{L^2(\Odo)}$ can be controlled by Lemma \ref{D_t^r-1 F-G lem}. The second term in \eqref{diffII} can be controlled in a similar way. On the other hand, since $\widehat{A}_{{}_{\!}I{}_{\,} i}^{\,\,a}=\delta^{a}_i$ on $\pa\Odo$, $\mathcal{B}$ can be controlled appropriately using the Lemma \ref{bdy estimate diff}.
 \end{proof}

\section{Estimates for commutators and $\F$} \label{commutatorsec}

In this section, we fix a vector field $V = V(t,y)$ on $\Omega$. We let
$x(t,y)$ denote the flow of $V(t,y)$, i.e. $D_t x=V$, $x|_{t=0}=x_0$,
and let $\xve(t,y)$ denote the tangentially smoothed flow, as in
\eqref{xvedef}. We suppose that
the mapping $y \mapsto \xve(t,y)$ is invertible for each $t$,
and we let $ A^i_{\m a}$ and $ A^a_{\m i}$
be the Jacobian matrix of $\xve$ and its inverse, respectively, see \eqref{udef}.
We will assume that $\xve$ and $V$ satisfy the bounds \eqref{uwassump}.

If $M^i_{\m a}$ is an invertible matrix with inverse $N_{\m i}^a$,
we recall the formula for the derivatives of $N_{\m i}^a$:
\begin{equation}
  D N_{\m i}^a = - N_{\m i}^b N_{\m j}^a \big(D M^j_{\m b}\big),
\end{equation}
where here $D = D_t$ or $D = \pa_c$. When $M^i_{\m a} = A^i_{\m a}$,
then this gives:
\begin{align}
  D_t A^a_{\m i} = - A_{\m j}^a \big(\pave_i \ssm V^j\big),
  &&
  \pa_c A^a_{\m i} = -A_{\m j}^a \big(\pave_i u^j_c\big).
  \label{dinv}
\end{align}

Using these formulas it is straightforward to calculate the following
commutators:
\begin{align}
  [D_t, \pave_i] & = -A_{\m i}^b A_{\m j}^a \big(\pa_b \ssm V^j\big) \pa_a
  = -\big(\pave_i \ssm V^j\big) \pave_j, \label{dtpavecomm}\\
  [\pa_c, \pave_i] &= -A_{\m i}^b A_{\m j}^a \big(\pa_c A^j_{\m b}\big) \pa_a
  =- \big(\pave_i A_{\m c}^j\big) \pave_j.
  \label{papavecomm}
\end{align}

We will need estimates for higher order derivatives of $A_{\m i}^a$.
As in section \ref{def T and FD},
given a set $\U = \{T_1,..., T_N\}$ of vector fields,
we write $\U^r = \U \times \cdots \times \U$ ($r$ times)
as well as
$\U^r V:\Omega \to \R^{3N+3}$.
The families of vector fields we will consider are $\U = \T$ (tangential
derivatives, $\U = \FD$ (mixed tangential and time derivatives),
$\U = \coord$ (mixed full space and time derivatives), and $\U = \{\pa_y\}$.
The point of the below estimate is just that derivatives of $A$ behave like
derivatives of $\pa_y \xve$. This lemma is in fact essentially the same
as Lemma \ref{vectcommlemma} but it is convenient to note this
estimate separately.
\begin{lemma}
  \label{inverselemma}
  With notation as in Section \ref{def T and FD},
  if $T^I \in \U^s$ where $\U = \T$,$\FD$,$\coord$,
  or $\{\pa_y\}$, then:
  \begin{equation}
   ||T^I\! A_{\m i}^a||_{L^2}
   + ||T^I\!\! g^{ab}||_{L^2}
   \leq C(M)\big(||T^I \xve||_{H^{1}} + P(||\U^{s-2} \xve||_{H^2})\big)
   \label{app:dxtu}
  \end{equation}
\end{lemma}
We note that taking $\U = \{\pa_y\}$ and summing over all $T^I \in \U^s$
gives:
 \begin{equation}
  ||A_{\m i}^a||_{H^s} + ||g^{ab}||_{H^s(\Omega)}
  \leq C(M)\big( ||\xve||_{H^{s+1}}
  + P( ||\xve||_{H^{s}})\big).
  \label{app:dxu}
 \end{equation}

\begin{proof}
The estimates for $g$ follow
from the estimates for $A$ and the definition $g^{ab} = \delta^{ij}
A_{\m i}^a A_{\m j}^b$ so we just prove the estimates for $A$.
For the sake of simplicity we will assume that all $T \in \U$ commutes
with $\pa_y$; this is only not the case if $\U = \T$ and in that
case the commutator is lower order and can be handled using similar arguments
to the below.
For $T^I \in \U^s$, repeatedly applying \eqref{dinv}, we have:
\begin{equation}
 T^{I} A_{\m i}^a =
 -A_{\m i}^b A_{\m j}^a T^I \pa_b \xve^j -
 {\sum}
 \big(\pave T^{I_1} \xve \big) \cdots \big(\pave T^{I_k} \xve\big)
 \label{sums}
\end{equation}
where the sum is taken over
a collection of multi-indices $I_1,..., I_k$
with $|I_1| + \cdots +  |I_k| = s$
with $|I_j| \leq s-1$ for $j = 1,..., k$.
The first term is bounded by the first term on the
right-hand side of \eqref{app:dxu}.
When $s \leq 3$, we bound the first $k-1$ factors in each summand in $L^\infty$
by $C(M)$ and the remaining factor in $L^2$ by $||\U^{s-1} \xve||_{H^1}$
and this is bounded by the right-hand side of \eqref{app:dxu} for all the
values of $\U$ we are considering.
We now assume that $s \geq 4$.
If any index $|I_j| \leq \min(3, s-3)$, we use the Sobolev estimate \eqref{vectsob}
to bound
$||\pave T^{I_{\!j\!}} x||_{L^{\!\infty}}\! \leq C(_{{}_{\!}}M_{{}_{\!}}) ||\pa_y T^{I_{\!j\!}} x||_{L^{\!\infty}}\!
\leq C(_{{}_{\!}}M_{{}_{\!}}) ||\U^{s-\!2} x||_{H^2}$. Therefore it suffices to deal with the
case when at least one index $ |I_j| \geq \max(4, s-4)$.
There can be at most one such index because if there are $\ell \geq 2$
such terms then  $4\ell \leq s $ so that $s \geq 8$
and that $\ell(s-4) \leq s $ so that $s \leq 4$. Since there is one such
index and $|I_j| \leq s-1$ we bound the corresponding term in $L^2$ by
$||\U^{s-1} \xve||_{H^1}$ which completes the proof.
 \end{proof}

Similarly, we have:
\begin{lemma}
  \label{inversedifflemma}
 Define $\xveu, \xvew, A_I, A_{\II}, \gveu, \gvew$ as in Appendix \ref{elliptic}.
 With notation as in Lemma \ref{inverselemma}, if $T^I \in \V^s$:
 \begin{equation}
   ||T^I\!(_{\!}A_{{}_{\!}I\, i}^{\,\,a}\! -\! A_{\II\, i}^{\,\,\,a})||_{L^2}\! + \! ||T^I\!(\gveu^{ab}\!\!- \!\gvew^{ab})||_{L^2}
 \! \leq \D_{\!s}
  ||\xveu\! - \xvew\!||_{H^{\ell+1}},\qquad
  \D_{\!s} \!=\! \D_{\!s}(M{}_{\!},_{\!} ||\U^{s-\!1\!}\xveu||_{H^{2}}, ||\U^{s-\!1\!}\xvew||_{H^{2}}).
  \label{dinvdiff}
 \end{equation}
\end{lemma}
\begin{proof}
  Applying \eqref{dinv} to $A_{{}_{\!}I\, a}^{\,\,i}$ and $A_{\II\, a}^{\,\,\,i}$ generates two sums of the
  form \eqref{sums}. Subtracting these two sums and arguing as in
  the proof of the previous lemma gives \eqref{dinvdiff}.
\end{proof}

The next lemma will be used at several places. Recall the definitions of
 $\Odo$, $\pahtu, \pahtw$ from Section \ref{extensionsec}.
 \begin{lemma}
   \label{commutator FD pave lemma diff}
Let with $r\geq 5$.
Then there is a continuous function
$$C_r = C_r\big(M', ||\xveu||_{H^r(\Omega)}, ||\xvew||_{H^r(\Omega)}, {\tsum}_{\ell\leq r-1}||D_t^\ell\Vu||_{H^{r-\ell}(\Omega)}, {\tsum}_{\ell\leq r-1}||D_t^\ell\Vw||_{H^{r-\ell}(\Omega)}\big).
$$
 such that with $\FD^r$ the mixed space-time tangential derivatives defined in Section \ref{def T and FD}:
 \begin{multline}
   ||[\FD^r,\pahtu] f - [\FD^r, \pahtw] g||_{L^2(\Odo\!)}\\
   \leq C_r \bigtwo(  {\sum}_{\ell \leq r-1} ||\FD^\ell \pahtu f - \FD^\ell \pahtw g||_{L^2(\Odo\!)}
   + ||\T \xveu||_{H^r(\Omega)}\big\{||\pahtu f - \pahtw g||_{L^\infty(\Odo\!)}
   +
   {\sum}_{\ell \leq 2} ||\FD^\ell \pahtu f - \FD^\ell \pahtw g||_{L^6(\Omega^{d_0}\!)}\big\}\\
   + \big\{||\T \xveu - \T \xvew||_{H^r(\Omega)}
   +{\sum}_{\ell\leq r-1}||D_t^{\ell}\big((\Vve_{\I}-\Vve_{\II}\big)||_{H^{r-\ell}(\Omega)}\big\}
   (||\pahtw g||_{L^\infty(\Odo\!)}+{\sum}_{\ell\leq 2}||\FD^{\ell}\pahtw g||_{L^6(\Odo\!)})\\
   +\big\{||\xveu -\xvew||_{H^r(\Omega)}+{\sum}_{\ell\leq r-1}||D_t^{\ell}\big(\Vve_{\I}-  \Vve_{\II}\big)||_{H^{r-\ell}(\Omega)}\big\}{\sum}_{\ell\leq r-1}||\FD^\ell \pahtw g||_{L^2(\Odo\!)}\bigtwo).
  {}
 \end{multline}
\end{lemma}
\begin{proof}
  We start by writing:
\begin{multline}
[\FD^r,\pahtu] f - [\FD^r, \pahtw] g = \underbrace{-\big((\pahtu \FD^r\xhtu)\pahtu f-(\pahtw \FD^r\xhtw)\pahtw g\big)}_{I}
\\+{\sum}_{\ell_1+\cdots+\ell_s=r,\ell_i\leq r-1}\underbrace{
(\pa \FD^{\ell_1}\xhtu)\cdots(\pa \FD^{\ell_{s-1}}\xhtu)(\FD^{\ell_s}\pahtu f)
-(\pa \FD^{\ell_1}\xhtw)\cdots(\pa \FD^{\ell_{s-1}}\xhtw)(\FD^{\ell_s}\pahtw g)}_{II}.
\end{multline}
We have:
\begin{multline}
||I||_{L^2(\Odo)}\leq ||(\pahtu \FD^r\xhtu-\pahtw \FD^r\xhtw)\pahtw g||_{L^2(\Odo)}+||(\pahtu \FD^r\xhtu)(\pahtu f-\pahtw g)||_{L^2(\Odo)}\\
\leq ||\pahtu \FD^r\xhtu-\pahtw \FD^r\xhtw||_{L^2(\Odo)}||\pahtw g||_{L^\infty(\Odo)}+ ||\pahtu \FD^r\xhtu||_{L^2(\Odo)}||\pahtu f-\pahtw g||_{L^\infty(\Odo)}\\
\leq \bigtwo(||\T (\xhtu-\xhtw)||_{H^r(\Odo)}+{\sum}_{\ell\leq r-1}||D_t^{\ell}(\Vhtu-\Vhtw)||_{H^{r-\ell}(\Odo)}\bigtwo)||\pahtw g||_{L^\infty(\Odo)}\\
+\bigtwo(||\T \xhtu||_{H^r(\Odo)}+{\sum}_{\ell\leq r-1}||D_t^{\ell}\Vhtu||_{H^{r-\ell}(\Odo)}\bigtwo)||\pahtu f-\pahtw g||_{L^\infty(\Odo)}.
\end{multline}
In addition, to control $II$ in $L^2$ one only needs to consider
\begin{equation}
II_1=(\pa \FD^{\ell_1}\xhtu)\cdots(\pa \FD^{\ell_{s-1}}\xhtu)(\FD^{\ell_s}\pahtu f-\FD^{\ell_s}\pahtw g),
\end{equation}
\begin{equation}
II_2=(\pa \FD^{\ell_1}\xhtu-\pa \FD^{\ell_1}\xhtw)(\pa \FD^{\ell_2}\xhtw)\cdots(\pa \FD^{\ell_{s-1}}\xhtw)\FD^{\ell_s}\pahtw g.
\end{equation}
When $r\!-\!1\!\geq\! \ell_s\!\geq\! 3$ then $\ell_j\!\leq\! r\!-\!3$ for $j\!\leq\!s\!-\!1$ and we control the terms involving $\xht$ in $L^\infty$. Hence
\begin{equation}
||II_1||_{L^2(\Odo)}\leq C_r'||\FD^{\ell_s}\pahtu f- \FD^{\ell_s}\pahtw g||_{L^2(\Odo)},
\end{equation}
where
$$C_r' = C_r'\big(M', ||\xhtu||_{H^r(\Odo)}, ||\xhtw||_{H^r(\Odo)}, {\tsum}_{\ell\leq r-1}||D_t^\ell\Vhtu||_{H^{r-\ell}(\Odo)}, {\tsum}_{\ell\leq r-1}||D_t^\ell\Vhtw||_{H^{r-\ell}(\Odo)}\big),$$
and
\begin{multline}
||II_2||_{L^2(\Odo)} \leq C_r' ||\pa \FD^{\ell_1}\xhtu-\pa \FD^{\ell_1}\xhtw||_{L^\infty(\Odo)}||\FD^{\ell_s}\pahtw g||_{L^2(\Odo)}\\
\leq C_r'\bigtwo( ||\xhtu - \xhtw||_{H^r(\Odo)}+{\sum}_{\ell\leq r-1}||D_t^{\ell}(\Vhtu-  \Vhtw)||_{H^{r-\ell}(\Odo)}\bigtwo)||\FD^{\ell_s}\pahtw g||_{L^2(\Odo)}.
\end{multline}
Second, when $\ell_j\geq r-2$ for $j=1,\cdots, s-1$, since $r\geq 5$, there is at most one $\ell_j$, say $\ell_1$, can be greater than or equal to $r-2$. In this case, we have:
\begin{multline}
||II_1||_{L^2(\Odo)}\leq C_r'||\pa \FD^{\ell_1}\xhtu||_{L^3(\Odo)}||\FD^{\ell_s}\pahtu f-\FD^{\ell_s}\pahtw g||_{L^6(\Odo)} \\
\leq C_r'\bigtwo(||\T \xht_I||_{L^2(\Odo)}+{\sum}_{\ell\leq r-1}||D_t^{\ell}\Vhtu||_{H^{r-\ell}(\Odo)}\bigtwo)||\FD^{\ell_s}\pahtu f-\FD^{\ell_s}\pahtw g||_{L^6(\Odo)},
\end{multline}
where $\ell_s\leq 2$, and
\begin{multline}
||II_2||_{L^2(\Odo)}\leq C_r'||\pa \FD^{\ell_1}\xhtu-\pa \FD^{\ell_1}\xhtw||_{L^3(\Odo)}||\FD^{\ell_s}\paht_{II} g||_{L^6(\Odo)}\\
\leq C_r'\bigtwo( ||\T(\xhtu - \xhtw)||_{H^r(\Odo)}+{\sum}_{\ell\leq r-1}||D_t^{\ell}(\Vhtu-  \Vhtw)||_{H^{r-\ell}(\Odo)}\bigtwo)||\FD^{\ell_s}\paht_{II} g||_{L^6(\Odo)}.
\end{multline}
This concludes the proof after adapting the Sobolev extension theorem.
\end{proof}

As a consequence, if we take
 $g = 0$, $\xveu = \xvew \equiv \xve$, we have:
\begin{cor}
\label{commutator FD pave lemma}
If $r\!\geq\! 5$
 then there is a constant $C_r\! =\! C_r(M_0, ||\xve||_{H^r(\Omega)}, {\sum}_{\ell\leq r-1}||D_t^\ell V||_{H^{r-\ell}(\Omega)})$ such that:
 \begin{equation}
   ||[\FD^r,\paht] f||_{L^2(\Odo)}\leq C_r\bigtwo(||\T\xve||_{H^r(\Omega)}\big(||\paht f||_{L^\infty(\Odo)}
   +{\sum}_{\ell\leq 2}||\FD^\ell\paht f||_{L^6(\Odo)}\big)
   +{\sum}_{\ell\leq r-1}||\FD^\ell \paht f||_{L^2(\Odo)}\bigtwo).
  \label{commutator FD pave}
 \end{equation}
 Here $\FD^r$ be the mixed space-time tangential derivatives defined in Section \ref{def T and FD}.
 In particular, one has:
 \begin{equation}
   ||[\FD^{r-1}D_t,\paht] f||_{L^2(\hD_t)}\leq C_r\bigtwo(||\paht f||_{L^\infty(\hD_t)}+{\sum}_{\ell\leq 2}||\FD^\ell\paht f||_{L^6(\hD_t)}
   +{\sum}_{\ell\leq r-1}||\FD^\ell \paht f||_{L^2(\hD_t)}\bigtwo).
 \end{equation}
\end{cor}

The following
lemma is similar to the previous one but is better adapted to proving
estimates for the wave equation. As in the previous lemma, the point is that
the commutator between $r$ derivatives and $\pave$ is a differential operator
of order $r$ with coefficients depending on $r+1$ derivatives of $\xve$.

In the following lemma, we will assume that we have the following a priori
bound for $\xveu, \xvew$:
\begin{equation}
 {\tsum}_{|I| + k \leq 3} |D_t^k \pa_y^I \xveu|
 + |D_t^k \pa_y^I \xvew| \leq  M.
 \label{Massumpapp}
\end{equation}
If we are considering vector fields which do not involve time derivatives,
we can instead assume that only:
\begin{equation}
 {\tsum}_{|I| \leq 3} |\pa_y^I \xveu|
 + |D_t^k \pa_y^I \xvew| \leq  M_0.
 \label{M0assumpapp}
\end{equation}
\begin{lemma}
  \label{vectcommlemma}
  Fix $s \geq 0$ and suppose that \eqref{Massumpapp} holds.
  If $\U = \coord, \FD$ or $\U = \T$, there is a constant
   $C_s = C_s(M, ||\U^{s-2} \xveu||_{H^2(\Omega)},
   ||\U^{s-2} \xvew||_{H^2(\Omega)})$ so that if
   $T^J \in \U^s$,
   with notation as in Section \ref{def T and FD},
   then:
  \begin{multline}
   ||[T^J, \paveu] f - [T^J, \pavew] g||_{L^2}
    \leq
C_s\big( ||T^J \xvew||_{H^1} + 1\big)
{\sum}_{j \leq s} ||\U^{j-2} (\paveu f - \pavew g)||_{H^1}
\\ +
C_s\big(
|||T^J(\xveu - \xvew)||_{H^1} +
||\U^{s-1} (\xveu - \xvew)||_{H^1}+
||\U^{2} (\xveu - \xvew)||_{C^1_{y,t}}\big)
{\sum}_{j \leq s} ||\U^{j-2} \pavew g||_{H^1},
\label{vectcomm}
\end{multline}
with $H^s = H^s(\Omega)$ and $||\alpha||_{C^1_{y,t}} =
\sum_{k + |J| \leq 1} ||\pa_y^J D_t^k \alpha||_{L^\infty([0,T] \times \Omega)}$.
If $\U = \{\pa_{y^1}, \pa_{y^2}, \pa_{y^3}\}$, the above estimate holds
assuming \eqref{M0assumpapp} holds with $M$ replaced by $M_0$.
\end{lemma}
Before proving this lemma we record a few useful instances of it which will be
used at several points. Taking $g= 0$ and writing $\xve =\xveu$,
 with $C_s' =C_s'(M, ||\T^{s-2} \xve||_{H^2(\Omega)})$, we have:
\begin{alignat}{2}
 ||[\pa_y^I, \pave ] f||_{L^2(\Omega)}
 &\leq C_s(M_0, ||\xve||_{H^{s+1}}) ||\pave f||_{H^{s-1}(\Omega)},
 \quad && |I| = s,\\
 ||[T^J, \pave]f||_{L^2(\Omega)}
 &\leq C_s'
(||T^J \xve||_{H^1(\Omega)} + 1){\sum}_{j \leq s-1}||\T^{j}\pave f||_{H^1(\Omega)},
 &&\quad T^J \!\!\in \T^s\! .
\end{alignat}
\begin{proof}
 Using \eqref{dinv}, we have:
 \begin{multline}
[T^J,\paveu] f - [T^J, \pavew] g =
\underbrace{-\big((\paveu T^J\xveu)\paveu f-( T^J \xvew)\pavew g\big)}_{I}
\\+{\sum}_{J_1+\cdots+J_m = J, |J_i| \leq s-1}\underbrace{
(\pa T^{J_1}\xveu)\cdots(\pa T^{J_{m-1}}\xveu)T^{J_m}\paveu f
-(\pa T^{J_1}\xvew)\cdots(\pa T^{J_{m-1}}\xvew)T^{J_m}\pavew g}_{II},
\label{Istep0}
\end{multline}
We control:
\begin{equation}
 ||(\paveu{}_{\!} T^J \!\xveu)\paveu\! f{}_{\!} - (\pavew {}_{\!} T^J\! \xvew)\pavew g||_{L^2(\Omega)}
 \leq || (\paveu {}_{\!} T^J\! \xveu - \pavew {}_{\!} T^J \!\xvew)\pavew g||_{L^2(\Omega)}
 + ||(\pavew {}_{\!} T^J \!\xvew) (\paveu\! f \!- \pavew g)||_{L^2(\Omega)}.
 \label{firstterm}
\end{equation}
If $|J| \leq 2$, then we control the factors involving $\xveu, \xvew$ in
$L^\infty$ and the result is bounded by the right-hand side of \eqref{vectcomm}.
If instead $|J| \geq 3$, we control the factors involving
$f,g$ in $L^\infty$ and note that since we must have $s \geq 3$, by
the Sobolev estimate \eqref{vectsob}, we have:
\begin{equation}
 ||\pavew g||_{L^\infty(\Omega)}
 \leq C{\sum}_{j \leq 2} ||\U^j \pavew g||_{H^1(\Omega)}
 \leq C{\sum}_{j \leq s} ||\U^j \pavew g||_{H^1(\Omega)},
\end{equation}
which is bounded by the right-hand side of \eqref{vectcomm}. Bounding
$||\paveu f - \pavew g||_{L^\infty(\Omega)}$ in the same way shows
that the left-hand side of \eqref{firstterm} is controlled by
the right-hand side of \eqref{vectcomm}.

To control $II$, it suffices to consider
\begin{equation}
II_1\!=\!(\pa T^{J_1\!}\xveu-\pa T^{J_1\!}\xvew\!)(\pa T^{J_2\!}\xvew\!){}_{\!}\cdots{}_{\!}(\pa T^{J_{m\!-\!1}\!}\xvew\!)
T^{J_m}\pavew g,
\qquad
II_2\!=\!(\pa T^{J_1\!}\xveu\!){}_{\!}\cdots{}_{\!}(\pa T^{J_{m\!-\!1}\!}\xveu\!)T^{J_m\!}(\paveu\! f\!-
\pavew g),
\end{equation}
where $J_1\! +\! ... + \!J_m \!=\! J$ and $|J_{\!1\!}|,..., |J_{m\!}| \leq \!s\!-\!1$.
We will just bound $||II_2||_{L^2(\Omega)}$, since the estimate for
$||II_{1}||_{L^2(\Omega)}$ is similar.
We start by noting that for each
$J_i$ with $|J_i| \leq 2, i \leq m-1$, we control the corresponding
factors of $\xveu$ in $L^\infty$ by the right-hand side of \eqref{vectcomm}.
Rearranging indices, it therefore suffices to control:
\begin{equation}
 ||(\paveu T^{J_1} \xveu)\cdots
 (\paveu T^{J_\ell} \xveu) (T^{J_m} \paveu (f-g))||_{L^2(\Omega)},
 \quad
 |J_1|, ..., |J_\ell| \geq 3, |J_1| + \cdots + |J_\ell| + |J_m| \leq s-1.
\end{equation}
If there are no factors of $\xveu$ present then the result is bounded
by $||\pave T^{J_m} (f-g)||_{L^2(\Omega)}$ and since $J_m \leq s-1$ we control
this by the right-hand side of \eqref{vectcomm}.
If there is at least one factor of $\xveu$ present,
Note that the conditions on the $|J_k|$ force $|J^m| \leq s-4$ and so
we control the last factor in $L^\infty$ by $\sum_{j \leq s-2}
||\V^j \pave (f-g)||_{H^1(\Omega)}$
by the Sobolev estimate \eqref{vectsob}. We now use Holder's inequality
and the Sobolev embedding \eqref{sob2} to control:
\begin{equation}
||\paveu T^{J_{{}_{\!}1}}\xveu\!\cdots \paveu T^{J_\ell} \xveu||_{L^{{}_{\!}2}(\Omega)}
\!\leq C
||\paveu T^{J_1}\xveu||_{L^{{}_{\!}2\ell}(\Omega)}
\!\cdots
||\paveu T^{J_\ell}\xveu||_{L^{{}_{\!}2\ell}(\Omega)}
\!\leq C ||\paveu T^{J_1}\xveu||_{H^{1{}_{\!}}(\Omega)}
\!\cdots ||\paveu T^{J_\ell} \xveu||_{H^{1{}_{\!}}(\Omega)}.
 {}
\end{equation}
We now note that since $|J_1| + ... + |J_\ell| \leq s-1$ and $|J_k| \geq 3$
for each $k = 1,..., \ell$, we in fact have $|J_k| \leq s-2$ for each $k$,
and so each of these factors is controlled by the right-hand side
of \eqref{vectcomm}.
\end{proof}

We also need a version  with pure time derivatives
in the proof of the estimates for the wave equation.
\begin{lemma}
 Fix $s \geq 0$. If \eqref{ubd2} holds, there is a constant
  $C_s = C_s(M, ||\xve||_s, ||V||_{\X^s})$ so that
 \begin{equation}
  || [D_t^{s+1}, \pave] f||_{L^2(\Omega)}
  \leq C_s
  (||\pave f||_{s,0} +
  (||V||_{\X^{s+1}} + 1)||\pave f||_{s-1}).
  \label{vectcommdt}
 \end{equation}
\end{lemma}
\begin{proof}
 By \eqref{Istep0} with $g = 0$ and $T^J = D_t^{s+1}$, we have:
 \begin{equation}
  [D_t^{s+1}, \pave_i] f =
  -(\pave D_t^{s+1} \xve)(\pave f)
  +{\sum}_{s_1 + ... + s_m = s+1,\,\, s_m\geq 1} (\pa D_t^{s_1} \xve)
  \cdots (\pa D_t^{s_{m-1}}\xve)
  (D_t^{s_m} \pave f).
  {}
 \end{equation}
 We now argue as in the previous lemma, but we want to
 point out explicitly how the norms of $V$ arise. We write the
 first term as $-(\pave D_t^s \ssm V)(\pave f)$. If $s \leq 2$ then
 we control the first factor in $L^\infty$ since $\ssm: L^\infty \to L^\infty$.
 If instead $s \geq 3$,
 we control the second factor using Sobolev embedding,
 $||\pave f||_{L^\infty(\Omega)} \leq C||\pave f||_{H^2(\Omega)}
 \leq C||\pave f||_{s-1}$, and now we note that
 $||\pave D_t^s \ssm V||_{L^2(\Omega)} \leq C(M) ||V||_{\X^{s+1}}$,
 using that $\ssm:L^2\to L^2$.

 The terms in the sum can be controlled using essentially the same argument
 as in the previous lemma. Rearranging indices it suffices to control:
 \begin{equation}
  ||(\pave D_t^{s_1}\xve)\cdots (\pave D_t^{s_{j}} \xve)
  (D_t^{s_m}\pave f)||_{L^2(\Omega)},
  \quad s_\ell \geq 3, \ell = 1,..., j, s_1 + \cdots s_j + s_m \leq s,
  s_m \geq 1.
  {}
 \end{equation}
 If there are no factors of $\xve$ present then we control this
 by $||D_t^{s_m}\pave f||_{L^2(\Omega)} \leq ||\pave f||_{s,0}$.
 and if there is at least one factor of $\xve$ present then
 we must have $s_m \leq s-3$ and so we can control
 $||D_t^{s_m}\pave f||_{L^\infty(\Omega)}
 \leq C||D_t^{s_m}\pave f||_{H^2(\Omega)}
 \leq C||\pave f||_{s-1}$. When $j = 1$ the result
 is obvious since $||\pave D_t^{s_1} \xve||_{L^2(\Omega)}
 \leq C||\xve||_{s_1 + 1}$ and $s_1 \leq s-1$.
 When $j \geq 2$ we have by Sobolev embedding \eqref{sob2}:
 \begin{equation}
  ||(\pave D_t^{s_1}\xve)\cdots (\pave D_t^{s_j} \xve)||_{L^2(\Omega)}
  \leq C ||\pave D_t^{s_1}\xve||_{L^{2j}(\Omega)}
  \cdots||\pave D_t^{s_j}\xve||_{L^{2j}(\Omega)}
  \leq C ||\pave D_t^{s_1}\xve||_{H^1(\Omega)}
  \cdots ||\pave D_t^{s_j} \xve||_{H^1(\Omega)}.
  {}
 \end{equation}
 Since each
 $s_\ell$ must satisfy $s_\ell \leq s-3$, each of these factors
 is bounded by $C(M)||\xve||_{s-1}$, as required.
\end{proof}
We also need to use the following commutator estimates in $\tD_t$.
\begin{lemma} \label{commutators in Eulerian coordinate}
Let $r \geq 7$ and $k+\ell \leq r+1$ with $k\geq 2$, we have:
\begin{equation}
||[D_t^{k-1}, \pave]\varphi||_{H^\ell(\D_t)} \leq P\big({\tsum}_{s\leq k-2}||D_t^s\ssm V||_{H^{r-s}(\tD_t)}\big){\sum}_{s\leq k-2}||D_t^s\varphi||_{H^{r-s}(\tD_t)}\label{ucommEu},
\end{equation}
and for $k+\ell=r$, we have:
\begin{equation}
||[D_t^{k-1}, \Dve]\varphi||_{H^\ell(\D_t)} \leq P\big({\tsum}_{s\leq k-2}||D_t^s\ssm V||_{H^{r-s}(\tD_t)}\big){\sum}_{s\leq k-2}||D_t^s\varphi||_{H^{r-s}(\tD_t)}\label{gvecommEu}.
\end{equation}
\end{lemma}
\begin{proof}
It is not hard to compute that $[D_t^{k-1}, \pave]$ consists of terms of the following forms:
\begin{align}
(\pave D_t^{s_1} S_\ve V)\cdots(\pave D_t^{s_{n-1}} S_\ve V)(\pave D_t^{s_n} \varphi),\quad\text{with}\quad
s_1\!+\cdots+s_n=k-n,\,\,\,n \geq 2,
\end{align}
so \eqref{ucommEu} follows. On the other hand, \eqref{gvecommEu} follows that  $[D_t^{k-1}\!\!, \Dve]$ consists of terms of the following form:
\begin{align}
(\pave D_t^{s_3} S_\ve V)\cdots (\pave D_t^{s_n} S_\ve V)(\pave^2 D_t^{s_1} S_\ve V)(\pave D_t^{s_2} \varphi),
\quad\text{with}\quad {s_1+\cdots+s_n=k-n,\,\,\,n\geq 2},\\
(\pave D_t^{s_3} S_\ve V)\cdots (\pave D_t^{s_n} S_\ve V)(\pave D_t^{s_1} S_\ve V)(\pave^2 D_t^{s_2} \varphi),
\quad\text{with}\quad {s_1+\cdots+s_n=k-n,\,\,\,n\geq 2}.\tag*{\qedhere}
\end{align}
\end{proof}

We now prove some estimates which are used in Sections \ref{waveests}
and \ref{enthsec} to control the terms on the right-hand side
of the various wave equations.
For these estimates we will assume the following bound
for $\varphi$:
\begin{equation}
  {\tsum}_{k + |J| \leq 3}  |D_t^k \pa_y^J \pave \varphi|
  + |D_t^k \varphi| \leq L.
 \label{Lassumpapp}
\end{equation}
\begin{lemma}
  \label{gestlem}
 If the equation of state
 satisfies \eqref{eq:closetoincompressible} for all $j \geq k + \ell \equiv s$,
 then there is a constant $C$ depending only on $c_1, c_2$ and a polynomial $P$
 so that:
 \begin{equation}
   || D_t^{k} (e'(\varphi)D_t^2\varphi ) - e'(\varphi) (D_t^{k+2} \varphi)||_{H^\ell}
  \leq
  C L ||D_t^{k+1} \varphi||_{H^{\ell}} +
  P(L, ||\varphi||_{s}).
  \label{nlinest}
 \end{equation}
\end{lemma}
\begin{proof}
  \!\!We just prove the $\ell \!=\! 0$ case since $\ell \!\geq \!1$ is similar.
  The main term in $D_t^{k} (e^{\prime\!}(\varphi)D_t^2\varphi ) - e^{\prime\!}(\varphi) (D_t^{k+2} \varphi)$ is:
\begin{equation}
 e''(\varphi) (D_t \varphi)D_t^{k+1} \varphi,
 \label{maintermg}
\end{equation}
and the remaining terms are of the form
\begin{align}
e^{(m)}(\varphi)D_t^{k_1}\varphi\cdots D_t^{k_m}\varphi,\quad 2\leq m\leq k,
\quad k_1+\cdots+k_m = k+1, \quad \text{ where } k_j \leq k, 1 \leq j \leq m.
\label{remaining}
\end{align}
The term \eqref{maintermg} is bounded by the right-hand side of
\eqref{nlinest}. To bound \eqref{remaining}, we note that if
$k_j \leq 3$ for $1 \leq j \leq m$ all of the terms are bounded by $L$.
If there are any terms with $k_j \leq k-2$ then by Sobolev embedding
we have $||D_t^{k_j} \varphi||_{L^\infty} \leq C ||\varphi||_k$. Therefore
it just remains to consider the case that there is at least one $j$ with
$k_j \geq \max(4, k-1)$ and in fact there can be at most one such term since
we also have $k_j \leq k$ for each $j$. In this case we put the corresponding
factor in $L^2$ and this proves \eqref{nlinest}.
\end{proof}

The following estimate is nearly the same as \eqref{nlinest} but will be
used in Section \ref{wavelwp} to bound quantities of the
form $e'(f) g$ when we know that $f$ is smoother than $g$.
\begin{lemma}
  \label{gest2}
  Under the hypotheses of Lemma \ref{gestlem}, if $k + \ell = s$
  then there is a constant $C$ depending only on $c_1, c_2$
  on is a polynomial $P$ so that if $\varphi$ satisfies \eqref{Lassumpapp}:
  \begin{equation}
   ||D_t^k e'(\varphi)||_{H^\ell}
   \leq C \big( ||D_t^k \varphi||_{H^\ell} + P(L, ||\varphi||_{s-1})).
  \end{equation}
\end{lemma}
\begin{proof}
 The proof is similar to the proof of Lemma \ref{gestlem}. If
 $|I|\! =\! \ell$ then $\pa_y^I D_t^k e'(\varphi)$ is a sum of terms
   \begin{equation}
   e^{m}(\varphi) (\pa_y^{I_1}D_t^{k_1}\varphi)\cdots (\pa_y^{I_m}
   D_t^{k_m}\varphi) \varphi),\qquad
   |I_1| + ... + |I_m| =|I|,\quad
   k_1 + ... + k_m=k.
   \end{equation}
   Using Sobolev embedding as in the proof of Lemma \ref{gestlem}
   gives the result.
\end{proof}

We will also need estimates for the derivatives of $\F = \F^1 + \F^2$, where
\begin{align}
 \F^1 = -(\pave_i \ssm V^j)(\pave_j V^i),
 \qquad
 \F^2 = - e''(h) (D_t h)^2-\rho(h).
\end{align}
\begin{lemma} \label{fbdslem}
  If \eqref{ubd2} holds and $h$ satisfies \eqref{Lassumpapp}, then
  for $k \geq 1$:
 \begin{align}
  ||D_t^k \F_1||_{L^2} &\leq C(M, ||\pa V||_{L^\infty})\big(||D_t^k V||_{H^1} + P(||V||_{\X^k})\big),
  \label{festdt}\\
  ||D_t^k \F_2||_{L^2} &\leq C L ||D_t^{k+1} h||_{L^2}
  + P(L, ||h||_{k,0}).
  \label{festdt2}
 \end{align}
 For $k \geq 0$, writing $k + \ell = s$, we also have:
 \begin{align}
  ||D_t^k \F_1||_{H^\ell} &\leq C(M,||\pa V||_{L^\infty})\big(
  ||D_t^k V||_{H^{\ell+1}} +
   ||\xve||_{H^{\ell+1}} + P(||V||_{s}, ||\xve||_{H^{\ell}})\big),
  \label{festmix}\\
  ||D_t^k \F_2||_{H^\ell} &\leq
  CL ||D_t^{k+1} h||_{H^{\ell}} +
  P(L, ||h||_{s}).\label{festmix2}
 \end{align}

\end{lemma}
\begin{proof}
  We we just prove the estimate \eqref{festmix}. The estimate \eqref{festdt} follows
  in a similar manner and the estimates \eqref{festdt2},
  \eqref{festmix2} follow as in the previous lemma.
The case $k = 0$ can be handled using interpolation and the
estimates \eqref{app:dxu}. When $k\geq 1$, we have:
 \begin{multline}
\pa^{\ell}\! D_t^k\mathcal{F}
=\sum_{l_1+\cdots+ l_n= k, \, n\geq 2} C_{l_1\cdots l_n}^k\pa^\ell\big((\pave_{j_1} D_t^{l_1} S_{\epsilon}V^{j_2})\cdots (\pave_{j_{n-1}} D_t^{l_{n-1}} S_{\epsilon} V^{j_n}) (\pave_{j_n} D_t^{l_n}V^{j_1})\big)\\
=\!\!\!\!\!\!\!\!\!\!\!\!\!\!\!\!\!\!\!\!\!\!\!\!\!\!\!
\sum_{\,\,\,\,\,\,\,\,\,\,\,\,\,\,\,\,\,\,\,\,l_1+\cdots+ l_n= k,\,\,
\sum |\beta_j|+|\gamma_j|=\ell-1\!\!\!\!\!\!\!\!\!} \!\!\!\!\!\!\!\!\!\!\!\!\!\!\!\!\!\!\!\!\!\!\!\!\!\!\!
 \widetilde{C}_{l_{\!1\!}\cdots l_n}^k(\pa^{\beta_{1}}\!\!
 A^{a_1}_{\m j_1}\!)\cdots(\pa^{\beta_{n}{}_{\!}}\! A^{a_n}_{\m j_n})
(\pa_{a_1} \pa^{\gamma_1} \!D_t^{l_1\!} \ssm V^{j_2}\!)\cdots(\pa_{a_{{}_{\!}j_{n\!-\!1}}\!} \pa^{\gamma_{n\!-\!1\!}}D_t^{l_{n\!-\!1\!}} \ssm V^{\!j_n\!})(\pa_{a_n} \pa^{\gamma_n{}_{\!}}D_t^{l_n} V^{j_1}\!),
  \label{fderivexp}
 \end{multline}
where we have used \eqref{dtpavecomm} repeatedly. The leading order term is of the form $$A_{\m i}^a (\pa^\alpha D_t^k \pa_a S_\ve V^j)(\pave_j V^i)+(\pa^\alpha A_{\m i}^a)(D_t^k\pa_aS_\ve V^j) (\pave_j V^j).$$
We bound the first term by:
 \begin{equation}
  C(M')||\pa V||_{L^\infty} ||D_t^k V||_{H^\ell(\Omega)},
  {}
 \end{equation}
and we bound the second term by:
 \begin{equation}
P(||V(t)||_{\X^{s}}, ||\xve(t)||_{H^{r-1}(\Omega)}), \q\text{when}\,\, k\geq 1,
  \quad\text{and}\quad
  C(M')||\pa V||_{L^\infty} ||\xve||_{H^r(\Omega)},\q\text{when}\,\, k=0.
 \end{equation}
The lower order terms in \eqref{fderivexp} is controlled via Sobolev embedding.
\end{proof}

Writing $\F_J\!\!=\!\!-(\pave_i S_\ve V_J^\ell)(\pave_\ell V_J^i)$ for $J\!\! =\!\! I_{\!},\II_{\!}$,
a simple modification of the proof of Lemma \ref{fbdslem} gives:
\begin{lemma} \label{N_r^1,2}
  Suppose that \eqref{ubd2} holds and let $s=k + \ell $. Then
  there is a continuous, positive function
  $C_s = C_s(M, ||V_I||_s,||V_{\II}||_s, ||\xveu||_{H^{s+1}},
  ||\xvew||_{H^{s+1}})$ so that:
\begin{equation}
  ||D_t^k \big(\F^1_{\!I\!} -\! \F^1_{\!I\!I}\big)||_{H^{\ell}}
 \leq C_s \big( ||D_t^k V_{\!I} - D_t^k V_{\!I\!I}||_{H^{\ell+1}}
 + ||\xveu \!-\! \xvew||_{H^{\ell+1}}
 + ||V_{\!I} - V_{\!I\!I}||_{s} +  ||\xveu \!- \! \xvew||_{C^4_{x,t}(\Omega)}\big),\qquad
 \label{fdiff}
\end{equation}
\begin{align}
  ||F^2(h_I) - F^2(h_{\II})||_{s,0} &\leq C_s \big(||h_I - h_{II}||_{s+1, 0}
  + ||h_I - h_{II}||_s + ||h_I - h_{II}||_{C^3_{x,t}} \big),\label{functionalest1}
  \\
  ||F^2(h_I) - F^2(h_{\II})||_{s-1} &\leq C_s \big(||h_I - h_{II}||_{s}
  + ||h_I - h_{II}||_{C^3_{x,t}} \big).
  \label{functionalest2}
\end{align}
\end{lemma}

\section{Existence of a sequence of compatible data for the smoothed problem}
\label{compat}
In this section, our goal is to prove:
\begin{theorem}
  \label{compatthm}
 Suppose that $V_0, h_0\!\in \! H^r\!\!$, $x_0\! \in\! H^{r+1}\!$ satisfy the compatibility conditions for
 Euler's equations \eqref{maincomp0} to order $r\!-\!1 \!\geq \! 7$.
 Then there is a sequence of data $V_0^\ve\!,h_0^\ve\! \in \!H^r\!\!, \, x_0^\ve \!\in\! H^{r+1}\!\!$
 satisfying the compatibility conditions for the smoothed Euler's equations \eqref{smcompatdef} to
  order $r{}_{\!}-{}_{\!}1$, and $(V_0^\ve\!, h_0^\ve, x_0^\ve)\! \to\! (V_0, h_0, x_0)$
 as $\ve\!\to\! 0$.
\end{theorem}
In the next section, we prove that if the compatibility conditions to order $r-1$ hold, given
sufficiently regular $V$, the wave equation
\eqref{wavedef} has a solution $h(t,\cdot) \in H^r(\Omega)$ with
$D_t h(t,\cdot) \in H^{r-1}(\Omega),..., D_t^{r-1} h(t,\cdot)
\in H^1_0(\Omega), D_t^r h(t,\cdot) \in L^2(\Omega)$ for $t > 0$.
 We modify the approach of Lindblad-Luo \cite{Lindblad2016}
  to construct functions $u_{-\!1}^\ve, u_0^\ve$ so that with
 $V_0^\ve \!\!=\! V_0\! +\! \pa_{x_0} u_{-\!1}^\ve$,
$h_0^\ve\! =\! h_0 \!+\! u_0^\ve$, and $x_0^\ve = x_0$, the initial data $V_0^\ve\!, h_0^\ve, x_0^\ve$ satisfy
the compatibility conditions \eqref{smcompatdef}. It will be convenient to
reformulate the conditions used in Sections \ref{compcondsec} and \ref{wavecompatcondn}
in a slightly more explicit way.
Suppose that
$\hat{x} \!=\! \!\sum x_k t^k\!/k!$, $\hat{V} \!\!=\! \!\sum V_k t^k\!/k!$, $\hat{h} \!=\! \!\sum V_k t^k\!/k!$ are formal power series
 solutions at $t \!=\! 0$ to
\eqref{mass}-\eqref{mom} with  $\hat{x}|_{t = 0} \!=\! x_0$
and $D_t^{\ell+1} \hat{x}|_{t = 0} \!=\!
 D_t^\ell\hat{V}|_{t = 0}$ and
$\hat{x}_\ve \!=\! \sum x_k^\ve t^k\!/k!$,
$\hat{V}_\ve \!=\! \sum V_k^\ve t^k\! / k!$, $\hat{h}_\ve
\!=\! \sum h_k^\ve t^k\!/k!$ are power series solutions at $t\! = \!0$
to the smoothed problem \eqref{smpbmdef}-\eqref{resultofproj}
with $\hat{x}|_{t = 0\!} \!=\! x_0$ and $
D_t^{\ell\!+\!1\!} \hat{x}^\ve|_{t = 0\!} \!=\! D_t^\ell\hat{V}^\ve|_{t = 0}$.
Define:
\begin{alignat}{3}
 F_k &= \big([D_t^k, \hat{\Delta}] \hat{h} +
  D_t^k  (\hat{\pa}_i \hat{V}^j)(\hat{\pa}_j \hat{V}^i)\big)\big|_{t = 0},
  &\qquad
  G_k &= \big(D_t^{k+1} (e'(\hat{h})D_t \hat{h}) -e'(\hat{h}) D_t^{k+2}\hat{h} +
  D_t^k\rho[\hat{h}]\big)\big|_{t = 0},\\
 F^\ve_k\! &= \!\big([D_t^k, \hat{\Dve}] \hat{h}_\ve \!+\!
  D_t^k  (\hat{\pave}_i \ssm \hat{V}_\ve^j)(\hat{\pave}_j \hat{V}_\ve^i)\big)\big|_{t = 0},
  &\qquad
  G^\ve_k \!&=\! \big(D_t^{k+1} (e'(\hat{h}_\ve)D_t \hat{h}_\ve)\! -e'(\hat{h}_\ve) D_t^{k+2}\hat{h}_\ve +
  D_t^k\rho[\hat{h}_\ve]\big)\big|_{t = 0},
\end{alignat}
as well as:
\begin{equation}
 C_k = [D_t^k, \hat{\pa}] (\hat{h} + \phi[\hat{x},\hat{h}]) |_{t = 0},
 \qquad C_k^\ve = [D_t^k,\hat{\pave}] (\hat{h}_\ve + \phi[\hat{x}_\ve,
 \hat{h}_\ve]|_{t = 0}.
 {}
\end{equation}
Here, we are writing:
\begin{equation}
 \hat{\pa}_i = \frac{\pa {y}^a}{\pa \hat{x}^i} \frac{\pa }{\pa y^a}
 ,\qquad
 \hat{\wpa}_i = \frac{\pa {y}^a}{\pa \hat{\widetilde{x}}{}^i} \frac{\pa }{\pa y^a},
 \qquad \hat{\Delta} = {\sum}_{i = 1}^3 \hat{\pa}^2_i,
 \qquad
 \hat{\Dve} = {\sum}_{i = 1}^3 \hat{\pave}_i{\!\!}^2.
\end{equation}
We are also writing
$\phi[x, h]$ for the map $x, h \mapsto \phi$ defined in \eqref{phidefsec2}.

Taking the divergence of Euler's equation \eqref{mom} at $t = 0$ and subtracting it from
the continuity equation
\eqref{mass} at $t = 0$ and performing the same manipulations to \eqref{smpbmdef}
and \eqref{resultofproj} gives that the coefficients $x_k, V_k, h_k, x_k^\ve, V_k^\ve, h_k^\ve$
must satisfy the relations:
\begin{align}
 x_k = V_{k-1}, \qquad V_k = -\pa_{x_0} H_{k-1} +C_k,
 \qquad
 e'(h_0) h_{k+2} = \Delta h_k + F_k + G_k,\label{vkandvkeps}\\
x^\ve_k = V_{k-1}^\ve, \qquad  V_k^\ve = -\pa_{x_0}  H_{k-1}^\ve + C_k^\ve,
 \qquad
 e'(h_0^\ve) h_{k+2}^\ve = \Delta h_k^\ve + F_k^\ve + G_k^\ve,
 \label{vkandvkeps2}
\end{align}
for $k \!\geq \!1$,
with $H_{k-\!1} \!=\! h_{k-\!1} \!+ \phi_{k-\!1}, H_{k-\!1}^\ve\! = \! h_{k-\!1}^\ve \!+ \phi_{k-\!1}^\ve$,
and where  $\phi_\ell\! =\! D_t^\ell \phi[\hat{x},\hat{h}]|_{t = 0}, \phi_\ell^\ve \!=\!
D_t^\ell \phi[\hat{x}_\ve, \hat{h}_\ve]$.

Expanding out the various definitions and replacing $x_k$ with $V_{k-1}$ for
$k\geq 1$, it follows that:
\begin{alignat}{3}
 F_k &=  F_k[x_0, V_0,..., V_k, h_0, ...., h_{k-1}],
 &\qquad
 G_k &= G_k[h_0,..., h_{k+1}],
 \label{fkgk}
 \\
 \phi_k &= K_k[x_0, V_0,..., V_{k}, h_0,...., h_k],
 &\qquad C_k &= C_k[x_0, V_0,..., V_{k-1}, H_0,..., H_{k-1}].
 \label{phikck}
\end{alignat}
for functionals $K_k$ and where these functionals depend on space derivatives
of their arguments, and similarly:
\begin{alignat}{3}
 F_k^\ve &=  F_k^\ve[x_0, V_0^\ve,..., V_k^\ve, h_0^\ve, ...., h_{k-1}^\ve],
 &\qquad
 G_k^\ve &= G_k^\ve[h_0^\ve,..., h_{k+1}^\ve],\label{fkgkeps}
 \\
 \phi_k^\ve &= K_k^\ve[x_0, V_0^\ve,..., V_{k}^\ve, h_0^\ve,...., h_k^\ve],
 &\qquad C_k^\ve &= C_k^\ve[x_0, V_0^\ve,..., V_{k-1}^\ve, H_0^\ve,..., H_{k-1}^\ve].
\end{alignat}
The formulas \eqref{phikck} combined with
the second identity in \eqref{vkandvkeps} shows that $V_k$ can be
expressed entirely in terms of $x_0, V_0$ and $h_0,..., h_{k-1}$
and similarly $V_k^\ve$ can be expressed entirely in terms of
$x_0, V_0, h_0^\ve,..., h_{k-1}^\ve$. Consequently we will
eliminate $V_k, V_k^\ve$ for $k \geq 1$ from our equations and
abuse notation slightly and write:
\begin{equation}
 F_k = F_k[x_0, V_0, h_0,..., h_{k-1}], \qquad
 F_k^\ve = F_k^\ve[x_0, V_0^\ve, h_0^\ve,..., h_{k-1}^\ve].
\end{equation}

\subsubsection{The perturbative system}
We start by considering the following system,
with $\Delta = \sum_{i = 1}^3 \pa_{i}^2$:
\begin{align}
 &\Delta u_{-1}^\ve = -e'(h_0 + u_0^\ve) u_1^\ve, &&\textrm{ in } \Omega,&
 \label{app:ell1}\\
 &\Delta u_k^\ve = F_k - \widetilde{F}_k^\ve
 + G_k - \widetilde{G}_k^\ve + e'(h_0 + u_0^\ve)u_{k+2}^\ve,
 &&\textrm{ in } \Omega, \textrm{ for } k = 0,..., r-2,\\
 &u_k^\ve = 0, &&\textrm{ on } \pa \Omega, \textrm{ for } k = -1,..., r-2,
 \label{app:bc}
\end{align}
with $\widetilde{F}_k^\ve= \widetilde{F}_k^\ve[u_{-1}^\ve,..., u_k^\ve]
,\widetilde{G}_k^\ve = \widetilde{G}_k^\ve[u_0^\ve,..., u_{k+1}^\ve]$
defined by:
\begin{equation}
 \widetilde{F}_k^\ve
  = F_k^\ve[x_0, V_0 + \pave u_{-1}^\ve,
 h_0 + u_0^\ve,..., h_{k-1} + u_{k-1}^\ve],
 \qquad
 \widetilde{G}_k^\ve
 = G_k^\ve[h_0 +  u_0^\ve,....,
 h_{k+1} +  u_{k+1}^\ve],
 {}
\end{equation}
and with the convention that $u_\ell^\ve = 0$ for $\ell \geq r-1$.

Suppose for the moment that this system has a solution $(u_{-1}^\ve,...,
u_{r-1}^\ve)$.
We claim that with $V_0^\ve = V_0 + \pa_{x_0} u_{-1}^\ve$ and
$ h_0^\ve =
h_0 + u_0^\ve$, the initial data $(V_0^\ve, h_0^\ve)$ satisfy the
compatibility conditions \eqref{hkepsdef} for the smoothed problem to order $r-1$.
Indeed, because $h_0 = 0$ on $\pa \Omega$ and because of the boundary
condition \eqref{app:bc} we have that $h_0^\ve = 0$ on $\pa \Omega$.
To see that $h_1^\ve = 0$ on $\pa \Omega$, we note that by construction:
\begin{equation}
  e'(h_0^\ve)h_1^\ve = - \div V_0^\ve = - \div V_0 -
  \Delta u_{-1}^\ve =  e'(h_0)h_1 + e'(h_0^\ve)u_1^\ve.
\end{equation}
By the compatibilty conditions for $V_0, h_0$, we have $h_1 \!=\! 0$ on
$\pa \Omega$ and by construction $u_1^\ve \!=\! 0$ on $\pa \Omega$
and so the first compatibility condition \eqref{smcompatdef} holds as well.
Using the definitions of $h_2^\ve, h_2$ from \eqref{vkandvkeps},\eqref{vkandvkeps2}, we have:
\begin{equation}
e'(h_0^\ve) h_2^\ve
 = \Delta h_0 + \Delta u_0^\ve + F_0^\ve + G_0^\ve
 = \Delta h_0  + F_0 + G_0
  + e'(h_0^\ve)u_2^\ve
 = e'(h_0)h_2 + e'(h_0^\ve) u_2^\ve,
\end{equation}
By the compatibility
conditions, $h_2 = 0$ on $\pa \Omega$ and this
combined with the boundary condition \eqref{app:bc}
shows that $h_2^\ve = 0$ on $\pa \Omega$ as well.
In general, this construction gives that:
\begin{equation}
 e'(h_0^\ve) h_k^\ve = e'(h_0) h_k + e'(h_0^\ve) u_k^\ve,
 \,\,
  k = 0,..., r-2, \qquad
 e'(h_0^\ve) h_k^\ve = e'(h_0)h_k, \,\,\,
 k = r-1, r,
\end{equation}
from which it immediately follows that the compatibility condition
of order $r-1$ holds for the smoothed problem
so long as the compatibility condition of order $r-1$ holds for the original
problem.

Because $e'$ is assumed to be small,
a simplified model for the above system is the following:
\begin{equation}
  \Delta w_{-1} = \kappa w_{1},\quad
 \Delta w_k = {\sum}_{\ell \leq k-1} A_k^\ell w_\ell + f_k +
 \kappa w_{k+2}, \quad k = 0,..., N, \qquad
 \text{ in } \Omega,
 \label{model}
\end{equation}
with the boundary condition $w_k = 0$ on $\pa \Omega$ for all $k$.
Here, $A_k^\ell, f_k$ are given functions, $\kappa$ is a small
parameter and we are writing $w_\ell = 0$ for $\ell \geq N+1$.
 When $\kappa = 0$, this system is lower-triangular and
can be solved directly
by successively solving for $w_0, w_1,...$.
To solve the model system \eqref{model} for nonzero but small
$\kappa$, one can use the following
iteration: $w^0_k \equiv 0$ for all $k$ and then, given $w^{\nu-1}_k$,
solve the following system for $w^\nu_k$:
\begin{equation}
  \Delta w_{{-1}}^\nu = \kappa w_2^{\nu-1},\quad
 \Delta w_k^{\nu} = {\sum}_{\ell \leq k-1} A_k^\ell w_\ell^\nu + f_k +
 \kappa w_{k+2}^{\nu-1}, \qquad k = 0,..., N-1,
 \label{}
\end{equation}
with $w^\nu_k = 0$ on $\pa \Omega$. Writing $W_k^\nu = w_k^\nu - w_k^{\nu-1}$,
by standard elliptic theory there are estimates of the form:
\begin{equation}
 ||W_k^{\nu}||_{H^{s-k}}
 \leq C\big({\sum}_{\ell \leq k-1} ||W_\ell^{\nu}||_{H^{s-k-2}}
 +    \kappa ||W^{\nu-1}||_{H^{s-k-2}}\big),\qquad k =-1...., N,
\end{equation}
where $C$ depends on norms of the coefficients $A$. Iterating this estimate
leads to an inequality of the form:
\begin{equation}
 ||W_k^\nu||_{s-k} \leq
 {\sum}_{\mu = 1}^{\nu-1} (C\kappa)^\mu.
\end{equation}
For $\kappa$ sufficiently small, the sequence $w^\nu\!$
converges as $\nu \!\to \!\infty$ to a solution $w \!= \!(w_{-1},\dots,w_{N})$
satisfying \eqref{model}.

%To see that the restriction on the size of $\kappa$ is not just an artifact of
%the proof, take $\lambda$ to be a Dirichlet eigenvalue of $\Delta$ on $\Omega$
%and set $\kappa = \lambda, A_k^\ell = 0$ for $\ell \leq k-2$ and $A_k^{k-1} = \lambda$.
%If $w_{-1},..., w_N$ is a solution of \eqref{model}, it follows that
% $W = {\sum}_{k = -1}^N w_k$ satisfies $\Delta W - \lambda W = \sum f_k$,
% which does not have a solution for arbitrary $f_k$.
% To consider large values of $\kappa$ one would need
% to rule out behavior of this type.
\subsubsection{The iteration to solve the system}
In order to solve the system \eqref{app:ell1}-\eqref{app:bc}, we will
use the following iteration. We set $u_k^0 \equiv 0$
in $\Omega$ for $k = -1,..., r$ and for
$\nu \geq 1$, we define $u_{-1}^\nu ,...,u_r^\nu$ by $u_{r-1}^\nu = u_r^\nu = 0$
and:
\begin{align}
 &\Delta u_{-1}^\nu = -e'(h_0 + u_0^{\nu-1}) u_1^{\nu-1},
 &&\textrm{ in } \Omega,&
 \label{app:ell1nu}\\
 &\Delta u_k^\nu = F_k - F_k^{\nu} + G_k-G_k^{\nu-1}+e'(h_0 +
 u_0^{\nu-1}) u_{k+2}^{\nu-1},
 &&\textrm{ in } \Omega, \textrm{ for } k = 0,..., r-2,\label{app:ell2nu}\\
 &u_k^\nu = 0, &&\textrm{ on } \pa \Omega, \textrm{ for } k = -1,..., r-2,
 \label{app:bcnu}
\end{align}
where we are writing:
\begin{equation}
 F_k^{\nu} = \widetilde{F}_k^\ve[u_{-1}^\nu,..., u_{k-1}^\nu],
 \qquad
 G_k^{\nu-1} = \widetilde{G}_k^\ve[u_0^{\nu-1},..., u_{k+1}^{\nu-1}].
\end{equation}
Let $u^\nu = (u_{-1}^\nu ,...,u_r^\nu)$.
To see that this system has a solution $u^\nu$ given $u^{\nu-1}$, one
just uses the fact that it is lower-triangular in $u^\nu$;
first solve \eqref{app:ell1nu} for $u_{-1}^\nu$
and then solve \eqref{app:ell2nu}-\eqref{app:bcnu} successively
for $k = 0, 1,...,r-2$.

We will prove that the sequence
$u^\nu$ is uniformly bounded in $\nu$ in the norm
\begin{equation}
 ||u^\nu||_r = || \pa_{x_0} u_{-1}^\nu||_{H^r(\Omega)}
+ ||u_{-1}^\nu||_{H^r(\Omega)}
 + {\sum}_{k = 0}^{r-2} ||u_k^\nu||_{H^{r-k}(\Omega)}.
\end{equation}

Set $E_0 = ||V_0||_{H^r}^2 + ||h_0||_{H^r}^2 + ||x_0||_{H^{r+1}}^2$.
In the following sections we will prove:
\begin{prop}
  \label{unuthm}
  Fix $r\! \geq\! 8$.
  There is a continuous function $C_r$ so that if $u^\nu$ satisfies  \eqref{app:ell1nu}-\eqref{app:bcnu}, then:
  \begin{equation}
   ||u^\nu||_r \leq C_r(E_0, ||u^{\nu-1}||_{r-1}) (\kappa ||u^{\nu-1}||_r + \ve),
   \label{unu}
  \end{equation}
  and there is a continuous function $D_r$ so that:
  \begin{equation}
   ||u^\nu - u^{\nu-1}||_{r}
   \leq D_r(E_0, ||u^\nu||_{r}, ||u^{\nu-1}||_{r}) \kappa ||u^{\nu-1} - u^{\nu-2}||_r.
   \label{unudiff}
  \end{equation}
\end{prop}
Let us now explain why one should expect estimates of this form.
The estimates \eqref{unu} follow from elliptic estimates applied
to the system \eqref{app:ell1nu}-\eqref{app:bcnu} and will ultimately
follow from estimates for $F_k - F_k^\nu, G_k - G_k^{\nu-1}$ in
Sobolev spaces. Let us consider the $k = 0$ case.
Using that $\hat{x}|_{t = 0} = x_0$ we have:
\begin{equation}
 F_0 - F_0^{\nu}
 = (\pa_i V_0^j)(\pa_j V_0^i) -
 (\pa_i \ssm (V_0^j + \delta^{jk}\pa_k u_{-1}^\nu)
 (\pa_j (V_0^i+\delta^{ik}\pa_k u_{-1}^\nu)).
 {}
\end{equation}
Expanding this out generates several terms but let us just consider
two of them:
\begin{equation}
 \pa_i V_0^j (\delta^{ij}\pa_j \pa_k u^{\nu}_{-1})
 \quad
 \text{ and }
 \quad
 (\pa_i V_0^j - \pa_i \ssm V_0^j)(\pa_j V_0^i).
 {}
\end{equation}
To control the $L^2(\Omega)$ norm, say, of the first term
we use the equation \eqref{app:ell1nu} and standard
elliptic theory to control
$||u^{\nu}_{-1}||_{H^2(\Omega)} \leq C ||e'(u_0^{\nu-1})
u_1^{\nu-1}||_{L^2(\Omega)}$. With $\kappa \geq \sup |e'|$
this type of term can be bounded by the first term in \eqref{unu}.
Also, we have $||V_0 - \ssm V_0||_{H^1(\Omega)} \leq C \ve ||V_0||_{H^2(\Omega)}$
so the second type of term can be bounded by the second term in \eqref{unu}.
Assuming that \eqref{unu}-\eqref{unudiff} hold for the moment, we give the proof:
\begin{proof}[Proof of Theorem \ref{compatthm}]
With the function $C_r$ from Proposition \ref{unuthm}, take
$C_0 = \max_{z \in [0,1]} C_r(E_0, z)$. Also take
$\kappa$ so small that $2\kappa C_0 \leq 1$ and $\ve$ so small
that $2\ve {\sum}_{\mu = 0}^\infty (\kappa C_0)^\mu \leq 1$. Since $u^{0} =0$,
it follows from \eqref{unu} that $||u^1||_r \leq C_0 \ve$. By induction
it then follows that $||u^\nu||_{r} \leq \ve \sum_{\mu = 0}^\nu (\kappa C_0)^\mu$,
and by the assumption on $\kappa$ the sum on the right-hand side is uniformly
bounded as $\nu \to \infty$.

Next, with the function $D_r$ from Proposition \ref{unuthm},
take $D_0 = \max_{z_1,z_2 \in [0,1]} D_r(E_0,z_1,z_2)$.
By induction and \eqref{unudiff} it follows that
$||u^\nu - u^{\nu-1}||_r \leq \ve (\kappa D_0)^{\nu}$. Therefore,
$u^\nu$ is a Cauchy sequence and so it converges to some limit
$u$ which satisfies the perturbative system \eqref{app:ell1}-\eqref{app:bc}
 by construction.
To prove the second point in the theorem, taking $\nu \!\to\! \infty$ in the
estimate for $u^\nu$ that we just proved shows $||u||_r \leq \ve \sum_{\mu = 0}^\infty
(\kappa C_0)^\mu \leq \ve$.
\end{proof}

We will use the following estimate, which is a straightforward consequence
of the elliptic estimate \eqref{sobell} at $t = 0$: If $s \geq 2$, there is
a constant $C_s = C_s(||x_0||_{H^{s+1}})$ so that if $f = 0$ on $\pa \Omega$, then:
\begin{equation}
  ||\pa_{x_0} f||_{H^{s}} \leq C_s ||\Delta_{x_0} f||_{H^{s-1}},
  \qquad
 ||f||_{H^s} \leq C_s ||\Delta_{x_0}f||_{H^{s-2}}.
 \label{compatelliptic}
\end{equation}

The estimates \eqref{unu} and \eqref{unudiff} follow after repeatedly applying
the next lemma:
\begin{lemma}
  There are continuous functions $C_{r,k}$ so that if $u^\nu = (u_{-1}^\nu,..., u_{r-2}^{\nu})$
   satisfies the approximate system \eqref{app:ell1nu}-\eqref{app:bcnu}
   and if the equation of state satisfies \eqref{eq:closetoincompressible2}, then:
 \begin{align}
  ||u_{-1}^\nu||_{H^r} &\leq
  C_{r,-1}(E_0, ||u^{\nu-1}_0||_{H^{r-2}}) \kappa ||u_1^{\nu-1}||_{H^{r-2}},\label{unu1} \\
  ||u_k^\nu||_{H^{r-k}} &\leq C_{r,k}(E_0, ||u^{\nu-1}||_r,
  \tsum_{\ell \leq k-1} ||u^{\nu}_\ell||_{H^{r-\ell-1}})
  \big( {\sum}_{\ell \leq k-1} ||u_\ell^\nu||_{H^{r-k}} + \kappa ||u^{\nu-1}||_r + \ve\big),
  \label{unu2}
 \end{align}
 and there are continuous functions $D_{r,k}$ so that with $U_k^\nu = u_k^\nu - u_k^{\nu-1}$:
 \begin{align}
  ||U_{-1}^\nu||_{H^r} &\leq D_{r,-1}(E_0, ||u_0^{\nu-1}||_r, ||u_0^{\nu-2}||_{r})
   \kappa ||U^{\nu-1}_1||_{H^{r-2}}, \label{Unu1}
   \\
  ||U_{k}^\nu||_{H^{r-k}} &\leq
  D_{r,k}(E_0, ||u^{\nu}||_r, ||u^{\nu-1}||_{r}, ||u^{\nu-2}||_r)
  \big(
  {\sum}_{\ell \leq k-1} ||U_\ell^{\nu}||_{H^{r-\ell}}
  + \kappa ||U_\ell^{\nu-1}||_{H^{r-\ell}}\big).
  \label{Unu2}
 \end{align}
\end{lemma}

\begin{proof}
 Using the
elliptic estimates \eqref{compatelliptic} and the fact that $H^{r-k-2}(\Omega)$ is an
algebra for $r -k \geq 4$, we have:
\begin{align}
 ||\pa_{x^0}u^\nu_{-1}||_{H^r} &\leq C \big( ||e'(h_0 + u_0^{\nu-1})||_{H^{r-1}} ||u_1^{\nu-1}||_{H^{r-1}}\big),
 \\
 ||u^\nu_k||_{H^{r-k}} &\leq C
 \big(||F_k - F_k^\nu||_{H^{r-k-2}} + ||G_k - G_k^{\nu-1}||_{H^{r-k-2}}
 + ||e'(h_0 + u_0^{\nu-1})||_{H^{r-k-2}} ||u_{k+2}^{\nu-1}||_{H^{r-k-2}}\big),
\end{align}
for $k = 0,..., r-2$,
with the convention that $u_\ell^\nu = 0$ for $\ell \geq r-1$, and with
constants depending on $||x_0||_{H^{r+1}}$.

Because $U^\nu = u^\nu - u^{\nu-1}$
satisfies the following system in $\Omega$:
\begin{align}
 \Delta U_{-1}^\nu &= e'(h_0 + u_0^{\nu-1}) u_1^{\nu-1} - e'(h_0 + u_0^{\nu-2}) u_1^{\nu-2},\\
 \Delta U_k^\nu &= F^{\nu-1}_k - F_k^\nu + G^{\nu-1}_k - G^\nu_k
 +e'(h_0 + u_0^{\nu-1}) u_{k+2}^{\nu-1} - e'(h_0 + u_{0}^{\nu-2}) u_{k+2}^{\nu-2},
 \quad k = 0,..., r-2,
 \label{}
\end{align}
with $U^\nu = 0$ on $\pa \Omega$, we also have:
\begin{equation}
 ||\pa_{x_0}U_{-1}^\nu||_{H^r} \leq
 C \big( ||e'(h_0 + u_0^{\nu-1}) - e'(h_0 + u_0^{\nu-2})||_{H^{r-1}} ||u_1^{\nu-1}||_{H^{r-1}}
 + ||e'(h_0 + u_0^{\nu-2})||_{H^{r-1}} ||U_1^{\nu-1}||_{H^{r-1}}\big),
 \end{equation}
 \begin{multline}
 ||U_k^\nu||_{H^{r-k}} \leq
 C\big( ||F^{\nu-1}_k - F^{\nu}_k||_{H^{r-k-2}}
 + ||G^{\nu-1}_k - G^{\nu}_k||_{H^{r-k-2}}\\
 +||e'(h_0 + u_0^{\nu-1}) - e'(h_0 + u_0^{\nu-2})||_{H^{s-k-2}} ||u_{k+1}^{\nu-1}||_{H^{r-k-2}}
 + ||e'(h_0 + u_0^{\nu-2})||_{H^{r-k-2}} ||U_{k+1}^{\nu-1}||_{H^{r-k-2}}.
\end{multline}
The estimates \eqref{unu1}-\eqref{Unu2} then follow from Proposition
\ref{Fkprop} and Lemmas \ref{Gnulem}-\ref{eprimenulem}.
\end{proof}

It remains to prove estimates for the terms on the right-hand sides
of \eqref{unu1},\eqref{unu2} and \eqref{Unu1},\eqref{Unu2}.
The proposition below is a  consequence of Lemmas \ref{Fklem},
\ref{Vkestlem}, whose proofs we postpone until Section \ref{Fkestsection}
\begin{prop}\label{Fkprop} Set $M_k^\nu = ||\pa_{x_0}u_{-1}^\nu||_{H^{r}}
  +  \tsum_{j \leq k}
  ||u_j^\nu||_{H^{r-j}}$.
There are continuous functions $K_k = K_{k}(E_0, M_{k-1}^\nu),
K_k' = K'_{k}(E_0, M_{k-1}^\nu, M_{k-1}^{\nu-1})$ so that
writing $j = r-k-2$:
\begin{equation}
  ||{}_{\!}F_{\!k} - F_{\!k}^{\nu}||_{{}_{\!}H^j}\! \!\leq\!
 K_{\!k}\big(
 ||u_{-\!1}^\nu\!||_{{}_{\!}H^{j\!+2}}+
 \!{\sum}_{\ell \leq k-\!1\!} ||u_{\ell}^\nu||_{{}_{\!}H^{j\!+2}} +\ve\big),
 \qquad
 ||{}_{\!}F_{\!k}^\nu\!\!- F_{\!k}^{\nu\!-\!1\!}||_{{}_{\!}H^j}
\!\!\leq\! K_{\!k}'
\big( ||U_{\!-\!1\!}^\nu||_{{}_{\!}H^{j\!+2}}
+ \!{\sum}_{\ell \leq k-\!1\!} ||U_{\ell}^\nu||_{{}_{\!}H^{j\!+\!2}}\!\big).
\end{equation}
\end{prop}

\begin{lemma}
 \label{Gnulem}
 There are continuous functions $K= K(E_0,
 ||u^{\nu-1}||_{r-1})$, $K' = K'(E_0,||u^{\nu-1}||_r, ||u^{\nu-2}||_{r})$ so that
 if $\sup_{r' \leq r+1} |e^{(r')}| \leq \kappa$ then:
\begin{equation}
 ||G_k - G_k^{\nu-1}||_{H^{r-k-2}} \leq
 \kappa K||u^{\nu-1}||_{r},
 \qquad
 ||G_k^{\nu-1} - G_k^{\nu-2}||_{H^{r-k-2}} \leq \kappa
 K' ||u^{\nu-1} - u^{\nu-2}||_r
\end{equation}
\end{lemma}
\begin{proof}
 Write $h_k^\nu = h_k + u_k^{\nu-1}$. Expanding out
 the definition of $G_k, G_k^{\nu}$ and applying $\pa_y^I$ for a multi-index
 $I$ with $|I| = r' \leq r-k - 2$, we see that $\pa_y^I (G_k - G_k^{\nu})$
 is a sum of terms of the form:
 \begin{equation}
  e^{(K)}(h_0^{\nu-1}) (\pa_y^{J_1} h_{k_1}^{\nu-1})
  \cdots (\pa_y^{J_j} h_{k_j}^{\nu-1}) -
  e^{(K)}(h_0) (\pa_y^{J_1} h_{k_1})
  \cdots (\pa_y^{J_j} h_{k_j}),
  {}
 \end{equation}
 with $|J_1| + \cdots |J_j| = r-k-2, k_1 + \cdots k_j = k+1, K \leq r-1$.
 Performing the usual manipulations, rearranging terms, and using that
 $h_k^{\nu-1} \!\!- h_k = u_k^{\nu-1}\!\!$, it suffices to
 control the $L^2(\Omega)$ of a sum of terms of the forms:
 \begin{equation}
  (e^{(K)}(h_0^{\nu-1}) - e^{(K)}(h_0))
  (\pa_y^{J_1}h^{\nu-1}_{k_1})
  \cdots
  (\pa_y^{J_j}h^{\nu-1}_{k_j}),
  \quad
  \text{ and }
  \quad
  e^{(K)}(h_0^{\nu-1}) (\pa_y^{J_1}u_{k_1}^{\nu-1})
  (\pa_y^{J_2}h^{\nu-1}_{k_2})
  \cdots
  (\pa_y^{J_j}h^{\nu-1}_{k_j}),
  {}
 \end{equation}
 the remaining terms being similar but with some of the factors
 of $h^{\nu-1}_\ell$ replaced by $h_\ell$. Let us just bound the second type
 of term here, the first type being identical after using the estimate
 $|e^{(K)}(h_0^{\nu-1}) - e^{(K)}(h_0)| \leq |\sup e^{(K+1)}| |u_0^{\nu-1}|$.
 For each $\ell$ with $|J_\ell| + k_\ell \leq r-3$, we bound the resulting
 term in $L^\infty$ by Sobolev embedding to get either
 $||\pa_y^{J_\ell} u^{\nu-1}_{k_\ell}||_{L^\infty(\Omega)}
 \leq C|| \pa_y^{J_\ell}u^{\nu-1}_{k_\ell}||_{H^2(\Omega)}$
 or $C(||\pa_y^{J_\ell} h_{k_\ell}||_{H^2(\Omega)}
 + ||\pa_y^{J_\ell} u_{k_\ell}^{\nu-1}||_{H^2(\Omega)})$.
 Since $|J_\ell|+ k_\ell + 2 \leq r-1$, the result can be bounded by
 $||u^{\nu-1}||_{s-1}$ or $||u^{\nu-1}||_{r-1} + E_0$, respectively.
 It therefore remains to handle terms with $|J_\ell| + k_\ell \geq r-2$.
 Since $r \geq 5$ there is at most one such term and so it is bounded by
 either $||u_{k_\ell}^{\nu\!-\!1}||_{H^{|\!J_\ell\!|}(\Omega)}
 \leq ||u^{\nu\!-\!1}||_{r-\!1}$ or
 $||h_{k_\ell}||_{H^{|\!J_\ell\!|}(\Omega)} + ||u_{k_\ell}^{\nu\!-\!1}||_{H^{|\!J_\ell\!|}(\Omega)}
 \leq E_0 + ||u^{\nu\!-\!1}||_{r-\!1}$, as required. The estimate
 for $G^{\nu-1} - G^{\nu-2}$ is similar.
\end{proof}

\begin{lemma}
 \label{eprimenulem}
 There are continuous functions $K'' \!\!= \!K''(E_0,\! ||u_0^{\nu-\!1\!}||_{H^{r\!-\!1}}\!)$,
 $K^{\prime\prime\prime} \!\!= \!K^{\prime\prime\prime}(E_0, ||u_0^{\nu-\!1\!}||_{H^{r\!-\!1}\!}, \!||u_0^{\nu-2\!}||_{H^{r\!-\!1}\!})$
 so that if
 $\sup_{k \leq r+1} |e^{(k)}| \leq \kappa$ then:
 \begin{equation}
  ||e'(h_0^{\nu-1})||_{H^{r}}
  \leq \kappa K''
  ||u_0^{\nu-1}||_{H^r(\Omega)},
  \qquad
  ||e'(h_0^{\nu-1}) - e'(h_0^{\nu-2})||_{H^r}\leq
  \kappa K''' ||u_0^{\nu-1} - u_0^{\nu-2}||_{H^r}.
  \label{eprimenuest}
 \end{equation}
\end{lemma}
\begin{proof}
 By the chain rule, if $I\!$ is a multi-index with $|I|\! = \!r' \!\!\leq \!r$,
 $\pa_y^I (e'(h_0 \!+\! u_0^{\nu-1}))\!$ is a sum of terms of the form:
  \begin{equation}
    e^{(K)}(h_0 + u_0^{\nu-1})
    (\pa_y^{J_1} h_0 + u_0^{\nu-1})
    \cdots
    (\pa_y^{J_j} h_0 + u_0^{\nu-1}) ,
    \qquad \tsum |J_j| = r', K\leq r'.
 \end{equation}
 We want to control the $L^2(\Omega)$ norm of this.
 For each $\ell$ with $|J_\ell|\! \leq \!r\!-\! 3$ we control the
 $L^\infty$ norm of the resulting factor by Sobolev embedding
 which shows that any such term is bounded by
 $C (||h_0||_{H^{r-1}(\Omega)} + ||u_0^{\nu-1}||_{H^{r-1}(\Omega)})$.
 To handle terms with $|J_\ell| \!\geq\! r\!-\!2$, note that since $r\! \geq\! 5$
 there can be at most one such term and we control it by
 $||u_0^{\nu-1}||_{H^r(\Omega)}$. Since
 $|e^{(K)}| \leq \kappa$ this gives the first estimate \eqref{eprimenuest}
 and the second is similar.
\end{proof}

\subsection{Estimates for $F_k - F_k^\nu$ and $F_k^\nu - F_k^{\nu-1}$}
\label{Fkestsection}
For these estimates it will be convenient to first state
the results in terms of the coefficients $V_k, V_k^\nu$ before
relating these to $h_k,u_k^\nu$, because they depend on each other
in a complicated way.
Recall the definitions of $S, \tilde{S}$ from \eqref{eq:highercommutatorssec2exp},
\eqref{sepsdef}. Given power series in time $t$ $\hat{V}\!, \hat{V}_\ve$ as in the beginning of this
section and evaluating at $t \!= \!0$, the $S$, $\tS$ are polynomials in the following arguments:
\begin{equation}
 S_\ell^k = S_\ell^k(\pa V_0, ..., \pa V_{k-\ell-1}),
 \qquad
 \tS_\ell^k = \tS_\ell^k(\pa \ssm V_0^\ve,..., \pa \ssm V_{k-\ell-1}^\ve).
 \label{}
\end{equation}
We note for later use that in fact we have:
\begin{equation}
  S_\ell^k(\pa V_0, ..., \pa V_{k-\ell-1})=
  \tS_\ell^k(\pa V_0,..., \pa V_{k-\ell-1}),
 \label{StSidentity}
\end{equation}
which follows from the formulas \eqref{eq:highercommutatorssec2exp}, \eqref{sepsdef}.
We then have the following formula for the $V_k$:
\begin{equation}
 V_k^i = -\delta^{ii'}\pa_{i'} H_{k-1} + {\sum}_{\ell \leq k-2}
 \delta^{ii'}S_{i'\ell}^{jk} \pa_j H_\ell, \qquad i = 1,2,3,
 \label{Vkdef}
\end{equation}
and, given $V_0^\nu$ we recursively define $V_k^\nu$ by:
\begin{equation}
 V_k^{i \nu} = -\delta^{ii'}\pa_{i'} H_{k-1}^\nu
 + {\sum}_{\ell \leq k-2} \delta^{ii'}\tS^{jk,\nu}_{i' \ell} \pa_j H_\ell^\ve,
 \qquad i = 1,2,3,
 \label{Vknudef}
\end{equation}
where, with $\tilde{S}^{jk}_{i\ell}$ defined in \eqref{sepsdef}, we are writing:
\begin{equation}
 \tS^{j k,\nu}_{i\ell} = \tS^{j k}_{i\ell}(
  \pa\ssm V_0^\nu,...,
  \pa\ssm V_{k-\ell-1}^\nu).
 \label{}
\end{equation}
We are also writing $H_\ell = h_\ell + \phi_\ell$ and
$H_\ell^\nu = h_\ell^\nu + \phi_\ell^\nu$ with
$\phi_\ell^\nu = D_t^\ell\phi[x,\hat{h}^\nu]|_{t = 0}$, where
$\hat{h}^\nu(t) = \sum h^\nu t^k\!/k!$.

If $T$ is a (2,2) tensor then we write:
\begin{equation}
 ||T_\ell^k||_{H^m}^2
 = {\sum}_{0 \leq |I| \leq m}
 \int_\Omega \delta^{ii'}\delta_{jj'}
 (\pa_y^I T_{i \ell}^{jk})(\pa_y^I T_{ i' \ell}^{j'k}) dy,
 {}
\end{equation}
and we have the following lemma which will be used repeatedly to control the commutators
$S, \tS$:
\begin{lemma}
  \label{Sestimates}
  Let $e_0^s = E_0 + \sum_{j \leq s} ||V_j||_{H^{s-j}}$.
 If $s \geq 2$ then there are continuous functions $C_{k,\ell}$ so that
 \begin{equation}
  ||S_\ell^k||_{H^{s}}
  \leq C_{k,\ell}(e_0^{s+1}),
  \qquad
  ||\widetilde{S}_{\ell}^{k,\nu}||_{H^{s}}
  \leq C_{k,\ell}(e_0^{s+1},m_{k-\ell+1,s+1}^{\nu}),\quad\text{where}\quad
  m_{r,s+1}^\nu = \tsum_{j \leq r} ||V_j^\nu||_{H^{s+1}},
 \end{equation}
 and there are continuous functions $D_{k,\ell}, D_{k,\ell}'$ so that:
 \begin{align}
  ||S_\ell^k - \widetilde{S}_\ell^k||_{H^s}
  &\leq D_{k,\ell}(e_0^{s+1}, m_{k-\ell+1,s+1}^\nu)
  ({\sum}_{j \leq k-\ell-1}||V_{j} - V_{j}^{\nu-1}||_{H^{m+1}}
   +\ve e_0^{s+2}),\label{Sdiffs1}
  \\
  ||\widetilde{S}_\ell^{k,\nu} - \widetilde{S}_\ell^{k,\nu-1}||_{H^{s}}
  &\leq D_{k,\ell}'(e_0^{s+1}, m_{k-\ell+1,s+1}^\nu, m_{k-\ell+1,s+1}^{\nu-1})
  {\sum}_{j \leq k-\ell-1}
  ||V_j^{\nu} - V_{j}^{\nu-1}||_{H^{s+1}}.
  \label{Sdiffs2}
\end{align}
\end{lemma}
\begin{proof}
  Because $s \geq 2, H^s(\Omega)$ is an algebra and so the first
  two estimates follow because $S, \tS$ are polynomials in their arguments
 (see \eqref{eq:highercommutatorssec2exp} and \eqref{sepsdef}).

To prove \eqref{Sdiffs1}, set
$A \!=\! (\pa V_0,..., \pa V_{k-\ell-1})$
and $ \widetilde{A}^\nu \!=\! (\pa \ssm V_0^\nu,..., \pa \ssm V_{k-\ell-1}^\nu)$.
Abusing notation, we write:
\begin{equation}
 S_\ell^k - \widetilde{S}^{k,\nu}_\ell
 = S^k_\ell(A) - \widetilde{S}_\ell^k(\widetilde{A}^\nu)
= S_\ell^k(A) - \widetilde{S}_\ell^k(A)
+ (\tS_\ell^k(\widetilde{A}^\nu) - \tS^k_\ell(A)).
 \label{}
\end{equation}
By \eqref{StSidentity}, the first two terms cancel. Since $\widetilde{S}$
is a polynomial in its arguments, we have:
\begin{equation}
 ||\tS_\ell^k(\widetilde{A}^\nu)
 - \tS_\ell^k(A) ||_{H^s}
 \leq C' {\sum}_{j \leq k-\ell-1} ||V_j - \ssm V_j^\nu||_{H^{s+1}},
 \label{}
\end{equation}
with $C = C(E_0^s, ||A||_{H^{s}}, ||A^\nu||_{H^{s}})$,
after additionally using that $\ssm$ is bounded on Sobolev spaces.
Now we write $||V_j - \ssm V_j^\nu||_{H^{s+1}} \leq ||V_j - \ssm V_j||_{H^{s+1}}
+ ||\ssm(V_j -  V_j^\nu)||_{H^{s+1}}$. Since $||V_j - \ssm V_j||_{H^{s+1}}
\leq C \ve ||V_j||_{H^{s+2}}$ by \eqref{nofreelunch}, this concludes the proof of the third estimate.
The proof of \eqref{Sdiffs2} is similar.
\end{proof}

We have the following technical estimate for $F_k -F_k^\nu$ and $F_k^\nu - F_k^{\nu-1}$ in
terms of $V_k, V_k^\nu$:
\begin{lemma}
 \label{Fklem}
 Set $m_k^\nu = ||V_0^\nu||_{H^{r}} + \tsum_{0 \leq j \leq k} ||V_j^\nu||_{H^{r-j-1}}$.
 There are continuous functions $K = K_{r,k}(E_0, m_{k}^\nu),
 K' = K'_{r,k}(E_0, m_k^\nu, m_k^{\nu-1})$
 so that, with $V_k, V_k^\nu$ defined by
 \eqref{Vkdef}-\eqref{Vknudef} and $j = r-k-2$:
 \begin{equation}
   ||F_k - F_k^{\nu}||_{H^j} \leq
  K\big(
  ||u_{-1}^\nu||_{H^{j+2}}+
  {\sum}_{\ell \leq k} ||V_\ell - V_\ell^\nu||_{H^{j+1}}
  + ||u^\nu_{\ell-1}||_{H^{j+2}}
  +\ve\big),
\end{equation}
\begin{equation}
  ||F_k^\nu - F_k^{\nu-1}||_{H^j}
 \leq K'
 \big( ||u_{-1}^\nu - u_{-1}^{\nu-1}||_{H^{j+2}}
 + {\sum}_{\ell \leq k} ||V_\ell^\nu - V_\ell^{\nu-1}||_{H^{j+1}}
 + ||u_\ell^\nu - u_\ell^{\nu-1}||_{H^{j+2}}\big).
\end{equation}
\end{lemma}
\begin{proof}

  We start by writing $F_k - F_k^\nu$ more explicitly in terms of the $S$
  and $\widetilde{S}$. With $V_k^\nu$ defined in \eqref{Vknudef} and with
  $h^\nu_k = h_k + u_k^\nu$, let
  $\hat{V}^\nu(t) = \sum_{k = 0}^N V_k^\nu t^k/k!$ and
  $\hat{h}^\nu(t) = \sum_{k = 0}^N h_k^\nu t^k/k!$, we write:
  \begin{equation}
    F_k - F_k^\nu = D_t^k \big( (\hat{\pa} \hat{V})(\hat{\pa} \hat{V})
    - (\hat{\pave} \ssm \hat{V}^\nu)(\hat{\pave} \hat{V}^\nu) \big)|_{t = 0}
    + \big([D_t^k, \hat{\Delta}] \hat{h} - [D_t^k, \hat{\Dve}] \hat{h}^\nu\big)|_{t = 0}
    \equiv f_k^\nu + g_k^\nu.
   \label{}
  \end{equation}

  For matrices $a_i^j, b_i^j$, write $a\cdot b
  = a_i^j b_j^i$ and if $T$ is a (2,2) tensor, write $T_\ell^{k} a$
  for the matrix with components $(T^{k}_{\ell} a)_m^n
  = T^{jk}_{m\ell} a_j^n$. We then have the following expression:
  \begin{equation}
    f_{k}^\nu =
   {\sum}_{k_1 + k_2 = k}
   {\sum}_{\ell \leq k_1}{\sum}_{\ell' \leq k_2} S_{\ell}^{k_1} \pa V_{\ell}\cdot
   S_{\ell'}^{ k_2}
   \pa V_{\ell'}
   - \widetilde{S}^{k_1,\nu}_{\ell}
   \pa \tV^{\nu}_{\ell}\cdot
    \widetilde{S}^{k_2,\nu}_{\ell'}
   \pa V^{\nu}_{\ell'},\qquad \tV^{\nu}_k = \ssm V^{\nu}_k.
  \end{equation}

  Using the commutator formulas \eqref{eq:highercommutatorssec2},\eqref{sepsdef}
  twice, we have:
  \begin{multline}
   g_k^\nu = {\sum}_{\ell \leq k-1} \delta^{ij}\big(\pa_i S^{i'k}_{j\ell} \pa_{i'} h_\ell
   -  \pa_i \tS^{i'k,\nu}_{j\ell}\pa h_\ell^\nu
   + S^{i'k}_{j\ell} \pa^2_{ii'} h_\ell
   -\tS^{i'k,\nu}_{j\ell} \pa^2_{ii'}h_\ell^{\nu}\big)\\
   + {\sum}_{\ell \leq k-1} {\sum}_{\ell' \leq \ell-1}
   \delta^{ij}\big(
   S_{i\ell}^{i'k} \pa_{i'} S_{j\ell'}^{j'\ell} \pa_{j'} h_{\ell'}
   -\tS_{i\ell}^{i'k,\nu} \pa_{i'} \tS_{j\ell'}^{j'\ell,\nu} \pa_{j'} h^\nu_{\ell'}
   + S_{i\ell}^{i'k} S^{j'\ell}_{j\ell} \pa_{i'j'}^2 h_{\ell'}
   - \tS_{i\ell}^{i'k,\nu} \tS^{j'\ell,\nu}_{j\ell} \pa_{i'j'}^2h_{\ell'}^\nu\big),
  \end{multline}
  where here we are writing $\pa = \pa_{x_0}$. Similarly, we have
  $F_k^{\nu} - F_k^{\nu-1} = f_k^{\nu,\nu-1} + g_k^{\nu,\nu-1}$, where:
    \begin{equation}
      f_{k}^{\nu,\nu-1} = f_k^\nu - f_k^{\nu-1} =
     {\sum}_{k_1 + k_2 = k}
     {\sum}_{\ell \leq k_1}{\sum}_{\ell' \leq k_2} \tS_{\ell}^{k_1,\nu} \pa \tV_{\ell}^\nu\cdot
     \tS_{\ell'}^{ k_2,\nu}
     \pa V^\nu_{\ell'}
     -\tS_{\ell}^{k_1,\nu-1} \pa \tV_{\ell}^{\nu-1}\cdot
     \tS_{\ell'}^{ k_2,\nu-1}
     \pa V^{\nu-1}_{\ell'},
    \end{equation}
    and
    \begin{multline}
     g_k^{\nu,\nu-1} = g_k^{\nu} - g_k^{\nu-1} = {\sum}_{\ell \leq k-1} \pa \tS^{k,\nu}_{\ell} \pa h^\nu_\ell-
      \pa \widetilde{S}^{k,\nu-1}_\ell \pa h_\ell^{\nu-1}
     + \tS^{k,\nu}_\ell \pa^2 h^\nu_\ell
     -\tS^{k,\nu-1}_\ell \pa^2 h_\ell^{\nu-1}\\
     + {\sum}_{\ell \leq k-1} {\sum}_{\ell' \leq \ell}
     \tS^{\ell,\nu}_{k-1} \pa \tS^{\ell',\nu}_\ell \pa h^\nu_{\ell'}
     -
     \tS^{\ell,\nu-1}_{k-1} \pa \tS^{\ell',\nu-1}_\ell \pa h^{\nu-1}_{\ell'}
     + \tS^{\ell,\nu}_{k-1} \tS^{\ell',\nu}_{\ell} \pa^2 h^\nu_{\ell'}
     - \tS^{\ell,\nu-1}_{k-1} \tS^{\ell',\nu-1}_{\ell} \pa^2 h_{\ell'}^{\nu-1}.
     \label{ghorrible}
    \end{multline}

    We first consider the case $r-k-2 \geq 2$.
After performing the usual manipulations and using that
$H^{r-k-2}$ is an algebra, to control $||f_k^\nu||_{H^{r-k-2}}$, it suffices
to prove that for $k' \leq k, \ell \leq k$, writing $ j =r-k-2$
\begin{equation}
 ||S^{k'}_\ell||_{H^{j}}
  + ||\tS^{k',\nu}_\ell||_{H^{j}}
  + ||\pa V_\ell||_{H^{j}}
  + {\sum}_{\alpha = 0,1} ||\pa \ssm^\alpha V_\ell||_{H^j}
  \leq K(E_0, m_{k}^\nu),
 \label{coefficients1}
\end{equation}
and that, with $K' = K'(E_0, m_{k}^\nu, m_k^{\nu-1})$:
\begin{equation}
 ||S^{k'}_\ell - \tS^{k',\nu}_\ell||_{H^{j}} +
 ||\pa V_\ell - \pa V_\ell^\nu||_{H^{j}} +
 ||\pa V_\ell - \pa \tV^{\nu}_\ell||_{H^{j}}
 \leq K'
 \big( {\sum}_{\ell \leq k} ||V_\ell - V_\ell^\nu||_{H^{j+1}}
 + ||u_{-1}^\nu||_{H^{j+2}} + \ve E_0\big).
 \label{differences1}
\end{equation}
 The first two terms in \eqref{coefficients1} are bounded by the right
 side of \eqref{coefficients1}
by Lemma \ref{Sestimates} and the other terms are
bounded  by the right side using the definition of the
$m_k^\nu$ and  the fact that
$\ssm$ is bounded on Sobolev spaces.

The first term in \eqref{differences1} is bounded
by the right-hand side of \eqref{differences1} by
Lemma \ref{Sestimates}, and the second term is directly
bounded by $||V_\ell - V_\ell^\nu||_{H^{r-k-1}}$. To control the
third term, we write $\tV_\ell^\nu = \ssm V_\ell^\nu
= \ssm V_\ell + \ssm (V_\ell^\nu - V_\ell)$. Using that
$||(1 - \ssm) \pa V_\ell||_{H^{r-k-2}} \leq C\ve ||V_\ell||_{H^{r-k}}$
and $||\ssm (V_\ell - V_\ell^\nu)||_{H^{s-k-1}} \leq C||V_\ell - V_\ell^\nu||_{H^{r-k-1}}$
gives the bound for $f_k^\nu$. The bound for $f_k^{\nu,\nu-1}$ follows in a nearly
identical way. The case $r -k-2 \leq 1$ is similar and follows the same lines
as the proof of e.g. \eqref{fbdslem}.

 We now bound $g_k^\nu$. We just prove estimates for the terms on the first line
 of \eqref{ghorrible} as the terms on the second line can be bounded in a similar
 manner. It suffices to prove that for $\ell \leq k-1$ with $j = r-k-2$:
 \begin{equation}
   {\sum}_{m = 0,1}
  ||\pa^m S^k_\ell||_{H^{j}} + ||\pa^m \tS^{k,\nu}_\ell||_{H^{j}}
  + ||\pa^{m+1} h_\ell||_{H^{j}} + ||\pa^{m+1} h_\ell^\nu||_{H^{j}}
  \leq K(E_0, m_k^\nu),
 \end{equation}
\begin{equation}
  {\sum}_{m = 0,1}
 ||\pa^m {}_{\!}S^k_\ell \!- \pa^m{}_{\!}\tS^{k,\nu}_\ell\!||_{H^{j}}\!
 + ||\pa^{m+1}\! h_\ell\! - \pa^{m+1} \!h_\ell^\nu||_{H^{j}}
 \!\leq \! K'\big( {\sum}_{\ell' \leq k}
 ||V_{\!\ell'} \!- \!V_{\!\ell'}^{\!\nu}||_{H^{j+1}} + ||u_k^{\nu}||_{H^{j+2}}
 + ||u_{-1}^{\nu}\!||_{H^{j+2}}\!\big).
\end{equation}
These estimates follow from
Lemma \ref{Sestimates} since $h_k^\nu \!= h_k + u_k^\nu$.
The estimates for $g_k^\nu \!- g_k^{\nu-1}\!\!$ are similar.
\end{proof}

To complete the proof of the estimates for $F_k-F_k^\nu$, we need
the following two estimates to relate $V_k, V_k^\nu$ to the initial data
$V_0, h_0$ and the perturbations $u^\nu$. We need a bit more notation.
Given a diffeomorphism $X:\Omega \to X(\Omega)$ and
a function $f:\Omega \to \R$, let $\Phi[X,f] = \phi\circ X^{-1}$, where
$\phi$ is defined by:
\begin{equation}
  (X, f) \mapsto \phi(x) = \int_{X(\Omega)} |x-x'|^{-1} \rho(f(x'))\,
  dx', \quad x \in \R^3.
\end{equation}
Set $\hat{x} = x_0 + t \sum_{k \geq 0}V_k t^k/(k+1)!$,
$\hat{x}^\nu(t) = x_0 + t \sum_{k \geq 0} V_k^\nu t^k/(k+1)!$
and write $x_\ell = D_t^\ell \hat{x}|_{t = 0},
x_\ell^\nu = D_t^\ell \hat{x}^\nu|_{t = 0}$.
Set
$\phi_\ell = D_t^\ell \Phi[\hat{x},\hat{h}]|_{t = 0}$ and $\phi_\ell^\nu =
D_t^\ell \Phi[\hat{x}^\nu, \hat{h}^\nu]|_{t = 0}$.
Then:
\begin{lemma}
  \label{Vkestlem}
  With notation as in the previous lemma, for each $k = -1,..., r-2$,
 there are continuous functions $K_0 = K_0(E_0), K_0' = K_0'(E_0,  ||u_{-1}^\nu||_{H^{r-k+1}},
 \tsum_{\ell \leq k-1} ||u^\nu_{\ell}||_{H^{r-k+1}})$,
 $K_0'' = K_0''(E_0, ||u^\nu||_r, ||u^{\nu-1}||_r)$ so that:
  \begin{equation}
  ||V_k||_{H^{r-k}} \leq K_0, \quad
  ||V_k^\nu||_{H^{r-k}} \leq K_0',
  \quad
  ||V_k - V_k^\nu||_{H^{r-k}} \leq K_0' \big( ||u^\nu_{-1}||_{H^{r-k+1}}
  + {\sum}_{\ell \leq k-1}
  ||u^\nu_\ell||_{H^{r-k+1}}\big),
\end{equation}
and with $U^\nu = u^\nu - u^{\nu-1}$:
\begin{equation}
  ||V_k^\nu - V_k^{\nu-1}||_{H^{r-k}} \leq
  K_0'' \big( ||U^\nu_{-1}||_{H^{r-k+1}} +
  {\sum}_{\ell \leq k-1}
  ||U^\nu_\ell||_{H^{r-k+1}}\big).
\end{equation}
\end{lemma}
\begin{proof}
 These estimates all follow from the definitions of $V_k, V_k^\nu$ in
 \eqref{Vkdef},\eqref{Vknudef}, the estimates for $S, \widetilde{S}$ from
 Lemma \ref{Sestimates}, and the estimates for $\phi_{k-1}, \phi_{k-1}^\nu$
 in Lemma \ref{phikestimates}.
\end{proof}
It still remains to control $\phi_{k-1}, \phi_{k-1}^\nu$. We shall not use this
observation to prove estimates, but we remark that we have
the following explicit representation formula for $\phi_{k-1}$:
\begin{equation}
 \phi_{k-1}(y)
 = {\sum}_{k_1 + \dots k_j + k' \leq k-1}
 \int_{\Omega} K_{k_1,\dots, k_j}(y,y') J_{k'}(y')\, dy',
\end{equation}
where, for some constants $d_{k_1\cdots k_j}$:
\begin{equation}
 K_{k_1,\dots, k_j}(y,y')
 =d_{k_{\!1}{\!}\dots{\!} k_{\!j}}\frac{\!\!(\delta  x_{k_1\!}\cdot_{\!} \delta x_{k_2\!})
 \cdots
  (\delta x_{k_{\!j\!-\!1}}\cdot\delta x_{k_j})\!\! }{|x_0(y)- x_0(y')|},
 \qquad
 J_{k'} = D_t^{k'} (\rho(\hat{h}) \hat{\kappa})|_{t = 0},
\end{equation}
with $\hat{\kappa} = \det (\pa y/ \pa\hat{x})$, and where we are
writing $\delta W(y,y)
= (W(y) -W(y'))/|x_0(y)- x_0(y')|$. Similarly:
\begin{equation}
 \phi^\nu_{k-1}(y)
 = {\sum}_{k_1+ \cdots k_j + k' \leq k-1}
 \int_{\Omega} K_{k_1,..., k_j}^\nu(y,y')
 J_{k'}^\nu(y')\, dy',
\end{equation}
where
\begin{equation}
  K_{k_1,\dots, k_j}^\nu(y,y')
  = d_{k_{\!1}{\!}\dots{\!} k_{\!j}}
  \frac{\!\!(\delta  x^\nu_{k_1\!}\cdot_{\!} \delta x^\nu_{k_2\!})
  \cdots
   (\delta x^\nu_{k_{\!j\!-\!1}}\cdot\delta x^\nu_{k_j})\!\! }{|x_0(y)- x_0(y')|},
  \qquad
  J_{k'}^\nu = D_t^{k'} (\rho(\hat{h^\nu}) \hat{\kappa^\nu})|_{t = 0},
  \qquad \widehat{\kappa}^\nu = \det (\pa \hat{x}^\nu/\pa y).
\end{equation}

\begin{lemma}
\label{phikestimates}
  With notation as in Lemma \ref{Fklem}, for each $k = 0,..., r-2$,
  there are continuous functions $ K_0 = K_0(E_0)$, $K_0' = K_0'(E_0, m_{k-1}^\nu)$,
  $K_0'' = K_0''(E_0, ||u^\nu||_s, ||u^{\nu-1}||_s)$
  so that:
 \begin{equation}
  ||\phi_{k-1}||_{H^{r-k+1}} \leq K_0, \qquad
  ||\phi_{k-1}^\nu||_{H^{r-k+1}} \leq K_0'.
   \end{equation}
 \begin{equation}
||\phi_{k-1} - \phi_{k-1}^\nu||_{H^{r-k+1}} \leq K_0'
  {\sum}_{\ell \leq k-1}
  ||u_{\ell}^\nu||_{H^{r-k}},
  \qquad
  ||\phi_{k-1}^\nu - \phi_{k-1}^{\nu-1}||_{H^{r-k+1}} \leq K_0^{''}
  {\sum}_{\ell \leq k-1} ||U_{\ell}^\nu||_{H^{r-k}}.
  \label{}
 \end{equation}
\end{lemma}
\begin{proof}
 The estimates  follows from Theorem \ref{main theorem, ell est D_t phi}, respectively Theorem \ref{diffellphi full thm}.
\end{proof}

\section{Existence and estimates for the wave equations}
\label{waveequationsection}

\subsection{Existence for the linear wave equations}
\label{linwaveexist}
Fixing $r \geq 7, T > 0,  V \in \X^{r+1}(T)$ and defining
$\Dve = \Dve[V]$ as in \eqref{laplsm},
the goal of this section is to solve the  linear wave equation:
\begin{align}
 D_t^2 \varphi - \eprimee \Dve \varphi
  &= \F,  \quad\textrm{ in }
 [0,T]\times \Omega,\quad\text{with}\quad
 \varphi=0, \textrm{ on } [0,T]\times \pa\Omega,\label{appwave1}\\
 \varphi(0,y) &= \varphi_0(y),\quad
 D_t\varphi(0,y) = \varphi_1(y),
 \quad \textrm{ in } \Omega\label{appwave3},
\end{align}
where  $\eprime = \eprime(t,y)$ satisfies
$0 < c_0 < \eprime  \leq c_1$ for some constants
$c_0, c_1$. We will omit the dependence on $c_0, c_1$ in what follows.
Compared with the estimates in Section
\ref{waveests}, we have divided by $\eprimee$ and abused notation
slightly to make the following computations simpler.

As in Section \ref{wavecompatcondn}, there are compatibility conditions
for \eqref{appwave1}-\eqref{appwave3}.\!\! We say $\varphi_0, \varphi_1$
satisfy the compatibility condition to order $s$ if there is a formal
power series in $t$, $\hat{\varphi} = \sum t^k \varphi_k$, satisfying
\eqref{appwave1}-\eqref{appwave3} and:
\begin{equation}
 \varphi_k \in H^1_0(\Omega), \qquad
 k = 0,...,s.
 \label{varphicc}
\end{equation}
Note that since $\hat{\varphi}$ satisfies \eqref{appwave1}, the coefficients
$\varphi_k$ can be computed recursively from $\varphi_0, \varphi_1$ and time
derviatives of $F$ at $t = 0$:
\begin{equation}
 \varphi_k = D_t^{k-2} \big( \sigma \Dve \widehat{\varphi} + F\big)\big|_{t = 0}.
 \label{}
\end{equation}

We will control solutions to \eqref{appwave1}-\eqref{appwave3}
using the quantities:
\begin{align}
 Y_s(t) = \Big({\sum}_{k \leq s}
 \int_\Omega  |D_t^{k+1} \varphi|^2 +
 \eprimee\delta^{ij} (D_t^{k}\pave_i \varphi)(D_t^{k}\pave_j \varphi) \, dy\Big)^{1/2}.
\end{align}

The main result we need is:
\begin{prop}
  \label{linearwavelwp}
  Fix $r\! \geq \!7$ and
  $T \geq 0$, and suppose that $V\! \!\in\! \X^{r+{}_{\!}1}(T)\!$ satisfies
  \eqref{admissible}. Suppose also that:
  \begin{align}
   \xve &\in L^\infty([0,T] ; H^r(\Omega)),\label{llwpxassump} \\
   D_t \xve \in L^\infty([0,T]; H^r(\Omega)),\quad \text{and}\quad
   D_t^k D_t \xve &\in L^\infty([0,T]; H^{r-k+1}(\Omega)),
    &&k = 1,..., r+1,\label{llwpvassump}\\
    D_t^k\eprime &\in L^\infty([0, T];H^{r-k}(\Omega)),
   &&k = 0,..., r\label{llwpsigassump},
\end{align}
and that the bound \eqref{ubd2} holds.
Also assume that for some $s$ with $0 \leq s \leq r$,
\begin{align}
   \F \in L^\infty([0,T]; H^{s-1}(\Omega)),\quad \text{and}\quad
   D_t^k \F &\in L^\infty([0,T]; H^{s-k}(\Omega)),
   &&k = 1,..., s,\label{llwpfassump}
  \end{align}
  and that the compatibility condition \eqref{varphicc} holds for $k = 0,...,s$.
  Take $K = K_{s,r}$ so that
  \begin{equation}
    {\sup}_{\,0 \leq t \leq T\,}\big( ||\xve(t)||_{r}
  + ||V(t)||_{r} + ||D_t V(t)||_r
  + ||\eprime(t)||_r +
  ||F(t)||_{s-1}\big) \leq K.
  \label{lwpkbd}
  \end{equation}
  Then the problem
 \eqref{appwave1}-\eqref{appwave3} has a unique solution
 $\varphi$ satisying:
 \begin{align}
  D_t^{s+1} \varphi \in L^\infty([0,T]; L^2(\Omega)),
  \qquad
  D_t^{\ell} \pave \varphi \in L^\infty([0,T];
  H^{s-\ell}(\Omega)),
  &&\ell = 0,..., s,
  \label{llwpreg}
 \end{align}
 and there are continuous functions $\C_s$ depending
 on $M, Y_{s-1}(0), T, $ and $K$ so that:
 \begin{equation}
   {\sup}_{\,0 \leq t \leq T} Y_s(t) \leq
   \C_s\Big(Y_s(0) +  \int_0^T ||\F(\tau)||_{s,0}
   \,d\tau \Big),
    \label{appwaveen}
   \end{equation}
 and  for $0 \leq t \leq T$
 \begin{align}
  ||\pave \varphi(t)||_{s}
  \leq \C_s ( Y_s(t) +
  ||\F(t)||_{s-1} ),
  \qquad
  ||\pave \varphi(t)||_{r}
  \leq \C_r \big( Y_r(t)  +
  ||\F(t)||_{r-1} +
  \ve^{-1}(||\sm x(t)||_r + 1)Y_{r-1}(t) \big).
  \label{appwaveell2}
 \end{align}
\end{prop}

This result is well-known (see e.g. \cite{Hormander2007} or
\cite{Evans2010}) and will follow from a Galerkin method.
However, we will need to be careful about the
regularity of $\xve$ and we will use our elliptic estimates from
Section \ref{elliptic}
 in place of ``standard'' elliptic estimates. We do not claim
 that this result is optimal with respect to the total number of
 derivatives of $\xve, V, D_tV,\eprimee$ required
 and in many of the following results it is obvious that one can
 do with much weaker assumptions on these variables.
We start by constructing weak solutions to the system
\eqref{appwave1}-\eqref{appwave3}.
Let $\{e_k\}_{k = 0}^\infty$ be the $L^2$-normalized eigenfunctions
in $H^1_0(\Omega)$ of the
Dirichlet Laplacian in the $y$-coordinates $\Delta_y \!= \pa_{1}^2\! + \pa_{2}^2\!
+ \pa_{3}^2$.
Let $\dm^k\! \in\! C^2([0,T])$, $k \!=\! 1,..., m$ solve the following system:
\begin{align}
 D_t^2 \dm^k +
 B_{k}(\dm)
 &= \int_\Omega \F e_k\,dy,
 && k =1,..., m, \label{ode}\\
 \dm^k(0) = (\varphi_0, e_k),  \qquad
 D_t \dm^k(0) &= (\varphi_1, e_k),  && k = 1,..., m,
\end{align}
where:
\begin{equation}
B_{k}(\dm)
=
{\sum}_{\ell = 1}^k
 \dm^{\,\ell} \int_\Omega \eprimee \delta^{ij}
 (\pave_i e_\ell)
 (\pave_j e_k)\, dy.
\end{equation}
 Define:
\begin{equation}
 \varphim(t) = {\sum}_{k = 1}^m \dm^k(t) e_k.
 \label{varphim}
\end{equation}
Multipling \eqref{ode} by $\dm^k$, summing over $k \leq m$ and using
\eqref{varphim}, we have:
\begin{align}
  \int_\Omega  D_t^2 \varphim e_k \, dy
  + \int_\Omega \eprimee \delta^{ij}
  (\pave_i \varphim )(\pave_j e_k)\,  dy =
  \int_\Omega \F e_k\,  dy, && k=1,..., m.
 \label{weak}
\end{align}
 We now prove the basic energy estimate:
  \begin{lemma}
   If $\varphim$ is as above, there is a constant $C_0 = C_0(M, K)$ so that:
   \begin{multline}
    {\max}_{\,0 \leq t \leq T} \big(||D_t \varphim(t)||_{L^2(\Omega)}
    + ||\pave \varphim(t)||_{L^2(\Omega)}\big)
    + || D_t^2 \varphim||_{L^2(0,T; H^{-1}(\Omega))}\\
    \leq C_0\big( ||\varphi_0||_{H^1(\Omega)} + ||\varphi_1||_{L^2(\Omega)}
    + ||\F||_{L^2(0,T;L^2(\Omega))}\big).
    \label{waveunif}
   \end{multline}
  \end{lemma}
\begin{proof}
  We multiply \eqref{weak} by $D_t d_m^k$ and sum over $k = 1,..., m$ to get:
  \begin{equation}
   \int_\Omega (D_t^2 \varphim) (D_t \varphim)
   \, dy
   + \int_\Omega \eprime \delta^{ij} (\pave_i \varphim) (\pave_j D_t \varphim)\,
    dy
   = \int_\Omega \F D_t \varphim dy.
   \label{enidentm}
  \end{equation}
  This first term is
  ${2}^{-1}{d} \,|| D_t \varphim||_{L^2(\Omega)}/{dt}$.
  We use \eqref{dtpavecomm} and write the second term as:
  \begin{multline}
   \int_\Omega \eprimee  \delta^{ij}(\pave_i \varphim)D_t (\pave_j\varphim) dy
   - \int_\Omega \eprimee \delta^{ij}(\pave_jV^\ell) (\pave_i \varphim)(\pave_\ell \varphim)
   dy\\
   =
   \frac{1}{2} \frac{d}{dt} ||\sqrt{\eprimee}\pave \varphim||_{L^2(\Omega)}^2
   - \int_\Omega \delta^{ij} (\pave_i \varphim)
   \big(\eprimee (\pave_j \ssm V^\ell)(\pave_\ell \varphi) + D_t \eprimee
   \pave_j \varphim\big)\, dy,
   {}
  \end{multline}
  and we can bound this last term by $C(M)(1 + ||D_t \eprime||_{L^\infty(\Omega)})
  ||\pave \varphim||_{L^2}^2$.

  Writing $Y_{(m)}(t) = ||D_t \varphim||_{L^2}
  + ||\sqrt{\eprimee} \pave \varphim||_{L^2}$, we have shown:
  \begin{equation}
   \frac{1}{2} \frac{d}{dt} Y_{(m)}^2
   \leq C(M) \Big( (1 +||D_t \eprime||_{L^\infty(\Omega)} ) Y_{(m)}^2
   +  ||\F||_{L^2} Y_{(m)}\Big),
   {}
  \end{equation}
  and so using that ${d}(Y_{(m)})^2 \!/{dt} =
  2 Y_{(m)} {d} Y_{(m)}/{dt}$, dividing both sides
  by $Y_{(m)}$ and multiplying by the integrating
  factor $e^{-C(M)(1 +||D_t \eprime||_{L^\infty(\Omega)}) t}$, we get:
  \begin{equation}
   {\sup}_{\,0 \leq t \leq T}
   \big(|| D_t \varphim(t)||_{L^{{}_{\!}2}}{}_{\!}
   + ||\sqrt{\eprimee}\pave \varphim(t)||_{L^{{}_{\!}2}} \!\big)
   \!\leq\!
   C \Big(
   ||  D_t\varphim(0)||_{L^{{}_{\!}2}}{}_{\!}
   + ||\!\sqrt{\!\eprime}\pave \varphim(0)||_{L^{{}_{\!}2}}{}_{\!}
   + \!\!\int_0^T \!\!\!\! \!||\F(\tau)||_{L^{{}_{\!}2}} d\tau\Big).
   {}
  \end{equation}
  where $C = C(M,\sup_{0 \leq t \leq T}||\xve||_{r} +
  ||\eprime(t)||_{L^\infty}, T)$.
  Using the orthogonality of the $e_k$, we have
  $$|| D_t \varphim(0)||_{L^2(\Omega)} + ||\sqrt{\eprime} \pa\varphim(0)||_{L^2(\Omega)}
  \leq ||\varphi_1||_{L^2(\Omega)} + ||\sqrt{\eprime} \varphi_0||_{H^1(\Omega)},$$
  which proves the first part of \eqref{waveunif}.
  We now control
  $||D_t^2 \varphim||_{H^{-1}(\Omega)}$.

  Let $v \in H^1_0(\Omega)$ so that $||v||_{H^1(\Omega)} = 1$, and
  split $v = v^1 \!\!+ v^2$ with $v^1$ in the span of $e_1,..., e_m$. Then we have:
  \begin{equation}
   \langle D_t^2 \varphim, v\rangle
   = \langle  D_t^2\varphim, v^1\rangle
   = ( D_t^2\varphim, v^1)_{L^2} =
   -(\eprimee \pave \varphim, \pave v^1)_{L^2}
    + (\F, v^1)_{L^2}.
   {}
  \end{equation}
  The right-hand side is bounded by
  $C(M)(
  ||D_t\varphim||_{L^2} + c_0||\pave \varphim||_{L^2(\Omega)}
  + ||\F||_{L^2})||v^1||_{H^1(\Omega)}$. Noting that
  $||v^1||_{H^1(\Omega)} \leq ||v||_{H^1(\Omega)} = 1$ and integrating
  in time gives the bound for $||D_t^2 \varphim||_{L^2(0,T;H^{-1}(\Omega))}$.
\end{proof}

\begin{lemma}
  \label{waveexist1}
 With assumptions as in Proposition
 \ref{linearwavelwp}, there is a unique $\varphi\! \in\! C([0,T];H^1_0(\Omega))$  satisfying \eqref{appwave1}-\eqref{appwave3}
 with
 \begin{equation}
  D_t \varphi \in L^\infty(0,T;L^2(\Omega)),
  \qquad
  D_t^2 \varphi \in L^\infty(0,T; H^{-1}(\Omega)).
 \end{equation}
\end{lemma}
\begin{proof}
  By the uniform estimate \eqref{waveunif} and
 Alaoglu's theorem, passing to a subsequence we see that there is
a $\varphi \in L^2(0,T;
H^1_0(\Omega))$ with $D_t \varphi \in L^2(0,T; L^2(\Omega)),
 D_t^2 \varphi \in L^2(0,T;
H^{-1}(\Omega))$ so that $\pave \varphim \to \pave\varphi$
weakly in $L^2(0,T; L^2(\Omega))$,
$D_t \varphim \to D_t \varphi$ weakly in $L^2(0,T; L^2(\Omega))$
 and $D_t^2 \varphim \to \varphi$ weakly
in $L^2(0,T; H^{-1}(\Omega))$. Concretely, this means that if $v \in H^1_0(\Omega)$ then:
\begin{equation}
 \int_0^T \!\!\!\!\int_{\Omega}\! (\pave^k\! \varphim(t,y)) \pave^k\!v(t,y)\, dydt
  \to\!\! \int_0^T\!\!\!\! \int_\Omega\!
 (\pave^k \!\varphi(t,y)) \pave^k\! v(t,y) \, dydt,\qquad k\!=\!0,\!1, \quad\text{and}\quad  \langle D_t^2 \varphim\!\!, v \rangle \to \langle D_t^2 \varphi, v \rangle.
\end{equation}

Now, given $v \in C^1([0,T]; H^1_0(\Omega))$
of the form:
\begin{equation}
 v(t) = {\sum}_{k = 1}^M v_k(t) e_k,
\end{equation}
we multiply the weak formulation \eqref{weak} by $v_k(t)$,
sum over $k$ and integrate over $[0,T]$ to get that:
\begin{equation}
 \int_0^T \langle   D_t^2\varphim, v\rangle\,
 dt
 + \int_0^T \int_\Omega \eprimee \delta^{ij}(\pave_i \varphim)(\pave_j v)
 \,
  dy dt
 = \int_0^T \int_\Omega \F v  \,dy dt.
 \label{stillweak}
\end{equation}
Taking $m \to \infty$ and using the above limits, we get that for
$v$ of the above form:
\begin{equation}
 \int_0^T \langle D_t^2 \varphi, v\rangle\,dt
 + \int_0^T \int_\Omega \eprimee\delta^{ij}(\pave_i \varphi) (\pave_j v)
 \, dy dt
 = \int_0^T \int_\Omega \F v \, dy dt.
 \label{lessweak}
\end{equation}
Since such $v$ are dense in $L^2(0,T; H^1_0(\Omega))$ this holds for
any $v$ in this space. Hence for almost every $t$
\begin{equation}
 \langle  D_t^2 \varphi(t), v(t)\rangle
 +  \int_\Omega \eprimee \delta^{ij}(\pave_i \varphi(t))(\pave_j v(t))
 \,dy
 =  \int_\Omega \F(t) v(t) dy,\qquad v \in H^1_0(\Omega).
\end{equation}

By an approximation argument
and the fundamental theorem of calculus, using that
$\varphi$ and its time derivatives are all in $L^2$ in time,
we also get that $\varphi \in C([0,T];L^2(\Omega))$ and
$D_t\varphi \in C([0,T];H^{-1}(\Omega))$ (see \cite{Evans2010}).
Hence \eqref{appwave3} makes sense.
We now have to check that $\varphi(0) = \varphi_0$ and $D_t \varphi(0) =
\varphi_1$. Let $v \in C^2([0,T];
H^1_0(\Omega))$ be such that $v(T) = D_t v(T) = 0$ and integrate by
parts twice in time in \eqref{lessweak} to get:
\begin{equation*}
 \int_0^T\!\!\! \int_\Omega\!\!
  D_t^2 v(t) \varphi(t) \,dydt
 + \int_0^T\!\!\! \int_\Omega\!\! \eprimee \delta^{ij}(\pave_i \varphi(t))( \pave_j v(t))\,
 dy dt
 \\
 = \int_0^T\!\!\! \int_\Omega\!\! \F(t) v(t) \, dy dt +
  \int_\Omega\!\!
 v(0) D_t\varphi(0) \! - \!
  D_t v(0) \varphi(0) \,  dy.
 {}
\end{equation*}
On the other hand we can integrate by parts twice in \eqref{stillweak}
and take $m \to \infty$ to get also:
\begin{equation*}
 \int_0^T \!\!\! \int_\Omega \!\! D_t^2 v (t) \varphi(t)\, dy dt
 + \int_0^T\!\!\! \int_\Omega\!\! \eprimee \delta^{ij}(\pave_i \varphi(t))
 ( \pave_j v(t))
 \,dy dt
 = \int_0^T\!\!\! \int_\Omega\!\! \F(t) v(t)\,
  dy dt +
  \int_\Omega\!\!
 v(0) \varphi_1 \!-\!
  D_t v(0) \varphi_0 \, dy.
 {}
\end{equation*}
Comparing these expressions  using that $v(0)$ and $D_t v(0)$ are arbitrary,
we have $\varphi(0)\! = \!\varphi_0$ and $D_t\varphi\! = \!\varphi_1$.
\end{proof}

We now want to show that we get improved regularity of $\varphi$ when
$\varphi_0, \varphi_1$ and $\F$ are more regular. The first step
is to show that the coefficients $d_m^\ell$ are more regular in this case
and for this we take time derivatives of the equation \eqref{appwave1}.
We apply $n \leq r-1$ time derivatives to \eqref{appwave1} and write
 $D_t \Dve \varphi = \delta^{ij} \pave_j (D_t \pave_i \varphi)
- \delta^{ij} (\pave_j V^\ell)\pave_\ell \pave_i \varphi$.
We write the result as:
\begin{equation}
 D_t^{n+2} \varphi
 -\delta^{ij}\pave_j (\eprime D_t^n \pave_i \varphi )
 -\delta^{ij}\pave_{j} (D_t \eprime D_t^{n-1} \pave_i\varphi)
 + \delta^{ij}\pave_{j'}\big( (\pave_j \ssm V^{j'})D_t^{n-1} \pave_i \varphi\big)
 = F^n,
 \label{ode5}
\end{equation}
where:
\begin{multline}
 F^n = D_t^n F
 - {\sum}_{s = 2}^{n}
 (D_t^s \eprime)(D_t^{n-s} \Dve \varphi)\\
 +
 {\sum}_{s = 1}^n
 \delta^{ij}(D_t^s A_{\m j}^b) (D_t^{n-s} \pa_b \pave_i \varphi)
 + \delta^{ij}(\pave_j \eprime + \pave_j D_t \eprime) D_t^n \pave_i\varphi
 - \delta^{ij} (\pave_{j'} \pave_j \ssm V^{j'}) D_t^{n-1}\pave_i \varphi.
 \label{Fkterm}
\end{multline}
We write \eqref{ode5} like this because the third and fourth terms
have as many space derivatives of $\varphi$ as the second term but fewer
time derivatives, and so we will need to integrate by parts in space and
time to handle them. The terms in $F^k$ will be lower-order and can
be bounded in $L^2$ directly.

Multiplying this by arbitrary $v \in H^1_0(\Omega)$ and integrating
by parts leads to the equation:
\begin{equation}
\int_\Omega (D_t^{n+2} \varphi) v\, dy
 + \int_\Omega \eprimee \delta^{ij} (D_t^n \pave_i \varphi)
 (\pave_j v)\, dy+
\int_\Omega \delta^{ij} (D_t^{n-1}\pave_i \varphi)
 \big( D_t \eprimee (\pave_j v)
\pave_j \ssm V^{j'})(\pave_{j'} v)\big)\, dy
 = \int_\Omega F^n v\, dy.
 {}
\end{equation}
With $d_m^\ell$ defined by \eqref{ode}, suppose that $d_m^\ell \in C^n([0,T])$
for some $n \geq 1$
and define:
\begin{equation}
 B_k^n  = B_k^n(d_m,..., D_t^n d_m) = {\sum}_{\ell \leq k} \int_\Omega \eprimee \delta^{ij}
 D_t^n(d_m^\ell \pave_i e_\ell)\pave_j e_k\, dy,
\end{equation}
\begin{equation}
 C_k^n = C_k^n(d_m,..., D_t^{n-1} d_m)  = {\sum}_{\ell \leq k} \int_{\Omega}
 \delta^{ij} D_t^{n-1}( d_m^\ell \pave_i e_\ell)
 \big( D_t \eprimee \pave_j e_k + (\pave_j \ssm V^{j'})\pave_{j'} e_k\big)\, dy.
 \label{}
\end{equation}
Also let $F^n(d_m)$ be $F^n$ with $\varphi$ replaced by
$\varphim = {\sum}_{k \leq m}  d_m^k e_k$. Let $\dot{d}_m^1,..., \dot{d}_m^k$
solve the ODE:
\begin{equation}
 D_t \dot{d}_m^k + B_k^n + C_k^n = \int_{\Omega} F^n(d_m) e_k\, dy,
 \quad
 \dot{d}_m^k(0) = (\varphi_{n}, e_k)_{L^2(\Omega)}, \quad k = 1,..., m,
 \label{odedot}
\end{equation}
where $\varphi_{n}$ is defined by \eqref{varphicc}. By the existence and uniqueness theorem for ODE, it follows that
$\dd^{\,k}(t)\! =\! D_t^{n} d_{m}^{\,k}(t)$ for $0\! \leq \! t\! \leq \!T$
and this implies that $d_{m}^{\,k} \!\!\in \!C^{n+1}(0, \!T)$.

Before proving that the sequence $\varphim$ converges in stronger
topologies, we will need to ensure that $\varphi$ satisfies the equation \eqref{appwave1}
almost everywhere. We start with:

\begin{lemma}
  \label{weakdt}
  Suppose that the
 hypotheses of Proposition \ref{linearwavelwp} hold.
 Let $\varphi$ be as in Lemma \ref{waveexist1}.
 If
 $\varphi_0 \in H^2(\Omega),\varphi_1 \in H^1_0(\Omega)$
 and $D_t\F \in L^2(0,T;L^2(\Omega))$, then we have the
 improved regularity:
 \begin{equation}
   D_t \varphi \in C([0, T]; H^1_0(\Omega)),\qquad
  D_t^2 \varphi \in L^\infty([0,T]; L^2(\Omega)),\qquad
  D_t^3\varphi \in L^\infty([0,T]; H^{-1}(\Omega)),
  \label{weakdtbds}
 \end{equation}
\end{lemma}
\begin{proof}
  Take $n = 1$,
  multiply \eqref{odedot} by $D_t^2 d_{m}^{\,k}$, use $\dot{d}_m^k= D_t d_m^{\,k}$
  and write:
  \begin{equation}
   \pave_j e_k D_t^2 d_{m}^{\,k}
   = \pave_j D_t^2 \varphim
   = D_t^2 \pave_j \varphim
   - (D_t^2 A_{\m j}^a) \pa_a \varphim
   - 2 (D_t A_{\m j}^a)D_t \pa_a \varphim
   \equiv D_t^2 \pave \varphim + R^1_j,
  \end{equation}
  which gives:
  \begin{multline}
   B^1_k D_t^2 d_{m}^{\,k}
   = \int_\Omega \eprimee \delta^{ij} (D_t \pave_i \varphim)
   (\pave_j e_k D_t^2 d_{m}^k)\,  dy
   = \frac{1}{2}\frac{d}{dt}\Big(\int_{\Omega}
   \eprime \delta^{ij} (D_t \pave_i \varphim)( D_t\pave_j \varphim)\,  dy\Big)\\
   -\int_\Omega (D_t \eprimee ) \delta^{ij} (D_t \pave_i \varphim)
   (D_t \pave_j \varphim)\,  dy
   -\int_\Omega \eprimee \delta^{ij} D_t \pave_i\varphim
   \big( (D_t^2 A_{\m j}^a) \pa_a \varphim - 2 (D_t A_{\m j}^a)D_t \pa_a\varphim\big)
   \, dy.
  \end{multline}
  and similarly:
  \begin{multline}
   C^1_k D_t^2 d_{m}^{\,k} =
   \frac{dC_1}{dt}\\
   - \int_\Omega (\pave_i \varphim)\big((D_t^2 \eprimee)
   (D_t \pave_j \varphim)
   - D_t \big( \eprimee (\pave_j \ssm V^{j'})\big)
   (D_t \pave_{j'} \varphim)\big)\, dy
   -\int_\Omega
   \eprime \delta^{ij}
   (D_t\pave_i \varphim)\big(R_j^1
-  (\pave_j \ssm V^{j'})R_{j'}^1 \big)
   \, dy,
   {}
  \end{multline}
  where:
  \begin{equation}
   C_{1} = C_{1}[\varphim]=  \int_\Omega (D_t \eprimee)\delta^{ij}
    (\pave_i \varphim ) (D_t \pave_j\varphim)
    -\delta^{ij}\eprimee (\pave_j \ssm V^{j'})
    (\pave_i\varphim)( D_t \pave_{j'}\varphim) dy.
   {}
  \end{equation}

  By Sobolev embedding and \eqref{dinv}:
  \begin{align}
   ||D_t^2 A_{\m j}^a||_{L^\infty(\Omega)}
   + ||D_t A_{\m j}^a||_{L^\infty(\Omega)}
   + ||D_t \pave_j\ssm V^\ell||_{L^\infty(\Omega)}
   &\leq C(M) ||\xve||_r,\\
   ||D_t \eprime||_{L^\infty(\Omega)}
   + ||\pave D_t \eprime |||_{L^\infty(\Omega)}
   + ||D_t^2 \eprime||_{L^\infty(\Omega)}
   &\leq C ||\eprime||_r.
   {}
  \end{align}
  With $Y^1_{\!\!m}\! = \! ||D_t^2 \varphim||_{L^2(\Omega)} \!+\!
  ||\sqrt{\eprime} D_t \pave \varphim||_{L^2(\Omega)}$, the above calculation
  shows that:
  \begin{equation}
   \frac{d}{dt} \Big((Y^1_{m})^2 -
   C_1[\varphim]\Big)
   \leq C(M, ||\xve||_r) (1 + ||\eprime||_r)\Big(
   Y^1_{m} + Y_{m} + ||F_{m}^1||_{L^2(\Omega)}\Big)Y^1_{m}.
   {}
  \end{equation}
  Multiplying both sides by the integrating factor
  $e^{-C(M, ||\xve||_r)(1 + ||\eprime||_r) t}$, integrating, and then using that
  $C[\varphim] \leq C(M) ||\xve||_r (\delta (Y^1_{m})^2 + \delta^{-1} Y_{m}^2)$
  for any $\delta > 0$, this implies that:
  \begin{equation}
   Y^1_{m}(t)^2
   \leq C(M, ||\xve||_r, ||\eprime||_r) \Big(
   Y^1_{m}(0)^2
    + Y_{m}(t)^2
   + \int_0^t Y_{m}^1(\tau)^2 +
    Y_{m}(\tau)^2 + ||D_t F_{m}^1(\tau)||_{L^2(\Omega)}^2\, d\tau\Big),
   {}
  \end{equation}
  and so by Gr\"{o}nwall's integral inequality, this implies:
  \begin{equation}
   {\sup}_{\,0 \leq t \leq T}
  Y^1_{m}(t) \leq
   C \Big( Y_{m}^1(0)
   +{\sup}_{\,0 \leq t \leq T} Y_{m}(t)
   + \int_0^T (1 + ||\eprime||_r)Y_{m}(\tau) +
   ||F_{m}^1(\tau)||_{L^2(\Omega)}\, d\tau\Big).
   {}
  \end{equation}
  Arguing as in the previous lemma, this implies that the sequence
  $D_t \varphim$ has limit $\dot{\varphi}$ with:
  \begin{equation}
   D_t \dot{\varphi} \in L^\infty(0,T;L^2(\Omega)),\qquad
   D_t^2 \dot{\varphi} \in L^\infty(0, T; H^{-1}(\Omega)).
   {}
  \end{equation}
  Since also $\varphim \to \dot{\varphi}$ in $L^2$
  by the previous lemma, it
  follows that $\varphi = \dot{\varphi}$ and in particular
  we get the first two statements in \eqref{weakdtbds}. To get that
  $D_t^3 \varphi \in L^\infty([0,T]; H^{-1}(\Omega))$, we argue
  as in the previous lemma. Also, since the compatibility conditions
  \eqref{varphicc} hold, we have that $Y^1_{m}(0) \to Y^1(0) =
  ||\varphi_2||_{L^2(\Omega)} + ||\sqrt{\eprimee(0)} \pave \varphi_1||_{L^2(\Omega)}$.
\end{proof}

We can now prove that $\varphi$ has enough regularity that
the elliptic estimates from the Section \ref{elliptic} hold:
\begin{lemma}
  \label{weak2}
 If $\varphi_0\! \in \!H^2(\Omega), \varphi_1 \!\in\! H^1_0(\Omega)$
and \eqref{ubd2},\eqref{lwpkbd} hold,
there is a constant $C_1\! \!=\! C_1( M, K, T)$ so that:
 \begin{multline}
  \esssup_{\,0 \leq t \leq T\,}
  \big( ||\varphi(t)||_{H^2(\Omega)} +
  ||D_t \varphi(t)||_{H^1_0(\Omega)}
  + || D_t^2 \varphi(t)||_{L^2(\Omega)}\big)
  + || D_t^3 \varphi||_{L^2(0,T;H^{-1}(\Omega))}
  \\ \leq
  C_1
  \big( ||\F||_{H^1(0,T; L^2(\Omega)} +
  ||\varphi_0||_{H^2(\Omega)} + ||\varphi_1||_{H^1(\Omega)}\big).
  \label{weaken2}
 \end{multline}
\end{lemma}
\begin{proof}
  By the previous lemma, we already have the second, third and
  fourth estimates in \eqref{weaken2} and it just remains to
  bound the first term. The point is that we do not yet know that
  the wave equation \eqref{appwave1} holds almost everywhere
  so we cannot use the elliptic estimate \eqref{sobell}.
  As in \cite{Evans2010},
   will instead prove an elliptic estimate for the
  approximate solution $\varphim$. We let $\{\lambda_\ell\}_{\ell = 0}^\infty$
  be the eigenvalues of $\Delta$ on $H^1_0(\Omega)$. Multiplying both
  sides of \eqref{ode} by $\lambda_\ell d_{(m)}^{\,\ell}$ and summing from
  $\ell = 1$ to $m$, we get that:
  \begin{equation}
   \int_\Omega \eprimee \delta^{ij}
   (\pave_i \varphim) (\pave_j \Delta \varphim)\, dy
   = \int_\Omega (\F - D_t^2 \varphim) \Delta \varphim\, dy.
   {}
  \end{equation}
  Since $\Delta \varphim = 0$ on $\pa \Omega$, we integrate by parts in the
  left-hand side and use the estimate \eqref{crossterm}, which gives:
  \begin{equation}
   ||\pave \varphim||_{H^1(\Omega)}
   \leq C(M) \big( ||D_t^2 \varphim||_{L^2(\Omega)}
   + ||D_t \pave \varphim||_{L^2(\Omega)}
   + ||\pave \varphim||_{L^2(\Omega)} +
   ||\varphim||_{L^2(\Omega)}\big).
   {}
  \end{equation}

Since $\varphi \in H^2(\Omega)$, we now have that $\varphi$
 solves the equation \eqref{appwave1}-
\eqref{appwave3} a.e. in $[0,T]\times\Omega$.
\end{proof}

\begin{proof}[Proof of Proposition \ref{linearwavelwp}]

  We argue by induction. We have just shown that the
  theorem holds for $s = 0,1$.
  We suppose that the theorem holds for
  $s = 1,..., n-1 \leq r-1$ and we now assume that
  the compatibility conditions \eqref{varphicc}
  hold for $s = 0,..., n$.
  By the inductive assumption, there is a unique $\varphi$ satisfying the
  equation \eqref{appwave1}-\eqref{appwave3} in the weak sense
  so that:
  \begin{equation}
    D_t^{n} \varphi \in L^\infty([0,T] ; L^2(\Omega)),\qquad
    D_t^{n - \ell} \pave \varphi \in L^\infty([0,T]; H^{\ell-1}(\Omega)),
    \qquad \ell = 0,...,n.
   \label{llwpind}
  \end{equation}
Moreover, with $\varphim$ as defined above, we have that
$D_t^n \varphim(t) \to D_t^s \varphi(t)$ in $L^2(\Omega)$ and
$D_t^{n-\ell}\pave\varphim(t) \to D_t^{n-\ell} \pave\varphi(t)$ in $H^{\ell-1}(\Omega)$
for $\ell = 0,...,n$ and $0 \leq t \leq T$.
We
multiply \eqref{ode} by $D_t^{n+1} d_{m}^\ell$ and write:
\begin{equation}
 (\pave_j e_k )D_t^{n+1} d^k_{m} = D_t^{n+1} \pave_j \varphim
 - {\sum}_{s = 1}^{n+1} (D_t^s A_{\m j}^a)D_t^{n+1-s}\pa_a \varphim
 \equiv D_t^{n+1} \pave_j \varphim - R^n_j,
 \label{Rjterm}
\end{equation}
and this leads to:
\begin{equation}
 B_{k}^n D_t^{n+1}\! d_{m}^k\!
 = \frac{1}{2} \frac{d}{dt}
 \int_\Omega \!\!\eprime \delta^{ij} (D_t^n \pave_i \varphim) D_t^n \pave_j \varphim
 \,dy
 - \int_\Omega\! (D_t \eprimee)\delta^{ij} (D_t^n \pave_i \varphim)
 D_t^n \pave_j \varphim\,dy
 - \int_\Omega \!\!\eprimee \delta^{ij} (D_t^n \pave_i\varphim )R_j^n\, dy,
\end{equation}
\begin{multline}
 C_{k}^n D_t^{n+1} d_{m}^k\!
 = \frac{dC_n}{dt}
 -\int_\Omega(D_t^{n-1} \pave_i \varphim)\big(
 (D_t^2 \eprime) (D_t^{n}\pave_j \varphim)
 - D_t \big(\eprimee (\pave_j \ssm V^\ell)\big)
 (D_t^{n} \pave_\ell \varphim)\big)\, dy\\
 -\int_\Omega \eprimee \delta^{ij} (D_t^n \pave_i \varphim)\big(R_j^n
 -  (\pave_j \ssm V^{j'}) R_{j'}^n\big)\, dy,
 {}
\end{multline}
where
\begin{equation}
 C_{n} =
 \int_\Omega (D_t \eprimee) \delta^{ij} (D_t^{n-1} \pave_i \varphim)
 (D_t^{n} \pave_j \varphim)
 - \delta^{ij}\eprime(\pave_j \ssm V^\ell)(D_t^{n-1} \pave_i \varphim)
 (D_t^{n}\pave_\ell\varphim)\,dy.
 {}
\end{equation}

Using Lemma \ref{existenceannoying} to control $R^n, F^n$
and arguing as in
the proof of Lemma \ref{weakdt}, we get:
\begin{align}
 \frac{d}{dt} \Big(Y^n_{m}(t)
 -C_n[\varphim]\Big) \leq C(M, ||\xve||_r, ||D_t \xve||_r, ||D_t^2 \xve||_r,
 ||\eprimee||_r)
 \Big( Y^n_{m} + Y_{m}^{n-1} +
 ||F^n_{m}||_{L^2(\Omega)}\Big),
 {}
\end{align}
and so applying the inductive assumption \eqref{llwpind}
and arguing as in the previous lemma, we get the result.
\end{proof}

\begin{lemma}
  \label{existenceannoying}
  Fix $r \geq 7$.
  Let $F^n_{m}$ be $F^n$ (defined in \eqref{Fkterm}) with
  $\varphi$ replaced by $\varphim$ and $R^{n}$ be as
  in
  \eqref{Rjterm}. There are constants $C_r$ depending on
  $
   M, ||\xve||_r, ||D_t \xve||_r, ||D_t^2 \xve||_r,
   ||\eprimee||_r,
  $
  so that if $k \leq r$,
  then:
  \begin{align}
   ||F^n_m||_{L^2(\Omega)}
   +
   ||R_{}^{n}||_{L^2(\Omega)}
   &\leq
   C_r\big(||\varphim||_{n} + ||\pave \varphim||_{n}
   + ||D_t^n F||_{L^2(\Omega)}\big).
   \label{annoyingest}
  \end{align}
\end{lemma}
We remark that unlike the estimates in Section \ref{waveests}, these estimates
depend on $||D_t^2 \xve||_{r}$. This is because the estimates in that section
are all in terms of $\pave \varphi$ i.e. we estimate
$||D_t^k \pave \varphi||_{L^2(\Omega)}$, but in the above proof we are forced to
consider what amounts to $||\pave D_t^k \varphim||_{L^2(\Omega)}$. The error
term this generates can be dealt with since in the application we have in
mind, $D_t^2 \xve = D_t \ssm V$ behaves like $\pave \varphi$.
\begin{proof}
  First, we control the first two terms in $F^n$ with $\varphi$ replaced
  by $\varphim$. When $s \leq r-2$, we
  have:
  \begin{equation}
   ||D_t^s \eprime||_{L^\infty(\Omega)} ||D_t^{n-s}\Dve \varphim||_{L^2(\Omega)}
   \leq ||\eprimee||_r ||D_t^{n-s} \Dve \varphim||_{L^2(\Omega)}.
   {}
  \end{equation}
  To control this second term, we use the commutator estimate \eqref{vectcomm}:
  \begin{equation}
   ||D_t^{n-s} \Dve \varphim||_{L^2(\Omega)}
   \leq C(M, ||\xve||_r) ||D_t^{n-s} \pave \varphim||_{H^1(\Omega)}.
   {}
  \end{equation}
  Since $s \geq 2$, we have $||D_t^{n-s} \pave \varphim||_{H^1(\Omega)}
  \leq ||\pave \varphim||_{n-1}$. If instead $s = r-1, r$, the result is bounded
  by:
  \begin{equation}
   ||D_t^s\eprimee||_{L^2(\Omega)} ||D_t^{n-s} \Dve \varphim||_{L^\infty(\Omega)}
   \leq C ||\eprimee||_r ||D_t^{n-s} \Dve \varphim||_{H^2(\Omega)},
   {}
  \end{equation}
  and so again applying the commutator estimate, this term is bounded
  by the right-hand side of  \eqref{annoyingest} provided
  $||\pave \varphim||_{n - s + 3} \leq
  ||\pave \varphim||_{n}$, and this follows since $r \geq n \geq s$,
  $s = r-1, r$ and $r \geq 7$.

  We now control the remaining terms from the definition of
  $F^n\!\!$. The last two terms are clearly bounded by the
  right-hand side of \eqref{annoyingest} so we just bound
  the terms in the sum. When $s \!\leq\! r\!-\!3$, we bound the terms by:
  \begin{equation}
   ||D_t^s A_{\m j}^b||_{L^\infty} ||D_t^{n-s} \pa_b \pave_i\varphim||_{L^2(\Omega)}
   \leq ||A_{\m j}^b||_{s+2} ||D_t^{n-s} \pave \varphim||_{H^1(\Omega)}
   \leq C(M)||\xve||_r ||\pave \varphim||_{n-s+1},
   {}
  \end{equation}
  and since $s \geq 2$, we have $n-s+1 \leq n-1$ as required.

  We now consider the remaining cases $r\!-\!2\!\leq \!s\!\leq\!r$.
  In these cases we instead bound the summands by:
  \begin{equation}
   ||D_t^s A_{\m j}^b||_{L^2(\Omega)}
   ||D_t^{n-s} \pa_b \pave_i \varphim||_{L^\infty(\Omega)}
   \leq C(M, ||\xve||_r)||D_t \xve||_r ||D_t^{n-s} \pave \varphim||_{H^3(\Omega)},
   {}
  \end{equation}
  and since in this case $n -s + 3 \leq n-1$ (because
  $r-2 \leq s \leq n$ and $r \geq 7$), this second factor is
  bounded by the right-hand side of \eqref{annoyingest} as well,
  and this completes the proof of the bounds for $F^n_{m}$.

  We now control $R^{n}_{}$. This follows in the same way as the bounds
  we have just proved but note that we also need to consider the
  case $s = r+1$. This is the reason that
  $||D_t^2 \xve||_r$ enters into the estimates. When $s \leq r-3$ we argue as above and the result is
  that:
  \begin{equation}
   ||D_t^s A_{\m j}^b||_{L^\infty(\Omega)}
   ||D_t^{n+1-s} \pa \varphim||_{L^2(\Omega)}
   \leq C ||A_{\m j}^b||_{s+2} ||\pa_b \varphim||_{n+1-s}
   \leq C(M, ||\xve||_r) ||\pa_b \varphim||_{n+1-s}.
   {}
  \end{equation}
  The remaining cases are $s = r-2,r-1,r,r+1$ and for these we bound the
  result by:
  \begin{equation}
   ||D_t^s A_{\m j}^b||_{L^2(\Omega)}
    ||D_t^{n+1-s} \pa_b \varphim||_{L^\infty(\Omega)}
    \leq C(M) ||D_t^2 \xve||_r
    ||\pa_y \varphim||_{n},
   {}
  \end{equation}
  where in the last step we used that $n+1-s \leq n$ when
  $s \geq r-2$ for $r \geq 7$. We now need to re-write
  $\pa_b \varphim = A_{\m b}^j \pave_j \varphim$ and we note that
  by similar arguments to the above we have:
  \begin{equation}
   ||A_{\m b}^j \pave_j \varphim||_n \leq C(M, ||\xve||_r)
   ||D_t \xve||_r ||\pave \varphim||_n.
   {}\tag*{\qedhere}
  \end{equation}
 \end{proof}

\subsection{Existence for a nonlinear wave equation}
\label{wavelwp}
We assume that
\eqref{eq:closetoincompressible2} hold and
that $e\!:\! (0,\infty)\! \to\! \R$ is a  function satisfying
\eqref{eq:closetoincompressible2}.
In this section we prove that the nonlinear wave equation:
\begin{align}
 e'(\varphi) D_t^2 \varphi - \Dve \varphi &= \F
\textrm{ in } [0,T]\times\Omega,\qquad\text{with}\qquad
 \varphi = 0 \textrm{ on } [0,T] \times \pa\Omega,\label{nl1}\\
 \varphi(0,y) &= \varphi_0(y) ,\qquad D_t \varphi(0,y) = \varphi_1(y)\textrm{ on } \Omega,
 \label{nl3}
\end{align}
has a unique strong solution $\varphi$ satisfying \eqref{llwpreg}.
We will construct a solution  so that for some $L=L[\varphi]<\infty$:
\begin{equation}
  {\tsum}_{k + |J| \leq 3} |D_t^k \pa_y^J \pave \varphi|
  +|D_t^k \varphi| \leq L, \quad \text{ in } [0,T]\times \Omega.
 \label{appendixLbd}
\end{equation}
We assume that $\F \!= \!\F_1\! + \F_2$ where
$\F_1\! = \F_1(t,y)$ is a function and $\F_2 \!= \F_2[\varphi,D_t\varphi](t,y)$
is a functional so that there are continuous functions $N_{\!s}\!=\!N_{\!s}(L[\varphi],\!||\varphi||_{s-\!1},||\varphi||_{s,0})$,
$ N_{\!s}'\!=\!N_{\!s}'(L[\varphi],\!||\varphi||_{s-\!1}\!)$
so that:
\begin{equation}
 ||D_t^s\! F_2[\varphi,D_t\varphi]||_{L^2(\Omega)}\!
 \leq N_{s}(||D_t^{s+1} \! \varphi||_{L^2(\Omega)}\!
 + ||\varphi||_{L^2(\Omega)}),
 \qquad
 ||\F_2[\varphi,D_t\varphi]||_{s\!-\!1}\!
 \leq N_s'||\varphi||_{s}.
  \label{F2estapp}
\end{equation}
 We will additionally assume that if $\varphi,\psi$ satisfy \eqref{appendixLbd}
 there are continuous functions $\overline{N}_s, \overline{N}_s^{\prime}$
 depending on $L[\varphi],L[\psi], ||\varphi||_{s}, ||\psi||_{s}$ with $\overline{N}_s$
 depending also on $||\varphi||_{s+1,0}, ||\psi||_{s+1,0}$ so that with $\dot{\varphi}\!=\!D_t \varphi$, $\dot{\psi}\!=\!D_t \psi$:
 \begin{equation}
  ||D_t^s \F_2[\varphi,\dot{\varphi}] - D_t^s \F_2[\psi,\dot{\psi}]||_{L^2(\Omega)}
  \leq \overline{N}_s||\varphi-\psi||_{s+1,0},
  \qquad\quad
  ||\F_2[\varphi,\dot{\varphi}] - \F_2[\psi,\dot{\psi}]||_{s-1}
  \leq \overline{N}_s^{\prime} ||\varphi-\psi||_{s}.
  \label{F2diffsapp}
\end{equation}
 In Section \ref{enthsec} we  take $F_2\! = e''(\varphi) (D_t\varphi)^2\!
 + \rho[\varphi]$, which satisfies these estimates.
 The energies we use are:
\begin{equation}
 Y_s(t) = \Big(\frac{1}{2}{\sum}_{ k \leq s}
 \int_\Omega e'(\varphi) |D_t^{k+1 }\varphi|^2
 + \delta^{ij}(D_t^k \pave_i  \varphi)(D_t^k\pave_j  \varphi)
 \, \kve dy\Big)^{1/2}.
\end{equation}

The initial data $\varphi_0, \varphi_1$
satisfy the compatibility conditions to order $s$ for the problem \eqref{nl1}-\eqref{nl3}
if there is a formal power series solution $\widehat{\varphi} = \sum t^k \varphi_k$
to \eqref{nl1} which additionally satisfies
\begin{equation}
 \varphi_k \in H^1_0(\Omega), \qquad
 k = 0,..., s
 \label{nlincc'}
\end{equation}

\begin{theorem}
 \label{nllwpthm}
 Fix $r \!\geq\! 7$ and suppose that
 $V\! \!\in\! \X^{r+{}_{\!}1}(T_1)$ for some $T_1 \!>\!0$ satisfies
 \eqref{admissible} and that the bound \eqref{ubd2} holds.
 Take $K$ so that
 \begin{equation}
   {\sup}_{\,0 \leq t \leq T_1\,}\big( ||\xve(t)||_{r} + ||V(t)||_{r}
   + ||D_t V(t)||_{r} +
  ||D_t \F_1(t)||_{r-1} + ||\F_1(t)||_{r-1}\big) \leq K.
 \end{equation}
 Suppose that
 \eqref{llwpxassump}-\eqref{llwpfassump}, \eqref{eq:closetoincompressible2}
 and the compatibility conditions \eqref{nlincc'} hold for
 some $s \leq r$.
 Let $L_0$ satisfy:
 \begin{equation}
  {\tsum}_{k + |J| \leq 3}
  ||\pa_y^L \pave \varphi_k||_{L^\infty(\Omega)}
  + ||\varphi_k||_{L^\infty(\Omega)} \leq L_0.
  {}
 \end{equation}
 There is a continuous function $G_r'$ so that if $T$ satisfies:
 \begin{equation}
  T G_r'(M, L_0, L_0^{-1}, Y_r(0),  K, T_1) \leq 1,\quad \text{and} \quad
  T \leq T_1,
  {}
 \end{equation}
  the problem
   \eqref{nl1}-\eqref{nl3} has a unique solution
 $\varphi$ satisfying:
 \begin{equation}
   D_t^s \varphi \in L^\infty([0,T]; L^2(\Omega)),\qquad
  D_t^{s+1-\ell} \pave \varphi \in L^\infty([0,T]; H^{\ell-1}(\Omega)),
 \qquad  \ell = 0,..., s+1,
  {}
 \end{equation}
 and there are constants $\C_s$ depending on
$M, L_0,\,Y_{s}(0), K$, and $T$
  so that the following estimates hold:
 \begin{equation}
  Y_s(t) \leq \C_s \Big( Y_s(0) + \int_0^t ||\F_1(\tau)||_{s,0}
  + ||\F_1(\tau)||_{s-1}
  \, d\tau\Big),\qquad
  ||\pave \varphi(t)||_{s} \leq  \C_s \big( Y_s(t)
  + ||F_1(t)||_{s-1}\big),
  \label{nlwaveenest}
 \end{equation}
 for $0 \leq t \leq T$ and
 \begin{equation}
  {\tsum}_{k + \ell \leq 2} |\pa^\ell D_t^k \varphi(t,y)|
  \leq 2 L_0, \qquad \textrm{ in } [0,T] \times \Omega.
  \label{Lbdapp}
 \end{equation}
\end{theorem}
We will construct a solution to \eqref{nl1}-\eqref{nl3} by considering the
sequence $\varphi^\nu$, $\nu = 0,1,...$, defined by:
\begin{align}
  \varphi^0 &= {\sum}_{k = 0}^s \varphi_k {t^k}\!/{k!} ,\\
  D_t^2 \varphi^\nu
  -{e'(\varphi^{\nu-1})^{-1}} \Dve \varphi^\nu &= {e'(\varphi^{\nu-1})^{-1}}
  \F^{\nu-1},
 \textrm{ in } [0,T] \times \Omega,\quad\text{with}\quad
  \varphi^\nu = 0\textrm{ on } [0,T] \times \pa \Omega,
  \label{nlina1}\\
  \varphi^\nu(0,y) &= \varphi_0(0, y),\qquad
  D_t\varphi^\nu(0,y) = \varphi_1(0,y),
 \textrm{ on } \Omega,
 \label{nlina3}
\end{align}
with $\F^{\nu-1} = \F_1 + \F_2[\varphi^{\nu-1}]$
and where $\widehat{\varphi} = \sum t^k \varphi_k$ is a given formal power series
solution to \eqref{nl1}.
Note that with this choice of $\varphi^0$, we have that
$D_t^j \varphi^0|_{t = 0} = \varphi_j^\ve$, $j \leq s$.
This system also has compatibility conditions which must
be satisfied to construct a sufficiently regular solution.
Given $\varphi^{\nu-1}$,
let $\widehat{\varphi^\nu} = \sum t^k \varphi^\nu_k$ be a formal power series solution
to \eqref{nlina1}. Taking time derivatives of \eqref{nlina1} and
restricting to $t = 0$, we see that the coefficients $\varphi^\nu_k$ must satisfy:
\begin{equation}
 \varphi^\nu_k = \big( {e'(\varphi^{\nu-1})^{-1}}
 \big(D_t^{k-2} \Dve \widehat{\varphi^\nu} + D_t^{k-2}\F^{\nu-1}
 + \G_k[\widehat{\varphi^\nu}, \varphi^{\nu-1}]\big)\big|_{t = 0},
 \label{nlcompatdef}
\end{equation}
where we are writing:
\begin{equation}
 \G_k[\widehat{\varphi^\nu}, \varphi^{\nu-1}] =
 D_t^{k-2} \big( e'(\varphi^{\nu-1}) D_t^2 \widehat{\varphi^\nu}
 \big) - e'(\varphi^{\nu-1}) D_t^{k}\widehat{\varphi^\nu}.
 {}
\end{equation}
The compatibility conditions for the system \eqref{nlina1}-\eqref{nlina3}
are then the requirement that:
\begin{align}
 \varphi^\nu_k \in H^1_0(\Omega), k = 0,...,s
 \label{nlincc}
\end{align}
Since $\varphi^\nu_0 = \varphi_0,
\varphi^\nu_1 = \varphi_1$ and  both of
these sequences are defined recursively, from
\eqref{varphicc} and \eqref{nlcompatdef} it follows that
$\varphi^\nu_k = \varphi_k$ for all $\nu \geq 0$ and so
the compatibility conditions for the approximate problem
\eqref{nlina1}-\eqref{nlina3} are satisfied so long
as the compatibility conditions \eqref{nlincc'} for the nonlinear problem
\eqref{nl1}-\eqref{nl3} hold.

We now argue by induction to show that the above problem
has a unique solution with bounds that hold uniformly
in $\nu$. Let $X^r_T$ be closure of $C^\infty([0,T]; C^\infty(\Omega))$
with respect to the norm:
\begin{equation}
 ||\varphi||_{X^r_T} = {\sup}_{\,0 \leq t \leq T}
 {\sum}_{s =0}^{r} ||D_t^{s+1} \varphi(t)||_{L^2(\Omega)}
 + ||D_t^{s}\pave \varphi(t)||_{H^{r-s}(\Omega)}.
 {}
\end{equation}
Assume that for $\nu \geq 1$ we have a solution
$\varphi^{\nu-1} \in X^r_T$ which moreover satisfies
 \eqref{Lbdapp}.
 Writing:
\begin{equation}
 Y^{\nu-1}_s(t)
 = {\sum}_{k \leq s} \Big(\frac{1}{2} \int_\Omega
 e'(\varphi^{\nu-2}) |D_t^{k+1} \varphi^{\nu-1}(t)|^2
 + |D_t^s \pave \varphi^{\nu-1}|^2\, \kve dy\Big)^{1/2},
 {}
\end{equation}
by Proposition \ref{linearwavelwp},
we have the estimate:
\begin{equation}
 Y^{\nu-1}_s(t) \leq \C_s \Big(Y_s(0)
 + \int_0^t ||D_t^{s-1} \F_1(\tau)||_{L^2(\Omega)}\, d\tau\Big),
 \qquad
 ||\pave \varphi^{\nu-1} ||_s
 \leq \C_s \big( Y_s + ||\F_1||_{s-1}\big),
 \quad s = 0 ,..., r.
 \label{iterateen}
\end{equation}
where here $\C_s$ depends on $M, Y_{s-1}(0)$, $K$ and
$\sup_{0 \leq t \leq T} ||e'(\varphi^{\nu-1}(t))||_{r}$.
Note that we are using that $Y_s^{\nu-1}(0) = Y_s(0)$
in \eqref{iterateen}.
By these estimates, \eqref{F2estapp}
and Lemma \ref{gest2} to control $e'(\varphi^{\nu-1})$, we have:
\begin{equation}
 ||\F^{\nu-1}||_{s, 0}
 + ||\F^{\nu-1}||_{s-1} + ||\eprime(\varphi^{\nu-1})||_{s,0}
 + ||\eprime(\varphi^{\nu-1})||_{r}
 \leq C_s(M, L_0, Y_r(0), K).
\end{equation}
By
\eqref{linearwavelwp},
there is a unique $\varphi^\nu\! \in X^r_T$ satisfying
\eqref{nlina1}-\eqref{nlina3}
and so that \eqref{llwpreg} holds. By the above estimates
and the inductive assumption we also have:
\begin{equation}
 Y_s^\nu(t) \leq \C_s \Big( Y_s(0)
 + \int_0^T ||\F_1(\tau)||_{s,0} +
 ||\F_1(\tau)||_{s-1} \, d\tau\Big),
 \label{ysnu}
\end{equation}
where $\C_{{}_{\!}s} \!=_{\!}\C_{{}_{\!}s}(\!M_{\!},_{\!} L_0, \!Y_{\!r}(_{\!}0_{\!}),_{\!} K_{\!})$ and we again are using that $Y_{\!s}(0)$ is independent of $\nu_{\!}$.
We note that by Sobolev
embedding, the estimate \eqref{ysnu} and the estimate \eqref{dtpsibd},
just as in the proof of Corollary \ref{bootstrapped}, we have that:
\begin{equation}
 L^{\!\nu}\!(t) \equiv \!{\tsum}_{k +_{\!} |_{\!}J| \leq 3} |\pa_y^J\! D_t^k \pave \varphi^\nu\!(t,_{\!}y)|
 + |D_{\!t\,}^k \!\varphi^\nu\!(t,_{\!}y)|
 \leq L_0 +  T P_0^\nu,\qquad\text{where}\quad P_0^\nu \!\equiv P_0^{\nu\!}(M_{\!}, {\sup}_{0 \leq t \leq T} L^{\!\nu\!}(t),\! Y_5^{\!\nu\,}\!(_{\!}0_{\!}),_{\!}K),
\end{equation}
and so a continuity argument
(see the proof of Corollary \ref{bootstrapped}) gives that $\sup_{0 \leq t \leq T}
L^\nu(t) \leq 2L_0$
provided that $T (2L_0)^{-1}P_0^\nu(M, 2L_0, Y_5^\nu(0), K) \leq 1$. Note
that in fact $P_0$ is independent of $\nu$ since $Y_5^\nu(0)$ is.

The sequence $\varphi^{\nu}$ is therefore uniformly bounded
in $X_{T_0}^r$ for a fixed  $T_0\! > \! 0$,
and therefore there is a $\varphi\!\in\!  X_{T_0}^r$ so that $\varphi^\nu\! \!\!  \to \! \varphi$
weakly. We now show that there is $\!T^*\!\! =\! T^*(M, L_0, Y_{\!r}(0), K) \!\leq T_0$
so that if $T_1\!\! \leq\! T^*\!$ then
\begin{equation}
 ||\varphi^{\nu} - \varphi^{\nu-1}||_{X^0_{T_1}}
 \leq 2^{-1} ||\varphi^{\nu-1} - \varphi^{\nu-2}||_{X^{0}_{T_1}}.
 \label{cauchy}
\end{equation}
Assuming that this holds for the moment, it follows that the
sequence $\varphi^\nu$ is a Cauchy sequence in $X_{T_1}^0$ and so converges
strongly to
some $\tilde{\varphi} \in X_{T_1}^{0}$. This limit has to coincide
with the $\varphi$ above and in particular this shows that the
$\varphi^\nu$ converges strongly to $\varphi$, and so $\varphi$ satisfies
the nonlinear equation \eqref{nl1}.

To prove \eqref{cauchy}, we  take $T^* \!\leq \! T_0$ and
set $\psi = \varphi^\nu \!- \varphi^{\nu-1}\!$ and note that with $\F^{\nu,\nu-1}\! = \F_2^\nu\! - \F_2^{\nu-1}\!$ we have:
\begin{equation}
 e'(\varphi^\nu) D_t^2 \psi - \Dve \psi
 = \F^{\nu,\nu-1} + (e'(\varphi^\nu) - e'(\varphi^{\nu-1})) D_t^2 \varphi^{\nu-1},
 \quad\text{with}\quad \psi|_{[0,T]\times\pa\Omega} = 0, \,
 \quad \psi|_{t = 0} = D_t\psi|_{t = 0} = 0.
 {}
\end{equation}
By the estimates
\eqref{F2diffsapp},
the estimate \eqref{appwaveen}, and the product estimate \eqref{product},
we have that:
\begin{equation}
 Y_0^{\nu,\nu-1}(t)\equiv \Big(\frac{1}{2}\int_\Omega e'(\varphi^\nu) |D_t
 \psi|^2 + | \pave \psi|^2\, \kve dy \Big)^{1/2} \leq \C_0
 \int_0^T ||\varphi^{\nu-1} - \varphi^{\nu-2}||_{1}\, dt
 \leq \C_0 T ||\varphi^{\nu-1} -\varphi^{\nu-2}||_{X_T^0},
 {}
\end{equation}
where $\C_{{}_{\!}s} \!=_{\!}\C_{{}_{\!}s}(\!M_{\!},_{\!} L_0, \!Y_{\!r}(_{\!}0_{\!}),_{\!} K_{\!})$.\!
Since $||\varphi^{\nu \!}\!\!- \!\varphi^{\nu\!-\!1\!}||_{X^0_{T}} \! \!\lesssim
_{\!}\sup_{\,0 \leq t \leq T\!} Y_0^{\nu_{\!},\nu-\!1\!}\!\!$,
 taking $T$ sufficiently small gives
\eqref{cauchy}.

\subsection{Proof of estimates for the wave equation}\label{sec:F3}
\begin{lemma}
  \label{dtWF3}
  For each $s \geq 0$,
  there is a continuous function $G'_s(t) \!=\! G'_s( M,
  ||\xve(t)||_s, ||V(t)||_{\X^{s}},W_{s-1}(t))$ and a polynomial $P$ so that
  if \eqref{phiwave}-\eqref{Lassumpwave} hold, then:
  \begin{equation}
   \frac{d}{dt} W_s\leq G'_s \Big(
   W_s + ||F_1||_{s,0} + ||F_1||_{s-1} + ||V||_{\X^{s+1}}  +
    P(L, W_{s-1}, ||F||_{s-2})W_{s}\Big).
   \label{dtWestF}
  \end{equation}
\end{lemma}

\begin{proof}
  We start by showing:
  \begin{equation}
   \frac{d}{dt} W_s^2 \leq G_s'' \Big( W_s + ||F||_{s,0}
   + ||\pave \varphi||_{s} + ||V||_{\X^{s+1}}
   + P(L,||\varphi||_s) W_s\Big) W_s,
   \label{showsthat}
  \end{equation}
  for a continuous function $G_s'' = G''_s(M, ||\xve||_s, ||V||_{\X^{s}})$.
 We have:
 \begin{multline}
  \frac{d}{dt} W_{\!\!s}^2 \!=\!\!
  ={\sum}_{k \leq s}
  \int_\Omega \eprime (D_t^{k+2}\varphi)(D_t^{k+1} \varphi)
  + \delta^{ij} (D_t^k \pave_i \varphi)\pave_j ( D_t^{k+1}\varphi)
  \, \kve dy\\
  +\!{\sum}_{k \leq s}\!\Big( \int_\Omega \delta^{ij} (D_t^k \pave_i \varphi) [\pave_j, D_t^{k+1}] \varphi\, \kve dy
  + \frac{1}{2}\int_\Omega (D_t \eprime) (D_t^{k+1} \varphi)^2
  + (D_t\log \kve)\big( (D_t^{k+1}\varphi)^2
  +|D_t^k \pave \varphi|^2)\, \kve dy \Big).
 \end{multline}
 The last line is bounded by $C(M)(1 + L) (W_s)^2$.
 Integrating by parts, the terms on the first line are:
 \begin{equation}
  \int_\Omega \big(\eprime D_t^{k+2} \varphi
  - \delta^{ij} (\pave_j D_t^k \pave_i \varphi)\big)(D_t^{k+1} \varphi)
  = \int_\Omega \big( \eprime D_t^{k+2} \varphi - D_t^k \Dve \varphi\big)(D_t^{k+1} \varphi)
  + \int_\Omega \delta^{ij}([D_t^k, \pave_j]\pave_i \varphi)( D_t^{k+1} \varphi)\,
  \kve dy.
 \end{equation}
 By Lemma \ref{gestlem}, we have:
 \begin{equation}
  ||\eprime D_t^{k}(D_t^2\varphi) -D_t^k (\eprime D_t^2 \varphi)||_{L^2(\Omega)} \leq
  P(L, ||\varphi||_{k-1})||\varphi||_{k},
 \end{equation}
 and by the commutator estimate \eqref{vectcommdt}:
 \begin{alignat}{2}
  ||[D_t^{k+1}, \pave_j] \varphi||_{L^2(\Omega)} &\leq
  C_k(M, ||\xve||_k, ||V||_{\X^{k}})
  \big( ||\pave f||_{k, 0} + (||V||_{\X^{k+1}} + 1) ||\pave f||_{k-1}\big),
  \\
  ||[D_t^k, \pave_j] \pave_i \varphi||_{L^2(\Omega)} &\leq
  C_k(M, ||\xve||_k, ||V||_{\X^k})
  (||\pave^2 f||_{k-1, 0} + ||\pave^2 f||_{k-2}).
\end{alignat}
 By \eqref{product}, $||\pave^{2\!\!} f||_{k-1\!,0}\! \leq \! C(M,\! ||V||_k)(1\! +\!
 ||V\!||_{\X^{k\!+\!1}} )||\pave f||_{k}$
 and since $D_t^k(\eprime D_t^{2} \varphi) - D_t^k \!\Dve \varphi \!=\! D_t^k \! F\!$,
 using \eqref{F2est} to control $D_t^kF$, we have \eqref{showsthat}.
 To prove \eqref{dtWestF} from \eqref{showsthat},
 we want to re-write $||\varphi||_s$ in terms of $||\varphi||_{s,0}$
 and $||\pave \varphi||_{s-1}$, and for this we re-write
 $\pa_a \varphi \!=\! A_{\m a}^i\pave_i \varphi$ and use \eqref{product}
 and Lemma \ref{inverselemma} to get:
 $
  || A_{\m a}^i\pave_i \varphi||_{s-1} \!\leq \! C(M)||\xve||_{s} ||\pave \varphi||_{s-1}.
 $
 This implies that
 $
  ||\varphi||_s \!\leq \! C_s \big(||\varphi||_{s,0} \!+ \!||\pave \varphi||_{s-1}\!\big),
 $
 and so inserting this into \eqref{showsthat}, applying \eqref{alldt1}
 and bounding $||V||_s \!\leq\! ||V||_{\X^{s+1}}$ and
 using that ${d}W_{\!s}^2\!/{dt} \!= \! 2 W_{\!s} {d}W_{\!s}/{dt} $
 gives \eqref{dtWestF}.
\end{proof}

\begin{lemma}
 There is a continuous function
  $G_s'' \!=\! G_s''(M,{}_{\!} ||\xve||_s)$ and $P_{\!s}$
  so that if
  \eqref{phiwave}-\eqref{Lassumpwave} hold, then:
\begin{equation}
 ||\pave \varphi||_s \leq G_s''(||\T \xve||_{H^s} + ||V||_{s})
  \big( ||\varphi||_{s+1,0} +
 ||\pave \varphi||_{s,0}
 + ||F||_{s-1}
 + P_s(L,||\varphi||_{s,0},
 ||\pave \varphi||_{s-1,0}, ||F||_{s-2})\big).
 \label{alldt1F}
\end{equation}
\end{lemma}
\begin{proof}
For $s = 0$ there is nothing to prove and so we assume that
\eqref{alldt1F} holds for $s = 0, 1,..., n-1$. To prove that it holds
for $s = n$, we will show that if $k + \ell = n$ then:
\begin{equation}
 ||D_t^k \pave \varphi||_{H^{\ell}}
 \leq G''_{\!n} \big( ||\varphi||_{n+1,0} + ||\pave\varphi||_{n,0}
 + ||F||_{n-1} +  P(L, ||\varphi||_{n,0}, ||\pave \varphi||_{n-1,0},
 ||F||_{n-2})\big),
 \label{alldt1step}
\end{equation}
with $G_n'' = G_n''(M, ||\xve||_n)$.
There is nothing to prove if $\ell = 0$ and so we assume that this estimate holds
for $\ell = 0, ..., \ell'-1$. To prove that it holds for $\ell = \ell'$,
 we use the estimate \eqref{sobell} when $ \ell' = n$:
\begin{equation}
 ||\pave \varphi||_{{}_{\!}H^{n{}_{\!}}}
 \!\leq \!C'_{\!n} {}_{\!}\big( ||\Delta \varphi||_{{}_{\!}H^{n\!-\!1}}\!
 + (||\T{}_{\!} \xve||_{{}_{\!}H^n}\! + ||\xve||_{{}_{\!}H^n}\!)||\varphi||_{{}_{\!}L^2}\!\big)
 \!\leq\!
 C'_{\!n} {}_{\!}\big( ||\eprime \!D_t^2 \varphi||_{{}_{\!}H^{\ell'\!\!-\!1}}\!
 + ||F||_{{}_{\!}H^{\ell'\!\!-\!1}}\!
 + (||\T {}_{\!} \xve||_{{}_{\!}H^n}\! + ||\xve||_{{}_{\!}H^n}\!)||\varphi||_{{}_{\!}L^2}\!\big),
\end{equation}
and the estimate \eqref{sobellmix} when $n-\ell' \geq 1$:
\begin{multline}
 ||D_t^{n-\ell'}\! \pave \varphi||_{H^{\ell'{}_{\!}}\!(\Omega)}
 \!\leq C'_n \big( ||\Delta \varphi||_{n-\ell'\!, \ell'\!-1}\!
 + (||D_t \xve||_n + ||\xve||_{n})||D_t^{n-\ell'}\!\! \varphi||_{L^2(\Omega)}\big)\\
 \leq
 C'_n \big( ||\eprime D_t^2 \varphi||_{n-\ell'\!, \ell'\!-1}\!
 + ||F||_{n-\ell'\!, \ell'\!-1}
 + (||D_t \xve||_n + ||\xve||_{n})||D_t^{n-\ell'}\!\! \varphi||_{L^2(\Omega)}\big),
\end{multline}

Using \eqref{nlinest},
the first term here is bounded by $C ||\varphi||_{n-\ell'+2, \ell'-1}
+  P(L, ||\varphi||_{n-1})$ and this second term can be bounded by
the right-hand side of \eqref{alldt1step} by the inductive assumption.
If $\ell' = 1$ then we have just proven
\eqref{alldt1step}. If $\ell' \geq 2$ we write $\pa\varphi/\pa {y^a} = A_{\m a}^i \pave_i \varphi$
and use the product estimate \eqref{product} and Lemma \ref{inverselemma}:
\begin{equation}
  ||\varphi||_{n-\ell'+2, \ell'-1} \leq
 ||\pa_y \varphi||_{n-\ell'+2, \ell'-2}
 + ||\varphi||_{n-\ell'+2, \ell'-2}
 \leq C(M, ||\xve||_n) (||\pave \varphi||_{n-\ell'+2, \ell'-2}
 + ||\varphi||_{n-1}),
\end{equation}
and noting that $||\xve||_n \leq C (||\xve||_{H^n} + ||V||_{n-1})$, this implies \eqref{alldt1step}.
\end{proof}

\begin{proof}[Proof of Lemma \ref{dtphidiffnew}]
We will show that:
\begin{equation}
 \frac{d}{dt\!\!}\, W_s^{{}_{\!}I\!,\II}
 \!\!\leq\! D_{\!{}_{\!}s}' \!\bigtwo(\! ||F_{\!I} \!-\! F_{\II\!}||_{s,0} \!+_{\!} ||F_{\!I} \!- \!F_{\II\!}||_{s-\!1}\!
 +_{\!} W_s^{{}_{\!}I\!,\II} \!\!
 +_{\!} ||\Vu \!-\! \Vw\!||_{\X^{{}_{\!}s+_{\!}1\!}} \big(||\varphi_{\II\!}||_{s+_{\!}2,0}\!
 +_{\!} ||\pavew \varphi_{\II\!}||_{s, 1}\!
 +_{\!} ||\varphi_{\II\!}||_{s+_{\!}1_{\!},0} \!+_{\!} ||\pavew \varphi_{\II\!}||_{s\!} \big)\! \bigtwo)_{\!},
 \label{y12goal}
\end{equation}
where $D_{\!s}'$ depends on $M, L, W_s^{I_{\!},\II\!}(t)$ and
$||\pave_J \varphi_{J\!}||_s, ||\varphi_{J\!}||_{s+1,0}$ for $J_{\!} \!=\! I_{\!},\II{}_{\!}$.
Arguing as in the proof of Theorem \ref{mainwavethm} and using
\eqref{dtpsibd} and \eqref{allpsibd}, this implies \eqref{energy esti
y^12}. Using \eqref{phijwave} and
\eqref{dtW}, it just remains to prove that the $L^2$ norms of
$D_t^s (\Dveu \!- \!\Dvew) \varphi_{\II}$ and $D_t^s ( (\eprime_I \!- \eprime_{\II})
D_t^2 \varphi_{\II})$ are bounded by  \eqref{y12goal}.
These terms are the reason that we lose derivatives of $\varphi_{\II}$ relative
to $\psi$ and why the coefficients $D_{\!s}$ will depend on $||\Vu||_{\X^{s+2}},
||\Vw||_{\X^{s+2}}$.

  We start by controlling $D_t^s (\Dveu - \Dvew)\varphi_{\II}$ in $L^2$.
  We write:
  \begin{equation}
    D_t^s \big( \wpa_{I i} \wpa_{I j} - \wpa_{\II i}\wpa_{\II i}\big)
   \varphi_{\II}
  \! =\!\big( \wpa_{I i} D_t^s\wpa_{I j} - \wpa_{\II i}D_t^s\wpa_{\II i}\big)
   \varphi_{\II}
   \!+\! \big( [\wpa_{I i},D_t^s ]\wpa_{I j} - [\wpa_{\II i},D_t^s]\wpa_{\II i}\big)
    \varphi_{\II}.
  \end{equation}
  The first term is bounded in $L^2(\Omega)$ by $C(M)
  ||(\paveu \!- \pavew{}_{\!})\varphi_{II}||_{s,1}$.
  By the product rule
  \eqref{product} this term is bounded by $C(M{}_{\!}, {}_{\!}||\Vu||_{\X^{s+1}},
  ||\Vw||_{\X^{s+1}}{}_{\!})||\Vu \!- \!\Vw||_{\X^{s+1}}||\pavew \varphi_{\II}||_{s+1}.$
  Using the commutator estimate \eqref{vectcomm}, the $L^2$
  norm of the second term  is bounded by the right-hand side of \eqref{y12goal}.

 To control $D_t^s (\eprime_I - \eprime_{\II}) D_t^2 \varphi_{\II}$, we use
  \eqref{product}:
  \begin{equation}
   ||D_t^s ( (\eprime_I - \eprime_{II} )D_t^2 \varphi_{\II})||_{L^2(\Omega)}
   \leq D_s'' ||\varphi_I - \varphi_{\II}||_{s,0} ||\varphi_{\II}||_{s+2,0},
  \end{equation}
  where $D_s''$ depends on $L,$ and $||\varphi_J||_{s-1}$ for $J = I, \II$.

  The estimate \eqref{elliptic y^12} follows in the same way as
  \eqref{allpsibd} using the elliptic estimate
  \eqref{ibpmixedelliptic} in place of \eqref{sobellmix} and:
  \begin{equation}
   \Dveu \varphi_I -\Dvew \varphi_{\II}
   = \eprime_I D_t^2 \psi + (\eprime_I - \eprime_{\II}) D_t^2 \varphi_{\II}
   +\F_I - \F_{\II}.
   {}\tag*{\qedhere}
  \end{equation}
\end{proof}

\section*{Acknowledgements}
The authors would like to thank Jeffrey Marino and Junyan Zhang for reading the first draft of our paper.

\bibliographystyle{abbrv}

\end{document}